\documentclass[oneside,reqno]{amsart}
\usepackage{amsfonts,amsmath,latexsym,verbatim,amscd,mathrsfs,color,array}

\usepackage[colorlinks=true,citecolor=blue]{hyperref}

\usepackage[left=2.8cm, right=2.8cm, top=2.4cm, bottom=2.4cm]{geometry}

\usepackage{amsmath,amssymb,amsthm,amsfonts,graphicx,color}
\usepackage{amssymb}
\usepackage{amsbsy}
\usepackage{epstopdf}

\usepackage[colorinlistoftodos]{todonotes}

\newcommand{\ass}{\quad\mbox{as}\quad}

\newcommand{\bphi}{ \tilde{\phi}  }
\newcommand{\EE}{{\mathcal E}  }

\newcommand{\inn}{{\quad\hbox{in } }}
\newcommand{\onn}{{\quad\hbox{on } }}
\newcommand{\ttt}{\tilde }

\newcommand{\TT}{{\mathcal T}  }

\newcommand{\nn}{ {\nabla}  }

\newcommand{\pp}{ {\partial} }

\newcommand{\vp}{\varphi}

\newcommand{\OO}{{\mathcal O}}

\newcommand{\RR}{{{\mathcal R}}}

\newcommand{\N}{\mathbb{N}}

\newcommand{\R} {\mathbb R}
\newcommand{\Z} {\mathbb Z}
\newcommand{\cuad}{{\sqcap\kern-.68em\sqcup}}

\newcommand{\foral}{\quad\mbox{for all}\quad}
\newcommand{\ve}{\varepsilon}

\newcommand{\be}{\begin{equation}}
	\newcommand{\ee}{\end{equation}}

\newcommand{\la}{\lambda}

\newcommand{\equ}[1]{(\ref{#1})}

\renewcommand{\div}{\mathop{\rm div}}
\newcommand{\curl}{\mathop{\rm curl}}

\newtheorem{lemma}{Lemma}[section]
\newtheorem{prop}{Proposition}[section]
\newtheorem{theorem}{Theorem}
\newtheorem{corollary}{Corollary}[section]
\newtheorem{remark}{Remark}[section]
\newcommand{\bremark}{\begin{remark} \em}
	\newcommand{\eremark}{\end{remark} }

\long\def\hide#1{}

\numberwithin{equation}{section}

\begin{document}
	
	\title{Leapfrogging vortex rings for the 3-dimensional incompressible Euler equations}
	
	\author[J. D\'avila]{Juan D\'avila}
	
	\address{Department of Mathematical Sciences, University of Bath, Bath, Ba2 7AY, UK.}
	
	\email{jddb22@bath.ac.uk}
	
	\author[M. del Pino]{Manuel del Pino}
	
	\address{Department of Mathematical Sciences, University of Bath, Bath, Ba2 7AY, UK.}
	
	\email{mdp59@bath.ac.uk}
	
	\author[M.Musso]{Monica Musso}
	
	\address{Department of Mathematical Sciences, University of Bath, Bath, Ba2 7AY, UK. }
	
	\email{m.musso@bath.ac.uk}
	
	\author[J. Wei]{Juncheng Wei}
	
	\address{Department of Mathematics, University of British Columbia, Vancouver, B.C., Canada, V6T 1Z2. }
	
	\email{jcwei@math.ubc.ca}

	\thanks{}


	\maketitle
	
	\begin{abstract}
		A classical problem in fluid dynamics concerns the interaction of multiple
		vortex rings sharing a common axis of symmetry 
		in an incompressible,
		inviscid $3$-dimensional fluid. Helmholtz
		(1858) observed that a pair of similar thin, coaxial vortex rings may pass through each other repeatedly due to 
		the induced flow of the rings acting on each other. This celebrated configuration, known as {\it leapfrogging}, has not yet been rigorously established.  
		We provide a mathematical justification for this phenomenon by constructing a smooth solution of the 3d Euler equations 
		exhibiting this motion pattern.

	\end{abstract}

	\section{Introduction}

	We consider the 3-dimensional Euler equation for an ideal incompressible  fluid  given by
	\be \label{euler0}
	\left \{ 
	\begin{aligned}
	&\ 	{\bf u}_t  + ({\bf u}\cdot \nn ){\bf u} = -\nn p,   \\
	&\ 	{\rm div}\, {\bf u} = 0 ,\quad  x\in \R^3, \ t\ge 0.   \\
	\end{aligned}\right.
	\ee
	For a solution  ${\bf u } (x,t)$ of  \equ{euler0},
	its vorticity is defined as
	$ \vec\omega = \curl {\bf u}  $. Then $\vec\omega$ 
	solves the Euler system \equ{euler0} in its vorticity form,  	
	\be \label{euler}
	\left \{ 
	\begin{aligned}
	&\ 	{\vec \omega}_t  + ({\bf u}\cdot \nn ){\vec \omega} =
	( \vec \omega \cdot \nn ) {\bf u}, \quad {\bf u} = \curl \vec \psi, \\ 
	&\vec \psi(x,t) = \left( 	-\Delta  \right)^{-1} \vec \omega \, =\,  {1\over 4 \pi} \int_{\R^3} { 1\over |x-y|}  \vec \omega (y,t)\,  dy ,  \quad x\in \R^3, \ t\ge 0  .   \\
	\end{aligned}\right.
	\ee


\medskip
A vortex ring is an 
 axially symmetric solution of \equ{euler} which does not change shape in time, whose
vorticity is mostly concentrated inside a solid torus which moves with constant speed along
the symmetry axis. 
The vortex lines form large circles that fill the torus, whereas fluid particles
spin around the vortex core within perpendicular cross sections  characterized with a  thin torus‐shaped region in which
the vorticity of the fluid is concentrated. 
These objects were first described by Helmholtz 
in his celebrated work \cite{H,H1}. He  considered with great
attention the situation where the vorticity field  is concentrated in a circular vortex-filament
of very small section, a thin vortex ring. 
Helmholtz also analyzed the interaction between two or more similar coaxial vortex rings with thin sections 
and similar translation speeds.
As pointed out by Jerrard and Smets \cite{js}, his description reads:

\smallskip
{\it We can now see generally how two ring-formed vortex-filaments having the same
axis would mutually affect each other, since each, in addition to its proper motion,
has that of its elements of 
fluid as produced by the other. If they have the same
direction of rotation, they travel in the same direction; the foremost widens and
travels more slowly, the pursuer shrinks and travels faster till finally, if their
velocities are not too different, it overtakes the first and penetrates it. Then the
same game goes on in the opposite order, so that the rings pass through each other
alternately.}

		\medskip
		
		The motion pattern described by Helmholtz is often termed
{\em leapfrogging} in 
fluid mechanics.
Leapfrogging vortex rings are solutions of the Euler equations where   several interacting
	vortex rings sharing a common axis of symmetry  move in the same direction along the symmetry axis and 
	pass through each other repeatedly due to the induced flow of the rings
	acting on each other as depicted in Figure \ref{1}.

\begin{figure}
\includegraphics[scale=0.6]{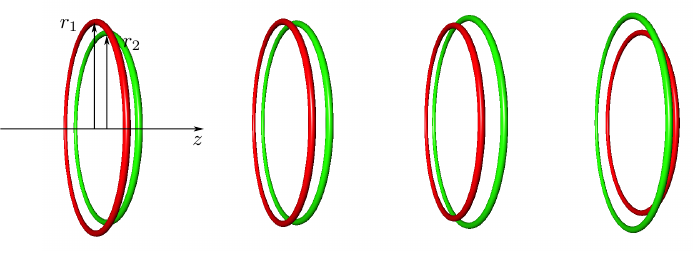}
\caption{Leapfrogging vortex rings.}
\label{1}
\end{figure}

	\medskip
	Even though this phenomenon
	has been widely studied since
Helmholtz, as far as we know it has not yet been mathematically justified.  	In this paper we present what seems to be  the first rigorous construction  of a solution to the Euler equations with 
a leapfrogging motion pattern. 

\medskip
The solution we construct belongs to the axisymmetric, no-swirl class, namely with a velocity field 
 ${\bf u} (x,t)$ which in standard cylindrical coordinates $(r, \theta, z)$ takes the form 
$$
	{\bf u} ( x ,t) = u^r(r, z, t) \, {\bf e}_r  + u^z(r, z,t) \, {\bf e}_z , \quad  x = (r \cos \theta , r \sin \theta , z) \in \R^3, 
	$$
 where
	${\bf e}_r = (\cos \theta , \sin \theta , 0 ),   \quad {\bf e}_\theta = (-\sin \theta, \cos \theta  , 0), \quad {\bf e}_z = (0,0,1),
	$ and 
		\begin{equation}\label{defSigma}
		(r,z) \, \in \, \Sigma =\{ (r,z) \, / \, r>0 , z \in \R\}.
	\end{equation}
	The corresponding vorticity $\vec \omega = \curl {\bf u}$ takes the form
	$$
	\vec \omega (x,t) =\omega^\theta (r,z,t) \,  {\bf e}_\theta, \quad {\mbox {where }} \omega^\theta =  \pp_z u^r - \pp_r u^z  .
	$$
	The divergence free condition $\div {\bf u} =0$ becomes
	$$
	\pp_r (r u^r ) + \pp_z (r u^z) =0 \inn \Sigma, \quad t>0.
	$$
	This implies the existence of a scalar function $\psi^\theta (r,z,t)$ such that
	$$
	u^r = - \pp_z \psi^\theta , \quad u^z ={1\over r}  \pp_r (r \psi^\theta ) \inn \Sigma,   \quad t>0.
	$$
The relation $\vec \omega = \curl {\bf u}$ yields
	$$
	-\left( \pp_z^2  + \pp_r^2 +{1\over r} \pp_r  -{1\over r^2} \right) \psi^\theta = \omega^\theta.
	$$
	We notice that the vector-valued function
	$
	\vec \psi ( x,t) = \psi^\theta (r,z,t) \, {\bf e}_\theta
	$
satisfies $- \Delta \vec \psi=\vec \omega$. 
	The  Euler equations \eqref{euler} thus become
	\begin{equation} \label{eu1}
\begin{array}{l}
			\pp_t \omega^\theta + u^r \pp_r \omega^\theta + u^z \pp_z \omega^\theta=  \frac{1}{r}  u^r \omega^\theta , \quad 
			- [ \pp_z^2 +\pp_r^2 +{1\over r} \pp_r  -\frac{1}{r^2}] \psi^\theta= \omega^\theta \quad (r,z) \in  \Sigma , \quad t>0.
		\end{array}
	\end{equation}
The axisymmetric Euler equation without swirl \eqref{eu1} has a formal singularity at $r=0$, which sometimes is inconvenient to work with. To remove the artificial singularity the following change of variable is usually used
$$
	 \omega= \frac{\omega^\theta}{r},  \quad \psi = \frac{\psi^\theta}{r}.
$$
The functions $\omega$, $\psi$ are respectively known as the {\it relative} vorticity and stream function. In terms of these new variables equation \eqref{eu1}  becomes
	\begin{equation}\label{eu2}
		\begin{array}{l}
		r\, 	\pp_t \omega + \nabla^\perp ( r^2 \psi) \cdot \nabla \omega =0, \quad 
			- \Delta_5 \psi= \omega \inn \Sigma  , \quad t >0,
		\end{array}
	\end{equation}
where
$
\nabla^\perp = (-\pp_z , \pp_r )
$
and 
\begin{equation}\label{D5}
\Delta_5 \psi = \partial_{rr} \psi + \frac{3}{r} \partial_{r} \psi 
	+ \partial_{zz} \psi , \quad x = (r,z),
	\end{equation}
supplemented with the conditions
$$
\pp_r \psi (0,z,t) = 0 , \quad \lim_{|(r,z)| \to \infty} \psi(r,z,t) = 0. 
$$
It is well-known that the initial value problem for \equ{eu2} is globally well-posed in 
$L^1 (rdrdz)\cap L^\infty (\Sigma)$, see \cite{uy,majda,choi}.   
 A bounded solution $\omega$ to \eqref{eu2} with sufficient space 
 decay actually satisfies  
\be\label{time}
t\mapsto  \int_{\R^2} r \omega(r,z,t) \,drdz = constant, \quad   t\mapsto \|w(\cdot, t)\|_{L^\infty (\R^2)}  =   constant.
\ee
 The same is true if the integral is taken in the time section of connected components of the support of $\omega$. We will use these facts in the formulation of a suitable first approximation 
 for a solution exhibiting the leapfrogging dynamics.


\medskip 
A {\bf vortex ring} moving with constant speed $ c$ along the $z$-axis  is a travelling wave solution of \eqref{eu2} with the form  
$$\omega(r,z,t) = W_0 (r, z- c\,  t), \quad \psi(r,z,t) =  \Psi_0 (r, z-c\,  t),$$
where $W_0$ and $\Psi_0$ solve
\begin{equation}\label{ring0}
\nabla^\perp \left( r^2 (\Psi_0 - {c\over 2} )   \right) \cdot \nabla W_0 =0, \quad 
			- [\partial_r^2+\frac{3}{r} \partial_r + \partial_z^2] \Psi_0= W_0 \quad \inn \Sigma .
\end{equation}
It is worth mentioning that if $\Psi_0(r,z)$ satisfies a semilinear equation of the form
$$
	-[\partial_r^2+\frac{3}{r} \partial_r + \partial_z^2] \Psi_0 =  f \left( r^2 (\Psi_0 -  {c \over 2}  ) \right) \inn \Sigma,
	$$
	for an arbitrary nonlinearity $f$, then $\Psi_0,\ W_0=f \left( r^2 (\Psi_0 -  {c \over 2}  ) \right)$ solves 
	\equ{ring0}. A first example of a solution with a compactly supported vorticity and  positive relative vorticity was exhibited by Hill \cite{lamb}. 
In 1970, Fraenkel \cite{f} (see also \cite{mar}) rigorously obtained a solution supported inside a torus with a tiny section of radius $\ve>0$ with a center located at a fixed distance $r_0 >0$ of the $z$-axis, and found the asymptotic expression
	\be \label{massi}
	c = \frac 2{r_0} |\log \ve|(1+ o(1)) \ass \ve \to 0 .
	\ee
		Vortex rings have been analyzed in larger generality in  \cite{fb,norbury,ambrosetti-struwe,dv,buffoni}.
	Expression 
	\equ{massi} corresponds to a special case of formal asymptotics 
	obtained by Da Rios in 1906 \cite{darios,darios1916},
	of thin vortex tubes around curves evolving by their binormal flow.
	See \cite{ricca,jseis,jerrard-smets}.  Helicoidal travelling vortex tubes have been recently found in \cite{ddmw1}.

\subsection*{Interacting vortex rings}  
We shall formally derive the dynamics 
of $k\ge 2$ thin, coaxial 
vortex rings with comparable speeds given at main order by \equ{massi} in \equ{eu2}.
It is convenient to introduce the new unknown $W(r,z,\tau)$ where 
\be
\nonumber
\omega (r,z,t)\,=\, W(r, z-  2{r_0^{-1}} \,  |\log\ve| \,  t,  |\log\ve|\, t\, ), \ee 
so that Problem  \eqref{eu2} in terms of $W(r,z,\tau)$, $ \tau = |\log \varepsilon| \, t$,   takes the form
	\begin{equation}
	\label{leap00}
		|\log\ve| \, r \, \pp_\tau W  + \nn^\perp (r^2\, ( \Psi - r_0^{-1} |\log \ve|) \, ) \cdot \nn W \,= \, 0  , \quad  -\Delta_5 \Psi = W.
\end{equation}	
Let us assume the presence of a solution $W(r,z,\tau)$ of 
\equ{leap00} which  has the approximate form
\begin{equation}
\label{Wansatz}
rW(r,z,\tau) 
 \,  = \, 
\sum_{j=1}^k  \frac 1 { \ve_j^{2}}U \left (
 \frac {x  -  P_j(\tau)}{\ve_j} \right),  \quad           x= (r,z) ,
 \quad P_j(\tau) = (  P_j^1(\tau) ,  P_j^2(\tau) ) , 
 \end{equation}
for small numbers $\ve_j\to 0$ and  centers $P_j(\tau)$ close to $(r_0,0)$ to be determined, where $U$ is the Kaufmann-Scully vortex 
\be\label{ks}
U(y) \, =\, \frac 8{(1+|y|^2)^2}, \quad y\in \R^2. 
\ee
This ansatz is consistent with the fact that $\int rW\, drdz $ should be preserved in time  
on components of the support, see \equ{time}.   On the other hand, the fact that the $L^\infty$-norm of $W$  should be preserved in time on each component suggests that 
$  P_j^1(\tau )\ve_j^2  $ should be constant in time, say $= r_0\ve^2$. Thus, we choose 
$$
\ve_j  =   \ve \frac{r_0}{\sqrt{P_j^1(\tau)}},$$
so that at main order, near $(r_0,0)$ we have 
$$\begin{aligned}
W(r,z,\tau) 
 \, & = \, 
\sum_{j=1}^k  \frac 1 { r_0\ve^2} U \left (
 \frac {x  -  P_j(\tau)}{\ve_j} \right),  \quad           x= (r,z) .
 \end{aligned}
$$ 
Since the Green function $G(r,z)$ solving 
\[
-[\partial_r^2+\frac{3}{r} \partial_r + \partial_z^2] G =   8\pi  \delta_{(r_0, 0)} 
\quad \text{in }\Sigma
\]
satisfies $ G(r,z) =  - 2 \log ( |x-(r_0,0)|^2 ) \, ( 1 - \frac 3{2r_0} (r-r_0) )$ 
near $(r_0,0)$, up to an additive constant and lower order terms, at main order we have
$$\begin{aligned}
\Psi _\ve (r,z,\tau) 
 \, & = \, 
\sum_{j=1}^k  \frac 2{P_j^1 } \log \Big  ( \frac 1 {|  
x-P_j (\tau ) |^2+ \ve_j^2}\Big )\, ( 1 - \frac 3{2 P_j^1 } (r-P_j^1)) ,  \quad           x= (r,z) ,
 \end{aligned}
$$ 
Substituting this expression into equation \equ{leap00} we obtain 
that at main order and for each $j$, 
\begin{equation}\label{12}
\Big [ - r |\log\ve| \frac{d P_j}{d\tau}    +
 \nn^\perp_x  \sum_{i\ne j} 4 r  \log \frac 1 {|x-P_i|}   \,   - 2 {r-r_0  \over r_0} |\log \ve | \, {\bf e}_2 
 \Big ]\cdot
\nn U \left (
 \frac {x     -  P_j(\tau)}{\ve_j} \right)\approx 0.
\end{equation}
It is convenient to use the  ansatz
$$ 
P_j(\tau)  =  (r_0,0)  +  \frac 1{\sqrt{|\log\ve|}} q_j(\tau) , \quad q_j(\tau) = (q_j^1(\tau), q_j^2(\tau)).  
$$
Imposing vanishing of the left factor in  \equ{12} at $x= P_j$,   neglecting lower order terms in a fixed interval $\tau\in [0,T]$, we arrive at the limiting system  
 \begin{equation}\label{car1}
\frac {dq_j}{d\tau}  =\begin{color}{blue} -  \end{color}4 \sum_{\ell \not= j} \,  \,   \frac{(q_j-q_\ell)^\perp }{ |q_j-q_\ell|^2}  -2 \, { q_j^1  \over r_0^2 } \,   \, \left( \begin{matrix} 0\\ 1 \end{matrix} \right), \quad \tau\in [0,T].  
\end{equation}
This is a Hamiltonian system for the energy  
$$
H_k(q_1,\ldots, q_k) = \begin{color}{blue} -\end{color} 2\sum_{i\ne j}  \log|q_i-q_j|  - \frac 1{r_0^2} \sum_{j=1}^k |q_j^1|^2.
$$
For instance, for $k=2$ and restricting ourselves to $q_1=-q_2 =q = (q^1 , q^2)$ 
we arrive at the system 
\begin{equation}\label{car2}
\frac {dq}{d\tau} \, =\, \begin{color}{blue} -\end{color}    2 \frac{q^\perp }{ |q|^2}  -2 \, { q^1  \over r_0^2 } \,   \, \left( \begin{matrix} 0\\ 1 \end{matrix} \right), \quad \tau\in [0,T].  
\end{equation}
The level curves for its Hamiltonian $q\mapsto H_2(q,-q)$ are depicted in Figure~\ref{fig2}. 
Any solution of this system  is periodic and lies on a closed level curve around the origin. 

\medskip
In short, the vorticity of  $k$ interacting 
similar thin vortex rings looks like the superposition of vorticities of $k$ individual rings with small cross section $\ve$, whose centers evolve in time approximately following the  {\em reduced leapfrogging dynamics} \equ{car1}  being located at mutual distances of order 
${1\over \sqrt{|\log \ve |}}$.

\begin{figure}
\includegraphics[scale=0.8]{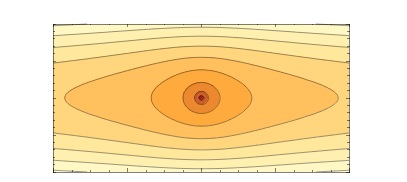}
\caption{ Level curves for $q \to H_2 (q, -q)$.}
\label{fig2}
\end{figure}

\medskip
Several authors investigated the reduced dynamics given by system \equ{car1},  see for instance \cite{dyson,dyson1, da, BCM, hicks,klein,lamb,marchioronegrini,aiki,bori}. Numerical simulations of the leapfrogging, such as in \cite{riley,naga,lim, chen}, provided theoretical evidence that this phenomenon should actually occur in the Euler equations. The leapfrogging phenomenon was experimentally confirmed in 1978  by  Yamada and Matsui \cite{yamada}. They used vortex rings made of air and used smoke for visualization and successfully created a leapfrogging pair of rings.

\medskip
Our main result states the existence of a true smooth solution $\omega(r,z,t)$ of Problem \equ{eu2} asymptotically 
obeying the dynamic law described above, for each given collisionless solution of System \equ{car1} on the time interval $[0,T].$

It is more convenient to express, for a small $\ve >0$, the problem in terms of the equivalent formulation \equ{leap00}

	\begin{equation}
	\label{leap01}
		|\log\ve| \, r \, \frac {\pp W}{\pp \tau}   + \nn^\perp (r^2\, ( \Psi - r_0^{-1} |\log \ve|) \, ) \cdot \nn W \,= \, 0  , \quad  -\Delta_5 \Psi = W,  \quad x \in \Sigma, \quad \tau >0  , 
\end{equation}
 $$
	 {\partial \over \partial r} \Psi (x,\tau ) = 0  \quad {\mbox {on}} \quad x\in \partial  \Sigma  , \quad  |\Psi(x,\tau )| \to 0 \quad \text{as } |x|\to \infty, \, \tau>0.
	 $$
We recall that, if $W(r,z,\tau)$ solves this problem, then 
$$\omega (r,z,t)\,=\, W(r, z-  2{r_0^{-1}} \,  |\log\ve| \,  t,\,   |\log\ve|\, t\, ) \quad 
$$
solves \equ{eu2}.

\medskip

\begin{theorem} \label{teo}

Let 
$q(\tau) = (q_1(\tau), \ldots q_k(\tau))$ be a solution of System $\equ{car1}$ without collisions in $[0,T]$ in the  sense that 
$ \inf \{ |q_i(\tau)-q_j(\tau)|\ / \tau\in [0,T], \ i\ne j\} \, >\, 0. $ Then there exists a smooth solution $W_\ve (r,z,\tau) $ of Problem $\equ{leap01}$ 
such that for certain points $P_j^\ve(\tau) $ with  the form 
$$
P_j^\ve  (\tau)  =  
(r_0,0) +  \frac 1{\sqrt{|\log\ve|}} q_j(\tau)  +  O \Big ( \frac { \log (|\log\ve|)}{|\log\ve|} \Big )   , 
\quad  
$$
we have 
\be\label{mmx}\begin{aligned} 
W_\ve (r,z,\tau )  &=\, \frac 1{\ve^2r_0 } \, \sum_{j=1}^k U \left (\frac{x-  P_j^\ve(\tau)}{\ve_j} \right)  +\varphi(r,z,\tau)  
, \quad  x= (r,z) 
\end{aligned} 
\ee
where $U(y) = \frac 8{(1+|y|^2)^2}$,  $\ve_j^2  =  \ve^2   \frac {r_0}{ P_j^1(\tau) }$. For 
some number $C>0$
we can estimate 
$$
|\varphi(r,z,\tau) |\,\le\,    C\sum_{j=1}^ k\frac {\ve^{2-\sigma}} { \ve^{3-3\sigma} +  
\left |x-  P_j^\ve(\tau) \right|^3} \foral x\in \Sigma, \ \tau \in [0,T],
$$
for some $\sigma >0$ small.
\end{theorem}

Theorem \ref{teo} is the first mathematically rigorous justification of the leapfrogging motion of vortex rings for the Euler equations since Helmholtz proposed it in 1858. 
Next we make various comments on our result and connections with the literature.

\medskip

$\cdot$ 
It has long been observed that vortex dynamics in the Euler equations exhibit strong analogies with the Ginzburg-Landau vortex evolution in two and three dimensions for the Gross-Pitaievskii equation. Jerrard and Smets \cite{js} constructed axisymmetric solutions to the problem
$$
iu_t  =  \Delta u + \frac 1{\ve^2}  (1-|u|^2) u \inn \R^3  \quad
$$
with several vortex rings precisely within the regime we are considering here,  exhibiting the leapfrogging dynamics.
 They state as a major open problem the corresponding result for the $3d$ incompressible Euler equations. See also \cite{js2021} for almost parallel GP vortex filaments, and \cite{jseis} for a conditional result for the so-called vortex filament conjecture.  
The interesting approach in \cite{js,js2021}, based on limiting Ginzburg-Landau energy configurations, does not appear to be suitable for the Euler setting.

\medskip 
$\cdot$   Marchioro and Pulvirenti 
\cite{MP93} analyzed smooth solutions with compactly supported, highly concentrated regular vorticities around a finite number of well-separated points for the $2d$ Euler equations establishing the connection with the Kirchhoff-Routh vortex dynamics.
 In \cite{MP83} they previously did so for tiny vortex patches.  
Extensions of those results to  the evolution of thin axisymmetric vortex rings have been considered in \cite{BCM,marchioronegrini,ButtaCM}. In the recent work  
 \cite{ButtaCM} evolution of compactly supported vortex rings under large separations has been recovered. However, the regime there considered does not allow to detect the leapfrogging dynamics. 
 The reason for this is that the method in those works does not provide a fine control on the evolution of the core of the vorticity. 
In particular, their initially tiny local supports could expand in time, not allowing a fine control on the form of the solutions near the concentration rings which 
in the leapfrogging dynamics are are at small order $|\log\ve|^{-\frac 12}$ distant to each other.

\medskip $\cdot$  
We follow here an entirely different strategy to overcome that difficulty, using a scheme based on our previous work \cite{ddmw} on vortex desigularization in the $2d$-Euler case.     
Problem \eqref{eu2}, \eqref{D5} can be regarded as an anisotropic version of the $2d$ Euler equation.  We first construct an accurate approximate solution, and then solve the full problem by linearization,
adjusting a small correction.  
We need to substantially refine the procedure in \cite{ddmw}. The anisotropy of the equation creates error terms which can not just be treated as perturbations. That makes the improvement of approximation be much more delicate. 
The approximate solution starts with the ansatz \eqref{Wansatz} where $U$ is the the Kaufman-Scully vortex \equ{ks}.
This profile  is a convenient local choice for the building blocks of the rings, while not at all the only one possible. Profiles with a sufficiently fast decay could also be employed. 
While a slight simplification in the approximation procedure could be obtained, that would be at the expense of complicating the a priori estimates.  
The basic blocks \equ{ks}  have been used to construct solutions with concentrated vorticities in $\R^2$ in \cite{ddmw} 
and helicoidal travelling waves in \cite{ddmw1}. It is worth mentioning that the nonlinear stability of fast-decay radial steady vortices is still an open problem. See the works 
 \cite{bedrossian,IJ, IJ2022} for recent linear stability results. The spectrum stability of vortex filament  in $3d$-Euler 
is a delicate matter, see
\cite{fs,gs,gs1,gsv}. Nonlinear stability of vortex ring solutions in $3d$-Euler case is a challenging problem. See the related remarkable  paper by Bedrossian and Masmoudi \cite{BM2015} on nonlinear asymptotic stability for plane Couette flows. 


\medskip


\medskip 
$\cdot$  While the solution built is global-in time (as it is classical for the axi-symmetric no swirl setting) the phenomenon described is not expected to persist as $t\to +\infty$. More precisely, for each $T>0$ we build a solution that approaches the leapfrogging dynamics in a suitable uniform sense as $\ve \to 0$ for $\tau\in [0,T]$ in the scaled time variable $\tau = t/|\log\ve|$. Numerical simulations \cite{yamada}, \cite{AN} show that the leapfrogging dynamics does not persist in time beyond a few crossings of the vortex rings involved. In fact, even in the two-dimensional case the typical long-time behavior is mixing or turbulence, so that long-time uniform localized vorticity is not generic.  

\medskip 
$\cdot$  We point out that the setting $k = 2$, $q_1 = - q_2$ does not simplify the situation: the Euler equation 
cannot be reduced to finding rotating or travelling wave type solutions. The existence of time-periodic solutions following the associated reduced periodic dynamics is a very difficult question. In case it exists it is probably extremely unstable. Global-in-time dynamics for other type of vortex configurations in the $2$-dimensional Euler equations has been constructed in \cite{zbarsky}.

\medskip
$\cdot$ 
We point out that the dynamics of a single vortex ring for the Navier-Stokes system has been considered in \cite{ButtaCM1}, 
and very recently in \cite{gsv2} with an approach closely connected with the one here, using the Lamb-Oseen self-similar vortex (which has a Gaussian decay). 


\medskip
$\cdot$  
In Theorem \ref{teo}, all individual rings 
in the vorticity expression  \equ{mmx} have similar positive circulations, which is what leads to the leapfrogging phenomenon. Another interesting situation is the case of rings with opposite sign circulations, namely dipole dynamics. This is related to various scenarios for potential singularities in axi-symmetric Euler flows, see \cite{child2008, child2018, choi2017, elgindi, hou2014}.   


\medskip
In the next section we explain in detail the scheme of the proof of Theorem \ref{teo}  which is carried out in the subsequent sections.

			\section{Scheme of the proof and organization of the paper.} 
\label{sec2}

The proposed strategy for our construction falls into two steps. First we derive an approximate solution, then we solve the full problem setting up an {\it inner-outer gluing scheme}.

\medskip

The  {\it inner-outer gluing scheme} has been a very 
powerful tool in singularity formation problems for nonlinear elliptic and parabolic equations, see for instance \cite{DKW2011, CDM2020, DDW2020}. In those applications, the use of maximum principle is essential. In \cite{ddmw} we  extend this scheme to the  Euler flow in $2$-dimensions and we find 
 regular
solutions with highly concentrated vorticities around a given number of moving points in the plane, the so-called {\it desingularized vortex problem}. This paper represents the first attempt to extend this scheme to the three-dimensional Euler flow and find regular solutions with highly concentrated vorticities around a given number of vortex rings interacting in accordance with the Leapfrogging dynamics. 

\medskip{}{}
We take advantage of the axi-symmetry and  no-swirl assumption to recast the $3$-dimensional problem in the $2$-dimensional setting \eqref{leap01}, which allows us to adapt some ideas from \cite{ddmw}. In the construction of solutions with  vorticity highly concentrated  around points for the $2$-dimensional Euler flow in \cite{ddmw}, each vortex is well described by a smooth radially symmetric function with fast decay at infinity, the Kaufman-Scully vortex $U$. The core of each vortex is of size $\ve$ and the interaction with other vortices can be easily controlled, as the distance of the centers of two different vortices is of order $O(1)$ as $\ve \to 0$. On the contrary,  radially symmetric profiles need an important non radial correction in order to well approximate vortex rings,  the reason being the anisotropy in equation \eqref{ring0}. The core of each ring is still of order $\ve$ but the relative distance between the centers of two rings is of size $|\log \ve |^{-{1\over 2}}$, which makes their interaction much stronger, and hence delicate to control. These are the features that  make the steps in our argument quite involved: the construction of the approximation requires several consecutive adjustments and the inner-outer scheme to find the remainder of the solution has to be designed to properly describe the transition of the problem in different-scaled regions.

\medskip

	\subsection*{Construction of an approximate solution.} \ \ 	The basic building block for the construction of the leapfrogging is a single approximate travelling vortex ring  with 
	highly $\ve$-concentrated vorticity near a point  $P_0= (r_0, z_0)$,  $r_0>0$. Here $r_0$ represents the radius of the ring, and $z_0$ its vertical displacement. This is achieved finding 
	a constant  $\alpha$ and a stream function  $ \Psi_\ve $ {\it almost} solving  the  equations for the travelling vortex ring
$$	\nn^\perp (r^2( \Psi_\ve  - \alpha |\log \ve|) ) \cdot \nn W_\ve   \sim 0,\quad
			{\mbox {where} } \quad 	  W_\ve =-\Delta_5 \Psi_\ve,
	$$
in a neighbourhood of $P_0$,	with the expectation that  the vorticity 
$$
		r \, W_\ve \rightharpoonup  8\pi  \delta_{P_0} , \quad \ass \ve \to 0.
$$
	In accordance with \eqref{ring0}, \eqref{massi} and the discussion in the introduction, the constant $\alpha$ is expected to satisfy $\alpha \sim {1\over r_0}$ as $\ve \to 0$ and the vorticity $r \, W_\ve$ to have the form ${1\over \ve^2 \, r_0 } U\left({x-P_0 \over \ve} \right)$.
More precisely,  we find
		\begin{align*}
		\alpha &= {1\over r_0} + O(|\log \ve |^{-1} )\\
		\Psi_\ve (x) &= {1\over r_0} \psi_\ve \left({x-P_0 \over \ve}\right), \quad \psi_\ve (y)= \left( \Gamma_0 \left( y \right) - 4 \log \ve  \right) \left( 1- {3\ve \over 2r_0} y_1 \right)+ O(1)\\
			W_\ve  (x)&= {1\over \ve^2  r_0} w_\ve \left({x-P_0 \over \ve}\right),  \quad 
			w_\ve (y )= U(y) \left( 1+{\ve y_1 \over 2  r_0} \Gamma_0 (y) \right) + U(y) O(\ve y_1)
		\end{align*}
		in the region around $P_0$ given by $|x-P_0|<\delta$, for some fixed $\delta >0$, uniformly as $\ve \to 0$.
The function   
$\Gamma_0(y) = \log U(y) $ satisfies the Liouville equation 
\be \label{defU} 
-\Delta \Gamma_0 = e^{\Gamma_0} = U(y) =\frac 8{(1+|y|^2)^2} \inn \R^2.
\ee
	 The precise derivation is carried out in Section \S \ref{sect}, Proposition \ref{f1}.

	 \medskip{}{}
The approximate leapfrogging solution for \eqref{leap01} will look at main order as the sum of $k$ approximate vortex rings. For simplicity we rename the time-variable $\tau$ in \eqref{leap01} to be $t$, so that $t \in [0,T]$.	 Let $k \geq 2$, and consider $k$ points
	 $$
	 P_1 (t ), \quad P_2 (t ), \quad \ldots \quad P_k (t ), \quad P_j = (P_j^1, P_j^2) \in \Sigma
	 $$
	 and $k$ scaling parameters $\ve_1 (t), \ldots , \ve_k (t)$. 
	 We allow these points and these parameters to evolve in time, for $t \in [0,T]$ and  we assume they have the form
	 $$
	 P_j = (r_0 , 0) + {1\over \sqrt{|\log \ve |}} q_j + \tilde P_j + {\bf a}_j, \quad P_j^1 (t) \, \ve_j^2 (t) = r_0 \, \ve^2 
	 $$
	 with $q_j$ satisfying the {\it reduced leapfrogging dynamics} \eqref{car1} and 
	 $$
	 \sqrt{|\log \ve |} \, |\tilde P_j | \to 0, \quad |{\bf a}_j| \lesssim \ve^3 \quad \quad \ass \ve \to 0,
	 $$
	 uniformly for $t \in [0,T]$. We refer to Subsection \S\ref{subsec41} for the detailed description of these points and the scaling parameters.
	 
	 \medskip
	 To each point $P_j$, we associate the approximate vortex ring  described before, namely a pair of stream function and vorticity $(\Psi_j, W_j)$, that then we  add up together. This gives a good description of the approximate leapfrogging close to the points $P_j$, but it is far from satisfying the required boundary conditions
	 $$
	 {\partial \over \partial r} \Psi (x,t) = 0  \quad {\mbox {on}} \quad \partial  \Sigma  \times [0,T), \quad  |\Psi(x,t)| \to 0 \quad \text{as } |x|\to \infty.
	 $$
	 We use a cut-off function and  correct the sum of stream functions to have at the same time the boundary conditions and the relation $-\Delta_5 \Psi = W$ satisfied.
	 
	 	\medskip
	For $N \in \N$, we set
	\begin{equation}\label{cutoff}
		\eta_N (s) = 1, \quad {\mbox {for}} \quad s<N, \quad = 0 \quad {\mbox {for}} \quad s>2N
	\end{equation}
	for a smooth cut-off function.
	Define
	$$
	\eta (x) = \eta_1 \left({4|x-(r_0, 0) |\over \delta } \right),
	$$
	with $\eta_1$as in \eqref{cutoff} and $\delta <r_0$. In Section \S \ref{sec3}, Subsection \S \ref{subsec42}, we prove the existence of  a smooth function $H^0 (x,t)$, which is uniformly bounded, as $\ve \to 0$, for $(x,t) \in \Sigma \times [0,T]$ so that the pair
	$$
	 \Psi^0 (x,t) = \eta (x) \, \sum_{j=1}^k  \Psi_j +H^0 (x,t), \quad {\mbox {and}}  \quad W^0 (x,t) = \eta (x) \, \sum_{j=1}^k W_j,
	$$
is defined for $(x,t) \in \Sigma \times [0,T]$ and satisfies
 $$
	 {\partial \over \partial r} \Psi^0 (x,t) = 0  \quad {\mbox {on}} \quad \partial  \Sigma  \times [0,T), \quad  |\Psi^0(x,t)| \to 0 \quad \text{as } |x|\to \infty, \quad -\Delta_5 \Psi^0 = W^0  \quad (x,t) \in \Sigma \times [0,T].
	 $$	
	 Let us define  the Euler operators
	\begin{equation}\label{defS}
		\begin{aligned}
			S_1(W,\Psi) &= |\log\ve| r \partial_t W + \nn^\perp (r^2( \Psi - r_0^{-1}  |\log \ve|) )\cdot \nn W , \quad \\
			S_2 (W , \Psi )&= \Delta_5 \Psi + W.
		\end{aligned}
	\end{equation}
	Hence
	$$
	S_2 (W^0,\Psi^0 ) = 0, \quad \quad (x,t) \in \Sigma \times [0,T].
	$$
	Besides $(\Psi^0, W^0)$ is a good approximate solution to \eqref{leap01} around each point $P_j$ in the following sense. 
	 For any $j=1, \ldots , k$, consider the small ball around $P_j$ of radius $|\log \ve |^{-1}$, $B(P_j, |\log \ve |^{-1} )$. Under our assumptions on the points $P_j$ and the scaling parameters $\ve_j$, one has that 
	 $$B(P_i,|\log \ve |^{-1} )  \cap B(P_j,|\log \ve |^{-1} ) = \emptyset ,\quad {\mbox {for}} \quad  i \not= j,
	 $$
	 for all $\ve $ small enough, at any $t\in [0,T]$.
	 In  Subsections \S \ref{subsec43} and \S \ref{subsec44} we derive 
	 \begin{align*}
	 \left|\, S_1 (W^0,\Psi^0)  \, (x,t) \, \right| \,& = \, O\left( {\ve^{-2} |\log \ve | \over 1+ |y|^4} \right), \quad x=(r,z)\\
	 {\mbox {where}} \quad & y={x-P_j \over \ve}, \quad {\mbox {for}} \quad x \in B(P_j, |\log \ve |^{-1} ), \, t \in [0,T].
	 	\end{align*}
	 This estimate can be achieved thanks to the choice of the points $P_j$ at their main order, namely $P_j = P_0 + {1\over \sqrt{|\log \ve |} } q_j$, and the fact that $q_j$ solve \eqref{car1}. Observe that on the boundary of the region $B(P_j , |\log \ve|^{-1} )$ one has $|S_1 (W^0, \Psi^0) | = O(\ve^2 |\log \ve |^5 )$, while it is of order $O(\ve^{-2} |\log \ve |)$ close around each $P_j$.

	 \medskip
		Next we modify $( W^0, \Psi^0)$ in order to produce a better approximate solution  $(W^*,\Psi^*)$.
		\medskip
	
	\medskip{}{} In Section \S \ref{sec4} and \S \ref{sec5} we prove the existence of points $\tilde P_j$, $j=1, \ldots , k$ in the decomposition of $P_j = (r_0 , 0) + {1\over \sqrt{|\log \ve |}} q_j + \tilde P_j + {\bf a}_j$, and functions
		$$
		\psi_j^*, \, \phi_j^*, \quad j=1, \ldots , k, \, \quad \psi^{*,out} , \, \phi^{*,out}
		$$ 
		so that
	 $(W^*,\Psi^*)$ given by
	 \begin{equation}\label{app1}
	\begin{aligned}
		\Psi^* (x,t) &= \Psi^0 + \sum_{j=1}^k   {\eta_{j2}\over r_j} \psi_j^* ({x-P_j \over \ve_j },t) + \psi^{*,out} (x,t) \\
		W^*  (x,t) &= W^0  + \sum_{j=1}^k    {\eta_{j1} \over r_j \ve_j^2 } \phi_j^* ({x-P_j \over \ve_j },t) + \phi^{*,out} (x,t) 
	\end{aligned}
	\end{equation}
	where
$$ \eta_{jN} (x,t) =\eta_N  \left(|\log \ve|^3 \, |x-P_j |\right),
$$
is a better approximate leapfrogging. The function $\eta_N$ is defined in \eqref{cutoff}.

In fact, the boundary conditions are satisfied and we get that
	$$
		\begin{aligned}
			| S_1(W^*, \Psi^*)(x,t)| \ \le &\     C\ve^{1-\sigma_*}\sum_{j=1}^k \frac 1 {1+ |y_j|^3} + C{\ve^{4-{2\sigma_* }} \over 1+ |x|^4} , \quad y_j = \frac {x-P_j(t)} {\ve}, \\
			|S_2(W^*, \Psi^*)(x,t)|\ \le &\   C\ve^{4-\sigma_*} 
			\eta_1 \left({4|x-(r_0, 0) |\over r_0} \right),
		\end{aligned}
		$$
		for all $x \in \Sigma$ and $t \in [0,T]$, and $|y_j|<\ve^{-1} |\log \ve |^{-3}$. Here
	$\eta_1$ is given by \eqref{cutoff} and $\sigma_*$ is a fixed positive number. The functions $(W^*, \Psi^*)$ depend on the points ${\bf a}_j$ which are left as parameters to adjust later. Notice that on the boundary of the region $B(P_j , |\log \ve|^{-1} )$ one has $|S_1 (W^*, \Psi^*) | = O(\ve^{4-{2\sigma_*}}   )$, while it is of order $O(\ve^{1-\sigma^*} )$ close to each $P_j$.
	
	\medskip{}{}
	Proposition \ref{Approximation} contains a precise description of the error term
	$
	S_1 (W^*, \Psi^*)
	$
	in the inner regions close to the points $P_j$, at distance $|x-P_j | \lesssim |\log \ve |^{-3}$, for each $j=1, \ldots , k$, as well as in the complementary outer region. Observe that the inner regions  $|x-P_j | \lesssim |\log \ve |^{-3}$ well separate the points $P_j$, whose relative distance is of order $|\log \ve |^{-{1\over 2}}$. The modification to the original stream function $\Psi^0$ is given by
	$$
	\sum_{j=1}^k    {\eta_{j2}\over r_j} \psi_j^* ({x-P_j \over \ve_j },t) + \psi^{*,out} (x,t).
	$$
	The terms $\psi_j^*$ are functions of the expanded variables $y={x-P_j \over \ve_j}$ and encode the local correction needed to improve the approximation near the points $P_j$. The cut-off functions $\eta_{j2}$ (as well as $\eta_{j1}$ for the vorticity) are designed in such a way to guarantee that $\Psi^0$ is still the main term in the decomposition of $\Psi^*$ in the region where the corresponding cut-off functions are non-zero.
	The function $\psi^{*,out}$ is a more regular function, expressed in terms of the original variable $x$ and it is responsible of the improvement of the size of the error far from the points $P_j$. Similar decomposition describes the modification of the initial vorticity $W^0$.
	
\medskip{}{}
In Section \S \ref{sec4}  we describe how we find the 
 improvement $(W^* , \Psi^*)$ and we  describe  the error of approximation. We make this statement precise in Proposition \ref{Approximation}. The proof of Proposition \ref{Approximation} is contained in Section \S \ref{sec5}.	The construction of the approximation  requires  several refinements  in regions close to the centers, which we call {\it inner improvements} (to get $\psi_j^*$ and $\phi_j^*$), and one    global adjustment in the region far from the centers, the {\em outer improvement} (to get $\psi^{*,out}$ and $\phi^{*, out}$). At the beginning of  Section \S \ref{sec5}, in Subsection \S \ref{strategy},  we describe the general strategy for improvement  and then we proceed with the detailed description of each step in Subsections \S \ref{prima} to \S \ref{ultima}.
For the sixth inner improvement and for the outer improvement we solve two linear transport equations. We study them respectively in Sections \S \ref{secProofTransportInner} and \S  \ref{secProofTransportOuter}.

\subsection*{Solving the full problem.} \ \ 
We look for a leapfrogging solution of  equations \eqref{leap01} of the form
\begin{equation}\label{finalform}
\begin{aligned}
    W&= W^* + \varphi\\
    \Psi &= \Psi^* + \psi
\end{aligned}
\end{equation}
where $\varphi$ and $\psi$ are small corrections of the previously found approximation. It is in order to find $\varphi$ and $\psi$ that we
 set up  an inner-outer gluing scheme. 
	The solution $(W, \, \Psi)$ will have the   form  \eqref{finalform} with
		\begin{align*}
			\varphi &= \sum_{j=1}^k \bar \eta_{j1}  {1\over r_j \ve_j^2 } \phi_j ({x-P_j \over \ve_j },t) + \phi^{out} (x,t) \\
				\psi &=  \sum_{j=1}^k \bar \eta_{j2} {1\over r_j} \psi_j ({x-P_j \over \ve_j },t) + \psi^{out} (x,t) 
		\end{align*}
	where
	$$\bar \eta_{jN} (x,t) =\eta_N (|\log \ve |^5 |x-P_j |),$$ and $\eta_N$ as in \eqref{cutoff}.

	Comparing with \eqref{app1}, you may notice that this ansatz has the same form as the one used  for the construction of the improved approximate leapfrogging of vortex rings $(W^*, \Psi^*)$. Observe though that here we are taking  cut-offs  slightly shorter than the ones taken in the definition of $(W^*, \Psi^*)$.

	Let $S_1$ and $S_2$ be the Euler operators introduced in \eqref{defS}.
	Then  the operator $S_1$  evaluated at $(W,\Psi)$ becomes
	\begin{align*}
		S_1(W,\Psi) &= \sum_{j=1}^k {\bar \eta_{1j} \over \ve_j^4} E_j [\phi_j , \psi_j \, \psi^{out} , P] + E^{out} [\phi^{out}, \Psi^{out} , \phi^{in}, \psi^{in} , P] \\
		&{\mbox {where}} \nonumber \\
		&\phi^{in}(y,t) = (\phi_1(y,t) , \ldots, \phi_k(y,t) ), \quad \psi^{in}(y,t)  = (\psi_1(y,t) , \ldots, \psi_k(y,t) ). \nonumber \end{align*}
The operators  $E_j$, $j=1, \ldots , k$,  and $E^{out}$ are defined respectively 
	\begin{align*}
		E_j^{in} & [\phi_j , \psi_j \, \psi^{out} , {\bf a} ] (y,t) := |\log \ve | \ve_j^2 (1 +{\ve_j \over r_j } y_1) \pp_t \phi_j   - |\log \ve | \left( \ve_j  \dot{\ve_j} (1+{\ve_j \over r_j} y_1)    \nabla \phi_j \cdot y  -{\ve_j^2 \over r_j} y_1 \dot P_j \cdot \nabla \phi_j \right) \nonumber \\
		& + \nabla^\perp \left(  (1+{\ve_j \over r_j} y_1)^2 (\Psi_j  -   |\log \ve | +\psi_j^*    + \psi_j + r_j \psi^{out} )  -\ve_j |\log \ve |  \dot{ {\bf a}_j } \right)\cdot \nabla \phi_j \nonumber \\
		&+ \nabla^\perp \left((1+{\ve_j y_1 \over r_j} )^2 (\psi_j +r_j \psi^{out} ) \right) \nabla  (\ve_j^2 r_j W^* ) \nonumber \\
		& +\ve_j^4 S_1 (W^*, \Psi^*) (\ve_j y + P_j),  \quad |y|< 3R , \quad R:= {1 \over \ve |\log \ve|^5} 
	\end{align*}
	with $\dot{} = {d \over dt}$ and $y={x-P_j \over \ve_j} $, and 
	\begin{equation} 
	\begin{aligned}\nonumber
		E^{out} &[\phi^{out}, \Psi^{out} , \phi^{in}, \psi^{in} , {\bf a}] (x,t) :=
		|\log \ve | \, r \, \phi^{out}_t\\
		&+ \nabla_x^\perp ( r^2 (\Psi^*  + \sum_{j=1}^k { \bar \eta_{j2}\over r_j} \psi_j ({x-P_j \over \ve_j }) + \psi^{out} -r_0^{-1} |\log \ve |)) \nabla_x \phi^{out} \nonumber \\
		&+\sum_{j=1}^k \left[ r \, |\log \ve |  \,  \pp_t \bar \eta_{j1} + \nabla_x^\perp ( r^2 (\Psi^*  + \sum_{j=1}^k  { \bar \eta_{j2}\over r_j} \psi_j ({x-P_j \over \ve_j }) + \psi^{out}-r_0^{-1} |\log \ve |)) \nabla \bar \eta_{1j} \right] {\phi_j \over \ve_j^2 r_j} \nonumber  \\
		&+ \left[ \sum_{j=1}^k (\bar \eta_{2j} - \bar \eta_{1j}) \nabla_x^\perp (r^2 ({\psi_j \over r_j} +\psi^{out}) ) + {r^2 \psi_j \over r_j} \nabla_x^\perp \bar \eta_{2j} \right] \nabla_x W^*  \nonumber \\
		&+ (1-\sum_{j=1}^k \bar \eta_{2j} ) \nabla^\perp (r^2 \psi^{out} ) \cdot \nabla W^* 
		+ (1-\sum_{j=1}^k \bar \eta_{j1} ) S_1 (W^*,\Psi^*)=0 \quad (x,t) \in \Sigma \times [0,T)  . \nonumber 
	\end{aligned}
	\end{equation}
With all this set up, we notice that the pair $(W, \Psi)$ is a solution to \eqref{leap01} if $(\phi^{in}, \psi^{in}, \phi^{out}, \psi^{out} )$ solve 
	the {\it inner-outer gluing} system given by the inner problems
$$
		\begin{aligned}
			E_j^{in} & [\phi_j , \psi_j \, \psi^{out} , {\bf a}] (y,t)  = 0, \quad (y,t) \in B(0; 3R ) \times [0,T)\\
			-\Delta_y\psi_j &-{3 \ve_j \over r_j + \ve_j y_1} \pp_{y_1} \psi_j = \phi_j,  \quad (y,t) \in B(0; 3R  ) \times [0,T)
		\end{aligned}
$$
	for all $j=1, \ldots , k$, 	coupled with the outer problem
$$
		\begin{aligned}
			E^{out} &[\phi^{out}, \psi^{out} , \phi^{in}, \psi^{in} , {\bf a}] (x,t) =0, \quad (x,t) \in \Sigma \times [0,T)\\
			E_1^{out} &[\phi^{out}, \psi^{out} , \phi^{in}, \psi^{in} , {\bf a}] (x,t) =0, \quad (x,t) \in \Sigma \times [0,T)
		\end{aligned}
		$$
	where
	\begin{align*}
		E_1^{out}& := 	\Delta_5 \psi^{out} + \phi^{out} + \sum_{j=1}^k (\bar \eta_{j1} -\bar \eta_{j2} ) {\phi_j \over r_j \ve_j^2} 
		+\sum_{j=1}^k ( {\psi_j \over r_j} \Delta_5 \bar \eta_{j2}  + 2 \nabla_x \bar \eta_{j2} \nabla_x {\psi_j \over r_j}  ), \quad (x,t) \in \Sigma \times [0,T)	
	\end{align*}
coupled 	with the boundary and decay conditions on $\psi^{out}$
	$$
		{\pp \over \pp r } \psi^{out} (x,t) = 0 , \quad (x,t) \in \pp \Sigma \times [0,T], \quad |\psi^{out} (x,t)| \to 0 , \quad \ass |x| \to \infty. 
	$$
	The solution predicted by Theorem \ref{teo} is obtained solving the inner-outer gluing system 
$$
 \begin{aligned}
			E_j^{in} & [\phi_j , \psi_j \, \psi^{out} , {\bf a}] (y,t)  = 0, \quad (y,t) \in B(0; 3R ) \times [0,T)\\
		-\Delta_y\psi_j &-{3 \ve_j \over r_j + \ve_j y_1} \pp_{y_1} \psi_j = \phi_j,  \quad (y,t) \in B(0; 3R  ) \times [0,T), \quad j=1, \ldots , k \\
			E^{out} &[\phi^{out}, \psi^{out} , \phi^{in}, \psi^{in} , {\bf a}] (x,t) =0, \quad (x,t) \in \Sigma \times [0,T)\\
			E_1^{out} &[\phi^{out}, \psi^{out} , \phi^{in}, \psi^{in} , {\bf a}] (x,t) =0, \quad (x,t) \in \Sigma \times [0,T)\\
   {\pp \over \pp r } \psi^{out} &(x,t) = 0 , \quad (x,t) \in \pp \Sigma \times [0,T], \quad |\psi^{out} (x,t)| \to 0 , \quad \ass |x| \to \infty.
		\end{aligned}
		$$
 In
order to obtain the desired solution (with initial conditions equal to zero in all the
parameter functions) we will formulate the system as a fixed point problem for a
compact operator in a ball of a suitable Banach space. We will find a solution
by means of a degree theoretical argument. That involves establishing a priori
estimates for a homotopical deformation of the problem into a linear one.   These arguments are performed in Sections \S\ref{sec9} and \S \ref{sec10}.
The rest of this paper is devoted to carrying out in detail the steps outlined
above.

	



		\section{Approximate travelling vortex ring} \label{sect}


	In this section we  define the basic building block for the construction of the leapfrogging. This object is an approximate travelling vortex ring  with 
	highly $\ve$-concentrated vorticity near a point  $P \in \Sigma$. It is achieved finding 
	 a constant  $\alpha$ and a stream function  $ \Psi_\ve $ "almost" solving in a neighbourhood of $P$ the  equations for the travelling vortex ring
	\begin{equation}\label{uu}
		\begin{aligned}
			S_\alpha (W_\ve , \Psi_\ve ) &:=	\nn^\perp (r^2( \Psi_\ve  - \alpha |\log \ve|) ) \cdot \nn W_\ve   \sim 0,\quad
		{\mbox {where} } \quad 	  W_\ve =-\Delta_5 \Psi_\ve
		\end{aligned}
	\end{equation}
	where $\Delta_5$ is the operator introduced in \eqref{D5}. We recall it here
$$
\Delta_5 \Psi = \partial_{rr} \Psi + \frac{3}{r} \partial_{r} \Psi 
+ \partial_{zz} \Psi , \quad x = (r,z).
$$
The point $P$ represents the centre of a travelling ring, and we take it of the form
$$
P= (\bar r , \bar z ) = P_0 + q, \quad P_0 = (r_0 , 0),
$$
with $r_0 >0$ a fixed number, and $|q| \to 0$ as $\ve \to 0$. In accordance with the discussion in the introduction, it is expected that the vorticity $r W_\ve$ satisfies
\begin{equation}\label{star} r \, W_\ve (r,z) =  \frac 1 {\ve^{2}}U \left ( \frac {x-P }\ve   \right) (1+ o(1)) , \quad x=(r,z)
\end{equation}
where $U$ is the rapidly decaying function in \eqref{ks},
and  that the constant $\alpha$  satisfies $\alpha = {1\over r_0} (1+ o(1)) $, with $o(1) \to 0$ as $\ve \to 0$. The associated stream function $\Psi_\ve $ will correspond to an $\ve$-regularization of
	the following Green's function
	$$
	-\Delta_5 G(x;P)  =  8\pi \delta_{P}, \quad {\partial \over \partial r} G (x;P) =0 , \quad {\mbox {on}} \quad \partial \Sigma, \quad G(x;P) \to 0 \ass |x|\to \infty.
	$$
	We write  the Green's function  $G$ as 
	$$
	G(x;P) =  \log \frac 1{|x-P|^4} \big (1  -  \frac 3{2\bar r} (r-\bar r )  + H(x;P) \big )  + K(x; P),
	$$
	where $H(\cdot ; P)$ and $K(\cdot ; P)$ satisfy respectively
	\begin{equation}\label{eqH}
		\Delta_5 \Big( \log \frac{1}{|x-P|^4} H(x;P) \Big)
		=
		-30 \frac{(r-\bar r )^2}{r \bar r  |x-P|^2}
		+\frac{9}{2 r \bar r }  \log \frac{1}{|x-P|^4} ,
	\end{equation}
	and
	\begin{equation}\label{eqK}
		\Delta_5 K(x;P) =0 .
	\end{equation}
	Let the $\ve$-regularization of the Green's function $G$ be given by  
		\begin{equation}\label{greg}
		G_\ve (x;P) := \log \frac 1{ (\ve^2 + |x-P|^2)^2} \big (1  -  \frac 3{2\bar r} (r-\bar r)  + H(x;P) \big ) + K(x;P).
		\end{equation}
Inserting this approximation in the operator defined in \eqref{uu} produces a term of size $\ve$ in a neighborhood of $P$, when expressed in the expanded variable $y={x-P \over \ve}$. We reduce the size to $\ve^2$ slightly modifying the approximation in a region close to $P$ .  For this purpose, we introduce the constant $A$  given by the relation
\begin{equation}\label{defA}
\int_{\R^2} y_1 U(y) \left( \Gamma_0 + A \right) \, {\partial \Gamma_0 \over \partial y_1} \, dy = 0,
\end{equation}
where $\Gamma_0$ is given in \eqref{defU},
and the function $\Gamma :\R^2 \to \R$ defined as
\begin{equation}\label{defGamma}
\Gamma (y) := \int_\rho^\infty {(1+ \eta^2)^2 \over  \eta^3} \int_0^\eta  {s^3 \over 1+s^2}  U(s) (\Gamma_0 (s) + A) \, ds \, d\eta, \quad \rho = |y|.
\end{equation}
A direct computation gives that $\Gamma $ solves 
$$
\Delta (y_1 \Gamma )  + U \, y_1 \, \Gamma  + y_1 U(y) \left( \Gamma_0 + A \right) = 0 \quad \inn \quad \R^2
$$
and satisfies
	$$
	\Gamma (y) = O\left( {\log (1+ |y| ) \over 1+ |y| }\right) , \quad \ass |y| \to \infty.
	$$

	\medskip
	\noindent
 		We make our construction  precise in the next Proposition. 
	
	\begin{prop}\label{f1}
		Let $r_0$ be a fixed positive number and $P= (\bar r , \bar z ) = P_0 + q$, with $P_0 = (r_0 , 0)$ and $|q| \to 0 $ as $\ve \to 0$. 
		Define $\alpha = \alpha (\ve , P)$ as
		\begin{equation}\label{defalpha}
			\alpha = {1\over \bar r} -{ A+ \log 8 - 6 - 2 \bar r \pp_r K (P;P) - 4K(P;P) \over 4 \bar r |\log \ve |} ,
			\end{equation}
where $A$ is given by \eqref{defA}, 		and $\Psi_\ve [P] (x)$ as 
		\begin{equation} \label{f}
			\begin{aligned}
				\bar r \, \Psi_\ve [P] (x) &= G_\ve (x;P)
			 +  {r-\bar r  \over  2 \bar r }  \Gamma ({x-P \over \ve} ) ,
			\end{aligned}
		\end{equation}
	where $G_\ve (x;P)$ is in \eqref{greg} and $\Gamma$ in \eqref{defGamma}. 
		Set
		$$
			W_\ve [P] (x ) = - \Delta_5 \Psi_\ve [P] (x).
		$$
		Then
		for any fixed $\delta >0$ small and any $x$ with $|x-P|<\delta$, 
		setting
		$$
		y= {x- P\over \ve}= \rho \,  e^{i\theta}, \quad |y| <{\delta \over \ve},
		$$
		$S_\alpha  (W_\ve , \Psi_\ve)$ as defined in \eqref{uu}  has the following expansion
		\begin{equation}\label{ee1}
			\begin{aligned}
				\ve^4 \,  & S_\alpha  (W_\ve , \Psi_\ve) (\ve y + P) =
				{ \ve^2 |\log \ve | \over 1+ |y|^4}  E_2(\rho,\theta,\varepsilon, P)	\\
				&+ {\ve^3 |\log \ve |  \over 1+|y|^3} |\log (\ve (1+  |y |))| [ E_1(\rho,\theta,\varepsilon, P) + E_3(\rho,\theta,\varepsilon, P)] \\
				&+O \left( {\ve^4 |\log \ve |^2  \over 1+ |y|^2} \, \right), \quad \ass \ve \to 0.
			\end{aligned}
		\end{equation}
		Here we have written $E_k(\rho,\theta,\varepsilon, P)$ for a function  of the form
		\begin{equation}\label{defEk}
			\begin{aligned}
				E_k(\rho,\theta,\varepsilon, P)
				= E_{k,1}(\rho,\varepsilon, P) \cos(k \theta)
				+ E_{k,2}(\rho,\varepsilon, P) \sin(k \theta)
			\end{aligned}
		\end{equation}
		where
		\begin{align*}
			\sum_{j=0}^2\Bigl|
			(1+\rho)^j \frac{\partial^j E_{k,i}}{\partial \rho^j }
			\Bigr| +  \Bigl| \nabla_{P} E_{k,i} \Bigl| \lesssim 1, \quad \ass \ve \to 0.
		\end{align*}

	\end{prop}

		\medskip
			\begin{proof}[Proof of Proposition \ref{f1}] 
Let us fix $\delta >0$ and consider the region of points $x$ with $|x-P| <\delta$. We use the expanded variable
		$$
		y= {x-P \over \ve}, \quad P= (\bar r , \bar z)
		$$
		which we also identify with polar coordinates $y= \rho e^{i\theta}$, $\rho = |y|$.
	
	\medskip{}{}
	From the definition of $\Psi_\ve$ given in  \eqref{f} we get the following expansions, for $|y|<{\delta \over \ve}$,
	\begin{equation}\label{barp00}
		\begin{aligned}
			 \Psi_\ve [P] (x) &= {1\over \bar r} \psi^0(y), \quad {\mbox {where}}  \\
				\psi^0(y)&= 
				\Gamma_0(y) - 4 \log\varepsilon - \log 8 + K (P;P)
				\\
				& \quad 
				+ \frac{\varepsilon y_1}{2\bar r }\Big(-3\Gamma_0(y) +A_\ve -4K(P ,  P) +\Gamma (y)
				\Big) 
				\, 
				+\ve^2 \theta_1[P] (y) + \ve^3 \theta_2 [P] (y) ,
			\end{aligned}
		\end{equation}
		and
		\begin{equation}\label{barp}
			\begin{aligned}
				& \frac {r^2}{\bar r}  \Big(\Psi_\ve [P] (x) - \alpha |\log\ve| \Big) 
				= 
				\Gamma_0(y) - (4-\alpha \bar r ) \log\varepsilon - \log 8 + K (P;P)
				\\
				& \quad 
				+ \frac{1}{2\bar r }\Big(\Gamma_0(y) +\bar A +\Gamma (y)
				\Big) 
				\, \varepsilon y_1
				+\ve^2 \theta_1[P] (y) + \ve^3 \theta_2 [P] (y) .
			\end{aligned}
		\end{equation}
		In the above expansions, $\bar A $ and $A_\ve$ are constants given by
		\begin{equation}\label{barA}
	\bar A = A -6, \quad A_\ve = 4 (4-\alpha \bar r)  \log \ve + \bar A + 4 \log 8 
\end{equation}
and $\Gamma = \Gamma (y)$ is as in \eqref{defGamma}. Moreover $	\theta_1[P] (y)$ and $	\theta_2[P] (y)$ denote generic reminders with the following form
			\begin{equation}\label{theta1n}
	\begin{aligned}
		\theta_1[P] (y) &=\Biggl[   a (P) + b_1 (P) \cos 2\theta + b_2 (P) \sin 2\theta  \\
		&+ |\log (\ve (1+|y|)) | ( a (P) + b_1 (P) \cos 2\theta + b_2 (P) \sin 2\theta )\Biggl] \, O (|y|^2)
	\end{aligned}
\end{equation}
and
\begin{equation}\label{theta2n}
	\theta_2[P] (y) = a(P) O (|y|^3 |\log (\ve (1+|y|)|))
\end{equation}
for $a$, $b_1$ and $b_2$ smooth  functions of $P$, uniformly bounded as $\ve \to 0$.
 Here $\Gamma_0(y) = \log U(y)$ as in \eqref{defU}.

		\medskip
	Formula \eqref{barp00} gives the asymptotic expansion of the stream function $\Psi_\ve$ associated to our approximate vortex ring. In order to describe the asymptotic expansion of the vorticity $W_\ve$, let us introduce
		\begin{equation}\label{el2}
 f( r^2 (\Psi - \alpha |\log \ve |))  = {8 \over \bar r } e^{-K (P;P)} \ve^{2 -\alpha  \bar r  }
f_0\Big( \frac{r^2}{\bar r}( \Psi - \alpha |\log \ve|) \Big), \quad {\mbox {with}} \quad f_0(s) = e^s
\end{equation}
and
\begin{equation} \label{el}
			\begin{aligned}
				E[\Psi] (x)&:= \Delta_5 \Psi +  f( r^2 (\Psi - \alpha |\log \ve |))   .
			\end{aligned}
		\end{equation}
		The proof of Proposition \ref{f1} follows from showing that $\alpha$ and   $\Psi_\ve $ are so that the following expansion for $E[\Psi_\ve]$, given by \eqref{el}-\eqref{el2}, holds true:
		for all $|y|<{\delta \over \ve}$ we have
		\begin{equation}\label{estin1}
			\begin{aligned}
				\ve^2 \, \bar r  \, E [\Psi_\ve] (x)&=  {\ve^2 \over 1+ |y|^2} \Biggl[   a (P) + b_1 (P) \cos 2\theta + b_2 (P) \sin 2\theta  \\
				&+ |\log (\ve (1+|y|)) | ( a (P) + b_1 (P) \cos 2\theta + b_2 (P) \sin 2\theta )\Biggl]\\
				&+ O({\ve^3 \over 1+ |y|} ) |\log (\ve (|y|+1) )|, \, \quad 
			\end{aligned}
		\end{equation}
		uniformly as $\ve \to 0$. Here 
		$a$, $b_1$ and $b_2$ denote smooth  functions of $P$, uniformly bounded together with their derivative as $\ve \to 0$, whose definition may change from line to line.

				\medskip
Assume expansion \eqref{estin1} is true. Since $W_\ve [P] = - \Delta_5 \Psi_\ve [P]$, we have
$$
S_\alpha (W_\ve , \Psi_\ve ) = \nabla^\perp  (r^2( \Psi_\ve  - \alpha |\log \ve|) ) \cdot \nabla (E [\Psi_\ve ] )
$$
and expansion \eqref{ee1} readily  follows from
		\eqref{barp} and \eqref{estin1}.

	\medskip\medskip
We also observe that, in the region $|y|<{\delta \over \ve}$, the vorticity of the approximate vortex ring can be described as
	\begin{equation}
		\label{Ubar}
		\begin{aligned}
			W_\ve [P] (x)&= {1\over \ve^2 \bar r} w_\ve [P] (y)\\
			w_\ve  [P] (y)&= U(y) \left( 1+{\ve y_1 \over 2 \bar r} (\Gamma_0 + \bar A + \Gamma) + \ve^2 \theta_1 [P] (y) + \ve^3 \theta_2 [P] (y)  \right)
		\end{aligned}
	\end{equation}
where $U(y)$ and $\Gamma_0 (y)$ are defined in \eqref{defU},  $\bar A$ in \eqref{barA}, $\theta_1$ and $\theta_2$ in \eqref{theta1n} and \eqref{theta2n}. 
 Hence $r W_\ve (x) $ approaches, locally around $P$, a Dirac delta, as $\ve \to 0$, in accordance with the expectation \eqref{star}. 
	We will make use of these estimates in the next Section.

		\medskip

		As we explained before, we want to prove the validity of \eqref{estin1}. 
		We start with the observation that, since the operator in \eqref{el} is invariant under translations in the $z$-direction, it is not restrictive to assume that $P= (\bar r , 0)$ and to work in the class of functions that are even in the variable $z$.

		To simplify notation, we drop the dependence on $P$ in the functions  $G_\ve$, $H$ and $K$. 
		
	 Then for $|x-P| < \delta$ we have 
		$$
		G_\ve (x)  = (\Gamma_0(y) - 4\log (\ve)  - \log 8) \Big (1  
		-  \frac 3{2\bar r } \ve  y_1  + H( P + \ve  y)\Big )    + K(P + \ve  y) , \quad y={x-P \over \ve}
		$$
		where we mean $H( P + \ve  y) = H( P + \ve  y; P)$, etc. Observe that $\pp_z K(P) =0$ by symmetry.
		From \eqref{eqH} we get that $H(x;P) $ has the following expansion as $x \to P$
		\begin{eqnarray} \label{expH}
			H(x)= a_1(P) (r-\bar r)^2 + a_2 (P)  z^2 + b(P) {|x-P|^2 \over \log |x-P|^2} + O(|x-P|^3)
		\end{eqnarray}
		where  $a_1$, $a_2$, $b$ are constants whose value depends on  $P$.
		Using \eqref{expH}, we expand 
		$$	\begin{aligned}
				G_\ve (x)  
				& = \Gamma_0(y)  - 4\log \ve -\log 8 +K (P)   -
				\frac{3}{2\bar r}\left[  \Gamma_0(y)  - 4\log \ve -\log 8 -{2 \bar r \over 3} \partial_r K(P)  \right] \varepsilon  y_1 
				\\
				& \quad 
				+\ve^2 \theta_1[P] (y) + \ve^3 \theta_2 [P] (y)
			\end{aligned}
		$$
		where $\theta_1$, $\theta_2$ are reminders that can be described as in \eqref{theta1n} and \eqref{theta2n}.
		
		For $
			\Psi_{P} (x) = \frac{1}{\bar r} G_\varepsilon(x), 
		$
		we get
		\begin{align*}
			& \frac {r^2}{\bar r}  \Big(\Psi_{P}(x) - \alpha |\log\ve| \Big) 
			= 
			\Gamma_0(y) - (4-\alpha \bar r ) \log\varepsilon - \log 8 + K (P)
			\\
			& \quad 
			+ \frac{1}{2\bar r}\Big(\Gamma_0(y) - (4-4\alpha \bar r )
			\log\varepsilon
			+2\bar r \partial_r K(P) + 4 K (P)
			-\log(8)
			\Big) 
			\, \varepsilon y_1
			\\
			& \quad 
			+\ve^2 \theta_1[P] (y) + \ve^3 \theta_2 [P] (y) ,
		\end{align*}
		and 
		\begin{equation}\label{c1}
			\begin{aligned}
				& {8 \over \bar r } e^{-K (P)}  \ve^{2-\alpha \bar r} 
				f\Big(  \frac {r^2}{\bar r}  ( \Psi_{P}(x) - \alpha |\log\ve|)  \Big) 
				\\
				&=
				\frac{1}{\varepsilon^{2}  \bar r } U(y)
				\exp\Big[
				\frac{1}{2\bar r }\Big(\Gamma_0(y) - (4-4\alpha \bar r )
				\log\varepsilon
				+2\bar r \partial_r K (P) + 4 K (P)
				-\log 8
				\Big) 
				\, \varepsilon y_1
				\\
				& \quad 
				+\ve^2 \theta_1[P] (y) + \ve^3 \theta_2 [P] (y)
				\Big] ,
			\end{aligned}
		\end{equation}
		for  functions $\theta_1$, $\theta_2$ of the form \eqref{theta1n} and \eqref{theta2n}.

		Next we compute 
		\begin{align*}
			\Delta_5 \Psi_{P} = (1) + (2)
		\end{align*}
		where
		\begin{align*}
			(1) &= \frac{1}{\bar r}
			\Delta_5\Big[ ( \Gamma_0 ({x-P \over \ve} )- 4 \log \varepsilon - \log8  ) \Big( 1 - \frac{3}{2 \bar r}(r-\bar r) \Big)
			\Big]
			\\         
			(2) &= 
			\frac{1}{\bar r }
			\Delta_5 [
			( \Gamma_0 ({x-P \over \ve} )- 4 \log \varepsilon - \log8  )
			H
			] ,
		\end{align*}
		since by definition $\Delta_5 K = 0$.
		Setting again  $\ve y = x-P$, we get
		\begin{align*}
			(1) &= 
			\frac{1}{\bar r}
			\Big[
			\Big( 
			-\frac{1}{\varepsilon^2}U(y) 
			+ \frac{3}{\varepsilon  (\bar r +\varepsilon y_1) } \Gamma_0'(y) \frac{y_1}{\rho}
			\Big) \Big( 1-\frac{3}{2 \bar r} \varepsilon \ y_1\Big)
			-\frac{3}{\varepsilon \bar r }
			\Gamma_0'(\rho) \frac{y_1}{\rho}
			-\frac{9}{2 r \bar r } ( \Gamma_0 - 4 \log \varepsilon - \log 8 )
			\Big]
			\\
			&= 
			\frac{1}{\bar r}
			\Big[
			-\frac{1}{\varepsilon^2 }U(y)
			+\frac{3}{2 \varepsilon \bar r} U(y) y_1
			- \frac{15}{2}\frac{\Gamma_0'(\rho)}{\rho} \frac{y_1^2}{r \bar r  }
			-\frac{9}{2 r \bar r } ( \Gamma_0 - 4 \log \varepsilon - \log 8  )
			\Big]\\
			&= \frac{1}{\ve^2 \bar r}
			\Big[
			-U(y)
			+\frac{3}{2  \bar r} \ve U(y) y_1
			+30 {(r-\bar r)^2 \over \bar r  r} {\ve^2 \over \ve^2 + |x-\bar P|^2 }
			-\frac{9}{2 r \bar r} \ve^2 \log {1\over (\ve^2 + |x-\bar P|^2 )^2}
			\Big].
		\end{align*}
		From \eqref{eqH}
		we get
		\begin{align*}
			(2) &=
			-30 \frac{(r-\bar r)^2}{r \bar r^2 |x-P|^2}
			+\frac{9}{2 r \bar r^2}  \log \frac{1}{|x-P|^4}+ 
			\frac{1}{\bar r}
			\Delta_5 [
			( \Gamma_0 ({x-P \over \ve} )- 4 \log \varepsilon - \log8 - \log {1\over |x-P|^4} )
			H
			] \\
			&=
			-30 \frac{(r-\bar r)^2}{r \bar r^2 |x-P|^2}
			+\frac{9}{2 r \bar r^2}  \log \frac{1}{|x-P|^4}+ 
			\frac{2}{\bar r}
			\Delta_5 [
			\log {|{x-P \over \ve}|^2\over 1 +|{x-P \over \ve}|^2}    
			H
			] 
		\end{align*}
		Combining the above expressions we conclude that
		\begin{equation}\label{c2} 
			\begin{aligned}
				- \Delta_5 \Psi_{P}
				&= 
				\frac{1}{\bar r}
				\Big[
				\frac{1}{\varepsilon^2}U(y) 
				- \frac{3}{\varepsilon \bar r } U(y) y_1
				+ O(\frac{1}{1+\rho^2})
				\Big]\\
				&= 
				\frac{1}{\bar r \varepsilon^2} U(y) 
				\Big[
				1
				- \frac{3}{ \bar r} \varepsilon y_1
				+ \ve^2 \theta_1[P] (y) + \ve^3 \theta_2 [P] (y)
				\Big],
			\end{aligned}
		\end{equation}
		where $\theta_1$ and $\theta_2$ are described in \eqref{theta1n} and \eqref{theta2n}.
		
		\medskip
		Putting together \eqref{c1} and \eqref{c2} we find that 
		\begin{equation}\label{err0}
			\ve^2 E[\Psi_{P}] = 
			\frac{1}{\bar r} [ \varepsilon E_0 + \Theta_\ve (y) ]
		\end{equation}
		where
		\begin{align*}
			E_0(y) &=
			\frac{ y_1}{ 2 \bar r }U(y)
			\Big(\Gamma_0(y) - 4(1-\alpha \bar r )
			\log\varepsilon
			+ 2 \bar r \partial_r K(P) + 4 K (P)
		-\log8 +6
			\Big)  ,
		\end{align*}
		and
		\begin{align*}
			\Theta_\ve (y)&=  U(y)
			\exp\Big[
			\frac{1}{2\bar r}\Big(\Gamma_0(y) - (4-4\alpha \bar r)
			\log\varepsilon
			+2\bar r\partial_r K (P) + 4 K (P)
			-\log 8
			\Big) 
			\, \varepsilon y_1
			\\
			& \quad 
			+\ve^2 \theta_1[P] (y) + \ve^3 \theta_2 [P] (y)
			\Big]\\
			&- U(y) \Big[ 1 + \frac{1}{2\bar r}\Big(\Gamma_0(y) - (4-4\alpha \bar r)
			\log\varepsilon
			+2 \bar r \partial_r K (P) + 4 K (P)
			-\log 8
			\Big) 
			\, \varepsilon y_1\Big]\\
			&+ U(y) \left[\ve^2 \theta_1[P] (y) + \ve^3 \theta_2 [P] (y) \right],
		\end{align*}
		where $\theta_1$, $\theta_2$ denote again generic functions of the form \eqref{theta1n}, \eqref{theta2n}.
		A closer look at this expression  
		gives that $\Theta_\ve (y) $ can be described as follows
		\begin{equation}\label{Thetaex}
			\begin{aligned}
				\Theta_\ve (y) &= { \ve^2 \over 1+ |y|^2}   \Biggl[   a (P) + b_1 (P) \cos 2\theta + b_2 (P) \sin 2\theta  \\
				&+ |\log (\ve (1+|y|)) | ( a (P) + b_1 (P) \cos 2\theta + b_2 (P) \sin 2\theta )\Biggl]   \\
				&+ O({\ve^3 \over 1+ |y|} ) \log (\ve (1+|y|) ) a(P) 
			\end{aligned}
		\end{equation}
		uniformly as $\ve \to 0$. Also here
		$a$, $b_1$, $b_2$ stand for generic smooth  functions of $P$, uniformly bounded as $\ve \to 0$.

		In order to reduce the size of the error term in \eqref{err0}  we solve 
		\begin{align}
			\label{appro}
			\Delta \bar \psi  + U(y) \bar \psi   + \varepsilon E_0(y)=0 \quad {\mbox {in}} \quad \R^2, \quad \lim_{|y|\to\infty } \bar \psi (y) = 0. 
		\end{align}
		It is here when we introduce the function $\Gamma$ and we use the explicit definition of $\alpha$ as given in \eqref{defalpha}.

		  \medskip{}{}
		  It is known that all bounded solutions to 
		$$
		\Delta \bar \psi  + U(y) \bar \psi  =0 \quad {\mbox {in}} \quad \R^2
		$$
		are given by linear combinations of
		$$
		\frac{\partial \Gamma_0}{\pp y_i }(y), \quad i=1,2, \quad 2 + \nabla \Gamma_0 \cdot y.
		$$
		This result can be found in \cite{bp}. By the standard Fredholm alternative for this problem we need 
		the following solvability condition satisfied
		$$
		\int_{\R^2}  E_0(y)  \frac{\partial \Gamma_0}{\pp y_1 }(y)\, dy\ =\ 0.
		$$
		Infact, the error term $E_0$ is by definition even with respect to the variable $y_2$. Hence
		the remaining solvability conditions 
		$$	\int_{\R^2}  E_0(y)  \frac{\partial \Gamma_0}{\pp y_2 }(y)\, dy\ =\ 0, \quad 	\int_{\R^2}  E_0(y)  \left( 2 + \nabla \Gamma_0 \cdot y \right) \, dy\ =\ 0$$ are automatically satisfied by symmetry.
		
		\medskip
		Let us then compute 
		\begin{align*}
			2\bar r  \int_{\R^2}  E_0(y)  \frac{\partial \Gamma_0}{\pp y_1 }(y)\, dy\ &= \left(  \int_{\R^2}  U(y) y_1   \frac{\partial \Gamma_0}{\pp y_1 }(y)\, dy \right)
			(	2 \bar r \partial_r K(P) + 4 K (P)
		-\log8 +6-4 (1-\alpha \bar r ) \log \ve )  \\
			&+  \int_{\R^2}  U(y) y_1   \frac{\partial \Gamma_0}{\pp y_1 }(y)  \Gamma_0 \, dy.
		\end{align*}
		Then $
		\int_{\R^2}  E_0(y)  \frac{\partial \Gamma_0}{\pp y_1 }(y)\, dy\ =\ 0
		$
		is satisfied  if we
		choose $\alpha $  as in \eqref{defalpha}. In terms of $P$ we have
		\begin{equation}\label{defbeta}
			\alpha = \frac{1}{\bar r} + \frac{\beta (P)}{|\log \varepsilon|}, \quad \beta (P) =  -{ A+ \log 8 - 6 - 2 \bar r \pp_r K (P;P) - 4K(P;P) \over 4 \bar r }.
		\end{equation}
	With this choice of $\alpha $
		we can solve \eqref{appro}.
		Writing in polar coordinates $y = \rho e^{i\theta}$ we observe that $E_0$ has the form
		$E_0(y) = Q(\rho) \cos(\theta)$ with $Q(\rho) = O( \frac{\log \rho}{\rho^3})$ as $\rho\to \infty$.
		A direct computation yields  to
		\begin{equation}\label{barpsi}
			\bar \psi(y; P )  = {\varepsilon \over 2\bar r } \Gamma (\rho) y_1.
		\end{equation}
		Using now the whole expression of $\Psi_\ve$ as in \eqref{f}, we decompose
		\begin{align*}
			\ve^2 E[ \Psi_\ve ] &= \ve^2 E[\Psi_{P} ] + {\ve^3 \over \bar r } [ (\pp_{rr} + \pp_{zz} )  \bar \Gamma + f' (\Psi_{P} ) \bar \Gamma ]
			+  \ve^3 {3\over \bar r r} \pp_r \bar \Gamma \\
			&+\ve^2 [  f (\Psi_{P} + \ve \bar \Gamma ) -  f (\Psi_{P}  ) -  f' (\Psi_{P}  ) \ve \bar \Gamma ]
		\end{align*}
		where $\bar \Gamma = {r-\bar r \over  2 \bar r}  \Gamma ({x-P \over \ve} ) $ and $f$ is
		$$
		f(s)=  {8 \over \bar r} e^{-K (P;P)} \ve^{2 -\alpha  \bar r }
		f_0\Big( \frac{r^2}{\bar r}( s - \alpha |\log \ve|) \Big).
		$$
		For $|y|<{\delta \over \ve}$,
		\begin{align*}
			f' (\Psi_{P} )&= \frac{1}{\varepsilon^{2}  } U(y)
			\exp\Big[
			\frac{1}{2\bar r}\Big(\Gamma_0(y) - (4-4\alpha \bar r)
			\log\varepsilon
			+2\bar r \partial_r K (P) + 4 K (P)
			-\log 8
			\Big) 
			\, \varepsilon y_1
			\\
			& \quad 
			+\ve^2 \theta_1[P] (y) + \ve^3 \theta_2 [P] (y)
			\Big] \times (1+{\ve y_1 \over \bar r})^2\\
			&= {1\over \ve^2} U(y)  \Big[ 1+ \frac{1}{2 \bar r}\Big(\Gamma_0(y) - (4-4\alpha \bar r)
			\log\varepsilon + a(P) \Big) \ve y_1 +\ve^2 \theta_1[P] (y) + \ve^3 \theta_2 [P] (y) \Big] 
		\end{align*}
		where $\theta_1$, $\theta_2 $ satisfy \eqref{theta1n}, \eqref{theta2n}.
		We have
		\begin{align*}
			{\ve^3 \over \bar r} & [ (\pp_{rr} + \pp_{zz} )  \bar \Gamma + f' (\Psi_{P} ) \bar \Gamma ]=
			{1\over \bar r} [\Delta_y \bar \psi + U(y) \bar \psi ]
			\\ &+ U(y)\bar \psi   \Big[  \frac{1}{2\bar r}\Big(\Gamma_0(y) - (4-4\alpha \bar r)
			\log\varepsilon + a(P) \Big) \ve y_1 +\ve^2 \theta_1[P] (y) + \ve^3 \theta_2 [P] (y) \Big]  .
		\end{align*}
		Recall now that $\ve^2 E [\Psi_P]$ can be written as in \eqref{err0}. Since $\bar \psi$ solves \eqref{appro} and has the form \eqref{barpsi}, we get
		\begin{align*}
			\ve^2 E& [\Psi_{P} ] + {\ve^3 \over \bar r} [ (\pp_{rr} + \pp_{zz} )  \bar \Gamma + f' (\Psi_{P} ) \bar \Gamma ] = \Theta_\ve (y) 
		\end{align*}
		where $\Theta_\ve$ can be described as  in \eqref{Thetaex}. Besides, using the form of the function $\Gamma$, as described in \eqref{f}, one sees with a direct inspection that
		$$
		\ve^3 {3\over \bar r r} \pp_r \bar \Gamma +\ve^2 [  f (\Psi_{P} + \ve \bar \Gamma ) -  f (\Psi_{P}  ) -  f' (\Psi_{P}  ) \ve \bar \Gamma ]
		=\Theta_\ve (y),
		$$
		with $\Theta_\ve$ another function of the form \eqref{Thetaex}. This concludes the proof of \eqref{estin1}.

	\end{proof}

	\medskip

	\section{First approximate leapfrogging}\label{sec3}

The rest of the paper is devoted to find a solution to  Problem \eqref{leap00} with the properties described in Theorem \ref{teo}. With a little abuse of notation, from now on we will use the variable $t$ instead of $\tau$. Given $r_0 >0$, we look for  $(\Psi , W)$ solving
	\begin{equation}\label{leap0}
		\begin{aligned} 
			\left \{ \begin{aligned} &
				|\log\ve| \,  r \, \pp_t W  + \nn^\perp (r^2( \Psi - r_0^{-1} \, |\log \ve|) ) \cdot \nn W = 0  \quad {\mbox {in}} \quad  \Sigma  \times [0,T) \\ & -\Delta_5 \Psi = W, \quad \quad {\mbox {in}} \quad  \Sigma  \times [0,T)   \\
				&{\partial \over \partial r} \Psi (x,t) = 0  \quad {\mbox {on}} \quad \partial  \Sigma  \times [0,T), \quad  |\Psi(x,t)| \to 0 \quad \text{as } |x|\to \infty.  \end{aligned} \right.   
		\end{aligned} 
	\end{equation}
We recall that  $\Sigma =\{ x=(r,z) \, : \, r>0, \, z \in \R \}$, see \eqref{defSigma},  and
	\[
	\Delta_5 \Psi = \partial_{rr} \Psi + \frac{3}{r} \partial_{r} \Psi 
	+ \partial_{zz} \Psi .
	\]


	\medskip This Section is devoted to define a first approximate solution to \eqref{leap0}, given as a sum of approximate travelling vortex rings, as built in Section \ref{sect}. These travelling vortex rings are centered at different points, at relative distance $|\log \ve |^{-{1\over 2}}$ one from each other, all of them collapsing to $(r_0, 0)$ as $\ve \to 0$. Let us be more precise.

	\subsection{The parameter functions}
	\label{subsec41}
	
	Fix an integer $k \geq 2$ and consider points $P_j = P_j(t)$, for   $j\in \{1, \ldots , k\}$, which evolve in time and have the form
	\begin{equation}\label{point}
		\begin{aligned}
			P_j &= P_j (t) = (r_j (t) , z_j (t) ), \quad t \in [0,T) \quad {\mbox {with}} \\
			P_j &= {\bf P}_j + {\bf a}_j (t), \quad {\bf P}_j= P_0  +P_j^0 (t) + P^1_j (t) , \quad P_0:= 	(r_0 , 0).
		\end{aligned}
	\end{equation}
	Let us describe the different terms in the decomposition of $P_j$. The points $P_j^0 (t)  = (r^0_j (t), z^0_j (t) )$ are explicit and will be determined towards the end of this section, in the form
	\begin{equation}\label{b0}
		P_j^0 ={1\over \sqrt{|\log \ve |}} q_j + Q_j , \quad 
		\| Q_j \|_{L^\infty [0,T)}  + \| \dot Q_j \|_{L^\infty [0,T)}  \lesssim  {\log |\log \ve | \over |\log \ve |},
	\end{equation}
	where $q_j$ are the given solutions to the  {\it leapfrogging dynamics} \eqref{car1}.
	
	The points $P_j^1  (t) = (r_j^1 (t) , z_j^1 (t))$ in \eqref{point} will also be determined in the process of the construction of an approximate leapfrogging solution and they will satisfy
	\begin{equation}\label{b1}
		\| P^1_j \|_{L^\infty [0,T)}  + \| \dot{P^1_j} \|_{L^\infty [0,T)}  \lesssim  \ve^{2-\sigma} ,
	\end{equation}
	for some $\sigma >0$, small and independent of $\ve$.

	The points ${\bf a}_j  (t) = (a_{j1} (t) , a_{j2} (t))$ are free parameters to adjust at the end of our proof. For the moment, we  ask they are continuous functions in $[0,T]$ for which $\dot {\bf a}_j$ exists and such that 
	\begin{equation}\label{b11}
		\| {\bf a}_j \|_{C^1 [0,T)}   \lesssim  \ve^{3+\sigma} .
	\end{equation}
	The following notation will be used to identify the different sets of points in the decomposition of $P_j$ given in \eqref{point}
$$
		\begin{aligned}
			P&= (P_1 , \ldots , P_k ) , \quad
			P^b = (P_1^b , \ldots , P_k^b ), \quad b=0,1, \quad {\bf a} = ({\bf a}_1 , \ldots , {\bf a}_k ).
		\end{aligned}
$$
Since the relative distance between two points is of order $|\log \ve |^{-{1\over 2}}$
we have 
$$ \{ x \, : \, |x-P_i| <|\log \ve |^{-1} \} \cap \{ x \, : \, |x-P_\ell| <|\log \ve |^{-1} \} = \emptyset
$$ for all $\ve $ small and $i \not= \ell$.
Besides, 	under assumptions \eqref{point}, \eqref{b0}, \eqref{b1} and \eqref{b11}, we have that
	$$
	\| P_j (t) - (r_0, 0) \|_{L^\infty [0,T]} ={c \over \sqrt{|\log  \ve |}},  \quad \forall j =1, \ldots , k,
	$$
	for some  constant $c$ uniformly bounded and bounded away from $0$ as $\ve \to 0$.
	
	\medskip
	
	Given the points $P_j$ as described in \eqref{point}, we introduce  positive functions 
	$\ve_j = \ve_j (t) >0$ such that, for all $j=1, \ldots , k$,
 \begin{equation}\label{ass11}
		r_j (t) \,  \ve_j^2 (t) = r_0 \,  \ve^2 \quad \foral \quad t \in [0,T),
	\end{equation}
where $r_j(t)$ is the first component of the point $P_j(t)$.

	From \eqref{point}--\eqref{b11} we get that for all $j=1, \ldots , k$,
	\begin{equation}\label{b2}
		\| \ve_j - \ve \|_{L^\infty [0,T)}  +  \| \dot \ve_j \|_{L^\infty [0,T)}  \lesssim  \ve |\log \ve|^{-{1\over 2}}, \quad \forall \quad j=1, \ldots , k, \quad \ass \ve \to 0.
	\end{equation}
	For such $\ve_j(t)$ and $P_j (t)$, let  $\alpha_j$ be given as in \eqref{defalpha},   Proposition \ref{f1} so that
	$$
	\alpha_j  [P_j ] (t) = {1\over r_j (t)} + {\beta (P_j (t) )\over  |\log \ve_j (t)|},
	$$
	where $\beta (s)$ is the smooth function defined as in \eqref{defbeta}.
	Observe that
	$$
	\alpha_j [P_0] (t) = {1\over r_0} + {\beta (P_0 )\over  |\log \ve|}
	$$ is constant in time.
We also have 
	\begin{equation}\label{alphajalpha0}
		r_0^{-1}  - \alpha_j [P_j] (t) = {r_j - r_0 \over r_0 r_j } - {\beta (P_j) \over |\log \ve_j|},
	\end{equation}
	and $ \| r_0^{-1} - \alpha_j \|_{L^\infty [0,T)} = O(  |\log \ve|^{-{1\over 2}} )$ as $\ve \to 0$.

	\subsection{The function $H^0$ and definition of the very first approximation}\label{subsec42}
	
	\medskip
	
	For any $j=1, \ldots , k$,  we define 
	\begin{equation}\label{defWj}
	\Psi_j^0  (x,t) := \Psi_{\ve_j (t)} [P_j(t)] (x,t), \quad		W_j^0  (x,t):=  - \Delta_5 \Psi_{\ve_j (t)} [P_j (t) ] (x), 
	\end{equation}
	where $\Psi_\ve [P] (x)$ is the approximate travelling vortex ring introduced  in 
	\eqref{f}, Proposition \ref{f1}. Since we are assuming that the point $P_j(t)$ evolves with time, the functions $	\Psi_j^0$ and $	W_j^0$ also depend on the time variable $t \in [0,T]$, and their dependence on time is through $P_j$ (and $\ve_j$).
	Writing
	\begin{equation}\label{fee2}
	\Psi_j^0 (x,t) = {1\over r_j}  \psi^0_j (y) , \quad y={x-P_j \over \ve_j}
	\end{equation}
	from \eqref{defWj} we get
	\begin{equation}\label{fe2}
		W_j^0 (x,t) = {1\over \ve_j^2 r_j} w_j^0 (y), \quad {\mbox {with}} \quad -\Delta_{5,j} \psi_j^0 =
		w_j^0
	\end{equation}
	where
	\begin{equation} \label{d5}
		\Delta_{5,j} = \pp_{y_1}^2 + \pp_{y_2}^2 +{3\ve_j\over r_j + \ve_j y_1} \, \pp_{y_1}.
	\end{equation}
	Let  $\delta >0$ be as in Proposition \ref{f1}. We can assume $\delta <r_0$. We have that in the region $|x-P_j |<\delta$
	\begin{equation}
		\label{Ubar1}
		\begin{aligned}
			w_j^0 (y)&= U(y) \left( 1+{\ve_j y_1 \over 2r_j} (\Gamma_0 + \bar A + \Gamma) + \ve^2 \theta_1 [P_j] (y) + \ve^3 \theta_2 [P_j] (y)  \right), \quad y= {x-P_j \over \ve_j}, 
		\end{aligned}
	\end{equation}
	where $\theta_1$ and $\theta_2$ are functions also described by \eqref{theta1n} and \eqref{theta2n} respectively. This  expansion has been obtained in Section \ref{sect}, formula \eqref{Ubar}.
	
\medskip
The starting point of our construction is to assume that the vorticity of a leapfrogging of vortex rings is at main order the sum of the vorticities of vortex rings.  We do the same with the stream functions, which we then multiply by a cut off function to make it $0$ at infinity.

	Having introduced the points
	$
	P = (P_1, P_2, \ldots, , P_k ) ,
	$
	define
	$$
	\bar \Psi^0 (x,t) = \eta (x) \, \sum_{j=1}^k  \Psi_j^0  (x,t), \quad W^0 (x,t) = \eta (x) \, \sum_{j=1}^k W_j^0   (x,t),
	$$
	where $\eta$ is the smooth cut-off function given by 
	$$
	\eta (x) = \eta_1 \left({4|x-(r_0, 0) |\over \delta } \right),
	$$
	with $\eta_1$ as in \eqref{cutoff}.  We then immediately see that $|\bar \Psi^0 (x,t)| \to 0 $ as $|x| \to \infty$ and that ${\pp \over \pp r} \bar \Psi^0 (x,t) = 0 $ on $\pp \Sigma $, for any $t \in [0,T]$. On the other hand we no longer have that $-\Delta_5 \bar \Psi^0 = W^0$.
	We shall then slightly modify $\bar \Psi^0$ by a function $H^0$ to have, for any $t \in [0,T]$ $$\Delta_5 \left( \bar \Psi^0 + H^0 \right)  (x,t) + W^0 (x,t) = 0, \quad x \in \Sigma. $$ For this purpose, consider the linear problem		\begin{equation}\label{ext0}
		\Delta_5  \psi + h =0, \quad  {\mbox {in }} \Sigma , \quad {\partial  \psi \over \partial r}  = 0 \quad {\mbox {on}} \, \,  \pp \Sigma , \quad \psi (x ) \to 0, \, \,  \ass |x|\to \infty,
	\end{equation}
	for a smooth function $h$ satisfying 
	\begin{equation}\label{decayh}
		|h (x)| \leq {C\over 1+ |x|^{2+\nu}}, 
	\end{equation}
	where $C>0$. Recall that $\Delta_5 = {\pp^2 \over \pp r^2} + {3\over r} {\pp \over \pp r} + {\pp^2 \over \pp z^2}$, for $x= (r,z) \in \Sigma$.

	We have
	\begin{lemma}\label{ll1} Assume $h$ satisfies \eqref{decayh}, with $\nu >0$. Then there exist a solution $\psi = {\mathcal T} (h) $ to \eqref{ext0} and a constant $C_1>0$ such that
		$$
		(1+ |x|) |\nabla \psi (x)|+	|\psi (x) | \leq {C_1 \over 1+ |x|^{\min (\nu , 3)}}.
		$$
	\end{lemma}
	\begin{proof}
		Recalling that the Laplacian of a radially symmetric function in $\R^n$ is 
		$$\Delta_Y = {d^2 \over d s^2 } + {n-1\over s} {d \over d s}, \quad Y= (Y_1, \ldots , Y_n), \quad s= \sqrt{ Y_1^2 + \ldots Y_n^2}$$
		we interpret the differential operator in \eqref{ext0} as the Laplacian in $\R^5$ and recast Problem \eqref{ext0} in $\R^5$.
		Define
		$$
		\Psi (Y_1, \ldots, Y_4, z) = \psi ( \sqrt{ Y_1^2 + \ldots Y_4^2},z), \quad r=\sqrt{ Y_1^2 + \ldots Y_4^2}, \quad  Y_5=z
		$$
		and
		$$
		H (Y_1, \ldots, Y_4, z) = h ( \sqrt{ Y_1^2 + \ldots Y_4^2},z), \quad  r=\sqrt{ Y_1^2 + \ldots Y_4^2}, \quad \quad Y_5=z.
		$$
		To solve	Problem \eqref{ext0} we find bounded (non-singular) solutions $\Psi$ to 
		$$
		\Delta_{\R^5} \Psi  + H =0, \quad  {\mbox {in }} \R^5 , \quad \Psi (Y ) \to 0, \quad \ass |Y|\to \infty.
		$$
		We define $\Psi$ using the Newtonian potential in $\R^5$ as
		$$
		\Psi (Y) = {1\over 15 \omega_5 }\int_{\R^5} {1\over |Z-Y|^3} H(Z) \, dZ,
		$$
		with $\omega_5$  the volume of the unit $5$-ball.  Moreover
			$
		\nabla \Psi (Y) = - {3\over 15 \omega_5 }\int_{\R^5} {(Z-Y) \over |Z-Y|^5} H(Z) \, dZ.
		$
		From \eqref{decayh} we get that $(1+ |Y|^{2+ \nu} ) |H(Y) | \lesssim 1$. Estimating the above integrals splitting them in the region $|Z|< {|Y| \over 2} $ and its complement we get
		$$
		(1+ |Y|) |\nabla \Psi (Y)|+	|\Psi (Y) | \lesssim {1\over 1+ |Y|^{\min (3,\nu)}} .
		$$
		Going back to the original variables $x= (r,z) \in \Sigma$ we get the required estimates.
	\end{proof}
	
	
	\medskip
For any $t \in [0,T]$,	we denote by  $ H^0 [P] = {\mathcal T} \left( \Delta_5 \bar \Psi^0 + W^0 \right) $ the  solution to 
	\begin{equation}\label{defH0}
		\Delta_5  H^0 + \Delta_5 \bar \Psi^0 + W^0 =0 \quad  {\mbox {in }} \Sigma , \quad {\partial  H^0 \over \partial r}  = 0 \quad {\mbox {on }} \pp \Sigma 
	\end{equation}
	with $H^0 [P] (x,t)\to 0$ as $|x|\to \infty$, for all $t \in [0,T]$.
	
	A direct computation gives
	\begin{equation}\label{fe1}
		\begin{aligned}
			\Delta_5 \bar \Psi^0 + W^0&= 
   (\Delta_5  \eta (x) ) \sum_{j=1}^k  \Psi_j^0  + 2 \sum_{j=1}^k  \nabla_{r,z} \eta \cdot \nabla_{r,z}  \Psi_j^0 .
		\end{aligned}
	\end{equation}
	The function $\Delta_5 \bar \Psi^0 + W^0$
	is smooth, uniformly bounded as $\ve \to 0$ and  with compact support.
	From Lemma \ref{ll1} and \eqref{fe1} we get that
	$$
	(1+ |x|) |\nabla H^0 [P] | +	| H^0 [P] | \lesssim {1 \over 1+ |x|^3}, \quad 
	$$
	uniformly for $t \in [0,T]$ as $\ve \to 0$.	
	
	\medskip
For
	$
	P = (P_1, P_2, \ldots, , P_k ) 
	$, we can now define the first approximate solution to \eqref{leap0} to be
	\begin{equation}\label{appro1}
		\begin{aligned}
			\Psi^0 [P]  &= \eta (x) \sum_{j=1}^k  \Psi_j^0 + H^0 [P], \quad
			W^0 [P] =  \eta (x)\sum_{j=1}^k W_j^0 .
		\end{aligned}
	\end{equation}
	
	\subsection{The very first error} \label{subsec43}
	We recast Problem \eqref{leap0}  as the problem of finding 
	$(W,\Psi)$ with
	$$
	S_1 (W,\Psi ) = S_2 (W,\Psi) = 0 \quad {\mbox {in}}
	\quad \Sigma \times [0,T), \quad {\partial  \Psi \over \partial r}  = 0 \quad {\mbox {on }} \pp \Sigma \times [0,T).
	$$
	Here $S_1$ and $S_2$ are the Euler operators introduced in \eqref{defS}.
	The leapfrogging of vortex rings are then solutions $(W,\Psi)$ with $W$ and $\Psi$ close respectively to $W^0$ and $\Psi^0$, and $|\Psi (x,t)| \to 0 $ as $|x| \to \infty$. 
	By construction what we have so far is that  
$$ \begin{aligned}
			S_2 (W^0 , \Psi^0 ) &= 0 \quad \Sigma \times [0,T), \quad {\partial \over \partial r} \Psi^0 (r,z,t) = 0 \quad \partial \Sigma \times [0,T),\\
			& |\Psi^0(x,t) | \to 0 , \quad \ass |x| \to \infty.
	\end{aligned}$$
	We shall now describe $S_1 (W^0, \Psi^0 )$. We can  write
	\begin{equation}\label{gru}
		\begin{aligned}
			S_1(W^0, \Psi^0 ) &= \sum_{j=1}^k E_j^0, \quad {\mbox {where}} \\
			E_j^0 &= |\log\ve| \,  r \, \eta \, \pp_t W^0_j  + \nn^\perp (r^2( \Psi^0 - r_0^{-1} |\log \ve|) ) \cdot \nn (\eta W_j^0) .
		\end{aligned}    
	\end{equation}
	We start with the following general remark.
	
	\medskip
	
	\begin{remark}\label{r1}	For a function $W(x,t)$ given in the form
	$$
			W(x,t) 
			= \frac{1}{r_j \varepsilon_j^2}
			w\Bigl( \frac{x-P_j}{\varepsilon_j},t\Bigr),
$$
		for some function $w (y,t)$, $y={x-P_j \over \ve_j}$ we have
		\begin{equation}\label{Sexp}
			\ve_j^4 |\log\ve| \, r \, \pp_t W=  \ve_j^2 |\log \ve| (1+{\ve_j \over r_j} y_1) \pp_t w  - \ve_j |\log\ve| \nabla w \cdot \dot{ P_j} + B_0 (w) ,
		\end{equation}
			where
	\begin{equation}\label{defB0}
		\begin{aligned}
			B_0 (\phi ) = 
			-\ve_j  \dot \ve_j (1+{\ve_j \over r_j} y_1)    \nabla \phi \cdot y  -{\ve_j^2 \over r_j} y_1 \dot P_j \cdot \nabla \phi
		\end{aligned}
	\end{equation}
		An equivalent expression, which will be useful in the sequel, is
			\begin{equation}\label{Sexp1}
			\begin{aligned}
				\ve_j^4 |\log\ve| r \pp_t W	= |\log \ve| &\Bigg[ - \ve_j \nabla w \cdot \dot P_j + \ve_j^2 \pp_t w  +\ve_j \dot \ve_j (y_1 \pp_1 w - y_2 \pp_2 w)  \\
				&\quad -{\ve_j^2 \over r_j}\dot z_j  y_1 \pp_2 w + {\ve_j^3 \over r_j} y_1 \pp_t w - {\ve_j^2 \over r_j} \dot \ve_j y_1 \nabla w \cdot y \Bigg].
			\end{aligned}
		\end{equation}
		
		\medskip{\bf Proof of \eqref{Sexp}}. \ \ 
		For $x=\ve_j y +P_j$, 
		\begin{align*}
			\ve_j^4 |\log\ve|\,  r \, \pp_t W&= |\log \ve| (1+{\ve_j \over r_j} y_1) \Biggl[\ve_j^2 w_t - ({\dot r_j \over r_j} +{2\dot \ve_j \over \ve_j}) \ve_j^2 w 
			- \ve_j \dot \ve_j \nabla w \cdot y - \ve_j \nabla w \cdot \dot P_j \Biggl].
		\end{align*}
		From our first assumption \eqref{ass11} on the parameters $r_j$ and $\ve_j$ we get
		\begin{equation}\label{ass1}
			\frac{\dot r_j}{r_j} = - 2 
			\frac{\dot \varepsilon_j }{\varepsilon_j} , \quad \foral t
		\end{equation}
		which gives \eqref{Sexp}. 
		
			\medskip{\bf Proof of \eqref{Sexp1}}. \ \ 
		This expression  follows from observing that
		\begin{align*}
			\ve_j^4 |\log\ve| r \pp_t W= |\log \ve| &\Bigg[ - \ve_j \nabla w \cdot \dot P_j + \ve_j^2 \pp_t w - \ve_j \dot \ve_j \nabla w \cdot y  \\
			&- {\ve_j^2 \over r_j} y_1 \dot P_j \cdot \nabla w  + {\ve_j^3 \over r_j} y_1 \pp_t w - {\ve_j^2 \over r_j} \dot \ve_j y_1 \nabla w \cdot y \Bigg]
		\end{align*} 
		and (again \eqref{ass1} for $P_j (t) = (r_j (t) , z_j (t))$)
		\begin{align*}
			&- \ve_j \dot \ve_j \nabla w \cdot y  - {\ve_j^2 \over r_j} y_1 \dot P_j \cdot \nabla w = \ve_j \dot\ve_j (y_1 \pp_1 w - y_2 \pp_2 w) -{\ve_j^2 \over r_j}\dot z_j  y_1 \pp_2 w.
		\end{align*}
	\end{remark}
	
	\medskip	
Recalling \eqref{Ubar1} and	using \eqref{Sexp} and \eqref{fe2} we easily get that the first term in $E_j^0$ in \eqref{gru} is given by 
	$$
		\begin{aligned}
			\ve_j^4 |\log\ve| \, r \,  \eta \, \pp_t W^0_j&= - \ve_j  |\log \ve| (1+{\ve_j \over r_j} y_1)  \, \eta \,  \dot P_j \cdot \nabla w^0_j \\
			&+  \ve_j |\log \ve| (1+{\ve_j \over r_j} y_1) \, \eta \,  \left[ \ve_j \pp_t w^0_j - \dot \ve_j \nabla w^0_j \cdot y \right].
		\end{aligned}
	$$
	We  now pass to the second term in $E_j^0$ given in \eqref{gru}. It is convenient to use the decomposition
	\begin{align*}
	    r^2 \, \left( \Psi^0 - r_0^{-1} |\log \ve | \right) &= r^2 \, \left( \eta \Psi_j^0 - \alpha_j |\log \ve_j | \right) \\
	    &+r^2 \, \left( \alpha_j |\log \ve_j|  - r_0^{-1} |\log \ve | \right) + r^2 \, \left( H^0+  \eta \sum_{\ell \not= j} \Psi^0_\ell \right).
	\end{align*}
Consider the region around the point $P_j$ defined by $\{ x \, : \, |x-P_j| <|\log \ve |^{-1} \} $. In this region one has  that $\eta (x) = 1$, for all $\ve $ small, as a consequence of the assumptions on the points $P_j$. 
	We write, for $y= {x-P_j \over \ve_j} $, $|y|< \ve_j^{-1} |\log \ve |^{-1}$,
	\begin{align*}
		\ve^4_j    \nn^\perp (r^2( \Psi^0 - r_0^{-1} |\log \ve|) ) \cdot \nn W_j^0 & =\nabla^\perp \left( (1+{\ve_j \over r_j} y_1)^2 (\psi_j^0  - r_j \alpha_j  |\log \ve_j | )  \right) \cdot \nabla  w^0_j\\
		&+   \nabla^\perp (\tilde \varphi_j (\ve_j y + P_j ; P) \, \, )  \,   \cdot \nabla w^0_j,
	\end{align*}
	where
	\begin{equation}\label{varphij1}
		\begin{aligned}
			\tilde \varphi_j (x;P)&= \tilde \varphi_j (\ve_j y + P_j ; P)= r_j (1+ {\ve_j \over r_j} y_1 )^2 \times \\
			&  \Biggl( \alpha_j |\log \ve_j| -  r_0^{-1} |\log \ve|  \, \, + H^0 [P] (\ve_j y + P_j ) + \sum_{\ell \not= j} {1 \over r_\ell} \,  \,   \psi_\ell^0 ({\ve_j \over \ve_\ell } y + {P_j -P_\ell \over \ve_\ell  }) \Biggl),
		\end{aligned}
	\end{equation}
  	with $H^0$  the correction introduced in \eqref{defH0}. 
In this region we have
\begin{align*}
    \ve_j^4 \, E_j^0 (P_j + \ve_j y) &= - \ve_j  |\log \ve| (1+{\ve_j \over r_j} y_1)  \,  \dot P_j \cdot  w^0_j +  \ve_j |\log \ve| (1+{\ve_j \over r_j} y_1) \,  \left[ \ve_j \pp_t w^0_j - \dot \ve_j \nabla w^0_j \cdot y \right] \\
   &+ \nabla^\perp \left( (1+{\ve_j \over r_j} y_1)^2 (\psi_j^0  - r_j \alpha_j  |\log \ve_j | )  \right) \cdot \nabla  w^0_j+   \nabla^\perp (\tilde \varphi_j (\ve_j y + P_j ; P) \, \, )  \,   \cdot \nabla w^0_j.
\end{align*}
Reordering the terms, 
$E_j^0$ takes the form
	\begin{align*}
		\ve_j^4 \, E^0_j (P_j + \ve_j y)&= \nabla^\perp {\mathcal R}_j  (y,t) \cdot \nabla w^0_j +  \ve_j^2 |\log \ve| (1+{\ve_j \over r_j} y_1)  \pp_t w^0_j  +|\log \ve|B_0 (w_j^0) \\
		&+ \nabla^\perp \left( (1+{\ve_j \over r_j} y_1)^2 (\psi_j^0  - r_j \alpha_j  |\log \ve_j | )  \right) \cdot \nabla  w^0_j, \quad y= {x-P_j \over \ve_j}
	\end{align*}
	where $B_0$ is defined in \eqref{defB0} and
\begin{equation}\label{errejey}
		{\mathcal R}_j  (y,t)= \ve_j |\log \ve| \,  \dot P_j^\perp \cdot y  + \tilde \varphi_j ( \ve_j y+ P_j;P ) - \tilde \varphi_j (P_j ;P).
\end{equation}
	In the same region, the contribution of $E_i^0$ for $i \not = j$ has the form
	\begin{align*}
		\ve_j^4 \, E^0_i (P_j + \ve_j y)&= {\ve_j^4 \over \ve_i^4} \, \ve_i^4 \, E^0_i (({\ve_j \over \ve_i} {P_j - P_i \over \ve_j} + {\ve_j \over \ve_i }y ) =O( {\ve^5 |\log \ve |^2  \over 1+ |y|} )  \quad \quad y= {x-P_j \over \ve_j}.
	\end{align*}
From the result in Proposition \ref{f1} we recognize that
		\begin{equation}\label{fee22}
			\begin{aligned}
				&(1+{\ve_j \over r_j} y_1)^2 (\psi_j^0  - r_j \alpha_j  |\log \ve_j | )
				= 
				\Gamma_0(y) - (4-\alpha_j r_j ) \log\varepsilon_j - \log 8 + K (P_j;P_j)
				\\
				& \quad 
				+ \frac{\ve_j \, y_1}{2r_j }\Big(\Gamma_0(y) +\bar A +\Gamma (y)
				\Big) \, 
				+\ve^2 \theta_1[P_j] (y) + \ve^3 \theta_2 [P_j] (y) ,
			\end{aligned}
		\end{equation}
		where $\bar A$ is the constant defined in \eqref{barA},
 $\Gamma = \Gamma (y)$ as in \eqref{defGamma}, $\theta_1 $ and $\theta_2$ have the form described in \eqref{theta1n} and \eqref{theta2n}. Moreover
\begin{align*}
	\nabla^\perp \left( (1+{\ve_j \over r_j} y_1)^2 (\psi_j^0  - r_j \alpha_j  |\log \ve_j | )  \right) \cdot \nabla  w^0_j & = r_j^2 \, \ve_j^4 \, S_{\alpha_j} (W_j^0; \Psi_j^0 ) (\ve_j y + P_j)
	\end{align*}
where the operator $S_\alpha$ is given by \eqref{uu}. From estimates \eqref{ee1}  on this term, we conclude that, for $\rho = |y|= |{x-P_j \over \ve_j}| < \ve_j^{-1}  |\log \ve |^{-1}$,   
	\begin{equation}\label{error1}
		\begin{aligned}
			\ve_j^4 \,  & S_1(W^0, \Psi^0 ) (P_j + \ve_j y) 
			=   \ve_j^4 \sum_{i=1}^k E^0_i (P_j + \ve_j y ) \\
			&= \nabla^\perp {\mathcal R}_j  (y,t) \cdot \nabla w^0_j 
			+  \ve_j^2 |\log \ve| (1+{\ve_j \over r_j} y_1)   \pp_t w^0_j + |\log \ve | B_0 ( w^0_j  ) \\
			&+  { \ve^2  |\log \ve| \over 1+ |y|^4}  E_2(\rho,\theta,t,\varepsilon)+ {\ve^3  |\log \ve |^2 \over 1+|y|^3}  [ E_1(\rho,\theta,t,\varepsilon) + E_3(\rho,\theta,t,\varepsilon)] 
			+ O\left( {\ve^4 |\log \ve |^2  \over 1+ |y|^2} \right)  \, 
		\end{aligned}
	\end{equation}
	as $\ve \to 0$. 
	To get this estimate we have used \eqref{ee1}. The functions $E_j(\rho,\theta, t, \varepsilon)$ have the following form
$$
		\begin{aligned} 
			E_j(\rho,\theta,t,\varepsilon)
			= E_{j,1}(\rho,\varepsilon, P (t) ) \cos(j \theta)
			+ E_{j,2}(\rho,\varepsilon, P (t)  ) \sin(j \theta)
		\end{aligned}
$$
	where
	\begin{align*}
		\sum_{i=0}^2 \Bigl|
		(1+\rho)^i \frac{\partial^i E_{j,\ell}}{\partial \rho^i}
		\Bigr| + \Bigl|
		\nabla_P E_{j,\ell }
		\Bigr| \lesssim 1 \quad \ell = 1,2, \quad \ass \ve \to 0,
	\end{align*}
	uniformly for $t \in [0,T)$.

	\subsection{Dynamics for $P^0_j$ and  local reduction of the  first error.} \label{subsec44} \ \ It is possible to reduce the size of the error term \eqref{error1} in each of the regions $|{x-P_j \over \ve_j}| < \ve_j^{-1}  |\log \ve |^{-1}$, $j=1, \ldots , k$, by choosing properly  the points $P_j^0$ as in \eqref{b0}. 
	Write the function $\tilde \varphi_j$ in \eqref{varphij1} as 
	$$
	\tilde \varphi_j (x;P)= \varphi_j (x;P) + \ve^2 \theta_{j\ve_j} (x;P)
	$$
	with
	\begin{equation}\label{varphij}
		\varphi_j (x;P)= {r^2\over r_j} \Biggl( \alpha_j |\log \ve_j| -  r_0^{-1} |\log \ve|  \, \, + H^0 (x;P ) + \sum_{\ell \not= j} {1 \over r_\ell} \,  \,  \bar  \psi_\ell^0 (x;P_\ell ) \Biggl)\end{equation}
	\begin{align*}
		\bar \psi_\ell^0 (x;P_\ell) &= 	 \log \frac 1{  |x-P_\ell|^4} \big (1  -  \frac 3{2r_\ell} (r-r_\ell)   + H(x;P_\ell) \big )  + K(x; P_\ell ) + \ve_\ell \Gamma (x; P_\ell) \\
		&{\mbox {where}} \quad \ve_\ell \Gamma (x;P_\ell) = {r-r_\ell \over 2 r_\ell } \Gamma ({x-P_\ell \over \ve_\ell} ) \quad {\mbox {and}}\\
		\theta_{j\ve_j} (x;P) &= {1\over \ve^2} \sum_{\ell \not= j} {r^2\over r_j r_\ell} \log \frac {|x-P_\ell |^4}{  (\ve_\ell^2 +|x-P_\ell|^2)^2} \big (1  -  \frac 3{2r_\ell} (r-r_\ell)   + H(x;P_\ell) \big ) .
	\end{align*} 
We refer to \eqref{eqH}, \eqref{eqK} and \eqref{f} for the definition of $H$, $K$ and $\Gamma$.	We Taylor expand
	\begin{equation}\label{ta}
		\begin{aligned}
			\ttt \vp_j &( P_j+   \ve_j y ; P)   = \  \ttt \vp_j ( P_j ; P ) + \ve_j \nn_x\vp_j ( P_j ; P )\cdot y  + \bar \vp_j  \\
			\bar \vp_j &= \sum_{k=2}^4 \frac{\ve_j^k}{k!}
			D_x^k \vp_j(P_j; P) [y]^k  \, +\,  \ve_j^3 D_x \theta_{j  \ve_j} (P_j; P ) [y]+{\ve_j^4 \over 2} D_x^2 \theta_{j \ve_j } (P_j ; P)[y]^2+ Q(\ve_j y , P)
		\end{aligned}
	\end{equation}
	where $Q(z,\zeta)$ is a function smooth in its arguments that satisfies
	$$
	|\nn _zQ(z , \zeta)|\ \le \ C\ve |z|^4.
	$$
Recalling the definition of ${\mathcal R}_j$ in \eqref{errejey} we can write
	\begin{align*}
	    \nabla^\perp 	{\mathcal R}_j  (y,t) &= 	- \ve_j 	|\log \ve| \,  \dot P_j   + \ve_j \nabla_x^\perp \varphi_j ( P_j; P) + \nabla^\perp \bar \vp_j,
	\end{align*}
	and we choose the points $P^0_j$ in \eqref{point}-\eqref{b0}
	to satisfy the ODEs system
	\begin{equation}\label{din1}
	-	|\log \ve| \,  \dot{ P_j^0}   + \nabla_x^\perp \varphi_j ( P_0+ P_j^0; P_0 + P^0) = 0,   \quad j=1, \ldots k.
	\end{equation}
Here $P_0 = (r_0, 0)$ and $P^0 = (P_1^0 , \ldots , P_k^0)$. We use \eqref{alphajalpha0} to write
	$$
	\nabla_x \varphi_j^\perp ( P_j ; P) = \Theta_j (P) =\Theta_j^0 ( P) + \Theta_j^1 ( P) 
	$$
	where for $P_j = (r_j , z_j)$,
	\begin{align*}
		\Theta_j^0 ( P)  &= -4 \sum_{\ell \not= j} \,  \,   \frac{(P_j-P_\ell)^\perp }{ |P_j-P_\ell|^2}  -2 \, {r_j-r_0 \over r_0^2  } \,  |\log \ve | \, {\bf e_2}.
	\end{align*}
	Recall now the form of the points $P_j^0$ from \eqref{b0}
$$
		P_j^0 ={1\over \sqrt{|\log \ve |}} q_j + Q_j.
$$
	Here $q_j$ are the given solutions to the  {\it leapfrogging dynamics} \eqref{car1}, which gets rewritten as 
	$$
	- |\log \ve |^{1\over 2} \dot q_j + \Theta_j^0 (P_0+ {1\over \sqrt{|\log \ve |}} q) = 0 \quad t \in [0,T].
	$$
	We choose $Q_j$ to solve the initial value problem
	\begin{align*}
	-\dot Q_j &+\tilde \Theta_j (t,Q) = 0 \quad t \in [0,T]\\
	\tilde \Theta_j (t,Q) &:= {1\over |\log \ve |} \left[ \Theta_j^0 (P_0 + P^0 ) -\Theta_j^0 (P_0 + {1\over \sqrt{|\log \ve |}} q) \right] + {1\over |\log \ve |}  \Theta_j^1 ( P_0 + P^0) ,   \\
	Q_j (0) &= 0.
	\end{align*}
	A direct computation gives that
	$$
	\| {1\over |\log \ve |}  \Theta_j^1 ( P_0 +{1\over \sqrt{|\log \ve |}} q) \|_{L^\infty ( [0,T] )} \lesssim {\log |\log \ve | \over |\log \ve |}
	$$
	The functions $\tilde \Theta_j$ are Lipschitz continuous in $Q$ in the set
		$$
\, \| Q_j \|_{L^\infty [0,T)} \lesssim  {\log |\log \ve | \over |\log \ve | },
	$$
	 and continuous in $t$, for $t \in [0,T]$. Standard ODEs theory ensures the existence of a solution $(Q_1, \ldots , Q_k)$ satisfying the bounds \eqref{b0}.

	\medskip
	\medskip
	This choice for $P^0$  automatically reduces from $\ve $ to $\ve^2$ the size of the term $ \nabla^\perp {\mathcal R}_j  (y,t) \cdot \nabla w^0_j$  in the first line of  the error of approximation described by \eqref{error1}.   
	Let us now analyze the terms in \eqref{error1} given by
	$$
	\ve_j |\log \ve| (1+{\ve_j \over r_j} y_1)  \left[ \ve_j \pp_t w^0_j - \dot \ve_j \nabla w^0_j \cdot y \right] -\ve_j^2 |\log \ve | y_1 \dot P_j \cdot \nabla w^0_j.
	$$
	They	have the same form as the terms in the following  line in \eqref{error1}, with the only difference that  in this case  the functions $E_j (\rho, \theta, t, \ve)$  depend on $\dot P$, not only on $P$. This fact is consequence of  \eqref{Sexp1} and of our assumptions on the points $P_j$.

	We are thus in a position to conclude that if we choose the points $P^0$ of the form \eqref{b0} to satisfy \eqref{din1}, the error of approximation \eqref{error1} in each region $|x-P_j| <|\log \ve|^{-1} $ can be described as, for $y={x-P_j \over \ve_j}$, $\rho = |y|$
	\begin{equation}\label{error11}
		\begin{aligned}
			\ve_j^4S_1(W^0, \Psi^0 )& (P_j + \ve_j y) 
			=   \nabla^\perp {\mathcal R}_j (y,t;P) \cdot \nabla w_j^0 +  { \ve^2 |\log \ve |  \over 1+ |y|^4}  E_2(\rho,\theta,t,\varepsilon)\\
			&+ {\ve^3 |\log \ve |^2  \over 1+|y|^3}  [ E_1(\rho,\theta,t,\varepsilon) + E_3(\rho,\theta,t,\varepsilon)] + O\left( {\ve^4 |\log \ve |^2  \over 1+ |y|^2} \right)  \, 
		\end{aligned}
	\end{equation}
	as $\ve \to 0$, 	where
	the functions $E_j(\rho,\theta, t, \varepsilon)$ now depends also on $\dot P$, and  have the form 
	\begin{equation}\label{Ekkn}
		\begin{aligned} 
			E_j [P, \pp_t P ] (\rho,\theta,t,\varepsilon)&
			= E_{j,1}(\rho,\varepsilon, P (t) , \dot P (t)) \cos(j \theta)
			\\
			&+ E_{j,2}(\rho, \varepsilon, P (t) , \dot P (t) ) \sin(j \theta),
		\end{aligned}
	\end{equation}
	where
	\begin{align*}
		\sum_{i=0}^2 \Bigl|
		(1+\rho)^i \frac{\partial^i E_{j,\ell}}{\partial \rho^i}
		\Bigr| + \Bigl|
		\nabla_P E_{j,\ell }
		\Bigr| +   \Bigl|
		\nabla_{\dot  P} E_{j,\ell }
		\Bigr|\lesssim 1 \quad \ell = 1,2, \quad \ass \ve \to 0.
	\end{align*}
	Besides,
	\begin{equation}\label{defRj}
		\begin{aligned}
			{\mathcal R}_j (y,t;P)&= \ve_j  |\log \ve|  {d \over dt} (P_j -P_j^0)^\perp \cdot y  + \tilde \varphi_j (\ve_j y + P_j ;  P )  - \tilde \varphi_j (P_j;P)  \\
			&-\nabla_x \varphi_j (\bar P_j^0; \bar P^0) \, \ve_j y .
	\end{aligned}\end{equation}
	where $ \bar P^0_j   (t) =  P_0  + P_j^0 (t)$.

\medskip

		In the complementary region $\cap_{j=1}^k \{ x \in \Sigma \, : |x-P_j |>|\log \ve|^{-1}  \, \}$, it is convenient to describe the error in the original variable $(x,t)$. Given the expression for $S_1 (W^0, \Psi^0)(x,t)$ from \eqref{gru}, this error is automatically zero in the region $|x-P_0| >{\delta \over 2}$. Thus we consider $|x-P_0| \leq {\delta \over 2}$. Taking possibly $\ve$ smaller, for points in this region we have that
  $$
  |x-P_j | < {\delta }, \quad \forall j=1, \ldots , k.$$
  Hence the expansions in Proposition \ref{f1} and \eqref{Ubar}  hold true. A direct computation gives that in the complementary region $\cap_{j=1}^k \{ x \in \Sigma \, : |x-P_j |>|\log \ve|^{-1} \,  \}$ one has
	\begin{equation}\label{gru1}
		S_1 (W^0, \Psi^0)(x,t) =  \sum_{j=1}^k {O(\ve^{2-\sigma})   \over (\ve^4  + |x-P_j|^4)} , \quad \nabla_x S_1 (W^0, \Psi^0)(x,t) =  \sum_{j=1}^k {O(\ve^{2-\sigma})   \over (\ve^5  + |x-P_j|^5)}  
	\end{equation}
 for some $\sigma >0$ arbitrarily small.

	\

	\section{Improvement of the approximation}\label{sec4}
	
	In the previous section we introduced the functions $\Psi^0 (x,t)$ and $W^0 (x,t)$  defined in \eqref{appro1}, where $P$ is a collection of $k$ points of the form 
	$$
	P= (P_1, \ldots , P_k ), \quad P_j (t)= P_0 + P^0_j (t) +P^1_j (t) + {\bf a}_j (t), \quad P_0 =(r_0, 0). 
	$$
	So far we have defined $P_j^0$ explicitly with the form \eqref{b0} to solve \eqref{din1}, while $P^1_j$ and ${\bf a}_j$ are still parameters that are assumed to satisfy \eqref{b1} and \eqref{b11} respectively.

	The next step in our argument is to modify $( W^0, \Psi^0)$ in order to produce a better approximate solution  $(W^*,\Psi^*)$. We do it taking $(W^*,\Psi^*)$ of the form
	\begin{align}
		\Psi^* (x,t;P) &= \Psi^0 + \sum_{j=1}^k    {\eta_{j2}\over r_j} \psi_j^* ({x-P_j \over \ve_j },t) + \psi^{*,out} (x,t) \label{f110}\\
		W^*  (x,t;P) &= W^0  + \sum_{j=1}^k    {\eta_{j1} \over r_j \ve_j^2 } \phi_j^* ({x-P_j \over \ve_j },t) + \phi^{*,out} (x,t) \label{f220}
	\end{align}
	where
	\begin{equation} \label{zeta} \eta_{jN} (x,t) =\eta_N  \left(|\log \ve|^\zeta \, |x-P_j |\right),
	\end{equation}
	and $\eta_N$ is given by \eqref{cutoff}. 
	Here $\zeta >1 $ is a positive number, independent of $\ve$. Since the relative distance between $P_j$ and $P_i$, $j\not= i$, is of the order $|\log \ve |^{-{1\over 2}}$,  we have
	$$
 P_i \in {\mbox {support}}\, ( \eta_{i2} ),
	\quad  {\mbox {support}}\, ( \eta_{i2} ) \cap  {\mbox {support}}\, ( \eta_{j2} ) =\emptyset , \quad i\not= j.
	$$ 
	It will be convenient to choose    $$\zeta =3.$$  This will guarantee that $\Psi^0$ and $W^0$ are the main terms in the decomposition of $\Psi^*$ and $W^*$ in the region where the corresponding cut-off functions $\eta_{j2}$, $\eta_{j1}$ are non-zero.

	\medskip
	If we insert the expressions of $\Psi^*$ and $W^*$ given by \eqref{f110} and \eqref{f220} in the Euler  operator $S_1$ (see \eqref{defS}) we get
	\begin{align}\label{fullS10}
		S_1(W^*,\Psi^*) = \sum_{j=1}^k {\eta_{1j} \over \ve_j^4} E_j^{in} [\phi_j^* , \psi_j^* \, \psi^{*,out} , P] 
		+ E^{out} [\phi^{*, out}, \psi^{*, out} , \phi^{*,in}, \psi^{*,in} , P],
	\end{align}    
	where
	$$
	\phi^{*,in}= (\phi_1^*, \ldots , \phi_k^*),\quad  \psi^{*,in} = (\psi_1^*, \ldots , \psi_k^*).
	$$
	The inner-operators $E_j^{in}$, $j=1, \ldots , k$,   are defined as
	\begin{align}\label{Ejin}
		E_j^{in} & [\phi_j , \psi_j \, \psi^{out} , P] (y,t) := |\log \ve | \ve_j^2 (1 +{\ve_j \over r_j } y_1) \pp_t \phi_j   + |\log \ve | B_0 (\phi_j )  \nonumber \\
		& + \nabla^\perp \left(  (1+{\ve_j \over r_j} y_1)^2 (\psi_j^0  - r_j \alpha_j  |\log \ve_j |  + \psi_j + r_j \psi^{out} )  +{\mathcal R}_j (y,t;P) \right)\cdot \nabla \phi_j \nonumber \\
		&+ \nabla^\perp \left((1+{\ve_j y_1 \over r_1} )^2 (\psi_j +r_j \psi^{out} ) \right) \nabla  (w^0_j + \sum_{i\not=j} {r_j \ve_j^2 \over r_i \ve_i^2 } w^0_i ) \nonumber \\
		& +\ve_j^4 S_1 (W^0, \Psi^0) (\ve_j y + P_j), \quad |y|< 3R_j , \quad R_j:= {1 \over \ve_j |\log \ve|^\zeta}
	\end{align}
	with $y={x-P_j \over \ve_j} $,  $t \in [0,T)$, and ${\mathcal R}_j$ defined in \eqref{defRj}. Moreover $B_0$ is the operator defined in \eqref{defB0},
	which can be equivalently written as
	\begin{align*}
		B_0 (\phi)= \ve_j \dot \ve_j (y_1 \pp_1 \phi - y_2 \pp_2 \phi) -{\ve_j^2 \over r_j} \dot \ve_j y_1 y \cdot \nabla \phi -{\ve_j^2 \over r_j} \dot z_j y_1 \pp_2 \phi
	\end{align*}
	see \eqref{Sexp}  and \eqref{Sexp1} for the derivation of $B_0$, and its equivalent form.
	The outer-operator $E^{out}$ is given by
	\begin{align}\label{Eout}
		E^{out} &[\phi^{out}, \psi^{out} , \phi^{in}, \psi^{in} , P] (x,t) :=
		|\log \ve | \, r \, \partial_t \phi^{out}\\
		&+ \nabla_x^\perp ( r^2 (\Psi^0  + \sum_{j=1}^k {  \eta_{j2}\over r_j} \psi_j ({x-P_j \over \ve_j }) + \psi^{out} -r_0^{-1} |\log \ve |)) \cdot  \nabla_x \phi^{out} \nonumber \\
		&+\sum_{j=1}^k \left[ r \, |\log \ve |  \,  \pp_t \bar \eta_{j1} + \nabla_x^\perp ( r^2 (\Psi^0  + \sum_{j=1}^k  {  \eta_{j2}\over r_j} \psi_j ({x-P_j \over \ve_j }) + \psi^{out}-r_0^{-1} |\log \ve |)) \nabla \bar \eta_{1j} \right] {\phi_j \over \ve_j^2 r_j} \nonumber  \\
		&+ \left[ \sum_{j=1}^k ( \eta_{2j} - \eta_{1j}) \nabla_x^\perp (r^2 ({\psi_j \over r_j} +\psi^{out}) ) + {r^2 \psi_j \over r_j} \nabla_x^\perp \eta_{2j} \right] \cdot \nabla_x W^0 \nonumber \\
		&+ (1-\sum_{j=1}^k \eta_{2j} ) \nabla^\perp (r^2 \psi^{out} ) \cdot \nabla W^0   
		+ (1-\sum_{j=1}^k  \eta_{j1} ) S_1 (W^0, \Psi^0) =0 \quad (x,t) \in \Sigma \times [0,T) .  \nonumber 
	\end{align}
		If we now insert the expressions of $\Psi^*$ and $W^*$ given by \eqref{f110} and \eqref{f220} in the Euler  operator $S_2$ (see \eqref{defS}) we get
	\begin{equation}\label{S20}
		\begin{aligned}
			S_2 & [ \Psi^* , W^*] = S_2[ \phi^{*,in}, \psi^{*,in}, \phi^{*,out}, \psi^{*,out}, P], \quad {\mbox {with}}\\
			S_2 & [ \phi^{in}, \psi^{in}, \phi^{out}, \psi^{out}, P]=  \sum_{j=1}^k  {\eta_{1j}\over r_j \ve_j^2} [ \Delta_{5,j} \psi_j+ \phi_j] +\Delta_5 \psi^{out} + \phi^{out} \\
			&+  \sum_{j=1}^k ( \eta_{j1} - \eta_{j2} ) {\Delta_{5,j} \psi_j \over r_j \ve_j^2}+ \sum_{j=1}^k ( {\psi_j \over r_j} \Delta_5  \eta_{j2}  + 2 \nabla_x  \eta_{j2} \nabla_x {\psi_j \over r_j}  )  . 
		\end{aligned}
	\end{equation}
	Here $\Delta_{5,j}$ is the differential operator in the expanded $y$-variable defined in \eqref{d5}, while $\Delta_5$ is the differential operator defined in the original $x$-variable as given in \eqref{D5}.
	We require conditions on the boundary and at infinity on $\psi^{*,out}$: for all $t \in [0,T]$
$$
		{\partial \psi^{*,out} \over \partial r} (x,t)=0, \quad {\mbox {on}} \quad \partial \Sigma \times [0,T], \quad 
		|\psi^{*, out} (x,t) | \to 0 , \quad \ass |x| \to \infty.
$$

	\medskip
	We are able to prove that there exist points $P_1^1, \ldots , P_k^1$ in \eqref{point} such that, for any choice of points ${\bf a}_1 , \ldots , {\bf a}_k$ in \eqref{point} satisfying \eqref{b11}, it is possible to construct a good approximate leapfrogging of vortex rings $(W^*, \Psi^*)$ with the form \eqref{f110}-\eqref{f220}. This is the content of next Proposition.

		\begin{prop}\label{Approximation}
			There exist  points 
			$$ {\bf P} = ({\bf P}_1, \ldots , {\bf P}_k) , \quad {\bf P}_j = P_0 + P^0_j + P^1_j, \quad P_0= (r_0, 0)
			 $$satisfying \eqref{b1}-\eqref{b11} for some $\sigma >0$ small, and functions
			$$ \phi^{*, in}:= (\phi_1^*(y,t) , \ldots, \phi_k^* (y,t) ),  \quad \psi^{*,in}:= (\psi_1^* (y,t) , \ldots, \psi_k^* (y,t) ), \quad \phi^{*,out}, \quad \psi^{*,out}
			$$
			in \eqref{f110}, \eqref{f220} 
			such that, for any points ${\bf a}$ satisfying \eqref{b11}, the following facts hold. For 
		$$
				P= {\bf P}+ {\bf a} , 
		$$
			we have
	\begin{align*}
				E_j^{in}  [\phi_j^* , \psi_j^* \, \psi^{*,out} , P ] (y,t) &= 
				\ve_j \nabla^\perp\left[\left(  |\log \ve|  \dot {\bf a }^\perp_j +  D_x \nn_x\vp_j ( {\bf P}_j ; {\bf P} ) [{\bf a}] \right)\cdot y   \right] \nabla U \\
				&	+ {\mathcal E}_j  [\phi_j^* , \psi_j^* \, \psi^{*,out} , {\bf a}] (y,t)
			\end{align*}
			where $\varphi_j$ is given in \eqref{varphij} and	 
		$$
			\Bigl|	{\mathcal E}_j  [\phi_j^* , \psi_j^* \, \psi^{*,out} , {\bf a}] (y,t) \Bigl|= O \left( {\ve^{5-\sigma_*} \over 1+ |y|^3 } \right) , \quad |y|< 3R_j , \quad R_j:= {1 \over \ve_j |\log \ve|^3}
		$$
			for all $j=1, \ldots , k$, and
		$$
			\left| E^{out} [\phi^{*,out}, \psi^{*,out} , \phi^{*,in}, \psi^{*, in} , P]  (x,t) \right| \lesssim {\ve^{4-2\sigma_* } \over 1+ |x|^4} , \quad (x,t) \in \Sigma \times [0,T]
		$$
  for some $\sigma_* >0$ arbitrarily small.
			We refer to \eqref{Ejin} and \eqref{Eout} for the explicit definitions of $E_j^{in}$ and $E^{out}$. 
			Besides, 
			\begin{align*}
				(1+ |y|) |\nabla \psi_j^* (y,t) |+  |\psi_j^* (y,t) | & \lesssim \ve^{2-\sigma_*} ,\quad |y|< 3R_j , \, t \in [0,T]\\
				(1+ |y|^2) |\phi_j^* (y,t)| &\lesssim \ve^{2-\sigma_*}  ,\quad |y|< 3R_j , \, t \in [0,T]\\
				(1+ |x|^4) |\phi^{*,out} (x,t) | + (1+ |x|^2 ) |\psi^{*, out} (x,t) | &\lesssim \ve^{2-2\sigma_*} , \quad (x,t) \in \Sigma \times [0,T].
			\end{align*}
		\end{prop}
		
		\medskip
		This result is telling that the new approximate solution 
		  $(W^*, \Psi^*)$ in $\equ{f110}$-$\equ{f220}$
		with the parameter functions considered above produces  total errors in $\equ{fullS10}$ and \eqref{S20} which can be estimated as follows
		$$
		\begin{aligned}
			| S_1(W^*, \Psi^*)(x,t)| \ \le &\     C\ve^{1-\sigma_*}\sum_{j=1}^k \frac 1 {1+ |y_j|^3} + C{\ve^{4-2\sigma_* } \over 1+ |x|^4} , \quad y_j = \frac {x-P_j(t)} {\ve}, \\
			|S_2(W^*, \Psi^*)(x,t)|\ \le &\   C\ve^{4-\sigma_*} 
			\eta_1 \left({4|x-(r_0, 0) |\over r_0} \right),
		\end{aligned}
		$$
		for all $x \in \Sigma$ and $t \in [0,T]$.
		By construction these errors are also  uniformly Lipschitz in the parameter points  ${\bf a}$.
		Besides, in combination with \eqref{el2}-\eqref{Ubar1} we get that, for all $j=1, \ldots , k$, 
		\begin{equation} \label{d}
		\begin{aligned}
		W^* (x,t;P) &= {1\over r_j \ve_j^2 } \,  f\left( 	(1+{\ve_j \over r_j} y_1)^2 (\psi_j^0  - r_j \alpha_j  |\log \ve_j |  + \psi_j^* ) \right) \left(1+ O(\ve^2) \right),\\
		{\mbox {where}} \quad f(s)&= 8 \, e^{-K(P_j , P_j)} \, \ve_j^{2-\alpha_j r_j} \, e^s, \quad {\mbox {and}} \quad  
		\Psi^* (x,t;P)= {1\over r_j} \, \psi_j^0 + O(\ve^2 |\log \ve|)
		\end{aligned}
		\end{equation}
		uniformly in the  region $|x-P_j |< |\log \ve |^{-3}$, with $x= P_j + \ve_j y$. The definition of $\psi_j^0$ is given in \eqref{fee2}, see also \eqref{fee22}.
		We also have 
	\begin{equation}\label{barp00n}
			\begin{aligned}
					(1+{\ve_j \over r_j} y_1)^2 (\psi_j^0  - r_j \alpha_j  |\log \ve_j |  + \psi_j^* )
			&	= 
				\Gamma_0(y) - (4-\alpha_j  r_j ) \log\varepsilon_j - \log 8 + K (P_j;P_j)
				\\
				& \quad 
				+ \frac{\ve_j y_1 }{2r_j }\Big(\Gamma_0(y) +\bar A +\Gamma (y)
				\Big) 
				\, 
				+b_j^{*} (y,t) +b_j^{**} (y,t)
			\end{aligned}
		\end{equation}
where
$$
	|\log \ve |^{1\over 2}	|\pp_t b_j^{*}(y,t)| + (1+ |y|) |  \,  \nn_y b_j^{*} (y,t)|  + |b_j^{*} (y,t)| \ \le  \  C \ve^2 \, (1+ |y|^2) \, 
	\log (2+ |y|).
$$
In $b_j^{**} (y,t)$ we collect all reminding terms which depend on $\dot {\bf a}$. On these terms we have the following control
$$
	 (1+ |y|) |  \,  \nn_y b_j^{**} (y,t)|  + |b_j^{**} (y,t)| \ \le  \  C \ve^{2+\bar \sigma} \, (1+ |y|^2) \, 
$$
for $y= {x-P_j \over \ve_j}$, $|y|< |\log \ve |^{-3}$, for some fixed $\bar \sigma >0$.
		Next section will be devoted to
		build this approximate solution and to prove Proposition \ref{Approximation}.

	\section{Improvement of the approximation: Proof of Proposition \ref{Approximation}} \label{sec5}
	
	\medskip 
	The approximate solution $(W^*, \Psi^*)$ predicted by Proposition \ref{Approximation} is constructed improving the inner errors $E_j^{in}$ defined in \eqref{Ejin} ten successive times, and improving the outer error $E^{out}$ defined in \eqref{Eout}  once.
	
	It is useful to have at hand a more explicit expression of the inner-operators $E_j^{in}$ in \eqref{Ejin}.
	The key observation is that in the region
	$ |x-P_j | < 3   |\log \ve |^{-\zeta }$, we have, for $y= {x-P_j \over \ve_j} $,
	\begin{equation}\label{newR}
		\begin{aligned}
			(1+{\ve_j \over r_j} y_1)^2 &(\psi_j^0  - r_j \alpha_j  |\log \ve_j | ) +	{\mathcal R}_j (y,t;P)  \\
			&= \Gamma_0(y) - (4-\alpha r_0) \log\varepsilon_j - \log 8 + K (P_j;P_j)
			\\
			& 
			+ \frac{\varepsilon_j y_1}{2r_j}\Big(\Gamma_0(y) +\bar A +\Gamma (y)
			\Big) 
			\,  + {\mathcal R}^0_j (y,t;P), \quad {\mbox {where}}\\
			{\mathcal R}^0_j (y,t;P)&=\ve_j  |\log \ve|   (\dot P_j -\dot P_j^0)^\perp \cdot y  + \ve_j [\nn_x\vp_j ( P_j ; P ) - \nn_x\vp_j ( \bar P_j^0 ; \bar P^0 )] \cdot y  \\
			& +\,  \ve_j^3 D_x \theta_{j  \ve_j} (P_j; P ) [y]
			+ \ve^2 |\log \ve| {\mathcal Q}_1 [P] (y) + \ve^3 |\log \ve |^{3\over 2} {\mathcal Q}_2 [P] (y) 
		\end{aligned}
	\end{equation}
	with $\bar P^0_j = P_0 + P_j^0$, $\bar P^0 = (\bar P^0_1, \ldots , \bar P_k^0 )$,
	${\mathcal Q}_1$, ${\mathcal Q}_2$  denote  functions of the form, for $y= \rho e^{i\theta}$,
	\begin{equation}\label{calQ}
		\begin{aligned}
			{\mathcal Q}_1[P] (y) &=\Biggl[   a (P) + b_1 (P) \cos 2\theta + b_2 (P) \sin 2\theta \Biggl] \, O (|y|^2)\\
			{\mathcal Q}_2[P_0] (y)& = a(P) O (|y|^3 ),
		\end{aligned}
	\end{equation}
	for $a$, $b_1$ and $b_2$ smooth  functions of the points $P$, which are uniformly bounded as $\ve \to 0$. Also: $\bar A$ is the explicit constant defined by \eqref{barA} and 
	$\Gamma = \Gamma (y)$ is the function introduced in \eqref{defGamma}.

	\noindent
	\begin{proof}[Proof of \eqref{newR}.]
		In order to prove \eqref{newR}, we combine an expansion of the term $ (1+{\ve_j \over r_j} y_1)^2 (\psi_j^0  - r_j \alpha_j  |\log \ve_j | )  $  following the lines to get \eqref{barp}. 
		From \eqref{ta} and \eqref{defRj} we get
		\begin{align*}
			{\mathcal R}_j (y,t;P) & = \ve_j  |\log \ve|   (\dot P_j -\dot P_j^0)^\perp \cdot y  + \ve_j [\nn_x\vp_j ( P_j ; P ) - \nn_x\vp_j ( \bar P_j^0 ; \bar P^0 )] \cdot y  \\
			& +\,  \ve_j^3 D_x \theta_{j  \ve_j} (P_j; P ) [y] + \sum_{k=2}^4 \frac{\ve_j^k}{k!}
			D_x^k \vp_j(P_j; P) [y]^k  +{\ve_j^4 \over 2} D_x^2 \theta_{j \ve_j } (P_j ; P)[y]^2+ Q(\ve_j y , P)
		\end{align*}
		Under our assumptions   \eqref{point}, \eqref{b1} and \eqref{b11} on the form and size of the points $P_j$, we have that 
		\begin{equation}\label{expRj}
			\begin{aligned}
				{\mathcal R}_j (y,t;P) & = \ve_j  |\log \ve|  (\dot P_j -\dot P_j^0)^\perp \cdot y  + \ve_j [\nn_x\vp_j ( P_j ; P ) - \nn_x\vp_j ( \bar P_j^0 ; \bar P^0 )] \cdot y  \\
				& +\,  \ve_j^3 D_x \theta_{j  \ve_j} (P_j; P ) [y] + \ve^2 |\log \ve| {\mathcal Q}_1 [P] (y) + \ve^3 |\log \ve |^{3\over 2} {\mathcal Q}_2 [P] (y) 
			\end{aligned}
		\end{equation}
		where
		${\mathcal Q}_1$, ${\mathcal Q}_2$ satisfy \eqref{calQ}. This fact gives \eqref{newR}. 
		
	\end{proof}

	\subsection{Strategy for the improvement} \label{strategy}
	
	Replacing \eqref{newR} in \eqref{Ejin} and using \eqref{ass11}   we re-write the inner operator $E^{in}_j$ (see \eqref{Ejin}) as
	\begin{align}\label{Ej1}
		E_j^{in}  [\phi_j  ,  \psi_j , \psi^{out} , P] (y,t) &:=L_1 (\phi_j) + L_2 (\phi_j) + L_3 (\psi_j) + Q(\psi_j , \phi_j)  
		+ L_3 (r_j \psi^{out})  + Q(\psi^{out} , \phi_j) \nonumber \\
		& +\ve_j^4 S_1 (W^0, \Psi^0) (\ve_j y + P_j),\end{align}
	where
	\begin{equation}\label{ll}
		\begin{aligned}
			L_1(\phi)&:= |\log \ve | \, \ve_j^2 (1 +{\ve_j \over r_j } y_1) \pp_t \phi   + |\log \ve | B_0 (\phi)  \\
			L_2 (\phi)& :=
			\nabla^\perp \left(\Gamma_0 +  \frac{\ve_j}{2r_j}\, y_1 \, \Big(\Gamma_0(y) +\bar A +\Gamma (y)
			\Big) 
			\, 
			+{\mathcal R}^0_j (y,t;P) \right) \cdot \nabla \phi \\
			L_3 (\psi ) &:=  \nabla^\perp ((1+{\ve_j y_1 \over r_1} )^2 \psi)\cdot  \nabla (w^0_j +\sum_{i\not = j} w_i^0) \\
			Q(\psi,\phi) &:=\nabla^\perp \left( (1+{\ve_j \over r_j} y_1)^2 \psi   \right) \cdot \nabla \phi  .
		\end{aligned}
	\end{equation}
We recall the definition of the operator $B_0$ in \eqref{defB0}.

	\medskip
	The  inner error is improved using three different types of mechanisms. The choice of the mechanism is dictated by the form of the part of the error we aim at removing.
	
	\medskip
	Suppose the  error  
	$
	E(y,t),
	$ with $y= \rho e^{i\theta}$, 
	has a Fourier decomposition in the $\theta$-variable  given by
	$$
	E(y,t) = E (\rho e^{i\theta} ,t) = \sum_{n \in \Z} E_n (\rho ,t ) e^{i\, n \, \theta} , \quad E_n (\rho ,t ) = \int_0^{2\pi} E(\rho e^{in\theta} ) e^{in \theta} \, d\theta.
	$$
	We call $E_n$ the $n-$th mode in the Fourier decomposition of $E$. 
	
	\medskip{}
	{\it The elliptic improvement.} \ \ If the part of the error we want to remove   has no $0$-th mode in its Fourier decomposition then we will solve using the elliptic operator
$$
		L[\psi]\,  :=    \, \nabla^\perp  \Gamma_ 0 \cdot   \nabla \phi + \nabla^\perp    \psi  \cdot \nabla U, \quad \phi= -\Delta_{5,j} \psi .
$$
	You recover $L[\psi]$ from $L_2 (\phi_j) + L_3 (\psi_j)$ in \eqref{ll} just formally taking $\ve_j=0$, $\omega_j^0 = U$ and $\omega_i^0= 0$ for $i\not= j$. 
	Using the fact that $ -\Delta \Gamma_0 = f_0(\Gamma_0) = U $ where $f_0(u) = e^u$ and $-\Delta_{5,j} \psi = \phi$ we see that
	\begin{align*}
		L[\psi]\  = & - \nabla^\perp \Gamma_0 \cdot \nabla [ \Delta_y \psi +{3\ve_j \over r_j + \ve_j y_1} \pp_{y_1} \psi + f_0'(\Gamma_0)\psi ].
	\end{align*}
	In polar coordinates $y=\rho e^{i\theta}$ in $\R^2$, we
	check that
	$$
	L[\psi]  =  - \frac 4{\rho^2 + 1} \frac{\pp}{\pp\theta}  [ \Delta_y \psi +{3\ve_j \over r_j + \ve_j y_1} \pp_{y_1} \psi + f_0'(\Gamma_0)\psi ].
	$$
	It is enough to use the simplified version of this operator given by
	\begin{equation}\label{gadd}
		L_0[\psi] := 	- \frac 4{\rho^2 + 1} \frac{\pp}{\pp\theta}  [ \Delta_y \psi + f_0'(\Gamma_0)\psi ],  \quad \phi= -\Delta_{5,j} \psi .
	\end{equation}
For terms in the error of Fourier mode $1$ or higher, we  will solve with the elliptic operator in \eqref{gadd}. When the error has mode $1$, it will be possible to solve only under certain orthogonality conditions, which requires a proper adjustment of the points $P_j^1$. We call this procedure the {\it elliptic improvement}. We will discuss the solvability and a-priori bounds for problems of the form $L_0 (\psi ) = E$ in Lemma \ref{alpha}.
	
	\medskip
{\it The ODEs improvement.} \ \ 	If the part of  inner error $E(y,t)$ we want to remove  has $0$-th mode in its Fourier decomposition, we will solve 
	$$
	|\log \ve | \, \ve_j^2  \, \pp_t \phi= E, \quad \phi (0, \cdot ) = 0.
	$$
	This is the main part of $L_1 (\phi)$ in \eqref{ll}. We can solve the above ODEs
	at the expenses of losing two power of $\ve_j$ (and gaining one power of $|\log \ve|$) in the estimates for the solution. We will see that the construction is decided by powers of $\ve$, in the sense that the exact number of powers of $|\log \ve|$ we gain will be easily absorbed in a slightly smaller  power of $\ve$. We call this procedure the {\it ODEs improvement}. It has the advantage of keeping track of the Fourier modes of the solutions.
	
	\medskip
{\it The transport improvement.} \ \ 	After the fifth inner improvement, we will need to reduce not only  the size of the error, expressed in terms of powers of $\ve$,  but also its decay rate in the $\rho= |y|$-variable. In this case we will solve the transport-type equation
	$$
	L_1 (\phi) + L_2 (\phi) = E.
	$$
	We call this procedure the {\it transport improvement}. A-priori estimates for solutions of this transport-type equation are contained in Lemma \ref{transport3-new}. After solving this linear equation, we will have lost control on the Fourier modes of the solution $\phi$, but the structure of the problem will automatically give a new error whose main term has no $0$-th mode in its Fourier decomposition. This will be crucial for the scheme of improvement to work. 
	
	\medskip
	We summarize the complete process of improvement in the following diagram: for $\rho = {|x-P_j| \over \ve_j}$
	$$
	e_1 = {\ve^2 |\log \ve | \over 1+ \rho^4} E_{ 2} \,  \underset{{\mathcal E}}{\Rightarrow} \, e_2 = {\ve^3 |\log \ve |^2 \over 1+ \rho^3} E_{1} \,  \underset{{\mathcal E \& P^1}}{\Rightarrow} \, e_3 = {\ve^4 |\log \ve |^4 \over 1+ \rho^2} E_{0 } \,  \underset{{\mathcal ODEs}}{\Rightarrow}\, e_4 = {\ve^3 |\log \ve |^6 \over 1+ \rho^3} E_{ 12} \sin \theta \,  \underset{{\mathcal E} \& P}{\Rightarrow}
	$$
	$$
	\, e_5 = {\ve^4 |\log \ve |^8 \over 1+ \rho^2} E_{ 2} \,  \underset{{\mathcal E }}{\Rightarrow} \, e_6 = {\ve^5 |\log \ve |^8 \over 1+ \rho} E_{0 } \,  \underset{{\mathcal T}}{\Rightarrow}
	\,
	e_7 = {\ve^3 |\log \ve |^8 \over 1+ \rho^5} E_{ 1} \,  \underset{{\mathcal Out} }{\Rightarrow} \, e_8 = {\ve^3 |\log \ve |^b \over 1+ \rho^5} E_{ 1} \,  \underset{{\mathcal E \& P^1 }}{\Rightarrow} \, 
	$$
	$$
	e_9 = {\ve^4 |\log \ve |^b \over 1+ \rho^4} E_{0 } \,  \underset{{\mathcal ODEs}}{\Rightarrow} \, e_{10} = {\ve^3 |\log \ve |^b \over 1+ \rho^5} E_{ 12} \sin \theta\,  \underset{{\mathcal E} \& P^1 }{\Rightarrow} \, e_{11} = {\ve^4 |\log \ve |^b \over 1+ \rho^4} E_{ 2} \,  \underset{{\mathcal E  }}{\Rightarrow} \, e_{12} = {\ve^5 |\log \ve |^b \over 1+ \rho^3} E_{0 } .\,  
	$$
	
	\medskip
	Here ${\mathcal E}$ stands for {\it inner elliptic improvement}, 
	${\mathcal E} \& P^1$ stands for {\it inner elliptic improvement with adjustment of the points $P^1$}, ${\mathcal T}$ stands for {\it inner transport improvement}, ${\mathcal ODEs}$ stands for {\it inner ODEs improvement}, and ${\mathcal Out}$ stands for {\it outer improvement}.
	Besides $b$ is a positive number whose value may change from line to line and within the same line. 
	
	Let us explain how to interpret the diagram: we start with an initial error of which we aim at eliminating the part of size $\ve^2 |\log \ve |$, decay in space ${1\over 1+ \rho^4}$ and Fourier mode $2$: we write it as $e_1={\ve^2 |\log \ve | \over 1+ \rho^4} E_{ 2} $. To do so, we proceed with the inner elliptic improvement: we write $\underset{{\mathcal E}}{\Rightarrow}$. After this correction is done, we have a new error. Of the new error we aim now at eliminating the main term, which has size $\ve^3 |\log \ve|^2$, decay in space ${1\over 1+ \rho^3}$ and Fourier mode $1$: we write  $e_2={\ve^3 |\log \ve |^2 \over 1+ \rho^3} E_{ 1} $. And so on. Notice that the errors $e_4$ and $e_{10}$ has mode $1$, but only with $\sin \theta$ (mode $1$, odd in $y_2$).
	
	\medskip
	Before starting the process of improvement of the approximation, 
	we observe that the expression for the initial error $	\ve_j^4S_1(W^0, \Psi^0 ) (P_j + \ve_j y) $
	in \eqref{error11}  gets a simpler form if we use \eqref{newR}. Using \eqref{expRj} and \eqref{Ubar}, we write 
	\begin{equation}\label{R00}
		\begin{aligned}
			{\mathcal R}_j (y,t;P) & = {\mathcal R}_j^{00} (y,t;P) + \ve^2 |\log \ve| {\mathcal Q}_1 [P] (y) + \ve^3 |\log \ve |^{3\over 2} {\mathcal Q}_2 [P] (y) 
			\\
			& \\
			{\mathcal R}_j^{00} (y,t;P) &:=	\ve_j  |\log \ve|   (\dot P_j -\dot P_j^0) \cdot y  + \ve_j [\nn_x\vp_j ( P_j ; P ) - \nn_x\vp_j ( P_j^0 ; P^0 )] \cdot y  \\
			& +\,  \ve_j^3 D_x \theta_{j  \ve_j} (P_j; P ) [y] ,
		\end{aligned}
	\end{equation}
	and
	\begin{align*}
		w_j^0  &= U(y) \left( 1+{\ve_j y_1 \over 2r_j} (\Gamma_0 + \bar A + \Gamma)+ \ve^2 |\log \ve| {\mathcal Q}_1 [P] (y) + \ve^3 |\log \ve |^{3\over 2} {\mathcal Q}_2 [P] (y)  \right),
	\end{align*}
	with ${\mathcal Q}_i$, $i=1,2$ as in \eqref{calQ}.
	Under the  constraints \eqref{point}-\eqref{b0}-\eqref{b1}-\eqref{b11}  on the points $P_j$, we get
	\begin{align*}
		{\mathcal R}_j (y,t;P) \cdot \nabla w_j^0 &= 	{\mathcal R}^{00}_j (y,t;P) \cdot \nabla U  +  { \ve^2 |\log \ve |  \over 1+ |y|^4}  E_2(\rho,\theta,t,\varepsilon)\\
		&+ {\ve^3 |\log \ve |^2  \over 1+|y|^3}   (E_1+ E_3) (\rho,\theta,t,\varepsilon) + O \left( {\ve^4 |\log \ve |^2  \over 1+ |y|^2}  \right) \,  ;
	\end{align*}
	hence
	$	\ve_j^4S_1(W^0, \Psi^0 ) (P_j + \ve_j y) $
	in \eqref{error11} becomes
	\begin{equation}\label{error111}
		\begin{aligned}
			\ve_j^4S_1(W^0, \Psi^0 )& (P_j + \ve_j y) 
			=   \nabla^\perp {\mathcal R}_j^{00} (y,t;P) \cdot \nabla U +  { \ve^2 |\log \ve |  \over 1+ |y|^4}  E_2(\rho,\theta,t,\varepsilon)\\
			&+ {\ve^3 |\log \ve |^2  \over 1+|y|^3}   (E_{1} + E_3) (\rho,\theta,t,\varepsilon) +O \left( {\ve^4 |\log \ve |^2  \over 1+ |y|^2} \right) \, 
		\end{aligned}
	\end{equation}
	as $\ve \to 0$, where
	the functions $E_i(\rho,\theta, t, \varepsilon)$ now depends also on $\pp_t P$, and  have the form \eqref{Ekkn}. These estimates are valid for $|y| <3 \ve_j^{-1} |\log \ve |^{-\zeta}$.

	\medskip
	We are now ready to start improving. At each step of improvement we analyze the term of the error we want to remove, we describe the strategy to do it and we compute the new error produced by the correction.

	\subsection{}\label{prima}\ \ {\it First inner improvement.} \ \  The first improvement of the error will remove  part of the mode-2 term of size $\ve^2 |\log \ve| $ given by
	$$
	{ \ve^2  |\log \ve| \over 1+ |y|^4}  E_2(\rho,\theta,t,\varepsilon)
	$$
	in the error $\ve_j^4 S_1 (W^0, \Psi^0) (\ve_j y + P_j)$ computed in \equ{error111}. This term can be decomposed as the sum of two parts, one depending on $\pp_t P$ and one not:
	$$
	{ \ve^2  |\log \ve| \over 1+ |y|^4}  E_2= { \ve^2  |\log \ve| \over 1+ |y|^4}  E_2^* [P] + { \ve^2  |\log \ve| \over 1+ |y|^4}  E_2^{**}[P, \dot P].
	$$
	In fact the origin of ${ \ve^2  |\log \ve| \over 1+ |y|^4}E_2^{**}[P, \dot P]$ is 
	$ \ve_j^2 |\log \ve| (1+{\ve_j \over r_j} y_1)  \pp_t w^0_j $,  $ |\log \ve |B_0 (w_j^0)$ and
	$	{\mathcal R}_j (y,t;P) \cdot \nabla w_j^0$, and the dependence on $\dot P$ is linear. We check that 
	$$
	{ \ve^2  |\log \ve| \over 1+ |y|^4} E_2^{**}[P, \pp_t P] =
	{ \ve^2  |\log \ve| \over 1+ |y|^4} E_2^{**}[r_0 {\bf e}_1 +P^0 + P^1,  (\dot P^0+ \dot P^1)]  +
	{ \ve^5  |\log \ve| \over 1+ |y|^4}  E_2 
	$$ 
	with $E_2$ satisfying \eqref{Ekkn}. See \eqref{point} for the assumptions on $P$, $P^0$ and $P^1$.
	We will remove the part of the error 
	${ \ve^2  |\log \ve| \over 1+ |y|^4} E_2 $ given by
	\begin{equation}\label{e2}
		e_1:={ \ve^2  |\log \ve| \over 1+ |y|^4}  E_2^* [P]+
		{ \ve^2  |\log \ve| \over 1+ |y|^4} E_2^{**}[(r_0, 0) + P^0+P^1,  (\dot P^0 +\dot P^1)],
	\end{equation}
	thus leaving out what depends on $\dot {\bf a}$.
	We will do it solving the simplified linear elliptic operator 
	$$
	L_0[\psi] := 	- \frac 4{\rho^2 + 1} \frac{\pp}{\pp\theta}  [ \Delta_y \psi + f_0'(\Gamma_0)\psi ].
	$$
	introduced in \eqref{gadd}.
	We freeze the time variable and consider the problem 
	\begin{equation}\label{linear1}
		\begin{aligned}
			{4\over 1+ \rho^2} 	\frac{\pp}{\pp\theta}&  \left[ \Delta \psi + e^{\Gamma_0(y)}\psi  \right] +  g(y)= 0 \inn B(0, {8    \over \ve |\log \ve |^\zeta } ), \quad  \\
			\psi & = 0 \inn \partial B(0, {8  \over \ve |\log \ve |^\zeta } ), \quad 	\phi  = -\Delta_{5,j} \psi ,
		\end{aligned}
	\end{equation}
	for a  bounded function $g: B(0, {8  \over \ve |\log \ve |^\zeta } ) \subset \R^2  \to\R$.
	A necessary condition for the solvability of \equ{linear1} is that
	\be \label{ort0} \int_0^{2\pi} g(\rho e^{i\theta} )\, d\theta =0 \foral  \rho\in (0,8R_\ve), \quad R_\ve ={1 \over \ve |\log \ve |^\zeta }. \ee
	As 
	$$
	L_0 [Z_\ell] = 0, \quad Z_\ell(y) = \pp_{y_\ell} \Gamma_0(y) .
	$$
	we also assume the orthogonality conditions
	\be \label{ort2} \int_{B (0,8 R_\ve)}(1+|y|^2)\, g(y)\, Z_\ell (y)\, dy\, =\, 0, \quad \ell=1,2. \ee
	We have the validity of the following result.

	\begin{lemma}\label{alpha} Assume that $3\le m \le 5$ and 
		\be \label{decay}
		|g(y)| \  \le\  (1+ |y|)^{-m}
		\ee
		There exists a constant $C>0$ such that for all $\ve >0$  sufficiently small and
		$g\in L^\infty (B (0,8R_\ve) )$ that satisfies conditions $\equ{ort0}$, $\equ{ort2}$ and $\equ{decay}$,
		there exists a unique solution  $(\psi,\phi)$ of equation $\equ{linear1}$ that satisfies
		$$ \int_0^{2\pi} \psi(\rho e^{i\theta})\, d\theta =0 \foral  \rho\in (0,8R_\ve) $$
		and the estimate
		$$\begin{aligned}
			& |\psi(y)| \, +\,  (1+|y|)\, |\nn \psi(y)|\,  + \,  (1+ |y|^2)\, | \phi(y)| 
			\nonumber
			\\
			& \le \  C{(1+|y|)^{4-m}} \,\begin{cases} \log \Big ( \frac {16 \ve^{-1}  } {|y|+1}  \Big) & \hbox{ if \quad} m =5   \\ 1  & \hbox{ if \quad}  3< m < 5  \\
				\log \Big ( \frac {16\ve^{-1}  } {|y|+1}  \Big) & \hbox{ if \quad} m =3   \end{cases}
			 \end{aligned}$$

	\end{lemma}

	\begin{proof}
		
		
		Let $y= \rho \, e^{i\theta}$ and decompose $\psi $ and $g$ in Fourier series in the $\theta$-variable
		$$
		g(\rho e^{i\theta} ) = \sum_{n \in \Z} g_n (\rho ) e^{i\, n \, \theta} , \quad \psi (\rho e^{i\theta} ) = \sum_{n \in \Z} p_n (\rho ) e^{i\, n \, \theta}.
		$$ 
		Condition \equ{ort0} amounts to $g_0 \equiv 0$. Imposing $p_0\equiv 0$,
		equation \eqref{linear1} decouples into the infinitely many problems.
		\be\label{pk}
		\mathcal L_n [p_n] := \partial^2_{\rho} p_n + \frac{1}{\rho}\partial_{\rho} p_n
		-\frac{n^2}{\rho^2} p_n + \frac{8p_n}{(1+\rho^2)^2}
		= \frac{i(1+\rho^2)}{4n} g_n(\rho), \quad p_n(8 R_\ve) = 0.
		\ee
		For each $n\ne 0$, there exists a positive function $\zeta_n(\rho)$ such that $\mathcal L_{n} [\zeta_n] =0$ and
		\[
		\zeta_n(\rho) =  \rho^{|n|} (1+ o(1))  \quad \text{as }\rho \to 0
		\]
		For $n=\pm 1$ we explicitly have
		$ \zeta_n(\rho) = \frac{\rho}{1+\rho^2} $, while for $|n|\ge 2$ we have
		\[
		\zeta_n(\rho) =  \rho^{|n|} (1+ o(1))  \quad \text{as }\rho \to +\infty.
		\]
		Problem \equ{pk} is uniquely solved by the formula
		$$
		p_{n} (\rho)   =  \mathcal L_n^{-1}[g_n]\, :=    \frac {i}{4n}\zeta_n(\rho)  \int_\rho^{8R_\ve} \frac {dr}{ r\zeta_n(r)^2} \int_0^r (1+s^2) g_{n}(s)\zeta_n(s)\,s\, ds . \quad
		$$
		Let us consider the case $|n|\ge 2$. We have
		$$
		|p_{n} (\rho)|\ \le\  \mathcal L_n^{-1}[(1+\rho)^{-\alpha} ]\le   c_n (1+\rho)^{|n|}  \int_\rho^{8R_\ve}   (1+ r)^{-|n|+3 -\alpha} dr
		$$
		so that
		\[
		|p_{n} (\rho)|\ \le\      c_n (1+\rho)^{4-\alpha}
		\]
		since
		$m+ |n| > 4$.
		Let us denote $ \bar P(\rho) =   \mathcal L_2^{-1}[ (1+ \rho)^\alpha ]$ we claim that for some $\gamma>0$ and all $|n|\ge 2$ we have the validity of
		the estimate
		$$
		|p_{n} (\rho)| \ \le \  \frac \gamma {n^3}\bar P(\rho) .
		$$
		That follows from the fact that the right hand side defines a positive supersolution for the real and imaginary parts of \equ{pk}. Indeed, if $\gamma$ is taken sufficiently large we get
		$$
		\mathcal L_n \big [ \frac \gamma {n^3}\bar P(\rho) \big ]  +  \frac {1+\rho^2}{4n} |g_n(\rho)|    \le     \gamma \frac {4-n^2} {10\rho^2} (1+ \rho)^{4+\alpha}
		$$
		Fourier modes $\pm 1$ need to be separately treated because $\zeta_1(\rho)$ decays at infinity.
		At this point we observe that
		$$
		Z_1(y) = \zeta_1(\rho) \cos\theta , \quad Z_2(y) = \zeta_1(\rho) \sin\theta
		$$
		and that the orthogonality conditions \equ{ort2} assumed are equivalent to
		$$
		\int_0^{8R_\ve} (1+\rho^2)\,g_{\pm1}(\rho) \zeta_1(\rho)\, \rho\, d\rho = 0  .
		$$
		Therefore we can write
		$$
		p_{\pm 1} (\rho)  =    \mp \frac {i}{4}\zeta_1(\rho)  \int_\rho^{8R_\ve}  \frac {dr}{ r\zeta_1(r)^2} \int_r^{8R_\ve}  (1+s^2) g_{\pm 1}(s)\zeta_1(s)\,s\, ds . \quad
		$$
		and we obtain, if we now assume $3 < m \le 5$,
		$$
		|p_{\pm 1} (\rho) | \ \lesssim \    \begin{cases}  (1+ \rho)^{-1} \log {\frac {16R_\ve}{\rho+1}}  &\hbox{ if } m = 5 \\ (1+ \rho)^{4 -m } &\hbox{ if }3<  m < 5  \\
			(1+  \rho ) \log {\frac {16R_\ve }{\rho+1}}  &\hbox{ if } m= 3   . \end{cases}
		$$
		The desired result then follows from addition of the above estimates since
		$$
		|\psi(y)| \ \le\ \sum_{n\in \Z} |p_n(|y|)| .
		$$
		Finally,  since $\psi$ satisfies the equation
		$$
		\Delta \psi + e^\Gamma_0 \psi  =  -\frac 14 (1+\rho^2) \int_0^\theta g(\rho,\theta)\, d\theta, \inn B (0,8R_\ve) , \quad \psi=0\onn\pp B (0,8R_\ve) ,
		$$
		the bounds for $\phi$ and $\nn \psi $ follow from standard elliptic estimates.
		
	\end{proof}



	\bigskip
	
	We extend the function $e_1$ in \eqref{e2}  to be equal to $0$ outside  $ B(0,3R_j)$. Under the assumptions \eqref{point} on $P$, $P^0$ and $P^1$, we have that
	$$
	(1+ |y|^4) 	\left( |e_1 (y,t) | + |\log \ve |^{1\over 2} |\pp_t e_1 (y,t) | \right) \lesssim \ve^2 |\log \ve|.
	$$ 
	Moreover the function $e_1$ satisfies authomatically the orthogonality conditions \eqref{ort0} and \eqref{ort2}.
	Let
	$\psi^1_{j} = \psi_j^1 (y,t)$ and $\phi_j^1 (y,t)= - \Delta_{5,j} \psi_j^1$  be the solution to \eqref{linear1} when $g= e_1$, as predicted by Lemma \ref{alpha}. We have 
	$$ |\psi^1_j(y,t)| \, +\,  (1+|y|)\, |\nn \psi_j^1(y,t)|\,  \lesssim \ve^2 |\log \ve|, \quad \inn B(0,3R_j) \times  [0,T].
	$$
	and
	$$
	\,  (1+ |y|^3)\, |\nabla_y \phi_j^1 (y,t)| + \,  (1+ |y|^2)\, |\phi_j^1 (y,t)| \lesssim \ve^2 |\log \ve|  \quad \inn B(0,3R_j)  \times  [0,T].
	$$
	Having left out what depends on $\dot {\bf a}$ in the part of the error $e_1$ we are considering (see \eqref{e2}), 	we can harmlessly differentiate  in time equations \eqref{linear1}, to also get
	$$
	\,  (1+ |y|^2)\, |\pp_t \phi_j^1 (y,t)| \lesssim \ve^2 |\log \ve|^{1\over 2}  \quad \inn B(0,3R_j) \times  [0,T].
	$$
	Besides, a consequence of the proof of Lemma \ref{alpha} is that the Fourier decomposition of $\psi_j^1$ only contains mode-2 terms. Hence the Fourier decomposition of the function $\phi_j^1 (y,t)$
	has a mode-$2$ term of size $\ve_j^2 |\log \ve|$, mode-$1$ and mode-$3$ terms of size  $\ve_j^3 |\log \ve|$ or smaller.
	
	Using these properties of the functions $\psi_j^1$, $\phi_j^1$, and the notations introduced in \eqref{ll}  we get
	\begin{align*}
		& L_1 (\phi_j^1)  =  {\ve_j^4 |\log \ve |^{3\over 2}   \over 1+|y|^2}   E_2   + {\ve_j^5 |\log \ve |^{3\over 2}   \over 1+|y|}   E_1,  \\
		&  \nabla^\perp \Gamma_0 \cdot \nabla [ {3\ve_j \over r_j + \ve_j y_1} \pp_{y_1} \psi_j^1  ]	 =  {\ve_j^3 |\log \ve |   \over 1+|y|^3}   (E_1 + E_3) +  {\ve_j^4 |\log \ve | \over 1+|y|^2}   (E_{ 0} + E_2)  ,\\
		&\quad {\mbox {and}}\\
		&L_2 [\phi_j^1] + L_3 [\psi_j^1 ] - L[\psi_j^1] + Q (\psi_j^1, \phi_j^1)  = {\ve_j^3 |\log \ve |^2  \over 1+|y|^3}  (E_1 + E_3)  + {\ve_j^4 |\log \ve |^2  \over 1+|y|^2}   E_{i\geq 0}  .\end{align*}
	In the above expression $E_j$ denotes the $j$-th mode in the Fourier decomposition of $E(y,t)$, in accordance with the notation introduced in \eqref{defEk}. 
	Inserting this information in \eqref{Ej1} we get the description of the new error, in the region $|x-P_j|<3 |\log \ve|^{-\zeta}$, for $y={x-P_j \over \ve_j}$,
	\begin{align*}
		E_j^{in} & [\phi_j^1  ,  \psi_j^1,0  , P] (y,t) := \nabla^\perp  {\mathcal R}_j^{00} (y,t;P)  \cdot \nabla U + {\ve_j^3 |\log \ve |^2  \over 1+|y|^3}   (E_1+E_3) \nonumber \\
		&+ {\ve_j^4 |\log \ve |^2  \over 1+|y|^2}  ( E_{ 0} + E_2)  +{\ve_j^5 |\log \ve |^2  \over 1+ |y|} \, E_{i\geq 0} ,\nonumber\end{align*}
	as $\ve \to 0$, where
	the functions $E_j(\rho,\theta, t, \varepsilon)$  have the form described in \eqref{Ekkn}.  For our next improvement it is convenient to write the error as a term 
	of size $\ve_j^3 |\log \ve |^2$-size in Fourier mode $1$ or higher, with spacial decay bounded by ${1\over 1+ |y|^3}$, and that does  depend on $P$, $\dot P^0$, but not on $\dot P^1 + \dot {\bf a}$. We use the fact that the error depends in a linear way on $\pp_t P$ and assumptions \eqref{b1}-\eqref{b11}  to conclude that
	\begin{align}\label{Ej2}
		E_j^{in} & [\phi_j^1  ,  \psi_j^1,0  , P] (y,t) := \nabla^\perp  {\mathcal R}_j^{00} (y,t;P)  \cdot \nabla  U   + {\ve^3 |\log \ve |^2  \over 1+|y|^3}  (E_1 + E_3)  [P, \dot P^0 ] + +O \left( {\ve^4 |\log \ve |^2  \over 1+ |y|^2} \right)
\end{align}
	
	

	\subsection{}\ \ {\it Second  inner improvement.} \ \ Our next step is to eliminate the terms in the error \eqref{Ej2} 
	of size $\ve^3 |\log \ve |^2$. A difference with the first improvement is that this time these terms posses a mode $1$.  
	We will solve again an elliptic problem of the form \eqref{linear1}, but only at the expenses of asking that the orthogonality conditions \eqref{ort2} for the right-hand side are satisfied. We shall see that this is possible with an adjustment of the points $P$.

	In the process of the construction of the approximation $(\Psi^*, W^*)$, we will need to correct these points several times. All these corrections are encoded in the point we called $P^1$  in \eqref{point}-\eqref{b1}. The final definition of $P^1$ is given by 
	\begin{equation}
		\label{P1}
		P^1(t) = \sum_{\ell =1}^4 P^{1\ell } (t).
	\end{equation}
	where $P^{1\ell}$ correspond to successive explicit adjustments of $P^1$. Take $\ell =1$, for the first adjustment, and
	for 
	$$\bar P^1= (\bar P^1_1, \ldots , \bar P^1_k) , \quad \bar P^1_j = (r_0 , 0 ) + P^0_j  + P_j^{11},
	$$ we take 
	$$
	e_2(y,t): = \nabla^\perp  {\mathcal R}_j^{00} (y,t;\bar P^1 )  \cdot \nabla U  + {\ve_j^3 |\log \ve |^2  \over 1+|y|^3}   (E_1 + E_3) [\bar P^1 , \pp_t P^0 ](\rho,\theta,t,\varepsilon).
	$$
	We directly check that $\int_0^{2\pi} e_2 (\rho e^{i\theta} , t) \, d\theta =0$, and $e_2$ satisfies the decay \eqref{decay} with $m =3$. In fact, under the assumptions \eqref{point} on $P$, we have
	$$
	(1+ |y|^3) 	( |e_2 (y,t) | + |\log \ve |^{1\over 2} |\pp_t e_2 (y,t) | ) \lesssim \ve^3 |\log \ve|^2, \quad |y| <3 R_j.
	$$
	The orthogonality conditions
	$$
	\int_{B_{8R_j}}  (1+ |y|^2) e_2 (y,y)  Z_\ell (y) \, dy =0, \quad \ell =1,2
	$$
	become a system of ODEs for the point $P^{11} = (P_1^{11} , \ldots , P_k^{11})$ that has the form
	\begin{equation}\label{reduced}
		\ve_j |\log \ve | 	 \left[ \dot P_j^{11} + A P^{11} + |\log \ve |^{-{1\over 2}} B (P^{11}) \right] = \ve^3 |\log \ve |^3 f (\bar P^1) 
	\end{equation}
	where $A$ is the $2 \times 2$ matrix defined by
	$$
	A = |\log \ve |^{-1} D^2_x\vp_j ( \bar P_j^0 ;  P^0 ), \quad t \in [0,T),
	$$
	and $B$
	$$
	B(P^{11} ) =|\log \ve |^{-1}  [\nn_x\vp_j ( \bar P_j^1 ; \bar P^1 ) - \nn_x\vp_j ( P_j^0 ; P^0 )
	- D^2_x\vp_j ( \bar P_j^0 ;  P^0 ) P^{11}].
	$$
	See \eqref{varphij} for the definition of $\varphi_j$. Under our assumptions \eqref{b0}  on the points $P^0$,  we have
	$$
	A= O(1) , \quad B(P^{11} ) = (|\log \ve |^{-{1\over 2} }) , \quad  t \in [0,T),
	$$
	uniformly  as $\ve \to 0$. Besides, $f (\bar P^1)$ is a smooth function of $P^{11}$, which is uniformly bounded, together with its derivative, for $t \in [0,T)$, uniformly as $\ve \to 0$.
	Standard ODEs theory gives that, for all $\ve >0$ small enough there exists a unique solution $P^{11}$ to \eqref{reduced} with initial condition $P^{11} (0) = 0$, which satisfy the bounds
$$
		\| P^{11}_j \|_{L^\infty [0,T)}  + \| \dot P^{11}_j \|_{L^\infty [0,T)}  \lesssim  \ve^2 |\log \ve |^2 .
$$
	We denote by  $\psi_j^2 = \psi_j^2 (y,t)$ and $\phi_j^2 (y,t)= - \Delta_{5,j} \psi_j^2$ the solution to \eqref{linear1}, with $g=e_2$,
	whose existence and estimates are  given by Lemma \ref{alpha}. We have that $\int_0^{2\pi} \psi_j^2 (\rho e^{i\theta} , t ) \, d\theta =0$ and  
	$$ |\psi^2_j(y,t)| \, +\,  (1+|y|)\, |\nn \psi_j^2(y,t)|\,  \lesssim \ve^3 |\log \ve|^3  \,  (1+|y|) , \quad \inn B (0, 3R_j)  \times  [0,T].
	$$
	We also have
	$$
	\,  (1+ |y|^2)\, |\nabla_y \phi_j^2 (y,t)| + \,  (1+ |y|)\, |\phi_j^2 (y,t)| \lesssim \ve^3 |\log \ve|^3  \quad \inn B (0, 3R_j) \times  [0,T].
	$$
	For the same reason we did it before, we can differentiate in time the equation, to also get
	$$
	\,  (1+ |y|)\, |\log \ve |^{1\over 2} \,  |\pp_t \phi_j^2 (y,t)| \lesssim \ve^3 |\log \ve|^3  \quad \inn B (0, 3R_j)  \times  [0,T].
	$$
	From Lemma \ref{alpha} we also get that the Fourier decomposition of $\psi_j^2$ only contains Fourier mode-$1$ and mode-$3$ terms. Hence the Fourier decomposition of the function $\phi_j^2 (y,t)$
	has  mode-$1$ and mode-$3$ terms of size $\ve^3 |\log \ve|^3$,  terms of mode-$0$, mode-$2$ or higher of size  $\ve^4 |\log \ve|^3$ or smaller.
	Using these properties and the notations introduced in \eqref{ll} we get
	\begin{align*}
		& L_1 (\phi_j^2)  =  {\ve^5 |\log \ve |^{7\over 2}   \over 1+|y|}   E_{i\geq 0},  \\
		&  \nabla^\perp \Gamma_0 \cdot \nabla [ {3\ve_j \over r_j + \ve_j y_1} \pp_{y_1} \psi^2_j  ]	 = {\ve^4 |\log \ve |^3 \over 1+|y|^2}   (E_{ 0} + E_{i\geq 2} ) +{\ve^5 |\log \ve |^3   \over 1+|y|}   E_{i\geq 0} ,\\
		&L_2 [\phi_j^2] + L_3 [\psi_j^2 ] - L[\psi_j^2] + Q (\psi_j^2, \phi_j^2)  =  {\ve^4 |\log \ve |^4 \over 1+|y|^2}   (E_{ 0} + E_{i\geq 2} ) +{\ve^5 |\log \ve |^4   \over 1+|y|}   E_{i\geq 0} \\
		&\quad {\mbox {and}}\\
		& Q (\psi_j^1, \phi_j^2) + Q (\psi_j^2, \phi_j^1) ={\ve^5 |\log \ve |^4   \over 1+|y|}   E_{i\geq 0} .\end{align*}
	Define
$$\tilde \phi_j^{2} =\sum_{i=1}^2 \phi_j^i , \quad \tilde \psi_j^{2} = \sum_{i=1}^2
		\psi_j^i.
$$
	We get the
	new error 
	\begin{align}\label{Ej3}
		E_j^{in} & [\tilde \phi_j^{2}  ,  \tilde \psi_j^{2},0  , P] (y,t) :=[ \nabla^\perp  {\mathcal R}_j^{00} (y,t;P)  - \nabla^\perp  {\mathcal R}_j^{00} (y,t;\bar P^1) ] \cdot \nabla U \nonumber \\
		& +{\ve^4 |\log \ve |^4  \over 1+ |y|^2} \, (E_{ 0} + E_{i\geq  2} )[P,\dot P]    +{\ve^5 |\log \ve |^4  \over 1+ |y|} \, E_{i\geq 0} [P, \dot P] .
	\end{align}
	Using the definition of ${\mathcal R}^{00}_j$ in \eqref{R00}, we check that
	$$
	\begin{aligned}
		{\mathcal R}_j^{00} (y,t;P) & -   {\mathcal R}_j^{00} (y,t;\bar P^1) =	\ve_j  |\log \ve|   (\dot P_j - {d \over dt} \bar  P_j^1)^\perp \cdot y \nonumber \\& + \ve_j [\nn_x\vp_j ( P_j ; P ) - \nn_x\vp_j ( \bar P_j^1 ; \bar P^1 )] \cdot y \\
		& +\,  \ve_j^3 D_x \theta_{j  \ve_j} (P_j; P ) [y] -  \ve_j^3 D_x \theta_{j  \ve_j} (\bar P^1_j; \bar P^1 ) [y].\nonumber
	\end{aligned}
	$$

	\subsection{}  {\it Third inner improvement.} \ \ Our next step is the elimination of the Fourier mode-$0$ term of size $\ve^4 |\log \ve|^4$ in formula \eqref{Ej3}.  We define $\phi_j^3$ as follows
	$$
	\phi_j^3 (y,t) = - { \ve^4 \, |\log \ve |^3 \over 1+ |y|^2} \, \int_0^t  \ve_j^{-2} (s)  E_0 (\rho, \theta, s , \ve) \, ds.
	$$
	It solves
	\begin{align*}
		|\log \ve | \ve_j^2 \pp_t \phi_j^3 & + {\ve^4 |\log \ve |^4  \over 1+ |y|^2} \, E_{ 0}   = 0 \quad (y,t) \in  B (0, 3 R_j) \times [0,T)\\
		\phi_j^3 (y,0) &= 0 ,
	\end{align*}
	it satisfies 
$$
		(1+ |y|) |\nabla_y \phi_j^3 (y,t) | +	|\phi_j^3 (y,t) | \lesssim {\ve_j^2 |\log \ve|^3 \over 1+ |y|^2} \, , \quad 
$$
	and	its Fourier decomposition  only has mode $0$.
	Let $\psi_j^3$ be the solution to 
	$$
	-\Delta_{5,j} \psi_j^3 = \phi_j^3 \quad (y,t) \in   B (0, 4 R_j) \times [0,T), \quad \psi_j^3 =0 \quad (y,t) \in \partial  B (0, 4 R_j) \times [0,T).
	$$
	Let $p(y,t)$, with $p=0$ on $\partial  B (0, 4 R_j)$ be the positive radial solution to
	$$
	\Delta_y p + {4 \over 1+|y|^2} = 0 , \quad y \in  B (0, 4 R_j) , \quad p(\rho, t) =4 \int_\rho^{4R_j} {\log (1+ s^2) \over s} \, ds.
	$$
	Then 
	$$
	\Delta_y p + {3 \ve_j \over r_j + \ve_j y_1  } \pp_1 p + {2 \over 1+|y|^2} \leq 0 ,
	$$
	thanks to the fact that  $R_j={\delta \over \ve_j |\log \ve|^\zeta}$ with $\zeta >1$. Take $\bar \psi (y,t) = M\, \ve^2 |\log \ve|^3 \,  p$, for some $M>0$. This is  a super solution for
	$$
	\Delta_{5,j} \psi +  \phi_j^3 \leq 0,
	$$
	and  gives that 
$$
		|\psi_j^3 (y,t) |  + |(1+ |y|) \nabla \psi_j^3 (y,t) | \lesssim \ve_j^2 \, |\log \ve |^5 .
	$$
	Besides the Fourier decomposition 
	of the function $\psi_j^3$
	has  mode-$0$  terms of size $\ve^2 |\log \ve|^5$,  terms of mode-$1$ or higher of size $\ve^3 |\log \ve|^5$ or smaller.
	
	It is important to notice that $
	B_0 (g)$ only contains Fourier mode-$2$ or mode-$1$ terms if $g$ is a Fourier mode-$0$ function. We refer to \eqref{defB0} and Remark \ref{r1} for the definition of $B_0$ and equivalent formulations. This observation yields that
	$$
	L_1 (\phi_j^3 )- |\log \ve | \ve_j^2 \pp_t \phi_j^3 ={\ve^4 |\log \ve |^4 \over 1+|y|^2} E_2 + {\ve^5 |\log \ve |^4 \over 1+|y|} E_1.
	$$
	If $f$ is a Fourier mode-$2$ function and $g$ is a Fourier mode-$0$ function, then
	$
	\nabla^\perp f \cdot \nabla g 
	$
	is a Fourier mode-$2$ function. We use this, together with the explicity expression of ${\mathcal R}_j^0$ in \eqref{newR} and the explicit form of the operator $L_2$ in \eqref{ll} to get
	$$
	L_2 (\phi_j^3) = {\ve^3 |\log \ve |^4 \over 1+ |y|^3} E_{1,2} \sin \theta +{\ve^4 |\log \ve |^4 \over 1+|y|^2} E_2 + {\ve^5 |\log \ve |^4 \over 1+|y|} E_{i\geq 0}.
	$$
	The key fact is that  the main error of size $\ve^3 |\log \ve|^4$ has a Fourier mode-$1$ term, but only containing $\sin \theta$, not $\cos \theta$: it is of the form described in \eqref{Ekkn} for $j=1$ and $E_{j,1}=0$. A similar expression is valid for $L_3 (\psi_j^3)$, with two more powers of $|\log \ve|$:
	$$
	L_3 (\psi_j^3) = {\ve^3 |\log \ve |^6 \over 1+ |y|^3} E_{1,2} \sin \theta +{\ve^4 |\log \ve |^6 \over 1+|y|^2} E_2 + {\ve^5 |\log \ve |^6 \over 1+|y|} E_{i\geq 0}.
	$$
	We also have
	$$
	Q(\psi_j^3, \phi_j^3) = {\ve^5 |\log \ve |^8   \over 1+|y|}   E_{i\geq 0}
	$$
	$$
	Q (\psi_j^1, \phi_j^2) + Q (\psi_j^2, \phi_j^1) ={\ve^4 |\log \ve |^6   \over 1+|y|^4}   E_{2} + {\ve^5 |\log \ve |^6   \over 1+|y|}   E_{i\geq 0}
	$$
	With this in mind, calling
$$
		\tilde \phi_j^{3} = \sum_{i=1}^3 \phi_j^{i} , \quad \tilde \psi_j^{3} = \sum_{i=1}^3
		\psi_j^{i}
$$
	we get the new error
	\begin{align}\label{Ej4}
		E_j^{in} & [\tilde \phi_j^{3}  ,  \tilde \psi_j^{3},0  , P] (y,t) :=[ \nabla^\perp  {\mathcal R}_j^{00} (y,t;P)  - \nabla^\perp  {\mathcal R}_j^{00} (y,t;\bar P^1) ] \cdot \nabla U  \nonumber \\
		& +  {\ve_j^3 |\log \ve |^6  \over 1+ |y|^3} \, E_{1,2}  \sin \theta    +{\ve_j^4 |\log \ve |^6  \over 1+ |y|^2} \, E_{ 2} [P, \dot P]+{\ve_j^5 |\log \ve |^8  \over 1+ |y|} \, E_{i\geq 0} [P, \dot P] .
	\end{align}
	If we compare this error $	E_j^{in}  [\tilde \phi_j^{3}  ,  \tilde \psi_j^{3},0  , P]$ with the error 
	$	E_j^{in}  [\phi_j^1  ,  \psi_j^1,0  , P] $ in \eqref{Ej2}, we observe  a crucial difference: even though their main term have size $\ve^3$ (multiplied by powers of $|\log \ve|$) and are in Fourier mode-$1$, in $	E_j^{in}  [\tilde \phi_j^{3}  ,  \tilde \psi_j^{3},0  , P]$ Fourier mode-$1$ enters only with a $\sin \theta$. We shall proceed as in the second improvement of the approximation, with adjusting the points $P^1$ and solving the same elliptic linear problem, but this time the new error will not have Fourier mode-$0$ of size $\ve^4$. This subtle fact allows us to proceed with the construction.

	\subsection{} {\it Fourth  inner improvement.} \ \ Our next step is to eliminate the term ${\ve_j^3 |\log \ve |^6  \over 1+ |y|^3} \, E_{1,2}  \sin \theta $ in the error \eqref{Ej4}.
	Take $\ell =2$ in the decomposition \eqref{P1} of $P^1(t)$ and 
	for $\bar P^2 =  (\bar P^2_1, \ldots , \bar P^2_k) $, $\bar P^2_j = P_0+ P_j^0 + P_j^{11} + P_j^{12}$, we take 
	$$
	e_4(y,t): = \nabla^\perp  {\mathcal R}_j^{00} (y,t;\bar P^2 )  \cdot \nabla U  +{\ve_j^3 |\log \ve |^6  \over 1+ |y|^3} \, E_{1,2} [\bar P^2,  (\dot P^0 + \dot P^{11})] \sin \theta .
	$$
	We have $\int_0^{2\pi} e_4 (\rho e^{i\theta} , t) \, d\theta =0$, and $e_4$ satisfies the decay \eqref{decay} with $m =3$. In fact, under the assumptions \eqref{point} on $P$, we have
	$$
	(1+ |y|^3) 	( |e_4 (y,t) | + |\log \ve |^{1\over 2} |\pp_t e_4 (y,t) | ) \lesssim \ve^3 |\log \ve|^5, \quad |y| <3 R_j.
	$$
	The orthogonality conditions
	$$
	\int_{B_{8R_j}}  (1+ |y|^2) e_4 (y,y)  Z_\ell (y) \, dy =0, \quad \ell =1,2
	$$
	become a system of ODEs for the point $P^{12} = (P_1^{12} , \ldots , P_k^{12})$ of the same form as \eqref{reduced}.
	Standard ODEs theory gives that, for all $\ve >0$ small enough there exists a unique solution $P^{12}$ to \eqref{reduced} with initial condition $P^{12} (0) = 0$, which satisfy the bounds
$$
		\| P^{12}_j \|_{L^\infty [0,T)}  + \| \dot P^{12}_j \|_{L^\infty [0,T)}  \lesssim  \ve^2 |\log \ve |^6 .
$$
	We denote by  $\psi_j^4 = \psi_j^4 (y,t)$ and $\phi_j^4 (y,t)= - \Delta_{5,j} \psi_j^4$ the solution to \eqref{linear1}, with $g=e_4$,
	whose existence and estimates are  given by Lemma \ref{alpha}. We have that $\int_0^{2\pi} \psi_j^4 (\rho e^{i\theta} , t ) \, d\theta =0$ and  
	$$ |\psi^4_j(y,t)| \, +\,  (1+|y|)\, |\nn \psi_j^4(y,t)|\,  \lesssim \ve^3 |\log \ve|^7  \,  (1+|y|) , \quad \inn B (0, 3R_j)  \times  [0,T].
	$$
	We also have
	$$
	\,  (1+ |y|^2)\, |\nabla_y \phi_j^4 (y,t)| + \,  (1+ |y|)\, |\phi_j^4 (y,t)| \lesssim \ve^3 |\log \ve|^7  \quad \inn B (0, 3R_j) \times  [0,T].
	$$
	We can differentiate in time the equation, to also get
	$$
	\,  (1+ |y|)\, |\log \ve |^{1\over 2} \,  |\pp_t \phi_j^4 (y,t)| \lesssim \ve^3 |\log \ve|^7  \quad \inn B (0, 3R_j)  \times  [0,T].
	$$
	From Lemma \ref{alpha} we infer that is that the Fourier decomposition of $\psi_j^4$  contains Fourier mode-$1$ terms, but only with $\cos \theta$. Hence the Fourier decomposition of the function $\phi_j^4 (y,t)$
	has  mode-$1$ terms  with $\cos \theta$ of size $\ve^2 |\log \ve|^7$,  terms of mode-$0$, mode-$2$ or higher of size $\ve^4 |\log \ve|^7$ or smaller.
	Using these properties and the notations introduced in \eqref{ll} we get
	\begin{align*}
		& L_1 (\phi_j^4)  =  {\ve^5 |\log \ve |^7   \over 1+|y|}   E_{i\geq 0},  \\
		&  \nabla^\perp \Gamma_0 \cdot \nabla [ {3\ve_j \over r_j + \ve_j y_1} \pp_{y_1} \psi^4_j  ]	 = {\ve^4 |\log \ve |^7 \over 1+|y|^2}   E_{i\geq 2}  +{\ve^5 |\log \ve |^7   \over 1+|y|}   E_{i\geq 0} ,\\
		&L_2 [\phi_j^4] + L_3 [\psi_j^4 ] - L[\psi_j^4] + Q (\psi_j^4, \phi_j^4)  =  {\ve^4 |\log \ve |^8 \over 1+|y|^2}    E_{i\geq 2}  +{\ve^5 |\log \ve |^4   \over 1+|y|}   E_{i\geq 0} \\
		&\quad {\mbox {and}}\\
		& Q (\tilde \psi_j^3, \phi_j^4) + Q (\psi_j^4, \tilde \phi_j^3) ={\ve^5 |\log \ve |^8   \over 1+|y|}   E_{i\geq 0} .\end{align*}
	Define
$$\tilde \phi_j^{4} =\sum_{i=1}^4 \phi_j^i , \quad \tilde \psi_j^{4} = \sum_{i=1}^4
		\psi_j^i.
$$
	We get the
	new error, for $y \in B(0, 3R_j)$,
	\begin{align}\label{Ej5-1}
		E_j^{in} & [\tilde \phi_j^{4}  ,  \tilde \psi_j^{4},0  , P] (y,t) :=[ \nabla^\perp  {\mathcal R}_j^{00} (y,t;P)  - \nabla^\perp  {\mathcal R}_j^{00} (y,t;\bar P^2) ] \cdot \nabla U \nonumber \\
		& +{\ve^4 |\log \ve |^8  \over 1+ |y|^2} \,  E_{i\geq  2}   +{\ve^5 |\log \ve |^8  \over 1+ |y|} \, E_{i\geq 0} [P, \dot P] .
	\end{align}
	
	The $0$-th Fourier mode of the new error comes with size $\ve^5$ (and powers of $|\log \ve |$).
	

	
	\subsection{} {\it Fifth inner improvement.} \ \ 
	We now remove   part of the Fourier mode-$2$ or higher terms of size $\ve^4 |\log \ve|^8 $ in \eqref{Ej5-1}. As in the first inner improvement, we take just the parts that depend on $P$ and on $\dot P^0$, $\dot P^1$, but not the ones depending on $\dot {\bf a}$, and we call it $e_5$.
	We extend the function $e_5$ in \eqref{e2}  to be equal to $0$ outside  $B_j = B(0,3R_j)$. Under the assumptions \eqref{point} on $P$, $P^0$ and $P^1$, we have that
	$$
	(1+ |y|^2) 	( |e_5 (y,t) | + |\log \ve |^{1\over 2} |\pp_t e_5 (y,t) | ) \lesssim \ve^4 |\log \ve|^8.
	$$ 
	Moreover the function $e_5$ satisfies authomatically the orthogonality conditions \eqref{ort0} and \eqref{ort2}.
	Let
	$\psi^5_{j} $ and $\phi_j^5 (y,t)= - \Delta_{5,j} \psi_j^5$  be the solution to \eqref{linear1} when $g= e_5$, as predicted by Lemma \ref{alpha}. We have 
	$$ |\psi^5_j(y,t)| \, +\,  (1+|y|)\, |\nn \psi_j^5(y,t)|\,  \lesssim \ve^4 |\log \ve|^8 (1+|y|^2), \quad \inn B (0,8R_j)  \times  [0,T].
	$$
	and
	$$
	\,  (1+ |y|)\, |\nabla_y \phi_j^5 (y,t)| + \,  \, |\phi_j^5 (y,t)| \lesssim \ve^4 |\log \ve|^8  \quad \inn B (0,8R_j)  \times  [0,T].
	$$
	We can differentiate in time equations \eqref{linear1}, to also get
	$$
	\,  \, |\pp_t \phi_j^5 (y,t)| \lesssim \ve^4 |\log \ve|^8  \quad \inn B (0,8R_j) \times  [0,T].
	$$
	Define
$$\tilde \phi_j^{5} =\sum_{i=1}^5 \phi_j^i , \quad \tilde \psi_j^{5} = \sum_{i=1}^5
		\psi_j^i.
$$
We get the
new error, for $y \in B (0,3R_j)$ 
\begin{align}\label{Ej5-2}
		E_j^{in} & [\tilde \phi_j^{5}  ,  \tilde \psi_j^{5},0  , P] (y,t) :=[ \nabla^\perp  {\mathcal R}_j^{00} (y,t;P)  - \nabla^\perp  {\mathcal R}_j^{00} (y,t;\bar P^2) ] \cdot \nabla U 
		+{\ve^5 |\log \ve |^{8}  \over 1+ |y|} \, E_{i\geq 0} .
\end{align}

\subsection{} {\it Sixth inner improvement.} \ \ Our next step is the elimination of  the $\ve_j^5 |\log \ve|^{8}$-term in \eqref{Ej5-2}, ${\ve_j^5 |\log \ve |^{8}  \over 1+ |y|} \, E_{i\geq 0} $, to get  faster decay in the $y$ variable. To do so,  rather than solving  an elliptic problem or an ODEs,
we solve the transport-type equation
\begin{align*}
		&L_1 (\phi_j) + L_2 (\phi_j) \\
		&=|\log \ve | \, \ve_j^2 (1 +{\ve_j \over r_j } y_1) \pp_t \phi_j   + |\log \ve | B_0 (\phi_j ) \\
		& +\nabla^\perp \left(\Gamma_0 +  \frac{\ve_j}{2r_j}\, y_1 \, \Big(\Gamma_0(y) +\bar A +\Gamma (y)
		\Big) 
		\, 
		+{\mathcal R}^0_j (y,t;P) \right) \cdot \nabla \phi_j 
		=E(y,t)  ,
	\end{align*}
	in $ B_{3R_j} \times [0,T]$,
	with initial condition $\phi(y,0)  =0$ in $B (0,3R_j)$.
	We write the operator $B_0$  defined in \eqref{defB0} as follows
		\begin{align*}
			|\log \ve | B_0 (\phi ) &= {\mathcal B}_j^0 (y,t;P) \cdot \nabla \phi \\
			{\mathcal B}_j^0 (y,t;P) &=- {|\log \ve| \over 2}  {d \over dt} (\ve_j^2 ) (1+{\ve_j \over r_j} y_1)    \, y  -|\log \ve| {\ve_j^2 \over r_j} y_1 \dot P_j .
		\end{align*}
	We will need uniform differentiability in $t$ of the coefficients  $\RR_j^0$ and ${\mathcal B}_j^0$. Thus we consider the following slightly-modified  transport equation
	\begin{align}
		\left\{
		\begin{aligned}
			\ve_j^2 |\log \ve | (1+ {\ve_j \over r_j} y_1) \pp_t  \phi 
			&+ \nabla_y^\perp(\Gamma_0(y) +  \tilde \RR_j ( y,t;P) ) \cdot \nabla_y \phi + {\mathcal B}_j 
			( y,t;P) \cdot \nabla_y \phi \\
			&=  E(y,t) ,
			\quad \text{in } B (0,3R_j) \times [0,T]
			\\
			\phi(y,0) & =0 , \quad \text{in } B (0,3R_j)
		\end{aligned}
		\right.
		\label{pico}\end{align}
	Here 
	$$
	\tilde \RR_j (y,t) =  \frac{\ve_j}{2r_j}\, y_1 \, \Big(\Gamma_0(y) +\bar A +\Gamma (y) \Big) + 	{\mathcal R}^1_j (y,t;P)
	$$
	where
$$
		{\mathcal R}^0_j (y,t;P) = {\mathcal R}^1_j (y,t;P) + \ve_j  |\log \ve|  \, \dot {\bf a}_j \cdot y
$$
	see \eqref{newR}. In other words, we leave out the term involving $\dot {\bf a}$. We do the same to ${\mathcal B}_j^0 $:  using the fact that ${d \over dt} \ve_j^2 = -r_0 \ve^2 {\dot r_j \over r_j^2}$, we write
	$$
	{\mathcal B}_j^0 (y,t;P) = {\mathcal B}_j (y,t;P) 
	+ {|\log \ve| \ve^2 r_0 \over 2}  {\dot {\bf a}_{j1} \over r_j^2} (1+{\ve_j \over r_j} y_1)     y  -|\log \ve| {\ve_j^2 \over r_j} y_1 \dot {\bf a}_j .
	$$
	It is straightforward to check that, under our assumptions \eqref{point}, \eqref{b0} and \eqref{b1}, for $y \in B (0,3R_j)$ we have
	\begin{equation}\label{papa-new}
		\begin{aligned}
			| \RR^1_j( y,t)| + |\pp_t  \RR^1_j( y,t)| & \ \le \ M \ve^2 |y|^2, \\
			|\nabla_y \RR^1_j( y,t) | + |\pp_t  \nabla_y \RR^1_j( y,t)| \ &\le \ M \ve^2 |y| ,\\
			|D^2_y \RR^1_j( y,t) | &\le \ M \ve^2, \\
			|{\mathcal B}_j( y,t) | + |\pp_t  {\mathcal B}_j ( y,t)| \ &\le \ M \ve^2 |\log \ve |^{1\over 2} |y| ,\\
			|\nabla_y {\mathcal B}_j( y,t) | \ &\le \ M \ve^2 |\log \ve |^{1\over 2} ,
		\end{aligned}
	\end{equation}
	for some positive constant $M$ independent of $\ve$.



	We consider a smooth cut-off function $\eta_4(s)$ as in \eqref{cutoff}, and take
	\begin{equation}\label{eta2epsilon}
		\eta_{4 \ve} (y) = \eta_ 4 ( |\log \ve |^\zeta \, \ve \, |y|).
	\end{equation}
	We will then have a solution to \equ{pico} by restricting to $B (0,3R_j)$ the solution of the Cauchy problem
	\begin{align}
		\left\{
		\begin{aligned}
			\ve_j^2 |\log \ve | (1+ \eta_{4 \ve} {\ve_j \over r_j} y_1) \pp_t \phi 
			&+ \nabla_y^\perp(\Gamma_0(y) + \eta_{4 \ve} \tilde \RR_j ( y,t;\xi) ) \cdot \nabla_y \phi
			+ \eta_{4\ve} {\mathcal B}_j 
			( y,t;P) \cdot \nabla_y \phi
			\\
			&= \eta_{4 \ve}  E(y,t) ,
			\quad \text{in } \R^2 \times [0,T]
			\\
			\phi(y,0) & =0 , \quad \text{in } \R^2
		\end{aligned}
		\right.
		\label{pico1}\end{align}
	
	\medskip
	In \S~\ref{secProofTransportInner} we prove the following result

	\begin{lemma}\label{transport3-new}
		Let us assume that $\tilde \RR_j$ and ${\mathcal B}_j$ satisfy $\equ{papa-new}$.
		Then there exist numbers $C,\delta>0$ such that for all sufficiently small $\ve$  and any function $E(y,t)$ that satisfies for some $C,\alpha\in \R$
		\[
	(1+|y|^2 ) |D^2_y E (y,t)|+	(1+|y|) \left( |\nn_y E (y,t)| + |\pp_t \nn_y E (y,t)|\right) + | E (y,t)| + |\pp_t E(y,t) |  \ \le \ C\, (1+ |y|)^\alpha ,
		\]
		the solution of $\equ{pico1}$ satisfies
	$$
		(1+|y|)|\nn_y \phi(y, t)    |  + |\phi(y, t)| \  \le\  C \ve^{-2} |\log \ve |^{-1}  C (1+|y|)^\alpha 
		$$
		for all $y \in \R^2$, with $|y| <4 \ve^{-1} |\log \ve|^{-\zeta}$, $t \in [0,T)$.
	\end{lemma}

	We want to apply Lemma \ref{transport3-new} for the right hand side 
	$$
	E(y,t)= e_6 , \quad 	e_6=  {\ve^5 |\log \ve |^{8}  \over 1+ |y|} \, E_{i\geq 0} (\rho,\theta,t,\varepsilon)
	$$
	where we freeze this term at $P_j=P_0 +P_j^0 + P_j^1$.
	We define $\phi_j^6$ to be the solution to problem \eqref{pico}, with $E$ as above, predicted by Lemma \ref{transport3-new}. It satisfies
	$$
	(1+|y|)|\nn_y \phi^6_j (y, t)    |  + |\phi^6_j (y, t)| \  \le\  C {\ve^{3} |\log \ve |^{7}  \over  1+|y|}.
	$$
	Let $\psi_j^6$ be the solution to 
	$$
	-\Delta_{5,j} \psi_j^6 = \phi_j^6 \quad (y,t) \in  B (0,4R_j) \times [0,T), \quad \psi_j^6 =0 \quad (y,t) \in \partial B (0,4R_j) \times [0,T).
	$$
	Let $p(y,t)$, with $p=0$ on $\partial B (0,4R_j)$ be the positive smooth radial solution to
	$$
	\Delta_y p + {4 \over 1+|y|} = 0 , \quad y \in B (0,4R_j) . 
	$$
	Then $p(y) = 16 (R_j - |y|) - 4 \int_{|y|}^{4R_j} {\log (1+ s) \over s} \, ds  $, and  $|p(y)| \lesssim (1+ |y|)$ and
	$$
	\Delta_y p + {3 \ve_j \over r_j + \ve_j y_1  } \pp_1 p + {2 \over 1+|y|} \leq 0 .
	$$
	Take $\bar \psi (y,t) = M\, \ve^3 |\log \ve|^7 \, p$, for some $M>0$. It is  a super solution for
	$$
	\Delta_{5,j} \psi +  \phi_j^6 \leq 0,
	$$
	and  gives that 
	$$
	(1+|y|^2)|D^2_y \psi^6_j (y, t)    |  +	(1+|y|)|\nn_y \psi^6_j (y, t)    |  + |\psi^6_j (y, t)| \  \le\  C \ve^{3} |\log \ve |^{8}    (1+|y|).
	$$
	We use this information to compute the following terms
	\begin{align*}
		& L_3 ( \psi_j^6 ) + Q (\psi_j^6  , \phi_j^6 )= {\ve^3 |\log \ve |^8 \over 1+ |y|^5} E_{i\geq 1} , \quad  Q (\psi_j^6  , \tilde \phi_j^5 )+  Q (\tilde \psi_j^5  , \phi_j^6 )=
		{\ve_j^5 |\log \ve |^{8} \over 1+ |y|^3} E_{i\geq 0}
	\end{align*}
	Define
$$\tilde \phi_j^{6} =\sum_{i=1}^6 \phi_j^i , \quad \tilde \psi_j^{6} = \sum_{i=1}^6
		\psi_j^i.
$$
	We get the
	new error 
	\begin{align}\label{Ej6}
		E_j^{in} & [\tilde \phi_j^{6}  ,  \tilde \psi_j^{6},0  , P] (y,t) :=[ \nabla^\perp  {\mathcal R}_j^{00} (y,t;P)  - \nabla^\perp  {\mathcal R}_j^{00} (y,t;\bar P^2) ] \cdot \nabla U \nonumber \\
		& + {\ve^3 |\log \ve |^8 \over 1+ |y|^5} E_{i\geq 1} +{\ve^5 |\log \ve |^{8}  \over 1+ |y|^3} \, E_{i\geq 0} .
	\end{align}
	If we compare this error with the one in \eqref{Ej2}, we see that both their main terms have size $\ve^3$ multiplied by a power of $|\log \ve|$, and have Fourier mode-$1$. The difference is in their decay in the $y$-variable: the error in \eqref{Ej6} has a much faster decay, which will be crucial to make the final argument of our construction work. We explain this in Section \ref{sec8} where the {\it inner-outer} scheme is described in detail.
	
	\medskip
	At this point of our construction we choose $\zeta $ in \eqref{zeta}. Take
	$$
	\zeta =3.
	$$
	Then
	$$
	W^0 (x) +    {\eta_{j1} \over r_j \ve_j^2 } \tilde \phi_j^6 ({x-P_j \over \ve_j },t) = {1\over r_j \ve_j^2} U({x-P_j \over \ve_j}) (1+ o(1) ), \ass \ve \to 0
	$$
	for $|x-P_j | < |\log \ve |^{-3}$.

	%

	\medskip

	\subsection{} {\it The outer improvement.} \ \ 
	So far we have modified the approximate solution in the inner regions, namely at a small distance from the points $P_j$. We will now improve the outer error given by
\begin{equation} \label{nik2}	\begin{aligned}
		E_0 (x,t;P)&= E^{out} [0, 0 , \tilde \phi^6, \tilde \psi^6 , P] (x,t) \\ &:=
		\sum_{j=1}^k \left[ r \, |\log \ve |  \,  \pp_t \bar \eta_{j1} + \nabla_x^\perp ( r^2 (\Psi^0  + \sum_{j=1}^k  {  \eta_{j2}\over r_j} \tilde \psi_j^{6} ({x-P_j \over \ve_j }) -r_0^{-1} |\log \ve |)) \nabla \bar \eta_{1j} \right] {\tilde \phi_j^{6} \over \ve_j^2 r_j} \nonumber  \\
		&+ \left[ \sum_{j=1}^k ( \eta_{2j} - \eta_{1j}) \nabla_x^\perp (r^2 {\tilde \psi_j^{6} \over r_j} )  + {r^2 \tilde \psi_j^6 \over r_j} \nabla_x^\perp \eta_{2j} \right] \nabla_x W^0 \\
		&
		+ (1-\sum_{j=1}^k  \eta_{j1} ) S_1 (W^0, \Psi^0)  \quad (x,t) \in \Sigma \times [0,T)   \nonumber 
	\end{aligned}
\end{equation}
	We also want to reduce the size of $	S_2  [ \tilde \phi^6, \tilde \psi^6, 0, 0, P]$.
	We refer to \eqref{defS} and also \eqref{S20}. 
	Observe that the function $\Psi^0  + \sum_{j=1}^k  {  \eta_{j2}\over r_j} \tilde \psi_j^6  ({x-P_j \over \ve_j }) $ satisfies the 
	conditions on the boundary and at infinity, see \eqref{appro1}: for all $t \in [0,T]$
$$
		{\partial \psi \over \partial r} (x,t)=0, \quad {\mbox {on}} \quad \pp \Sigma, \quad 
		|\psi (x,t)  | \to 0 , \quad \ass |x| \to \infty.
$$
	
	To reduce the outer error, we first solve in $\phi^{out}$ the {\it outer transport equation}
	\begin{equation}\label{transport-outer}
		\begin{aligned}
			|\log \ve | \, r \, \pp_t \phi^{out}
			&+ \nabla_x^\perp ( r^2 (\Psi^0   -r_0^{-1} |\log \ve |)) \nabla_x \phi^{out} =E_0, \quad {\mbox {in}} \quad  \Sigma \times [0,T)\\
			\phi^{out} (x,0) &= 0, \quad \quad {\mbox {in}} \quad  \Sigma.
		\end{aligned}
	\end{equation}
	We define
	\begin{equation}\label{ou1}
		\begin{aligned}
			\tilde E_0 (x,t) = {E_0 (x,t) \over r }, \quad
			|\log \ve | B(x,t)= {1\over r} \nabla_x^\perp ( r^2 (\Psi^0   -r_0^{-1} |\log \ve |)) 
		\end{aligned}
	\end{equation}
	and re-write \eqref{transport-outer} as
	\begin{equation}\label{to1}
		\begin{aligned}
			|\log \ve | \, \pp_t \phi^{out}
			+|\log \ve | B \cdot  \nabla_x \phi^{out} &=\tilde E_0, \quad {\mbox {in}} \quad  \Sigma \times [0,T)\\
			\phi^{out} (x,0) &= 0, \quad \quad {\mbox {in}} \quad  \Sigma.
		\end{aligned}
	\end{equation}
	From the very definition of $\Psi^0$ 
	the following properties for  $B(x,t)$ follow: if  $B (x,t) = (B_1 (x,t), B_2 (x,t))$, we have
	\begin{equation}\label{estB}\begin{aligned}
			B_1(x,t) &= 0 \inn (x,t) \in \pp \Sigma \times [0,T]
	\end{aligned} \end{equation}
	and
	\begin{equation}\label{estB1}\begin{aligned}
			B(x,t) &= O(1) \inn (x,t) \in  \Sigma \times [0,T], \quad \ass \ve \to 0.
	\end{aligned} \end{equation}
	
	Since the function  $B$  is continuous and log-Lipschitz in $x$ uniformly in $t$, for $(x,t) \in \Sigma \times [0,T)$,
	we can represent the solution to \eqref{to1}
	using the Duhamel's representation formula
	\begin{equation}\label{phi3}
		\phi^{out} (x,t) = |\log \ve |^{-1} \int_0^t \tilde E_0 (\bar x (s;x,t), s) \, ds
	\end{equation}
	where $\bar x (s;x,t)$ are the characteristic curves defined as
	\begin{equation}\label{char}
		\, 	{d \over ds} \bar x (s; x,t) = B (\bar x (s; x,t), s ) \quad s \in (0,t), \quad \bar x (t; x,t) = x,
	\end{equation}
	which exist and are unique, mainly thanks to \eqref{estB} as we shall prove later.

	The right-hand side of our equation \equ{transport-outer} is supported away from the vortices $P_j(t)$ at a distance proportional to $|\log \ve |^{-3}$. We assume that for some fixed number $\delta'$ we have
	\begin{equation}
		\tilde E_0 , \, E_0 \equiv 0  \inn  \Big \{ (x,t) \in \Sigma \times [0,T] \ /\  x\in  \bigcup_{j=1}^k B (P_j (t) , \delta' |\log \ve |^{-3} )   \Big \}.
		\label{assE}\end{equation}
	
	

	\medskip We have the validity of the following 
	
	\begin{lemma}\label{int1}
		For $x\in \Sigma$ the characteristic curves \eqref{char} satisfy
		$\bar x(s; t, x)\in \Sigma$ for all $0\le s\le t$.  The solution of $\equ{transport-outer}$ given by $\equ{phi3}$ satisfies for any $1\le p\le +\infty$,
		\begin{equation*}
			\|\phi(\cdot ,t)\|_{L^p(\Omega)}  \le    |\log \ve |^{-1}  t \, \sup_{0\le s\le t} \|\tilde E_0 (\cdot, s)\|_{L^p (\Omega)} .
		\end{equation*}
		If $E_0$ satisfies $\equ{assE}$,  there exists a number $\delta^*>0$ independent of  $\delta'$ and $\ve $ such that  the solution of $\equ{transport-outer}$
		satisfies
		\[
		\phi \equiv 0  \inn  \Big \{ (x,t) \in \Sigma \times [0,T] \ /\  x\in  \bigcup_{j=1}^k B (P_j(t),  \delta^* |\log \ve |^{-3} )   \Big \}.
		\]
		Moreover, 	if $\tilde E_0$ satisfies $\equ{assE}$
		and
		$$
		(1+ |x|)  |\nn_x \tilde E_0(x,t)|  + |\tilde E_0 (x,t) |  \le {A   \over 1+ |x|^4} \foral (x,t) \in \Sigma \times [0,T],
		$$ then there exists $C>0$ such that
		the solution of $\equ{transport-outer}$ satisfies
		the estimate
		\be\label{poto1} (1+ |x|)   |\nn_x \phi(x,t)| + |\pp_t \phi (x,t)|  + |\phi(x,t) |\ \le\ C\, |\log \ve |^{-1}  {A \over 1+ |x|^4} . \ee
	\end{lemma}

	For the proof of Lemmas~\ref{int1} see \S~\ref{secProofTransportOuter}.

	\medskip

	From \eqref{gru1}, we explicitly compute
	\begin{align*}
		{	(1-\sum_{j=1}^k  \eta_{j1} )	S_1(W^0, \Psi^0 ) \over r } &= |\log \ve| (1-\sum_{j=1}^k  \eta_{j1} ) \, \eta \sum_{j=1}^k  \biggl[ \pp_t W^0_j  
		+ {r \over |\log \ve |} \nn^\perp (\Psi^0 - r_0^{-1} |\log \ve| ) \cdot \nn W_j^0   \\
		&- {2\over |\log \ve |}  ( \Psi^0 - r_0^{-1} |\log \ve|) \, {\bf e_2} \cdot \nn W_j^0 
		\Biggl].
	\end{align*}    
	This fact, together with \eqref{nik2} and the estimates on $\tilde \phi^6$, $\tilde \psi^6$ give that $\tilde E_0 (x,t)$ has no singularity at $r=0$ and satisfies
	\begin{equation}\label{est-tilde-E0}
		\begin{aligned}
			\tilde E_0 (x,t)& =  \sum_{j=1}^k  (1-  \eta_{j1} ) {O(\ve^2 |\log \ve | )\over (|\log \ve |^{-3} + |x-P_j|)^4} , \\
				\nabla_x \tilde E_0 (x,t)& =  \sum_{j=1}^k (1-  \eta_{j1} ) {O(\ve^2  |\log \ve | )\over (|\log \ve |^{-3} + |x-P_j|)^5} .
		\end{aligned}
	\end{equation}
We define $\phi^{out}_1$ to solve \eqref{transport-outer}. Since \eqref{est-tilde-E0}, from Lemma \ref{int1} we obtain that 
$$ (1+ |x|)  |\nn_x \phi^{out}_1(x,t)| + |\pp_t \phi^{out}_1(x,t)|  + |\phi^{out}_1 (x,t) |\ \le\ C\,  {\ve^2 |\log \ve |^b \over 1+ |x|^4}  $$
	where $b>0$ denotes a number whose exact value will change from line to line and also within the same line. At the end of this construction, we will absorb any positive power of $|\log \ve |$ in an arbitrarily smaller power of $\ve$. 
	
	\medskip
	We next introduce $\psi_1^{out}$ to solve 
	\begin{align*}
		-\Delta_5 \psi^{out}_1 &= \phi^{out}_1 + S_2  [ \tilde \phi^6, \tilde \psi^6, 0, 0, P], \quad {\mbox {on}} \quad \Sigma  \times [0,T)   , \quad  
		\\
		{\partial \psi^{out}_1 \over \partial r}  & = 0  \quad {\mbox {on}} \quad \partial \Sigma  \times [0,T), \quad  |\psi^{out}_1(x,t)| \to 0 \quad \text{as } |x|\to \infty .
	\end{align*}
From the estimates on $\tilde \phi^6$ and $\tilde \psi^6$ we get
	\begin{align*}
	S_2 & [ \tilde \phi^6, \tilde \psi^6, 0, 0, P]=   \sum_{j=1}^k ( \eta_{j1} - \eta_{j2} ) {\Delta_{5,j} \tilde \psi_j^6 \over r_j \ve_j^2}\\
	&+ \sum_{j=1}^k ( {\tilde \psi_j^6 \over r_j} \Delta_5  \eta_{j2}  + 2 \nabla_x  \eta_{j2} \nabla_x {\tilde \psi_j^6 \over r_j}  ) =O( \ve^2 |\log \ve |^{b} ) {\bf 1}_{\{ |\log \ve |^{-3} < |x-P_j| < 4|\log \ve |^{-3}\}} . 
\end{align*}
	With the aid of Lemma \ref{ll1}, we have
	\begin{equation}\label{psiout1} 
		(1+ |x| ) |\nabla \psi_1^{out} (x,t) | + |\psi_1^{out} (x,t) | \lesssim {\ve^2 |\log \ve |^{b} \over 1+ |x|^2},
	\end{equation}
	and hence
	\begin{equation}\label{S2000}
	S_2  [\tilde  \phi^6, \tilde \psi^6, \phi_1^{out}, \psi_1^{out}, P]= 0 \quad (x,t) \in \Sigma \times [0,T)    .
	\end{equation}
	Having introduced the outer corrections $\phi_1^{out}$ and $\psi_1^{out}$, we compute  the new outer error $E^{out} [\phi^{out}_1, \psi^{out}_1, \tilde \phi^6, \tilde \psi^{6} , P]$ and the new inner error $	E_j^{in}  [\tilde \phi_j^6   , \tilde \psi_j^6  , \psi_1^{out} , P]$.

	\medskip
	 The new outer error is given by
	\begin{align*}
		E^{out} &[\phi^{out}_1, \psi^{out}_1, \tilde \phi^6, \tilde \psi^{6} , P] (x,t) :=
		\nabla_x^\perp ( r^2 (\sum_{j=1}^k  {  \eta_{j2}\over r_j} \tilde \psi_j^6  + \psi^{out}_1) ) \cdot  \nabla_x \phi^{out}_1 \nonumber \\
		&+\sum_{j=1}^k   \nabla_x^\perp ( r^2   \psi^{out}_1 ) \nabla \bar \eta_{1j}  {\tilde \phi_j^6 \over \ve_j^2 r_j}+  \sum_{j=1}^k ( \eta_{2j} - \eta_{1j}) \nabla_x^\perp (r^2 \psi^{out}_1)  \cdot \nabla_x W^0 \nonumber \\
		&+ (1-\sum_{j=1}^k \eta_{2j} ) \nabla^\perp (r^2 \psi^{out}_1 ) \cdot \nabla W^0  \quad (x,t) \in \Sigma \times [0,T)    
	\end{align*}
	We see that 
	$	E^{out} [\phi^{out}_1, \psi^{out}_1, \tilde \phi^6, \tilde \psi^{6} , P] (x,t) $ is a regular function, which is $\equiv 0$ in the region
	$$
	\Big \{ (x,t) \in \Sigma \times [0,T] \ /\  x\in  \bigcup_{j=1}^k B (P_j (t) , \bar \delta |\log \ve |^{-3} )   \Big \}
	$$
	for some $\bar \delta>0$, and satisfies the bounds
	\begin{equation}\label{Eout1-2}
		| E^{out} [\phi^{out}_1, \psi^{out}_1, \tilde \phi^6, \tilde \psi^{6} , P] (x,t) | \lesssim {\ve^4 |\log \ve|^{b} \over 1+ |x|^4}.
	\end{equation}
	The new inner error is given by
	\begin{align*}
		E_j^{in} & [\tilde \phi_j^6   , \tilde \psi_j^6  , \psi_1^{out} , P] (y,t) := 
		\nabla^\perp ((1+{\ve_j y_1 \over r_1} )^2 r_j \psi^{out}_1 )\cdot  \nabla (w^0_j + \sum_{i\not= j} w_i^0)\nonumber \\
		&+ \nabla^\perp \left( (1+{\ve_j \over r_j} y_1)^2  r_j \psi^{out}_1    \right) \cdot \nabla  \tilde \phi_j^6 +	E_j^{in}  [\tilde \phi_j^6  ,  \tilde \psi_j^6,0  , P] (y,t).\end{align*}
	The biggest term in this expression comes from $\nabla^\perp ((1+{\ve_j y_1 \over r_1} )^2 r_j \psi^{out}_1 )\cdot  \nabla (w^0_j + \sum_{i\not= j} w_i^0)$. Using \eqref{psiout1} we see that the size of this term is ${\ve^3 |\log \ve |^{b} \over 1 + |y|^5}$. Moreover, since $w_j^0$ satisfies \eqref{Ubar1}, its  Fourier modes-$0$ come at order $\ve^4 |\log \ve |^{b}$ or smaller. In combination with \eqref{Ej6}, we get
	\begin{align}\label{Ej7}
		E_j^{in} & [\tilde \phi_j^6   , \tilde \psi_j^6  , \psi_1^{out} , P] (y,t) := 
		[ \nabla^\perp  {\mathcal R}_j^{00} (y,t;P)  - \nabla^\perp  {\mathcal R}_j^{00} (y,t;\bar P^2) ] \cdot \nabla U \nonumber \\
		& + {\ve^3 |\log \ve |^{b} \over 1+ |y|^5} E_{i\geq 1}+  {\ve^4 |\log \ve |^{b} \over 1+ |y|^4} E_{0}  +{\ve^5 |\log \ve |^{b}  \over 1+ |y|^3} \, E_{i\geq 0} .
	\end{align}
	Recall that $b>0$ denotes a number whose exact value will change from line to line and also within the sale line.
	
	
	\subsection{} {\it Seventh inner  improvement.} \ \ 
	Our next step is to eliminate the terms in the error \eqref{Ej7} given by
	$ {\ve^3 |\log \ve |^{13} \over 1+ |y|^5} E_{i\geq 1} $. We proceed as in the second and fourth improvements.
	Take $\ell =3$ in the decomposition \eqref{P1} of $P^1(t)$ and 
	for $\bar P^3= (\bar P_1^3, \ldots , \bar P_k^3 )$, $\bar P^3_j = P_0 + P_j^0 + P_j^{11} + P_j^{12} + P_j^{13}$, we take 
	$$
	e_7(y,t): = \nabla^\perp  {\mathcal R}_j^{00} (y,t;\bar P^3 )  \cdot \nabla U  +{\ve_j^3 |\log \ve |^{b}  \over 1+ |y|^5} \, E_{i\geq 1} [\bar P^2, \dot P^0 + \dot P^{11} + \dot P^{12}] .
	$$
	We have $\int_0^{2\pi} e_7 (\rho e^{i\theta} , t) \, d\theta =0$, and $e_7$ satisfies the decay \eqref{decay} with $m =5$. In fact, under the assumptions \eqref{point} on $P$, we have
	$$
	(1+ |y|^5) 	( |e_7 (y,t) | + |\log \ve |^{1\over 2} |\pp_t e_7 (y,t) | ) \lesssim \ve^3 |\log \ve|^{b}, \quad |y| <3 R_j.
	$$
	The orthogonality conditions
	$$
	\int_{B_{8R_j}}  (1+ |y|^2) e_7 (y,y)  Z_\ell (y) \, dy =0, \quad \ell =1,2
	$$
	become a system of ODEs for the point $P^{13} = (P_1^{13} , \ldots , P_k^{13})$ of the same form as \eqref{reduced}.
	Standard ODEs theory gives that, for all $\ve >0$ small enough there exists a unique solution $P^{13}$ to \eqref{reduced} with initial condition $P^{13} (0) = 0$, which satisfy the bounds
$$
		\| P^{13}_j \|_{L^\infty [0,T)}  + \| \dot P^{13}_j \|_{L^\infty [0,T)}  \lesssim  \ve^2 |\log \ve |^c ,
	$$
	for some $c>0$.
	We denote by  $\psi_j^7 = \psi_j^7 (y,t)$ and $\phi_j^7 (y,t)= - \Delta_{5,j} \psi_j^7$ the solution to \eqref{linear1}, with $g=e_7$,
	whose existence and estimates are  given by Lemma \ref{alpha}. We have that $\int_0^{2\pi} \psi_j^7 (\rho e^{i\theta} , t ) \, d\theta =0$ and  
	$$ |\psi^7_j(y,t)| \, +\,  (1+|y|)\, |\nn \psi_j^7(y,t)|\,  \lesssim  \, {\ve^3 |\log \ve|^{b}  \over  1+|y|}, \quad \inn B (0, 3R_j)  \times  [0,T].
	$$
	We also have
	$$
	\,  (1+ |y|^4)\, |\nabla_y \phi_j^7 (y,t)| + \,  (1+ |y|^3)\, |\phi_j^7 (y,t)| \lesssim \ve^3 |\log \ve|^{b}  \quad \inn B (0, 3R_j) \times  [0,T].
	$$
	We can differentiate in time the equation, to also get
	$$
	\,  (1+ |y|^3)\, |\log \ve |^{1\over 2} \,  |\pp_t \phi_j^7 (y,t)| \lesssim \ve^3 |\log \ve|^{b}  \quad \inn B (0, 3R_j)  \times  [0,T].
	$$
	Define
	$$\tilde \phi_j^{7} =\sum_{i=1}^7 \phi_j^i , \quad \tilde \psi_j^{7} = \sum_{i=1}^7
	\psi_j^i.
	$$
	Arguing as in the second improvement we get 
	\begin{align}\label{Ej8}
		E_j^{in} & [\tilde \phi_j^{7}  ,  \tilde \psi_j^{7},\psi_1^{out}  , P] (y,t) :=[ \nabla^\perp  {\mathcal R}_j^{00} (y,t;P)  - \nabla^\perp  {\mathcal R}_j^{00} (y,t;\bar P^3) ] \cdot \nabla U \nonumber \\
		& +{\ve^4 |\log \ve |^{b}  \over 1+ |y|^4} \,  E_{i\geq  0}   +{\ve^5 |\log \ve |^{28}  \over 1+ |y|^3} \, E_{i\geq 0} [P, \dot P] .
	\end{align}

	\subsection{} {\it Eighth inner  improvement.} \ \ Our next step is the elimination of the Fourier mode-$0$ term of size $\ve^4 |\log \ve|^{26}$ in formula \eqref{Ej8}.  We proceed as in the third inner improvement and  define $\phi_j^8$ as follows
	$$
	\phi_j^8 (y,t) = - { \ve^4 \, |\log \ve |^{b} \over 1+ |y|^4} \, \int_0^t  \ve_j^{-2} (s)  E_0 (\rho, \theta, s , \ve) \, ds.
	$$
	It solves
	\begin{align*}
		|\log \ve | \ve_j^2 \pp_t \phi_j^8 & + {\ve^4 |\log \ve |^b \over 1+ |y|^4} \, E_{ 0}   = 0 \quad (y,t) \in  B (0, 3R_j) \times [0,T)\\
		\phi_j^8 (y,0) &= 0 \inn B(0, 3R_j) ,
	\end{align*}
	it satisfies 
	$$
	(1+ |y|) |\nabla_y \phi_j^8 (y,t) | +	|\phi_j^8 (y,t) | \lesssim {\ve^2 |\log \ve|^b \over 1+ |y|^4} \, , \quad 
	$$
	and	its Fourier decomposition  only has mode $0$.
	Let $\psi_j^8$ be the solution to 
	$$
	-\Delta_{5,j} \psi_j^8 = \phi_j^8 \quad (y,t) \in  B_{4R_j} \times [0,T), \quad \psi_j^8 =0 \quad (y,t) \in \partial B_{4R_j} \times [0,T).
	$$
	Thanks to the space decay in $\phi_j^8$ we get
	$$
	|\psi_j^8 (y,t) |  + |(1+ |y|) \nabla \psi_j^8 (y,t) | \lesssim {\ve^2 \, |\log \ve |^{b} \over 1+|y|^2} .
	$$
	As in the third inner improvement we get, for
	$$
	\tilde \phi_j^{8} = \sum_{i=1}^8 \phi_j^{i} , \quad \tilde \psi_j^{8} = \sum_{i=1}^8
	\psi_j^{i},
	$$
	\begin{align}\label{Ej9}
		E_j^{in} & [\tilde \phi_j^{8}  ,  \tilde \psi_j^{8},\psi_1^{out}  , P] (y,t) :=[ \nabla^\perp  {\mathcal R}_j^{00} (y,t;P)  - \nabla^\perp  {\mathcal R}_j^{00} (y,t;\bar P^3) ] \cdot \nabla U  \nonumber \\
		& +  {\ve^3 |\log \ve |^{b}  \over 1+ |y|^5} \, E_{1,1}  \sin \theta    +{\ve^4 |\log \ve |^{b}  \over 1+ |y|^4} \, E_{ 2} [P, \dot P]+{\ve^5 |\log \ve |^{b}  \over 1+ |y|^3} \, E_{i\geq 0} [P, \dot P] .
	\end{align}
	In comparison with \eqref{Ej8}, we loose one power of $\ve$ and produce a new error in mode $1$, in the form of $\sin \theta$. We have already seen how to deal with such a situation in the fourth inner improvement and in this way we proceed next.

	\subsection{} {\it Ninth inner  improvement.} \ \ We eliminate the terms in the error \eqref{Ej9} 
	${\ve^3 |\log \ve |^{b}  \over 1+ |y|^5} \, E_{1,1}  \sin \theta $.
	Take $\ell =4$ in the decomposition \eqref{P1} of $P^1(t)$ and 
	for  $\bar P^4= (\bar P_1^4, \ldots , \bar P_k^4 )$, $\bar P^4_j = P_0 + P_j^0 + P_j^{11} + P_j^{12} + P_j^{13} + P^{14}_j$, we take 
	$$
	e_9(y,t): = \nabla^\perp  {\mathcal R}_j^{00} (y,t;\bar P^4 )  \cdot \nabla U  +{\ve^3 |\log \ve |^{b}  \over 1+ |y|^5} \, E_{1,1} [\bar P^4, \dot P^4 - \dot P^{14}] \sin \theta .
	$$
	We have $\int_0^{2\pi} e_9 (\rho e^{i\theta} , t) \, d\theta =0$, and $e_9$ satisfies the decay \eqref{decay} with $m =5$. In fact, under the assumptions \eqref{point} on $P$, we have
	$$
	(1+ |y|^5) 	( |e_9 (y,t) | + |\log \ve |^{1\over 2} |\pp_t e_9 (y,t) | ) \lesssim \ve^3 |\log \ve|^{b}, \quad |y| <3 R_j.
	$$
	The orthogonality conditions
	$$
	\int_{B_{8R_j}}  (1+ |y|^2) e_8 (y,y)  Z_\ell (y) \, dy =0, \quad \ell =1,2
	$$
	become a system of ODEs for the point $P^{14} $ of the same form as \eqref{reduced}.
	Standard ODEs theory gives that, for all $\ve >0$ small enough there exists a unique solution $P^{14}$ to \eqref{reduced} with initial condition $P^{14} (0) = 0$, which satisfy the bounds
$$
		\| P^{14}_j \|_{L^\infty [0,T)}  + \| \dot  P^{14}_j \|_{L^\infty [0,T)}  \lesssim  \ve^2 |\log \ve |^c 
$$
	for some $c>0$.
	We denote by  $\psi_j^9 = \psi_j^9 (y,t)$ and $\phi_j^9 (y,t)= - \Delta_{5,j} \psi_j^9$ the solution to \eqref{linear1}, with $g=e_9$,
	whose existence and estimates are  given by Lemma \ref{alpha}. We have that $\int_0^{2\pi} \psi_j^9 (\rho e^{i\theta} , t ) \, d\theta =0$ and  
	$$ |\psi^9_j(y,t)| \, +\,  (1+|y|)\, |\nn \psi_j^9(y,t)|\,  \lesssim { \ve^3 |\log \ve|^{b}  \over  1+|y|} , \quad \inn B (0, 3R_j)  \times  [0,T].
	$$
	We also have
	$$
	\,  (1+ |y|^4)\, |\nabla_y \phi_j^9 (y,t)| + \,  (1+ |y|^3)\, |\phi_j^9 (y,t)| \lesssim \ve^3 |\log \ve|^{b}  \quad \inn B (0, 3R_j) \times  [0,T].
	$$
	We can differentiate in time the equation, to also get
	$$
	\,  (1+ |y|^3)\, |\log \ve |^{1\over 2} \,  |\pp_t \phi_j^9 (y,t)| \lesssim \ve^3 |\log \ve|^{27} \quad \inn B (0, 3R_j)  \times  [0,T].
	$$
	Arguing as in the fourth improvement we get, for
	$$\tilde \phi_j^{9} =\sum_{i=1}^9 \phi_j^i , \quad \tilde \psi_j^{9} = \sum_{i=1}^9
	\psi_j^i.
	$$
	We get the
	new error 
	$$
	\begin{aligned}
		E_j^{in} & [\tilde \phi_j^{9}  ,  \tilde \psi_j^{9},\psi_1^{out}  , P] (y,t) :=[ \nabla^\perp  {\mathcal R}_j^{00} (y,t;P)  - \nabla^\perp  {\mathcal R}_j^{00} (y,t;\bar P^4) ] \cdot \nabla U \nonumber \\
		& +{\ve^4 |\log \ve |^{50}  \over 1+ |y|^4} \,  E_{i\geq  2}   +{\ve^5 |\log \ve |^{52}  \over 1+ |y|^3} \, E_{i\geq 0} [P, \dot P] .
	\end{aligned}
	$$

	\subsection{} \label{ultima} {\it Tenth  inner  improvement.} \ \ Our last step consists in removing the term ${\ve^4 |\log \ve |^{b}  \over 1+ |y|^4} \,  E_{i\geq  2} $ and we argue exactly as in the first inner improvement, but with two more power  of $\ve$. We find $\psi_j^{10}$, $\phi_j^{10}$ satisfying 
	$$ |\psi^{10}_j(y,t)| \, +\,  (1+|y|)\, |\nn \psi_j^{10}(y,t)|\,  \lesssim \ve^4 |\log \ve|^{b}, \quad \inn B (0,8R_j)  \times  [0,T].
	$$
	and
	$$
	\,  (1+ |y|^3)\, |\nabla_y \phi_j^{10} (y,t)| + \,  (1+ |y|^2)\, |\phi_j^{10} (y,t)| \lesssim \ve^4 |\log \ve|^{b}  \quad \inn  B (0,8R_j)  \times  [0,T],
	$$
	such that, for 
	$$\tilde \phi_j^{10} =\sum_{i=1}^{10} \phi_j^i , \quad \tilde \psi_j^{10} = \sum_{i=1}^{10}
	\psi_j^i.
	$$
	we get
	\begin{align*}
		E_j^{in} & [\phi_j^{10}  ,  \psi_j^{10},\psi_1^{out}  , P] (y,t) := [ \nabla^\perp  {\mathcal R}_j^{00} (y,t;P)  - \nabla^\perp  {\mathcal R}_j^{00} (y,t;\bar P^4) ] \cdot \nabla U +{\ve^5 |\log \ve |^{b} \over 1+ |y|^3} \, E_{i\geq 0} [P, \dot  P] \nonumber\end{align*}
	as $\ve \to 0$, where
	the functions $E_j(\rho,\theta, t, \varepsilon)$  have the form described in \eqref{Ekkn}.  Using the definition of ${\mathcal R}^{00}_j$ in \eqref{R00}, we check that
	\begin{align*}
		{\mathcal R}_j^{00} (y,t;P) & -   {\mathcal R}_j^{00} (y,t;\bar P^4) =	\ve_j  |\log \ve|  \dot {\bf a}_j \cdot y \nonumber \\& + \ve_j [\nn_x\vp_j ( P_j ; P ) - \nn_x\vp_j ( \bar P_j^4 ; \bar P^4 )] \cdot y \\
		& +\,  \ve_j^3 D_x \theta_{j  \ve_j} (P_j; P ) [y] -  \ve_j^3 D_x \theta_{j  \ve_j} (\bar P^4_j; \bar P^4 ) [y].\nonumber
	\end{align*}  
	We take
	$$
	\phi_j^*= \bar \phi_j^{10}, \quad \psi_j^* = \bar \psi_j^{10}, \quad \psi^{*,out} = \psi_1^{out}, \quad \phi^{*,out} = \phi_1^{out}.
	$$
	Recall from \eqref{S2000} that 
	$$
	S_2  [\tilde  \phi^6, \tilde \psi^6, \phi_1^{out}, \psi_1^{out}, P]= 0 \quad (x,t) \in \Sigma \times [0,T)    .
$$
Hence
	$$
		|S_2[ \phi^{*,in}, \psi^{*,in}, \phi^{*,out}, \psi^{*,out}, P]	|\ \le \   C\ve^{4-\sigma} 
			\eta_1 \left({4|x-(r_0, 0) |\over r_0} \right).
			$$
			Besides, from \eqref{Eout1-2} we get
				$$
			\left| E^{out} [\phi^{*,out}, \psi^{*,out} , \phi^{*,in}, \psi^{*, in} , {\bf a}]  (x,t) \right| \lesssim {\ve^4 |\log \ve |^{b} \over 1+ |x|^4} , \quad (x,t) \in \Sigma \times [0,T].
			$$
	This concludes the proof of Proposition \ref{Approximation}.

	\section{The inner modified transport equation}
	\label{secProofTransportInner}

	This section is devoted to discuss the result of Lemma \ref{transport3-new}, on the solution of the Cauchy Problem \eqref{pico1}.  We write it as
	\begin{align}
		\begin{aligned}
			\ve_j^2 |\log \ve |  \pp_t \phi
			&	+ {\nabla_y^\perp(\Gamma_0(y) +  \eta_{4 \ve}  \tilde \RR_j ( y,t) ) + \eta_{4\ve} {\mathcal B}_j 
				( y,t)\over (1+ {\ve_j \over r_j} \eta_{4 \ve} y_1)} \cdot \nabla_y \phi 
			= {\eta_{4 \ve} E(y,t) \over (1+ {\ve_j \over r_j} \eta_{4 \ve} y_1)} ,
			\quad \text{in } \R^2\times [0,T],
			\\
			\phi(y,0) & =0 , \quad \text{in } \R^2.
		\end{aligned}
		\label{pico1-new}\end{align}
We refer to \eqref{eta2epsilon} for the definition of $\eta_{4\ve}$.	Problem \eqref{pico1-new} is a transport equation. It is convenient to introduce the change of variable in time and let 
	$t = t(\tau )$, 
	$$t= |\log \ve |\int_0^\tau \ve_j^2 (s) \, ds , \quad T=|\log \ve| \int_0^{\tau_T} \ve_j^2 (s) \, ds .$$ Problem \eqref{pico1-new} gets transformed into a problem for $\phi= \phi (y, t(\tau)) = \phi(y,\tau)$ as
	\begin{align}
		\begin{aligned}
\partial_\tau			\phi
			&+  {\nabla_y^\perp(\Gamma_0(y) +  \eta_{4 \ve} \tilde \RR_j ( y,t) ) +\eta_{4\ve} {\mathcal B}_j 
				( y,t)  \over (1+ {\ve_j \over r_j} \eta_{4 \ve} y_1)}   \cdot \nabla_y \phi 
			= {\eta_{4 \ve} E(y,t) \over (1+ {\ve_j \over r_j} \eta_{4 \ve} y_1)} ,
			\quad \text{in } \R^2\times [0,\tau_T],
			\\
			\phi(y,0) & =0 , \quad \text{in } \R^2
		\end{aligned}
		\label{pico2}\end{align}
	We solve \equ{pico2} by the method of characteristics.

	\medskip
	The characteristic curve $\bar y(s;\tau,y)$  is by definition the solution
	$\bar y (s)$ of the ODE system
	\begin{align}
		\label{characteristic1-new}
		\left\{
		\begin{aligned}
			\frac{d \bar y }{d s} (s) & =
			{\mathcal B} (\bar y(s) , t(s) )
			\\
			\bar y(\tau) &= y  .
		\end{aligned}
		\right.
	\end{align}
	where
	$$
	{\mathcal B} (y, t(\tau))= {\nabla_y^\perp(\Gamma_0(y) +  \eta_{4 \ve} \tilde \RR_j ( y,t) ) +\eta_{4\ve} {\mathcal B}_j 
		( y,t)  \over (1+ {\ve_j \over r_j} \eta_{4 \ve} y_1)} .
	$$
	Observe that $(1+ {\ve_j \over r_j} \eta_{4\ve} y_1) = 1+O(|\log \ve |^{-3})$, uniformly in $\ve$ and that ${\mathcal B}$ is
	log-Lipschitz in $y$ uniformly in $t$, that is, for some $L>0$,
	\[
	| {\mathcal B} (y_1,t) - {\mathcal B} (y_2,t) | \leq L |y_1-y_2| ( 1+ |\log|y_1-y_2| \, | )
	\]
	and continuous in its two variables. Hence system \eqref{characteristic1-new} has a unique solution.
	For a locally bounded function $E$, the unique solution of \equ{pico2} is then represented by the formula
	\begin{align}
		\label{phi}
		\phi(y,\tau) =
		\int_0^\tau {\mathcal E} (\bar y(s;\tau,y) , t(s))\,ds , \quad {\mathcal E}= {\eta_{4 \ve} E(y,t) \over (1+ {\ve_j \over r_j} \eta_{4 \ve} y_1)}.
	\end{align}
	Under the assumptions in Lemma \ref{transport3-new}, we get that
\begin{equation}\label{cotita}
(1+|y|^2) |D_y^2     \mathcal{E} (y,t)| + (1+|y|) |\nn_y \mathcal \pp_t \mathcal E (y,t)|+		(1+|y|) |\nn_y \mathcal E (y,t)| + |\mathcal  E (y,t)| + |\pp_t \mathcal E (y,t) |  \ \le \ C\, (1+ |y|)^\alpha .
	\end{equation}
		
	\medskip
	Our first result is
	
	\begin{lemma}\label{transport1-new}
		Let $1\le p\le +\infty$ and $m \in \R$. There exists a number $C>0$ such that
		for any function  $E(y,t)$ that satisfies
		\[
		\sup_{t\in [0,T]}
		\| (1+|\cdot |)^{-m } E(\cdot , t)\|_{L^p(\R^2)}  \ <\ +\infty
		\]
		we have that for all sufficiently small $\ve$, the solution of $\equ{pico1-new}$ satisfies
		$$
		\sup_{t\in [0,T]}
		\| (1+|\cdot |)^{-m } \phi (\cdot , t)\|_{L^p(\R^2)} \ \le \
		C \ve^{-2} |\log \ve|^{-1} \sup_{t\in [0,T]}
		\| (1+|\cdot |)^{-m } E(\cdot , t)\|_{L^p(\R^2)}  .
		$$
	\end{lemma}
	
	\begin{proof}
		Let
		$$
		H(y,s) = \Gamma_0(y) +  \eta_{4 \ve} \tilde \RR_j ( y,t (s)) .
		$$
		We use the explicit expression for $\RR_j ( y,t (s))$ to get that  along the characteristic curves defined in \eqref{characteristic1-new} one has
		\begin{align*}
			{d \over ds} &H(\bar y (s), t(s))= \nabla_y (\Gamma_0(\bar y) +  \eta_{4 \ve} \tilde \RR_j ( \bar y,t (s)) )\cdot 	\frac{d \bar y }{d s} (s) + \pp_s (\eta_{4 \ve} \tilde \RR_j ( \bar y,t (s)) ) \\
			&= \nabla_y (\Gamma_0(\bar y) +  \eta_{4 \ve} \tilde \RR_j ( \bar y,t (s)) )\cdot {\eta_{4\ve} {\mathcal B}_j 
				( \bar y,t (s)) \over (1+ {\ve_j \over r_j} \eta_{4 \ve} y_1) } + \pp_s (\eta_{4 \ve}  \tilde \RR_j ( \bar y ,t (s)) ).
		\end{align*}
		Observe that 
		\begin{align*}
			\pp_s (\eta_{4 \ve}  \tilde \RR_j ( \bar y ,t (s)) )	&= \eta_{4\ve}  {d \over dt} (\frac{\ve_j}{2r_j} ) {dt \over ds} \, \bar y_1 \, \Big(\Gamma_0(\bar y) +\bar A +\Gamma (\bar y) \Big) + \eta_{4\ve}	\pp_t \tilde {\mathcal R}^1_j (\bar y,t;P) {dt \over ds}\\
			&= \eta_{4\ve}  {d \over dt} (\frac{\ve_j}{2r_j} ) \ve_j^2 |\log \ve|  \, \bar y_1 \, \Big(\Gamma_0( \bar y) +\bar A +\Gamma (\bar y) \Big) + \eta_{4\ve}	\pp_t \tilde {\mathcal R}^1_j (\bar y,t;P) \ve_j^2 |\log \ve|,
		\end{align*}
		as ${dt \over ds} = \ve_j^2 |\log \ve|$. 
		We integrate both sides from $\tau$ to  $s $  and  get
		$$
		\begin{aligned}
			H (\bar y(s), t(s) )&- H(y,t(\tau))=  \int_\tau^s
			\nabla_y (\Gamma_0(\bar y (\xi) ) +  \eta_{4 \ve} \tilde \RR_j ( \bar y (\xi),t (\xi)) )\cdot {\eta_{4\ve} {\mathcal B}_j 
				( \bar y (\xi),t (\xi )) \over (1+ {\ve_j \over r_j} \eta_{4 \ve} \bar y_1 (\xi)) }
			\\ &+ |\log \ve|\int_\tau^s
			\eta_{4\ve}  {d \over dt} (\frac{\ve_j}{2r_j} ) \ve_j^2   \, \bar y_1 (\xi) \, \Big(\Gamma_0( \bar y (\xi)) +\bar A +\Gamma (\bar y (\xi)) \Big) \, d\xi \\
			&+ |\log \ve|\int_\tau^s \eta_{4\ve}	\pp_t \tilde {\mathcal R}^1_j (\bar y (\xi) ,t;P) \ve_j^2 (t(\xi)) \, d\xi. 
		\end{aligned}
$$
		By assumptions \eqref{point} and \eqref{b2}, we have that ${d \over dt} (\frac{\ve_j}{2r_j} ) = O(\ve |\log \ve |^{-{1\over 2}})$. Furthermore, $s \in [0, \tau)$, $\tau \in [0, \ve_j^{-2} |\log \ve |^{-1} ]$, and using \eqref{papa-new} we get
		$$
		H (\bar y(s), t(s) )- H(y,t(\tau))=  |\log \ve |^{-{1\over 2}} g(s)
		$$
		for $g(s) =O(1)$ as $s \in [0,\tau)$, uniformly as $\ve \to 0$, such that $g(\tau )=0$.
		This relation is equivalent to 
		\begin{align*}
			&(1+ \eta_{4\ve} \frac{\ve_j}{2r_j} \bar y_1 (s)  )	\Gamma_0 (\bar y(s))  + \eta_{4\ve} \frac{\ve_j}{2r_j}\,\bar  y_1 (s) \, \Big(\bar A +\Gamma (\bar y (s)) \Big) + \eta_{4\ve} \tilde  {\mathcal R}^1_j (\bar y(s) ,t(s);P)  \\
			&= (1+ \eta_{4\ve} \frac{\ve_j}{2r_j}  y_1   )	\Gamma_0 (y)  + \eta_{4\ve} \frac{\ve_j}{2r_j}\, y_1  \, \Big(\bar A +\Gamma (y) \Big) + \eta_{4\ve} \tilde {\mathcal R}^1_j (y ,t(\tau);P) +  O(|\log \ve |^{-{1\over 2}}).
		\end{align*}
		A closer look at this identity gives
		\begin{align*}
			\log (1+ |\bar y(s)|^2) = (1+ O(\eta_{4\ve} |\log \ve|^{-1}))	\log (1+ | y|^2) +  O(|\log \ve |^{-{1\over 2}}),
		\end{align*}
		and hence, taking the exponential on both sides and using the fact that $(\ve |\log \ve |)^{-|\log \ve |^{-1}} = O(1)$ as $\ve \to 0$,
		\begin{align*}
			(1+ |\bar y(s)|^2) &= 	(1+ | y|^2)\,  [ 2+ O({\log |\log \ve| \over |\log \ve | }) +O(|\log \ve |^{-{1\over 2}})] \\
			&= (1+ | y|^2) \,  [ 2 +O(|\log \ve |^{-{1\over 2}})].
		\end{align*}
		Therefore there are positive constants $a,b$ independent of $\tau\in (0, \tau_T)$ and $\ve$ such that 
		\be\label{bounds1}
a (	1+ | y(s)|^2)	\leq 1+ |\bar y(s)|^2 \leq b(	1+ | y(s)|^2 ) \quad \forall s \in [0,\tau_T).
		\ee
		We recall the representation formula \equ{phi}
		$$
		\phi(y,\tau) = \int_0^ \tau  {\mathcal E} ( \bar y (s;\tau,y) ,t(s)) \, ds  , \quad {\mathcal E} = {\eta_{4 \ve} E(y,t) \over (1+ {\ve_j \over r_j} \eta_{4 \ve} y_1)}	$$
		Using \equ{bounds1} we readily get
		$$
		\sup_{t\in [0,T]}
		\| (1+|\cdot |)^{-m } \phi (\cdot , t)\|_{L^p(\R^2)} \ \le \
		C \ve^{-2} |\log \ve|^{-1} \sup_{t\in [0,T]}
		\| (1+|\cdot |)^{-m } E(\cdot , t)\|_{L^p(\R^2)}  .
		$$
		for any $1\le p\le +\infty$, as desired.
	\end{proof}

	\medskip A consequence of the bounds \eqref{bounds1} on the characteristic curves is a control on the spacial support of $\phi$ if the spacial support of the function $E$ stays
	at a uniform large distance of the origin then so does the solution of \equ{pico1-new}. 
	
	\begin{lemma}\label{transport2}  There exist numbers $\ve_0>0$, $\beta>0$  such that for any sufficiently small $0<\ve <\ve_0$ and any locally bounded function $E$ such that
		$$E(y,t) = 0 \foral (y,t) \in B_{\ve^{-1} |\log \ve |^{-\zeta}}(0)\times [0,T] $$
		we have that the solution of $\equ{pico1-new}$ satisfies
		$$\phi(y,t) = 0 \foral (y,t) \in B_{\beta  \ve^{-1} |\log \ve |^{-\zeta}}(0)\times [0,T]. $$
	\end{lemma}
	
	\begin{proof}
		This is an immediate consequence of estimates \equ{bounds1} for the characteristics and the representation formula \equ{phi}.
	\end{proof}

	\medskip
	We can now prove Lemma \ref{transport3-new}. 
	
	\noindent{\it Proof of Lemma \ref{transport3-new}.}\ \ 
	
	The bound on the function $|\phi (y,t)|$ follows directly from the control on the characteristic curves we obtained in \eqref{bounds1} and the representation formula \eqref{phi}. 
	
	To estimate $	\pp_{y_\ell}\phi(y,\tau)$, we formally differentiate  \equ{phi} with respect to $y_\ell$, $\ell=1,2$ and obtain
	\begin{align}
		\label{phi1}
		\pp_{y_\ell}\phi(y,\tau) =
		\int_0^\tau \nn _{\bar y} {\mathcal E}(\bar y(s;\tau,y) , t(s) ) \cdot \bar y_{y_\ell}(s;\tau,y) \,ds .
	\end{align}
Recall that
	\begin{align*}
	 \left(1+ {\ve_j \over r_j} \eta_{4 \ve} (\bar y) \bar y_1 \right)\, 	\frac{d \bar y }{d s} (s) & = \nabla_y^\perp(\Gamma_0(\bar y ) +  \eta_{4 \ve} \tilde \RR_j ( \bar y,t (s)) ) +\eta_{4\ve} {\mathcal B}_j 
			( \bar y,t)  , \quad 
		\bar y(\tau) = y  ,
	\end{align*}
	(see \eqref{characteristic1-new}). Under assumptions \eqref{papa-new}, the above system has the form
		\begin{align*}
	 \left(1+ {\ve_j \over r_j} \eta_{4 \ve} (\bar y) \bar y_1 \right)\, 	\frac{d \bar y }{d s} (s) & = \nabla_y^\perp \left(\Gamma_0(\bar y ) +  \eta_{4 \ve}\frac{\ve_j}{2r_j}\, y_1 \, (\Gamma_0(y) +\bar A +\Gamma (y) ) \right) +\ve^2 |\log \ve |^{1\over 2} \, \eta_{4\ve} \,  \mathcal A (\bar y , s)  , \quad 
		\bar y(\tau) = y  ,
	\end{align*}
	where $\mathcal A$ satisfies 
		\begin{equation}\label{matha}
\left| \pp_s \mathcal A (y,s)\right|+	\left|  \mathcal A ( y , s ) \right| \leq M  | y|, \quad \left|  \nabla_y \mathcal A ( y , s ) \right| \leq M
	\end{equation}
	for some constant $M$ independent of $\ve$. Consider the point  $a= (a_1 ,a_2)$, close to $y=0$,  such that
	$$
	\Big( \nabla^\perp \left(\Gamma_0( y ) +  \eta_{4 \ve}\frac{\ve_j}{2r_j}\, y_1 \, (\Gamma_0(y) +\bar A +\Gamma (y) ) \right)  +\ve^2 |\log \ve |^{1\over 2} \, \eta_{4\ve} \,  \mathcal A ( y , s)  \Big)_{y=a} =0.
	$$
	A direct inspection gives
	$$
	a_1 = -{\ve_j \over 8 r_j} (\bar A + \Gamma (0) ) + a_{1\ve} (s) , \quad a_2 = a_{2\ve} (s) , \quad {\mbox {with}} \quad |\log \ve |^{1\over 2} |{d \over ds} a_{i\ve} | + |a_{i\ve} | \lesssim \ve^2 |\log \ve |^{1\over 2}  , \quad i=1,2.
	$$
	Introducing the change of variable $\bar z=\bar y-a$, we obtain
		\begin{align*}
	 \left(1+ {\ve_j \over r_j} \eta_{4 \ve}  (\bar z_1 + a_1) \right)\, 	\frac{d \bar z }{d s} (s) & = B_a (\bar z) +\ve^2 |\log \ve |^{1\over 2} \, \eta_{4\ve} \,  \mathcal A (\bar z+a , s)  , \quad 
		\bar z(\tau) = y -a (\tau ),
	\end{align*}
	where $\mathcal A$ stands for a function satisfying \eqref{matha} and 
	\begin{align*}
	B_a (z ) &=\nabla_z^\perp \left(\Gamma_0( z +a ) +  \eta_{4 \ve}\frac{\ve_j}{2r_j}\, ( z_1 +a_1) \, (\Gamma_0(z+a) +\bar A +\Gamma (z+a) ) \right) \\
	&= { (z+a)^\perp \over |z+a| } \left( \Gamma_0' ( |z +a| ) +  \eta_{4 \ve}\frac{\ve_j}{2r_j}\, ( z_1 +a_1) \, (\Gamma_0' (|z+a|)  +\Gamma' (|z+a|) ) \right) \\
	&+   \eta_{4 \ve}\frac{\ve_j}{2r_j} \, (\Gamma_0(z+a) +\bar A +\Gamma (z+a) ) \, {\bf e}_2 + \ve^2 |\log \ve |^{1\over 2} \, \eta_{4\ve} \,  \mathcal A_1 ( z+a , s)
	\end{align*}
	where $\mathcal A_1$ also satisfies \eqref{matha}. 
	The vector field $B_a(z)$ is smooth and the choice of $a$ gives
	\begin{equation}\label{mathb}
	B_a (0) +\ve^2 |\log \ve |^{1\over 2} \, \eta_{4\ve} \,  \mathcal A (a , s) =0, \quad \forall s.
	\end{equation}
	Let us introduce polar coordinates around $a$:
	$$
	\bar z = \bar y - a = \bar \rho e^{i \bar \theta}.
	$$
	Hence
	$$
	{d \bar z \over ds} = \bar \rho_s \hat \rho + \bar \rho \bar \theta_s \hat \theta , \quad \hat \rho = e^{i \bar \theta} , \quad \hat \theta = i e^{i \bar \theta}.
$$	
	We have that 
	\begin{align*}
	    B_a ( z) \cdot \hat \rho &=  { a^\perp \cdot \hat \rho\over |z+a| } \left( \Gamma_0' ( |z +a| ) +  \eta_{4 \ve}\frac{\ve_j}{2r_j}\, ( z_1 +a_1) \, (\Gamma_0' (|z+a|)  +\Gamma' (|z+a|) ) \right)  \\
	    	&+   \eta_{4 \ve}\frac{\ve_j}{2r_j} \, (\Gamma_0(z+a) +\bar A +\Gamma (z+a) ) \, \sin \theta  + \ve^2 |\log \ve |^{1\over 2} \, \eta_{4\ve} \,  \mathcal A_1 ( z+a , s) \cdot \hat \rho 
	    	 \,\\
	    	  \quad \quad {\mbox {and}} \quad  &\\
	    B_a (z) \cdot \hat \theta &= { z^\perp \cdot \hat \theta \over |z+a| } \left( \Gamma_0' ( |z +a| ) +  \eta_{4 \ve}\frac{\ve_j}{2r_j}\, ( z_1 +a_1) \, (\Gamma_0' (|z+a|)  +\Gamma' (|z+a|) ) \right)  \\
	    &+{ a^\perp \cdot \hat \theta \over |z+a| } \left( \Gamma_0' ( |z +a| ) +  \eta_{4 \ve}\frac{\ve_j}{2r_j}\, ( z_1 +a_1) \, (\Gamma_0' (|z+a|)  +\Gamma' (|z+a|) ) \right)  \\
	    	&+   \eta_{4 \ve}\frac{\ve_j}{2r_j} \, (\Gamma_0(z+a) +\bar A +\Gamma (z+a) ) \, \cos \theta  + \ve^2 |\log \ve |^{1\over 2} \, \eta_{4\ve} \,  \mathcal A_1 ( z+a , s) \cdot \hat \theta .
	\end{align*}
	Set
\begin{align*}
    B_{a1} (z,s) &=  { a_1 \over |z+a| } \left( \Gamma_0' ( |z +a| ) +  \eta_{4 \ve}\frac{\ve_j}{2r_j}\, ( z_1 +a_1) \, (\Gamma_0' (|z+a|)  +\Gamma' (|z+a|) ) \right)  \\
	    	&+   \, \eta_{4 \ve}\frac{\ve_j}{2r_j} \, (\Gamma_0(z+a) +\bar A +\Gamma (z+a) ).
\end{align*}
In the region we are considering
$$
B_{a1} (0,s) = O(\ve^2 |\log \ve |^{1\over 2}) , \quad |B_{a1} (z,s)| = O(\ve |\log \ve |), \quad |\pp_s B_{a1} (z,s)| = |\pp_t B_{a1} | \, \ve^2 \, |\log \ve| = O (\ve^3 |\log \ve |^{-{1\over 2}}).
$$
With the definition of $B_{a1}$, we can write
\begin{align*}
    B_a ( z) \cdot \hat \rho &=  B_{a1} \,  \sin \theta   + \ve^2 |\log \ve |^{1\over 2} \, \eta_{4\ve} \,  \mathcal A_1 ( z+a , s) \cdot \hat \rho \\
     B_a (z) \cdot \hat \theta &= 
	    	{ \rho \over |z+a| } \left( \Gamma_0' ( |z +a| ) +  \eta_{4 \ve}\frac{\ve_j}{2r_j}\, ( z_1 +a_1) \, (\Gamma_0' (|z+a|)  +\Gamma' (|z+a|) ) \right)  \\ 
	    	&+ B_{a1} \, \cos \theta + \ve^2 |\log \ve |^{1\over 2} \, \eta_{4\ve} \,  \mathcal A_1 ( z+a , s) \cdot \hat \theta .
\end{align*}
Inserting these computations in the ODEs system for ${d \bar z \over ds}$ we get
\begin{equation}\label{mathc}
	\begin{aligned}
	 \left(1+ {\ve_j \over r_j} \eta_{4 \ve}  (\bar z_1 + a_1) \right)\, 	\bar \rho_s & = B_{a1} \,  \sin \bar \theta   + \ve^2 |\log \ve |^{1\over 2} \, \eta_{4\ve} \,  \mathcal A ( \bar z+a , s) \cdot \hat \rho  , \\
	 \left(1+ {\ve_j \over r_j} \eta_{4 \ve}  (\bar z_1 + a_1) \right)\, \bar \rho	\bar \theta_s & = 	{ \bar \rho \over |\bar z+a| } \left( \Gamma_0' ( |\bar z +a| ) +  \eta_{4 \ve}\frac{\ve_j}{2r_j}\, ( \bar z_1 +a_1) \, (\Gamma_0' (|\bar z+a|)  +\Gamma' (|\bar z+a|) ) \right)  \\ 
	    	&+ B_{a1} \, \cos \bar \theta + \ve^2 |\log \ve |^{1\over 2} \, \eta_{4\ve} \,  \mathcal A ( \bar z+a , s) \cdot \hat \theta .
	\end{aligned}
\end{equation}
Let us analyse the second equation in \eqref{mathc}. Using the control on the characteristics  in the proof of Lemma \ref{transport1-new}, we have the bound \eqref{bounds1} also for $\bar z(s)$. Thanks to \eqref{mathb}, we can divide the second equation in \eqref{mathc} by $\bar \rho$ and we write it as
\begin{align*}
	\bar \theta_s & = -{4 \over 1+ \bar \rho^2} \left( 1+ g_\ve (s,\bar \rho) \right)   \,  \, + \tilde B_{a1}  \, \cos \bar \theta \,  \, + \ve^2 |\log \ve |^{1\over 2} \, \eta_{4\ve} \,  \mathcal A ( z+a , s) \cdot \hat \theta,
\end{align*}
where
$$
(1+ |z|) \, |\tilde B_{a1} (z,s)| = O(\ve |\log \ve |), \quad (1+ |z|) \, |\pp_s \tilde B_{a1} (z,s)| = O (\ve^3 |\log \ve |).
$$
Here $|g_\ve (s,\bar \rho) | $ is a function such that $|g_\ve (s,\bar \rho) | = O(|\log \ve |^{-\zeta +1})$, as $\ve \to 0$, where $\zeta >2$ is the number introduced  in the definition of $\eta_{4\ve}$ as in \eqref{eta2epsilon}. Let $\bar \theta_0$ be 
$$
\bar \theta_0 (s)  = - \int_{\tau}^s {4  \over 1+ \bar \rho^2 (\eta) } \, d\eta
$$
and let $g$ 
to solve $\pp_s	\left( \bar \theta_0 (1+g) \right)  = -{4 \over 1+ \bar \rho^2} \left( 1+ g_\ve (s,\bar \rho) \right)   $. This is possible for
$$
\bar \theta_0 g  = -\int_{\tau}^s {4 \over 1+ \bar \rho^2 (\eta) }  g_\ve (\eta ,\bar \rho)   d\eta .
$$
Let $\rho = \bar \rho (\tau)$. Following the proof of Lemma \ref{transport1-new}, we have that $\bar \rho (s)^2 = \rho^2 (1+ o(1))  $ as $\ve \to 0$, we get
$$
(1+g) \bar \theta_0 = -{4 (s-\tau) \over 1+  \rho^2  } \, \left( 1 + O\left( |\log \ve|^{-1}   \right) \right).
$$
Let  $\bar \theta_1$ solve at main order
$$
\pp_s \bar \theta_1 = \tilde B_{a1}  \, \cos \left( \bar \theta_0 \right).
$$
An integration by parts gives, for $z= \bar \theta_0 (s)   $
\begin{align*}
\bar \theta_1 &=-\int \tilde B_{a1}  {1+ \bar \rho^2 \over 4}  \cos z \, dz = \tilde B_{a1} {1+ \bar \rho^2 \over 4}  \sin \left( \bar \theta_0  \right) - \int_s^\tau {d \over ds} (\tilde B_{a1}  ) \, \, {1+ \bar \rho^2 \over 4}  \sin \left( \bar \theta_0  \right) \\
&= \tilde B_{a1}  {1+ \bar \rho^2 \over 4}  \sin \left(\bar \theta_0  \right) \left( 1+ O( |\log \ve |^{-{1\over 2}} )\right).
\end{align*} 
We write $\bar \theta = \bar \theta_0 + \bar \theta_1 + \theta_2$. We get  
 $\theta_2 = O(|\log \ve |^{-{1\over 2}})$ and $\pp_s \theta_2 = O(\ve^2 |\log \ve |^{1\over 2})$ and 
we conclude that
\begin{equation}\label{perparti0}
\begin{aligned}
    \bar \theta &= -{4 (s-\tau) \over 1+ \bar \rho^2 } \, \left( 1 + O\left( |\log \ve|^{-1}  \right) \right) +  \tilde B_{a1}  {1+ \bar \rho^2 \over 4}  \sin \left( \bar \theta_0  \right)  \left( 1+ O( |\log \ve |^{-{1\over 2}} )\right) + O(|\log \ve |^{-{1\over 2}})\\
\end{aligned}
\end{equation}
Inserting this result into the first equation in \eqref{mathc}, we get
\begin{equation}\label{perparti}
\bar \rho (s) =\rho  + \rho O (\ve  \rho \eta_{4\ve}) \cos ({4 (s-\tau) \over 1+ \rho^2}).
\end{equation}
Since $a = a(s)$, we have that
$$
{d \bar y \over dy_\ell} = {d \bar z \over d y_\ell } = \bar \rho_{y_\ell} \hat \rho + \bar \rho \bar \theta_{y_\ell} \hat \theta,
$$
and using  \eqref{phi1} we write
	\begin{align*}
		\pp_{y_\ell}\phi(y,\tau) =
		\int_0^\tau \nn _{\bar y} {\mathcal E}(\bar y , t(s) ) \cdot \left( \bar \rho_{y_\ell} \hat \rho + \bar \rho \bar \theta_{y_\ell} \hat \theta \right)  \,ds .
	\end{align*}
	First we estimate $	\int_0^\tau \nn _{\bar y} {\mathcal E}(\bar y , t(s) ) \cdot \bar \rho_{y_\ell} \hat \rho \, ds$. From \eqref{perparti} we obtain
	$$
	\bar \rho_{y_\ell}=  \left( 1 + O(\ve \rho \eta_{4\ve}) \cos {4(s-\tau) \over 1+ \rho^2} \right)\,  + 
	O(\ve \rho \eta_{4\ve})  \, {s-\tau \over 1+ \rho^2} \sin {4(s-\tau) \over 1+ \rho^2} \, .
	$$
	Set $g(\bar y, s) := \nabla_{\bar y}  {\mathcal E}( \bar y (s) , t(s) ) \cdot \hat \rho$, 
	\begin{equation}\label{boh}
	\begin{aligned}
	    \int_0^\tau \nn _{\bar y} {\mathcal E}(\bar y , t(s) ) \cdot \bar \rho_{y_\ell} \hat \rho \, ds&=  \int_0^\tau g   \left( 1 + O(\ve \rho \eta_{4\ve}) \cos {4(s-\tau) \over 1+ \rho^2} \right)\,  \, ds\\
	    &+  \int_0^\tau g \, 	O(\ve \rho \eta_{4\ve})  \, {s-\tau \over 1+ \rho^2} \sin {4(s-\tau) \over 1+ \rho^2} \,  \, ds = A+B.
	\end{aligned}
	\end{equation}
	From \eqref{cotita} we readily get
	$$
	|A| \lesssim \ve^{-2} |\log \ve |^{-1} (1+ |y|)^{\alpha -1}.
	$$
	To estimate $B$, we integrate by parts, we use that $\int x \sin x \, dx = - x \cos x + \sin x +C$  and obtain
	\begin{align*}
	    |B| & \leq O(\ve \rho \eta_{4\ve}) ) \left( |g(\bar y (\tau ), t (\tau) ) | \,  \tau \, 
	    | \cos {4\tau \over 1+ \rho^2}| + |g| \, (1+ \rho^2) \, | \sin {4\tau \over 1+ \rho^2}| \right)\\
	    &+  O(\ve \rho \eta_{4\ve}) ) \, \ve^2 |\log \ve | \, \left( C \, (1+ |y|^{\alpha -1}) \, \int_0^\tau  s \cos {4 s \over 1+ \rho^2} + \sin {4s \over 1+ \rho^2} \right) \\
	    &\lesssim  O(\ve \rho \eta_{4\ve}) ) \, \ve^{-2} |\log \ve |^{-1} \, (1+ |y|^{\alpha -1} ). 
	\end{align*}
	These estimates allow us conclude that
	\begin{align*}
	    \int_0^\tau \nn _{\bar y} {\mathcal E}(\bar y , t(s) ) \cdot \bar \rho_{y_\ell} \hat \rho \, ds \lesssim \ve^{-2} |\log \ve |^{-1} (1+ |y|)^{\alpha -1}.
	    \end{align*}
	    Next we treat the term we estimate $	\int_0^\tau \nn _{\bar y} {\mathcal E}(\bar y , t(s) ) \cdot  \bar \rho \bar \theta_{y_\ell} \hat \theta \, ds$.
	    To this purpose we observe that
	    $$
	    {d \over ds} \mathcal E = \nabla_{\bar y} \mathcal E {d \bar y \over ds} + \ve^2 |\log \ve |\, {d \over dt} \mathcal E .
	    $$
	    Since ${d \bar y \over ds} = \bar \rho_s \hat \rho + \bar \rho \bar \theta_s \hat \theta$, letting $\gamma= {\bar \theta_{y_\ell} \over \bar \theta_s}$ we get
	    $$
	    \nn _{\bar y} {\mathcal E}(\bar y , t(s) ) \cdot  \bar \rho \bar \theta_{y_\ell} \hat \theta = \gamma {d \over ds} \mathcal E - \gamma \nabla_{\bar y} \mathcal E \bar \rho_s \hat \rho - \gamma \, \ve^2 |\log \ve |\, {d \over dt} \mathcal E.
	    $$
	    We evaluate separately the following integrals
	    $$
	    I= \int_0^\tau \gamma {d \over ds} \mathcal E \, ds  , \quad II = \int_0^\tau  \gamma \nabla_{\bar y} \mathcal E \bar \rho_s \hat \rho \, ds , \quad III= \int_0^\tau \gamma \, \ve^2 |\log \ve |\, {d \over dt} \mathcal E\, ds.
	    $$
	    From \eqref{perparti0} it follows that
\begin{align*}
    \pp_{y_\ell} \bar \theta &= {8 (s-\tau) \bar \rho \bar \rho_{y_\ell} \over (1+ \bar \rho^2)^2 } \, \left( 1 + O\left( |\log \ve|^{-1}   \right) \right)   \\
    &+ \left( ( \pp_{y_\ell} \tilde B_{a1}  {1+ \bar \rho^2 \over 4}   +\tilde B_{a1}  { \bar \rho \pp_{y_\ell} \bar \rho \over 2} ) \sin \bar \theta_0  
    + \tilde B_{a1}  {1+ \bar \rho^2 \over 4}  \cos \left( \bar \theta_0  \right)  \pp_{y_\ell} \bar \theta_0 \right)
  \left( 1+ O( |\log \ve |^{-{1\over 2}} )\right))\\
    &+ O(|\log \ve |^{-{1\over 2}}).
\end{align*}
Under the assumptions on $\tilde B_{a1}$, we get
$$
\left|  ( \pp_{y_\ell} \tilde B_{a1}  {1+ \bar \rho^2 \over 4}   +\tilde B_{a1}  { \bar \rho \pp_{y_\ell} \bar \rho \over 2} ) \sin \bar \theta_0   \right| \lesssim \ve |\sin \bar \theta_0|.
$$
We thus have
\begin{align*}
    \pp_{y_\ell} \bar \theta &= {8 (s-\tau) \bar \rho \bar \rho_{y_\ell} \over (1+ \bar \rho^2)^2 } \, \left( 1 + O\left( |\log \ve|^{-1}  \right) + O( \ve \bar \rho \eta_{4\ve} )\right)  + O(|\log \ve |^{-{1\over 2}}).
\end{align*}
	    Hence
	    \begin{align*}
	        \gamma &:= {\bar \theta_{y_\ell} \over \bar \theta_s}
	        ={2 (s-\tau) \bar \rho \bar \rho_{y_\ell} \over (1+ \bar \rho^2) } \, \left( 1 + O\left( |\log \ve|^{-1}  \right) + O( \ve \bar \rho \eta_{4\ve} )\right) + O(|\log \ve |^{-{1\over 2}}) ,
	    \end{align*}
	    and
	    $$
	    {d\gamma \over ds} = O \left( {  \bar \rho \bar \rho_{y_\ell} \over (1+ \bar \rho^2) }\right) .
	    $$
	    Since
	    \begin{align*}
	        I&= \int_0^\tau  {d \over ds} (\gamma \mathcal E ) \, ds - \int_0^\tau \mathcal E {d \gamma \over ds} \, ds,
	    \end{align*}
	    we easily get
	    $$
	   | \int_0^\tau  {d \over ds} (\gamma \mathcal E ) \, ds | \lesssim | \mathcal E (y,\tau )| \, {\tau \over 1+ \rho} \lesssim \ve^{-2} \,  |\log \ve |^{-1} \, (1+ |y|)^{\alpha -1},
	   $$
	   and
	   $$
	   \left| \int_0^\tau \mathcal E {d \gamma \over ds} \, ds \right| \lesssim \ve^{-2} \,  |\log \ve |^{-1} \, (1+ |y|)^{\alpha -1},
	   $$
	    as expected. We now treat integral $II:= \int_0^\tau  \gamma \nabla_{\bar y} \mathcal E \bar \rho_s \hat \rho \, ds$. From the first equation in \eqref{mathc} and \eqref{perparti0}, we get
	    \begin{align*}
	        II&=\int_0^\tau \, B_{a1} \,  (\nabla_{\bar y} \mathcal E \cdot  \hat \rho  ) \, \bar \rho \, \bar \rho_{y_\ell} {2 (s-\tau ) \over 1+  \rho^2}
	          \,  \sin  {2 (s-\tau ) \over 1+  \rho^2}  \, (1+ o(1) ) \, ds  + O \left( \ve^2 |\log \ve |^{1\over 2}  \,  \int_0^\tau \, \gamma (\nabla_{\bar y} \mathcal E  \cdot \hat \rho  ) \, \right).
	    \end{align*}
	    To estimate the first integral in the above formula we proceed as in the estimate of the term $B$ in formula \eqref{boh}. The second integral can be bounded as follows
	    $$
	    \left| \ve^2 |\log \ve |^{1\over 2}  \,  \int_0^\tau \, \gamma (\nabla_{\bar y} \mathcal E  \cdot \hat \rho  ) \right| \lesssim \ve^2 |\log \ve |^{1\over 2} \, (1+ |y|)^{\alpha -1} \,  \int_0^\tau (s-\tau ) \, ds \lesssim  \ve^{-2} |\log \ve |^{-{3\over 2}} \, (1+ |y|)^{\alpha -1}.
	    $$
	   We thus get
	    $$
	   \left| II \right| \lesssim \ve^{-2} \,  |\log \ve |^{-1} \, (1+ |y|)^{\alpha -1}.
	   $$
	   In an analogous way we can treat the last integral $III=\int_0^\tau \gamma \, \ve^2 |\log \ve |\, {d \over dt} \mathcal E\, ds $ and we obtain
	    $$
	   \left| III \right| \lesssim \ve^{-2} \,  |\log \ve |^{-1} \, (1+ |y|)^{\alpha -1}.
	   $$
	   Collecting all the above estimates, we arrive to the desired bound
	   $$
	   \left| \pp_{y_\ell}\phi(y,\tau)  \right|  \lesssim \ve^{-2} \,  |\log \ve |^{-1} \, (1+ |y|)^{\alpha -1}.
	   $$
This concludes the proof of Lemma \ref{transport3-new}.
	


	\section{The outer modified transport equation}
	\label{secProofTransportOuter}

	\begin{proof}[Proof of Lemma \ref{int1}] Let $x \in \Sigma$. We shall prove that $\bar x (s;t,x) \in \Sigma$, for all $s \in (0,t)$. For the definition of $\bar x$ we refer to \eqref{char}. By contradiction, assume there exists $s_0 \in (0,t)$ such that  $\bar x (s;t,x) \in \Sigma$, for all $s \in (s_0,t)$ and $\bar x (s_0 ; t,x) \in \partial \Sigma$. Setting $\bar x (s_0 ; t,x) = (\bar r (s_0 ; t,x), \bar z (s_0 ; t,x))$, it holds
		$\bar r (s_0 ; t,x) =0$. Setting $B (x,t) = (B_1 (x,t), B_2 (x,t))$, we get from \eqref{ou1} that
		$B_1 (\bar x (s_0 ; t,x), s ) =0$. See \eqref{estB}. Then the ODEs
		$$
		{d \over ds} \bar x (s; x,t) \cdot {\bf e_1} = B_1 (\bar x (s;t,x),s), \quad \bar x (s_0; x,t) \cdot {\bf e_1} =0
		$$
		has the trivial solution $\bar x (s; x,t) \cdot {\bf e_1} \equiv 0$, which contradicts uniqueness of solutions for ODEs. Formula \equ{phi3} is then well-defined.
		
		In order to get estimates on the solution, we write the solution using
		Duhamel's representation formula 
		$$
		\phi^{out} (x,t) = |\log \ve |^{-1} \int_0^t  \tilde E_0 (\bar x (s;x,t  ), s) \, ds
		$$
		where $\bar x (s;x,t)$ are the characteristic curves defined as
		$$
		\, 	{d \over ds} \bar x (s; x,t  ) = B (\bar x (s; x,t  ), s ) \quad s \in (0,t ), \quad \bar x (t ; x,t ) = x.
		$$
		Using the estimates on $B(x,t)$ described in \eqref{estB}	we get 
		\[
		|\phi(x,t)|   \le |\log \ve |^{-1}   t  \|E\|_{L^\infty (\Sigma \times [0,t])} .
		\]
		We also have that for $1\le p <+\infty $,
		$$
		\int_\Sigma | E(\bar x(s;t,x),s) |^p\,dx \leq  C \int_\Sigma | E( x; t,s) |^p\,dx
		$$
		for some $C>0$, 
		since the differential of area for the change of variable  $x\mapsto \bar x(s;t,x)$ is a bounded function. From this we readily get
		\begin{equation*}
			\| \phi(\cdot ,t)\|_{L^p(\Sigma)} \ \le |\log \ve |^{-1}  \ t \sup_{s\in [0,t]} \| E(\cdot ,s)\|_{L^p(\Sigma)}.
		\end{equation*}

		\medskip
		With the change of variables $y= \frac {x-P_j(t)}{\ve}$  an equation of the type \equ{pico1-new}
		is satisfied. The result then follows from
		Lemma \ref{transport2}: since $E_0$ and $\tilde E_0$ are zero in the spacial region $\sum_{j=1}^k B(P_j,2 |\log \ve |^{-3})$ for all $t \in [0,T)$, there exists $\beta >0$ such that the solution $\phi$ to \eqref{transport-outer} is zero in  $\sum_{j=1}^k B(P_j,\beta |\log \ve |^{-3})$, for all $t \in [0,T)$. An alternative way to see this is to estimate the characteristics for points $x \in B(P_j , \delta |\log \ve |^{-3} )$ and use \eqref{estB1}. 
		
		\medskip
		This fact allows us to say that $\phi$  satisfies an equation of the form
		$$
				|\log \ve | \pp_t 	\phi  + \nn_x^\perp H \cdot \nn \phi  =  E(x,t)
				\quad \text{in } \Sigma \times [0,T], \quad 
				\phi(x,0)  =0 \quad \text{in } \Sigma,
			$$
		where  $$ H(x,t)= |\log \ve | \Big(1- \sum_{j=1}^k\eta_j(x,t) \Big ) B$$
		and
		$\eta_j$ is a smooth function with $\eta_j(x,t) =1$ whenever $|x-P_j(t)|<\beta  |\log \ve |^{-3}$  and $=0$ if  $|x-P_j(t)|>2 \beta  |\log \ve |^{-3}$.
		Using the explicit expression of the coefficient $B$ given in \eqref{ou1}-\eqref{estB1}, we see that
		$$
		H(x,t)= \Big(1- \sum_{j=1}^k\eta_j(x,t) \Big )  O(|\log \ve | ), \quad \text{in } \Sigma\times [0,T],
		$$
		uniformly as $\ve \to 0$. Hence
		$$
		|\bar x (s;t,x) | = |x| + O(1), \quad \ass \quad \ve \to 0, 
		$$
		and bounds \eqref{poto1} readily follow.
	\end{proof}

	\subsection{Uniform continuity}

	Let us consider an equation of the form
	\be\label{ff}
	\left\{
	\begin{aligned}
		|\log \ve | \, r \, \pp_t \phi  + \nn_x^\perp (r^2 (\Psi^0 - \alpha_0 |\log \ve | + e )  ) \cdot \nn \phi  =  &E(x,t)
		\quad \text{in } \Sigma \times [0,T],
		\\
		\phi(x,0) & =0 , \quad \text{in } \Sigma .
	\end{aligned}
	\right.
	\ee
 Let
$$
		|\log \ve |	{d \over ds}  {\bar x} (s; t ,x ) = \bar r \nabla^\perp 
		H \left( \bar x (s;t ,x) \right) + 2 H \left( \bar x (s;t ,x ) \right) {\bf e}_2 , \quad
		\quad \bar x (t ; t , x ) = x
	$$
	where 
	$$
	H (x,t):= \Psi^0 (x,t) - \alpha_0 |\log \ve | + e (x,t), \quad x=(r,z).
	$$
	We know that the characteristics  $\bar x = \bar x (s; t,x)$  are well defined in $\Sigma \times [0,T]$ if $\nabla e$, $r \Delta e 
	\in L^\infty (\Sigma  \times [0,T])$. Besides the formula
$$
		\phi(x,t)  = \int_0^t E(\bar x(s; t,x) , s) \, ds\
$$
	gives the solution of \equ{ff}.
	We have

	\begin{lemma} \label{modc1}
		
		For all $\varrho>0$ there exists a positive number
		$$\delta = \delta (\varrho,\| \nabla  e\|_{L^\infty(\Sigma \times [0,T])} , \| r\Delta e\|_{L^\infty(\Sigma \times [0,T])}, T, \Sigma) $$
		such that
		for all $(x_1,t_1  ), \, (x_2,t_2  ) \in \bar \Sigma \times [0,T]$ we have
		$$
		|t_1 - t_2 | + |x_1 - x_2 | < \delta \implies  \left| \bar x  (s; t_1 , x_1 ) - \bar x (s; t_2 , x_2) \right| < \varrho.
		$$
		If  $E(x,t)$ is a bounded function that  satisfies
	$$
		\sup_{t\in [0,T], \ |x_1-x_2|< \mu} |E(x_1,t)-E(x_2,t)|\le \Theta (\mu)
		$$
		for a certain function $\Theta$ with  $\Theta(\mu)\to 0$ as $\mu\to 0$, then  the solution $\phi(x,t)$ of  \equ{ff} satisfies that
		for all $(x_1,t_1  ), \, (x_2,t_2  ) \in \bar \Sigma  \times [0,T]$ we have
		$$
		|t_1 - t_2 | + |x_1 - x_2 | < \delta \implies  \left| \phi( x_1,t_1 ) - \phi (x_2,t_2) \right| < \varrho.
		$$
	\end{lemma}
	
	\begin{proof}
		By definition
		\begin{align*}
			|\log \ve |\, 	{d \over ds}  {\bar x} (s; t_i ,x_i ) &= \bar r (s;t_i ,x_i )  \nabla^\perp H\left( \bar x (s;t_i ,x_i ) \right) + 2 H \left( \bar x (s;t_i ,x_i ) \right) {\bf e}_2 , \quad \\
			\bar x (t_i ; t_i , x_i ) &= x_i
		\end{align*}
		for $i=1,2$. Let $h(s) = \bar x (s; t_1 , x_1 ) - \bar x (s; t_2 , x_2 )$. Then
		$$
		| {d \over ds} h (s) | 
		\leq C \,  A \, | h(s) | \, |\log (|h (s) |) |,
		$$
		where $A= \left( \| \nabla e \|_\infty +\| r \Delta  e \|_\infty \right) $.
		Setting $\beta (s) := |h(s) |^2$, we can assume that $0< \beta (s) <1$. Then we get
		$$
		\left| {d \over ds} \left( \log \left( \log {1\over \beta (s) }\right) \right) \right| \leq C  \, A .
		$$
		Integrating, we  obtain
		\be \label{cott}
		e^{-C \, A  \, T} \, \log {1 \over \beta (t_1 ) } \leq \log {1\over \beta (s) } \leq  e^{ C \, A \, T} \log {1\over  \beta (t_1 ) }.
		\ee
		Observe now that
		\be \label{cott1}
		|h (t_1)|  = |x_1 - x_2 | + |x_2 - \bar x (t_1 ; t_2 , x_2 ) | \leq C \, A \,  \left( |x_1 - x_2 | + |t_1 - t_2| \right) .
		\ee
		Combining \eqref{cott} and \eqref{cott1}, we obtain the first statement of the Lemma.
		
		\medskip The second statement of the Lemma  is a direct consequence of  the representation formula
		$$
		\phi(x,t)  = \int_0^t E(\bar x(s; t,x) , s) \, ds\
		$$
		for the solution of \equ{ff}.
	\end{proof}

	\medskip
	Let us now consider
	\be\label{ff1}
	\left\{
	\begin{aligned}
		\partial_\tau \phi  +  \nn_x^\perp  \left( \Gamma_0 (y) + b(y,t) \right) \cdot \nn \phi  =  &E(y,t)
		\quad \text{in } \R^2 \times [0, \tau_T],
		\\
		\phi(y,0) & =0 , \quad \text{in } \R^2 ,
	\end{aligned}
	\right.
	\ee
	with $b(y,t) = 0= E(y,t)$ for $|y| >R$ and 
	$$t= |\log \ve |\int_0^\tau \ve_j^2 (s) \, ds , \quad T=|\log \ve| \int_0^{\tau_T} \ve_j^2 (s) \, ds .$$ 
	This problem has the same form as the one in \eqref{pico2}, and under our assumption
	$$
	\tau_T \sim \ve^{-2} |\log \ve |^{-1}, \quad R \to \infty \ass \ve \to 0.
	$$
	As in the previous problem, the modulus of continuity for the characteristics
	$\bar y = \bar y (s; \tau , y)$ for \eqref{ff1} depends only on $\| \Delta_y  b \|_{L^\infty (\R^2 \times [0,T] ) }$, and that of the solution
	only on a uniform bound for $E$ and for its  modulus of continuity. Arguing as in the proof of Lemma \ref{modc1}, we can show that

	\begin{lemma} \label{modc2}
		Assume that
		$$
		\sup_{\tau \in [0,\tau_T], \ |y_1-y_2|< \mu} |E(y_1,\tau )-E(y_2,\tau )|\le \Theta (\mu)
		$$
		for a certain function $\Theta$ with  $\Theta(\mu)\to 0$ as $\mu\to 0$. Then
		for each $\varrho>0$ there exists a positive number
		$$\delta = \delta (\varrho,\|\Delta b\|_{L^\infty(\R^2 \times [0,\tau_T ])},\|E\|_\infty, \Theta, T, \ve) $$ such
		that the solution $\phi(y, \tau )$ of  \equ{ff1} satisfies that
		for all $(y_1,\tau_1  ), \, (y_2,\tau_2  ) \in  \R^2 \times [0,\tau_T ]$ we have
		$$
		|\tau_1 - \tau_2 | + |y_1 - y_2 | < \delta \implies  \left| \phi( y_1,\tau_1 ) - \phi (y_2,\tau_2) \right| < \varrho.
		$$
	\end{lemma}

	\medskip 
	
	These results will be useful in Section \ref{sec10} for the final argument in our construction.

	\section{The inner-outer gluing procedure}\label{sec8}
	
	In the previous sections 
	we proved the existence of  points $P_1^1, \ldots , P_k^1$ in \eqref{point} such that, for any choice of points ${\bf a}_1 , \ldots , {\bf a}_k$ in \eqref{point} satisfying \eqref{b11}, there exists  a good approximate leapfrogging of vortex rings $(W^*, \Psi^*)$ with the form \eqref{f110}-\eqref{f220}, as described in Proposition \ref{Approximation}. 
	
	\medskip
	We now set up  the inner-outer gluing procedure which will lead us  to find an exact solution $(\Psi , W)$ to \eqref{leap0} close to  the approximation $(\Psi^*, W^*  )$.  
	
	\medskip
	The solution $(\Psi,W)$ will have the   form 
	\begin{equation}\label{f11}
		\begin{aligned}
			\Psi &= \Psi^* (x,t) + \sum_{j=1}^k \bar \eta_{j2} {1\over r_j} \psi_j ({x-P_j \over \ve_j },t) + \psi^{out} (x,t) \\
			W &= W^* (x,t) + \sum_{j=1}^k \bar \eta_{j1}  {1\over r_j \ve_j^2 } \phi_j ({x-P_j \over \ve_j },t) + \phi^{out} (x,t) 
		\end{aligned}
	\end{equation}
	where
	$$\bar \eta_{jN} (x) =\eta_N  ({|\log \ve |^5 |x-P_j |\over  \delta  }).$$ Here $\eta_N$ is the cut-off function introduced in \eqref{cutoff}.  This ansatz has the same form as the one used in \eqref{f110} and \eqref{f220} for the construction of the improved approximate leapfrogging of vortex rings $(W^*, \Psi^*)$. We notice though that here we are taking  cut-offs  slightly shorter than the ones in \eqref{zeta}.

	Let $S_1$ and $S_2$ be the Euler operators introduced in \eqref{defS}.
	Then  the operator $S_1$  evaluated at $(W,\Psi)$ becomes
	\begin{align}\label{fullS1}
		S_1(W,\Psi) &= \sum_{j=1}^k {\bar \eta_{1j} \over \ve_j^4} E_j [\phi_j , \psi_j \, \psi^{out} , P] + E^{out} [\phi^{out}, \Psi^{out} , \phi^{in}, \psi^{in} , P] \\
		&{\mbox {where}} \nonumber \\
		&\phi^{in}(y,t) = (\phi_1(y,t) , \ldots, \phi_k(y,t) ), \quad \psi^{in}(y,t)  = (\psi_1(y,t) , \ldots, \psi_k(y,t) ). \nonumber \end{align}
	In \eqref{fullS1} $E_j$, $j=1, \ldots , k$,  and $E^{out}$ are defined respectively 
	\begin{align}\label{Ej}
		E_j^{in} & [\phi_j , \psi_j \, \psi^{out} , {\bf a} ] (y,t) := |\log \ve | \ve_j^2 (1 +{\ve_j \over r_j } y_1) \pp_t \phi_j   + |\log \ve | B_0 (\phi_j )  \nonumber \\
		& + \nabla^\perp \left(  (1+{\ve_j \over r_j} y_1)^2 (\psi_j^0  - r_j \alpha_j  |\log \ve_j | +\psi_j^*    + \psi_j + r_j \psi^{out} )  +{\mathcal R}_j (y,t;P) \right)\cdot \nabla \phi_j \nonumber \\
		&+ \nabla^\perp \left((1+{\ve_j y_1 \over r_j} )^2 (\psi_j +r_j \psi^{out} ) \right) \nabla  (\ve_j^2 r_j W^* ) \nonumber \\
		& +\ve_j^4 S_1 (W^*, \Psi^*) (\ve_j y + P_j),  \quad |y|< 3R , \quad R:= {1 \over \ve |\log \ve|^5} 
	\end{align}
	with $y={x-P_j \over \ve_j} $, where $B_0 $ is defined in \eqref{defB0}, and 
	\begin{equation}
	\label{Eoutn}
	\begin{aligned}
		E^{out} &[\phi^{out}, \Psi^{out} , \phi^{in}, \psi^{in} , {\bf a}] (x,t) :=
		|\log \ve | \, r \, \partial_t \phi^{out}\\
		&+ \nabla_x^\perp ( r^2 (\Psi^*  + \sum_{j=1}^k { \bar \eta_{j2}\over r_j} \psi_j ({x-P_j \over \ve_j }) + \psi^{out} -r_0^{-1} |\log \ve |)) \nabla_x \phi^{out}  \\
		&+\sum_{j=1}^k \left[ r \, |\log \ve |  \,  \pp_t \bar \eta_{j1} + \nabla_x^\perp ( r^2 (\Psi^*  + \sum_{j=1}^k  { \bar \eta_{j2}\over r_j} \psi_j ({x-P_j \over \ve_j }) + \psi^{out}-r_0^{-1} |\log \ve |)) \nabla \bar \eta_{1j} \right] {\phi_j \over \ve_j^2 r_j}   \\
		&+ \left[ \sum_{j=1}^k (\bar \eta_{2j} - \bar \eta_{1j}) \nabla_x^\perp (r^2 ({\psi_j \over r_j} +\psi^{out}) ) + {r^2 \psi_j \over r_j} \nabla_x^\perp \bar \eta_{2j} \right] \nabla_x W^*   \\
		&+ (1-\sum_{j=1}^k \bar \eta_{2j} ) \nabla^\perp (r^2 \psi^{out} ) \cdot \nabla W^* 
		+ (1-\sum_{j=1}^k \bar \eta_{j1} ) S_1 (W^*,\Psi^*)=0 \quad (x,t) \in \Sigma \times [0,T)    
	\end{aligned}
	\end{equation}
	It is straightforward to check that a pair $(W, \Psi)$ of the form \eqref{f11} is a solution to \eqref{leap0} if $(\phi^{in}, \psi^{in}, \phi^{out}, \psi^{out} )$ solve 
	the {\it inner-outer gluing} system given by the inner problem
	\begin{equation} \label{inner}
		\begin{aligned}
			E_j^{in} & [\phi_j , \psi_j \, \psi^{out} , {\bf a}] (y,t)  = 0, \quad (y,t) \in B(0; 3R ) \times [0,T)\\
			-\Delta_{5,j} \psi_j &= \phi_j,  \quad (y,t) \in B(0; 3R  ) \times [0,T)
		\end{aligned}
	\end{equation}
	for all $j=1, \ldots , k$, 	coupled with the outer problem
	\begin{equation}\label{out}
		\begin{aligned}
			E^{out} &[\phi^{out}, \psi^{out} , \phi^{in}, \psi^{in} , {\bf a}] (x,t) =0, \quad (x,t) \in \Sigma \times [0,T)\\
			E_1^{out} &[\phi^{out}, \psi^{out} , \phi^{in}, \psi^{in} , {\bf a}] (x,t) =0, \quad (x,t) \in \Sigma \times [0,T)
		\end{aligned}
	\end{equation}
	where
	\begin{align}
\nonumber
		E_1^{out}& := 	\Delta_5 \psi^{out} + \phi^{out} + \sum_{j=1}^k (\bar \eta_{j1} -\bar \eta_{j2} ) {\phi_j \over r_j \ve_j^2} \\
  \label{Eout1}
		&+\sum_{j=1}^k ( {\psi_j \over r_j} \Delta_5 \bar \eta_{j2}  + 2 \nabla_x \bar \eta_{j2} \nabla_x {\psi_j \over r_j}  ), \quad (x,t) \in \Sigma \times [0,T)	
	\end{align}
	with boundary and decay conditions on $\psi^{out}$
	\begin{equation}\label{boundary}
		{\pp \over \pp r } \psi^{out} (x,t) = 0 , \quad (x,t) \in \pp \Sigma \times [0,T], \quad |\psi^{out} (x,t)| \to 0 , \quad \ass |x| \to \infty. 
	\end{equation}
	The rest of the paper is devoted to solve the inner-outer gluing system \eqref{inner}-\eqref{out}-\eqref{boundary}.
	
	\subsection{Setting the inner problems in the whole $\R^2$} 	The inner problems  in \eqref{inner} can be extended to the whole space $\R^2$ with the use of cut-off functions. We will use two types of cut-off functions: one with large support much beyond the region $B(0,3R)$, the other with support bigger that   $B(0,3R)$ but of comparable size.
	
	Using the notation introduce in \eqref{cutoff}, we define
	\begin{equation}\label{ccuts}
	\chi (y_1) = \eta_{10} \left({|y_1| \over \ve \, |\log \ve |^3}\right), \quad \eta_{4\ve} (y) = \eta_4 \left({|y| \over R}\right), \quad {\mbox {where}} \quad R:= {1 \over \ve |\log \ve|^5} .
\end{equation}
We use $\chi$ to extend the operator $\Delta_{5,j}$  to $\R^2$. The functions $\phi_j$ and $\psi_j$ are related via the operator $\Delta_{5,j}$ as described in \eqref{inner}. This operator can be written as
	$$
		-\div \left(  (1+ {\ve_j \over r_j} y_1)^3  \nabla \psi_j \right) = (1+ {\ve_j \over r_j} y_1)^3 \, \phi_j,  \quad (y,t) \in B(0, 3R  ) \times [0,T).
		$$
Letting $\mathcal P (y_1, t)$ be defined as
\begin{equation}\label{defP}
		\begin{aligned}
			{\mathcal P}  (y_1,t) &= p(y_1 , t)^3 , \quad  \quad {\mbox {where}} \quad p(y_1 , t)= \left( 1+ {\ve_j \over r_j} y_1  \chi \right),
		\end{aligned}
	\end{equation}
 it is clear that, if $(\psi_j, \phi_j)$ are such that
	$$
		-\div \left(  \mathcal P   \nabla \psi_j \right) = \mathcal P \, \phi_j,  \quad \inn \quad \R^2 \times [0,T]
	$$
	then their restriction to $B(0,3R)$ satisfy the second relation in \eqref{inner}. 
	
	\medskip{}{}
	Let us introduce the functions
	\begin{equation}
	    \label{mm}
	 \begin{aligned}
  \psi_{0j} (y,t) &=  \Gamma_0(y)
				+ \frac{\ve_j y_1 }{2 r_j  }\, \chi \, \Big(-3 \Gamma_0(y) +A_\ve - 4 K(P_j, P_j) +\Gamma (y)
				\Big)\\
   \Gamma_{0j}(y,t) &=  \Gamma_0(y)
				+ \frac{\ve_j y_1 }{2 r_j  }\, \chi \, \Big(\Gamma_0(y) +\bar A +\Gamma (y)
				\Big) .
    \end{aligned}
	\end{equation}
	For $y \in B(0, 3R  )$ (and beyond) these functions  coincide with part of the expansion of $\psi_j^0 + \psi_j^*$ and $ (1+{\ve_j \over r_j} y_1)^2 (\psi_j^0  - r_j \alpha_j  |\log \ve_j | +\psi_j^*  )$ respectively, as given in \eqref{barp00n}:
			\begin{equation}\label{mm1}\begin{aligned}
   \psi_j^0    + \psi_j^* &=  \psi_{0j} - 4 \log\varepsilon_j - \log 8 + K (P_j;P_j) +c_j^* (y,t)  +c_j^{**} (y,t) \\
					(1+{\ve_j \over r_j} y_1)^2 (\psi_j^0  - r_j \alpha_j  |\log \ve_j |  + \psi_j^* )
			&	= 
				\Gamma_{0j}(y) - (4-\alpha_j  r_j ) \log\varepsilon_j - \log 8 + K (P_j;P_j)\\
    &
				+b_j^{*} (y,t) +b_j^{**} (y,t)
			\end{aligned}
   \end{equation}
where
\begin{equation}\label{similar}
\begin{aligned}
|\log \ve |^{1\over 2}	&|\pp_t c_j^{*}(y,t)| + (1+ |y|) |  \,  \nn_y c_j^{*} (y,t)|  + |c_j^{*} (y,t)| \ \le  \  C \ve^2 \, (1+ |y|^2) \, \log (1+ |y|) \\
	& (1+ |y|) |  \,   \nn_y c_j^{**} (y,t)|  + |c_j^{**} (y,t)| \ \le  \  C \ve^{2+\bar \sigma} \, (1+ |y|^2) \, 
\end{aligned}
\end{equation}
for $y= {x-P_j \over \ve_j}$, $|y|< |\log \ve |^{-3}$, for some fixed $\bar \sigma >0$. Similar estimates hold true for $b_j^*$ and $b_j^{**}$. We refer to \eqref{barp00n}.
Besides, from \eqref{d} we get that
	$$
	\ve_j^2 r_j W^* = f \left( (1+{\ve_j \over r_j} y_1)^2 (\psi_j^0  - r_j \alpha_j  |\log \ve_j | +\psi_j^*  ) \right) (1+ O(\ve^2))\sim  f_0 \left( \Gamma_{0j}(y,t) + b_j^{*} +b_j^{**}\right) , \quad {\mbox {with}} \quad f_0 (s) = e^s.
	$$
	Let
 $$
B_0 (\phi ) = 	\mathcal B_j^0  \cdot \nabla \phi , \quad 
		\mathcal B_j^0 := -\ve_j \pp_t \ve_j (1+{\ve_j \over r_j} y_1)     y  -{\ve_j^2 \over r_j} y_1 \pp_t P_j. 
 $$
 With abuse of notation, we can write
 $$
 (1+{\ve_j \over r_j} y_1 \chi )^2 (\psi_j^0  - r_j \alpha_j  |\log \ve_j | + \psi_j^*)       + |\log \ve | \, \eta_{4\ve} \, \mathcal B_j^0 +\eta_{4\ve} \, {\mathcal R}_j (y,t;P)
 $$
 to be equal to
 $$
 \Gamma_{0j}+ \eta_{4\ve } \, b_j^{*} +\eta_{4\ve} b_j^{**}
 $$
 for some new $b_j^*$, $b_j^{**}$ satisfying \eqref{similar}. Let us also introduce
	\begin{equation}\label{e}
	\begin{aligned}
	b_j (\psi^{out} , {\bf a} )&=   (1+{\ve_j \over r_j} y_1 \chi )^2 \,  \eta_{4\ve} \, ( \hat \psi_j  +   r_j \psi^{out} )  \\
		\ve_j^2 r_j W^* &= f_0 \left( \Gamma_{0j} + b_j^{*} \right) + U^* 
		, \quad f_0(s) = e^s.
	\end{aligned}
	\end{equation}
	The inner problem \eqref{Ej}-\eqref{inner} are the restriction to $B(0,3R)$ of the following equations
	\begin{equation}\label{Ejn}
 \begin{aligned}
		E_j  [\phi_j , \psi_j \, \psi^{out} , {\bf a} ] (y,t) &:= |\log \ve |  \, \ve_j^2 \,  (1 +{\ve_j \over r_j } y_1 \chi ) \, \pp_t  \phi_j     \\
		& + \nabla^\perp \left( \Gamma_{0j} + \eta_{4\ve } b_j^{*} +\eta_{4\ve } b_j^{**}  +  b_j ( \psi^{out}, {\bf a} ) \right)\cdot \nabla   \phi_j  \\
		&+ \nabla^\perp ( (1 +{\ve_j \over r_j } y_1 \chi )^2   \psi_j) \cdot \nabla \left( f_0\left( 	\Gamma_{0j} +  b_j^{*} \right) + \eta_{4\ve}  U^* \right) \\
		&+	 \nabla^\perp \left(  \eta_{4\ve} (1+{\ve_j y_1 \over r_1} \chi )^2 r_j \psi^{out}  \right)   \nabla \left( f_0\left( 	\Gamma_{0j} +  b_j^{*}  \right) + \eta_{4\ve}   U^* \right)   \\
&	+ {\mathcal E}_j  [\phi_j^* , \psi_j^* \, \psi^{*,out} , {\bf a} ] (y,t) 	\quad \inn \R^2 \times [0,T],
\\
	-\div \left(  \mathcal P   \nabla \psi_j \right)& = \mathcal P \, \phi_j,  \quad \inn \quad \R^2 \times [0,T] 
	\end{aligned}
 \end{equation}
	where $ {\mathcal E}_j  [\phi_j^* , \psi_j^* \, \psi^{*,out} , P]$ is defined and estimated in Proposition \ref{Approximation}.

	\subsection{Decomposition of $\phi_j$ and $\psi_j$}
	For a function $h :\R^2 \to \R$, consider the problem
	\begin{equation}\label{an}
		- \div \left( {\mathcal P} \nabla \psi \right) = h  \quad \inn \quad \R^2 \times [0,T].
	\end{equation}
	A necessary condition for solvability of \eqref{an} is that
	\begin{equation}\label{an1}
		\int_{\R^2} h (y,t) \, dy =0, \quad \forall t \in [0,T).
	\end{equation}
	We will make use of the following lemma
	
	\begin{lemma}\label{ii} Assume $h:\R^2 \times [0,T] \to \R$ satisfies \eqref{an1} and
		\begin{equation}\label{an2}
			\int_{\R^2} U^{-1} (y) h^2(y,t) \, dy < \infty, \quad \forall t \in [0,T], \quad {\mbox {where}} \quad U(y) = {8 \over (1+ |y|^2)^2} .
		\end{equation}
		Then there exists a unique solution $\psi$  to \eqref{an} such that, for all $t \in [0,T]$,  $\psi (\cdot , t) \in C^{0, \sigma } (\R^2)$ and
		$$
		|\psi (y,t) | \lesssim { \| U^{-{1\over 2} } h \|_{L^2 (\R^2) }\over 1+ |y|^{1-\sigma}}
		$$ 
		for some $\sigma \in (0,1)$, and 
		$$
		\| U^{-{1\over 2} +{1\over p}} \nabla \psi \|_{L^p (\R^2)} \lesssim \| U^{-{1\over 2} } h \|_{L^2 (\R^2) }
		$$
		for any $p>2$.
		If $\| U^{{1\over q}-1} h \|_{L^q (\R^2) } <\infty$ for some $q>2$, then 
		$$
		|\psi (y,t)| + (1+ |y|) |\nabla \psi (y,t) | +
		(1+ |y|)^{1+ \sigma } [\nabla \psi ]_\sigma \lesssim {\| U^{{1\over q}-1} h \|_{L^q (\R^2) }  \over 1+ |y|},
		$$ 
		where we denote
		$$
		[\nn \psi ]_\sigma (y) = \sup_{y_1,y_2\in B_1(y)} \frac {|\nn \psi (y_1) -\nn \psi (y_2) |}  {|y_1-y_2|^\sigma} 
		$$
		and		$\sigma \in (0,1)$.
	\end{lemma}
	
	\begin{proof} It is convenient to  pull equation \equ{an}  into the sphere $S^2$ by means of  the stereographic projection
		$$
		\pi \, : \,  S^2 -\{(0,0,1)\}: \to \R^2
		$$
		\be \label{stereo}
		\pi(z) = (    \frac{z_1}{1-z_3},\frac{z_2}{1-z_3}), \quad z \in S^2 -\{(0,0,1)\},
		\ee
		whose inverse is given by $\pi^{-1} : \R^2 \to S^2 -\{(0,0,1)\}$
		$$
		\pi^{-1} (y) =   (\frac{2y_1}{1+ |y|^2}, \frac{2y_2}{1+ |y|^2},\frac{|y|^2-1 }{1+|y|^2}), \quad y\in \R^2.
		$$
		For a function  $h(y):\R^2\to \R$ we denote by $h(z)$ the function $h(\pi(z))$ defined on $S^2$. Then the differential of volume is given by $d \pi (z) = {4 \over (1+ |y|^2)^2} \, dy = {U(y) \over 2} \, dy$, where the function $U$ is the one in \eqref{defU}. Thus  we get
		\begin{align*}
			\nabla_{\R^2} h(y) = {U^{1\over 2} \over \sqrt{2}} (\nabla_{S^2} h ) (z), \quad 	\Delta_{\R^2} h(y) = {U\over 2}   (\Delta_{S^2}  h ) (z) ,  \quad  y = \pi(z).
		\end{align*}
		and
		$$
		\int_{\R^2}h(y)\, dy \  =\ 2 \int_{S^2}  h(z)\,  U^{-1}(z) \, d\pi (z)  .
		$$ 
		For a vector-valued function  $F(y):\R^2\to \R^2$ we have
		$$
		{\div }_{S^2} F  (z) = \sqrt{2} U^{-1} {\div}_{\R^2} (U^{1\over 2} F) (y).   
		$$
		Let us denote
		$$
		\ttt h(y,t) = \frac {2\, h(y,t)}{U(y)}
		$$
		Then Equation \equ{an}  gets transformed into
		\be\label{poisson1}
		-{\div}_{S^2} (\mathcal P \nabla_{S^2} \psi )  = \ttt h \inn S^2  .
		\ee
		From \eqref{defP} we get that $\mathcal P (y_1,t) = 1+ O(\ve R)$, for all $t \in [0,T]$, $y_1 \in \R$.  Assumptions \eqref{an1} and \eqref{an2} become
		$$
		\int_{S^2} \ttt h^2 (z) \, d\pi(z) \ =\ 2 \int_{\R^2}  h^2(y) U(y)^{-1}\, dy  <\infty , \quad \int_{S^2} \ttt h (z) \, d\pi (z) \ =\ 0.
		$$
		The latter condition in the above formula  implies the existence of a unique solution of \equ{poisson1} with mean value zero. This solution is in $H^2(S^2)$, hence it is H\"older continuous of any order. After adding a proper constant, we choose the solution which vanishes at the sounth pole  $S=(0,0,1)$. Pulling back this function to a $\psi(y)$ defined in $\R^2$ we see that
		it satisfies equation \equ{an} and it is the only solution that vanishes as $|y|\to \infty$.  This and the H\"older condition yields that $\psi$  satisfies
		\be\label{cota1psi}
		|\psi(y)| \ \le \  \frac {C}{1+ |y|^{1-\sigma}}\,  \| h\, U^{-\frac 12} \|_{L^2(\R^2)}
		\ee
		for an arbitrarily small $\sigma>0$. Moreover, for all $p>2$ we have an $L^p(S^2)$-gradient estimate for $\ttt \psi(z)$ of the form
		$$
		\|\nn_{S^2} \ttt\psi \|_{L^p(S^2)} \ \le \ C  \| \ttt \phi \|_{L^2(S^2)}
		$$
		which yields for $\psi$ in \equ{an}
		\be\label{gradp}
		\| U^{-\frac 12 + \frac 1{p}} \nn \psi \|_{L^p(\R^2)} \ \le \ C  \| h\, U^{-\frac 12} \|_{L^2(\R^2)}.
		\ee
		If  in addition we have that $\ttt h\in L^q(S^2)$ for some $q>2$, then a solution of
		\equ{poisson1} satisfies
		$$
		\|\nn_{S^2} \ttt\psi \|_{C^{0,\sigma} (S^2)} \ \le \ C  \| \ttt h \|_{L^q(S^2)}.
		$$
		for some $0<\sigma<1$.
		This estimate translates for $\psi $  into
	$$  |\psi(y)| \, + \,  (1+|y|) \, |\nn \psi(y) | + (1+|y|)^{1+\sigma} [\nn \psi ]_\sigma (y)    \ \le \frac C{1+|y|}  \| U^{\frac 1q -1} h \|_{L^q(\R^2)}. $$
		The proof is completed.
	\end{proof}
 
\medskip

We shall make a  decomposition of the functions $\phi_j$, $\psi_j$ introduced in \eqref{f11}. 

\medskip
	Let
	\begin{align*}
		\mathcal Z_{20}(y) = \Gamma_0(y),\quad
		\mathcal Z_{23}(y)  =  2\Gamma _0(y) + \nn_y \Gamma_0 (y)\cdot y  ,  
	\end{align*}
	where $\Gamma_0$ is given in \eqref{defU}.
	We  write $\phi_j(y,t)$ in the form
	\begin{equation}\label{desco1}
		\begin{aligned}
\mathcal P 		\phi_j(y,t) &=  \mathcal P \hat \phi_j(y,t) + \sum_{l=0,3}  \beta_{jl}(t)  \mathcal P \mathcal Z_{1l}(y,t), \quad {\mbox {where}} \\
		{\mathcal Z}_{1\ell} &= -{1\over {\mathcal P} }\, \div ({\mathcal P} \nabla \mathcal Z_{2\ell} ) ,\quad \ell =0,3,
		\end{aligned} 
	\end{equation}
	for some smooth parameter functions $\beta_{j \ell }$ that we will determine later on. 
	A direct inspection gives
	\begin{align*}
		{\mathcal Z}_{10}& = U(y) -{3\ve_j \over (r_j + \ve_j y_1 \chi )} \,  ( \chi + y_1 \chi' ) \, \pp_{y_1} \mathcal Z_{20} (y)  ,\quad \\
		{\mathcal Z}_{13}  &=  2 U(y) + \nn_y  U(y) \cdot y  -{3\ve_j \over (r_j + \ve_j y_1 \chi )} \, ( \chi + y_1 \chi' ) \, \pp_{y_1} \mathcal Z_{23} (y),
	\end{align*}
where $\chi$ is defined in \eqref{defP}.
	On $\hat \phi_j$
	we impose the following orthogonality conditions
	\begin{equation}
		\int_{\R^2} \mathcal P \hat \phi_j {\bf Z} _\ell \, dy \, = \, 0, \quad \ell=0,1,2,3 \foral t\in [0,T] ,
		\label{orto1}\end{equation}
	with
	\begin{equation}\label{bernieb}
		\begin{aligned}
			{\bf Z}_0(y) &= 1, \quad  p^2  \, {\bf Z}_3(y,t) = G\left( \Gamma_{0j} \right)   + b_3(y,t)  \\
	 \,  \,  p^{2}{\bf Z}_1(y,t) &= y_1 (1+ b_1 (y_1,t))  \, {\bf 1} _{B_{8R}}(y)  , \quad
		 \,	p^2 \, {\bf Z}_2(y,t) = \left( y_2 + \ve y_1 y_2 \right) {\bf 1} _{B_{8R}}(y)  , \quad
		\end{aligned}
	\end{equation}
	where  $p(y_1, t)$ is defined in \eqref{defP},  ${\bf 1} _{B_{8R}}(y) =1$ for $y\in B(0,8 R)$, $=0$ otherwise. The function $\Gamma_{0j}$ is given in \eqref{mm} and $G$ is such that
	$$
	G(\Gamma_0 (y) ) = \nabla \Gamma_0 \cdot y + 2 = {1-|y|^2 \over 1+ |y|^2}.
	$$
	Since $G' (\Gamma_0 (y)) \sim |y|^{-2}$ as $|y| \to \infty$, one has
	$$
	G(\Gamma_{0j} ) = G(\Gamma_0 ) (1+ O(\ve R) ).
	$$
	Besides, the functions 
 $b_3 , b_1 $ will be explicitly defined in \eqref{defb3}; for the moment we think at it as $b_3= O(\ve^2)$ and $b_1 = o(1)$, as $\ve \to 0$. 

 \medskip{}{}
  For $\phi_j$ as in \eqref{desco1}, we write
\begin{equation}\label{nota1}
\hat \phi = \left( \hat \phi_1 , \ldots , \hat \phi_k \right), \quad \beta = \left( \beta_1 , \ldots , \beta_k \right), \quad \beta_j = \left( \beta_{j0}, \beta_{j3} \right).
\end{equation}
	We assume that
	\begin{equation}\label{inte}
		\int_{\R^2} U^{-1} (y) ( \mathcal P \hat \phi_j )^2 (y,t) \, dy <\infty, \quad \forall t \in [0,T].
	\end{equation}
	Let $\hat \psi_j (y,t)$ be the solution of
	\begin{equation}\label{psi1h}
		-\div (\mathcal P \nabla \hat \psi_j ) = \mathcal P \hat \phi_j, \quad \inn \quad \R^2,
	\end{equation}
	predicted by Lemma \ref{ii},
	and define
	\begin{equation}\label{desco2}
		\psi_j(y,t) =  \hat \psi_j (y,t)  + \sum_{l=0,3}  \beta_{jl}(t)  \mathcal Z_{2l}(y) .
	\end{equation}
	Then $\psi_j$ 
	satisfies the second condition in \eqref{inner}, namely 
	$$
	-\Delta_{5,j} \psi_j = \phi_j  \quad (y,t) \in B(0, 3R  ) \times [0,T)
	.	$$
	Recall from \eqref{Ej} that
	\begin{equation}
		\label{defR}
		R= {1 \over \ve |\log \ve|^5 }.
	\end{equation}

	The decomposition for $\phi_j$ in \eqref{desco1} is motivated by the validity of a key estimate for a quadratic form.

	\begin{lemma}\label{passl}
		There exists a number $\gamma>0$ such that for any  sufficiently small $\ve$  and all $\phi$ satisfying conditions $\equ{orto1}$-$\equ{inte}$, the following holds:
		let $g$ be given by
		\[
		g = \left( f_0' (\Gamma_{0j} ) \right)^{-1} \left(   \mathcal P \phi  - \mathcal P f_0' (\Gamma_{0j} ) \, p^2 \,  \psi \right), 
		\]
		where $\psi$ solves $ - \div (\mathcal P \nabla \psi ) = \mathcal P \phi$. Then 
		we have
		\begin{equation}\label{pass}
			\int_{\R^2} \mathcal P \phi \, g \geq \frac{\gamma}{|\log R|}
			\int_{\R^2} (\mathcal P \phi)^2 U^{-1},
		\end{equation}
		where $R$ is given by \eqref{defR}. We recall the definition of $\mathcal P (y_1, t) $ and $p(y_1, t)$ in \eqref{defP}, $\Gamma_{0j}$ as in \eqref{mm} and that $f_0 (s) = e^s$.
	\end{lemma}

	\begin{proof}
	The proof is divided into two steps: we first prove \eqref{pass} with the test function $g$ replaced by $g_0$, with
	$$
	g_0 = U^{-1} \left( \mathcal P \phi - U \psi \right).
	$$
	where $U$ is defined in \eqref{ks}. Then we prove \eqref{pass}.
	
	\medskip{}{}
		Let
		$ \ttt \phi = 2\, U^{-1} \mathcal P \phi $. Using the stereographic projection $\pi$ introduced in  \eqref{stereo}, we get
		$$
		\int_{\R^2} \mathcal P \phi g_0 = 2\int_{\R^2} (\mathcal P \phi )^2 U^{-1} - \int_{\R^2} \mathcal P \phi \psi = \int_{S^2} \tilde \phi^2 - \int_{S^2} \tilde \phi \psi.
		$$ 
		Consider the orthonormal basis in $L^2 (S^2)$ of spherical harmonics $(e_\ell)_\ell $, where $ -\Delta_{S^2} e_\ell = \la_\ell e_\ell$.
		Here $ \la_0=0$ and $ e_0$ is constant, while  $ \la_1=\la_2=\la_3 = 2$, with $ e_\ell (z) =z_\ell$, for $\ell=1,2,3$. We decompose $\tilde \phi$ as
		$$
		\ttt \phi =   \sum_{\ell=0}^\infty  \ttt\phi_\ell e_\ell (z) = \sum_{\ell=0}^3 \tilde  \phi_\ell e_\ell (z)  +\ttt\phi^\perp, \quad \tilde \phi_\ell= \int_{S^2} \phi \, e_\ell \, d\pi (z).
		$$
		Since $\int_{\R^2} \mathcal P \phi = \int_{S^2} \tilde \phi =0$, we get $\tilde \phi_0=0$. For $\ell \geq 1$, let $\psi_\ell$ be defined by
		$$
		-{\div}_{S^2} (\mathcal P \nabla_{S^2} \psi_\ell )= e_\ell .
		$$
	From Lemma \ref{ii} we get that $\psi_\ell \in H^2 (S^2)$ and $\| \psi_\ell \|_{H^2 (S^2)} \leq C$ for some $C>0$.	Since $-{\div}_{S^2} (\mathcal P \nabla_{S^2} \psi )= \tilde \phi$, we have that $\psi = \sum_{\ell=1}^\infty \tilde \phi_\ell \psi_\ell.$
		The equation satisfied by $\psi_\ell$ becomes $-\div_{\R^2} (\mathcal P \nabla \psi_\ell ) = U e_\ell$ in $\R^2$.

		It is convenient to write
		$
		\psi_\ell =  (- \Delta_{S^2} )^{-1} e_\ell + \bar \psi_\ell = {2 \over \la_\ell } e_\ell + \bar \psi_\ell.
		$
		Then
		$$
		-{\div}_{S^2} ( \nabla_{S^2} \bar \psi_\ell )= h_\ell, \quad h_\ell = { \nabla_{S^2} \mathcal P \nabla_{S^2} \psi_\ell \over \mathcal P} + {(1-\mathcal P) e_\ell \over \mathcal P}.
		$$
		A direct computation gives $\int_{S^2} h_\ell =0$. Besides,
		\begin{align*}
			\int_{S^2} h_\ell^2 & \leq \int_{S^2 } |{\nabla_{S^2} \mathcal P \over \mathcal P} \nabla_{S^2} \psi_\ell |^2 + \int_{S^2 } |{(1-\mathcal P) \over \mathcal P} e_\ell|^2 \leq c  |\log \ve |^{-6}\left(  \| \psi_\ell \|_{H^2 (S^2)}^2 +1  \right) 
		\end{align*}
		for some $c>0$.
		We argue as in Lemma \ref{ii} to get  $\psi_\ell \in H^2 (S^2) $.
		Thus
		\begin{align} 	\label{bernie}  \int_{\R^2} \mathcal P\phi \, g_0
			&	=    \sum_{\ell=4}^\infty \left ( 1- \frac 2{\la_\ell} \right )\ttt\phi_\ell^2  + O(\ve R)  \| \tilde \phi \|^2_{L^2 (S^2)} \ \ge \ \\ & c_1\|\ttt\phi^\perp \| _{L^2(S^2)}^2  + O(\ve R) \| \tilde \phi \|^2_{L^2 (S^2)} \nonumber
		\end{align}
		for some uniform $c_1>0$. We also have, for $j=1,2$
		$$
		0 = \int_{B (0,5R)} \mathcal P \phi {\bf Z}_j  =  c_2   \ttt\phi_j  +     O( \|\ttt \phi^\perp\|_{L^2(S^2)}) \, |\log R|^{\frac 12 }
		$$
		for some uniform $c_2>0$, as $\ve \to 0$.  On the other hand, we have
		$$
		0 = \int_{\R^2} \mathcal P \phi {\bf Z}_3   =  \ttt\phi_3   +    O(\ve R)  \|\ttt\phi \| _{L^2(S^2)}.
		$$
		From the above relations we get that for some $c>0$ independent of $\ve$,
		\begin{align*}
			\|\ttt \phi^\perp\|_{L^2(S^2)}^2 \, & \ge  \,  \|\ttt \phi\|^2_{L^2(S^2)}  - c \sum_{\ell = 1}^3 |\ttt \phi_\ell|^2 ,  \\
			& \ge\,  (1- O(\ve R) ) \,\|\ttt \phi\|_{L^2(S^2)}^2 \, - \, c |\log R|\, \|\ttt \phi^\perp\|_{L^2(S^2)}^2.
		\end{align*}
		From here we get
		$$
		(1+ c |\log R|) \, \|\ttt \phi^\perp\|_{L^2(S^2)}^2 \, \geq  (1- o(1) ) \,\|\ttt \phi\|_{L^2(S^2)}^2
		$$
		with $o(1) \to 0 $ as $\ve \to 0$,
		which combines with \equ{bernie} to get
		$$
		\int_{\R^2} \mathcal P \phi \, g_0  \ge \frac{\gamma}{  |\log R|} \int_{S^2} \ttt \phi^2
		$$
		for some uniform $\gamma>0$.
		
		\medskip{}{}
		We next estimate $\int_{\R^2} \mathcal P g $ in terms of $\int_{\R^2} \mathcal P g_0 $. To this purpose, we write
		\begin{align*}
		    g&= g_0 + \left( \left( f_0' (\Gamma_{0j} ) \right)^{-1} - U^{-1} \right) \left( \mathcal P \phi - U  \psi \right) + \left( f_0' (\Gamma_{0j} ) \right)^{-1} \left( U - \mathcal P f_0' (\Gamma_{0j} ) p^2 \right) \psi 
		\end{align*}
		Recalling the definition of $\Gamma_{0j}$ in \eqref{mm}, we get
		$$
f_0' (\Gamma_{0j} ) = U(y) \left(1+ O(|\log \ve |^{-2} ) \right);
		$$
		hence we get
		$$
		\int_{\R^2} \mathcal P \phi \,  \, \left( \left( f_0' (\Gamma_{0j} ) \right)^{-1} - U^{-1} \right) \left( \mathcal P \phi - U  \psi \right) = O(|\log \ve |^{-2} ) 	\int_{\R^2} \mathcal P \phi \, g_0 .
		$$
		Besides
		\begin{align*} \left| \int_{\R^2} \mathcal P \phi \left( f_0' (\Gamma_{0j} ) \right)^{-1} \left( U - \mathcal Pf_0' (\Gamma_{0j} ) p^2 \right) \psi  \right|& \lesssim |\log \ve |^{-2} \int_{\R^2} |\mathcal P \phi \psi |
		\lesssim  |\log \ve |^{-2} \| U^{-{1\over 2}} \phi \|_{L^2 (\R^2)}^2  .
		\end{align*}
		Thus we get 
		\begin{align*}
		    	\int_{\R^2} \mathcal P \phi \, g  \ge (1+ O(|\log \ve |^{-2} ) ) 	\int_{\R^2} \mathcal P \phi \, g_0  - O(|\log \ve |^{-2} ) \int_{\R^2} U^{-1} (\mathcal \phi)^2 \ge {\tilde \gamma \over |\log R |} \int_{\R^2} U^{-1} (\mathcal \phi)^2 
		\end{align*}
		for some new $\tilde \gamma $ which is uniformly positive as $\ve \to 0$. This concludes the proof.
		\end{proof}

	We will make use of this result to establish a-priori bounds for solutions to a projected version of the inner problem \eqref{inner}. In Section \ref{sec9} we will first establish a-priori bounds in weighted $L^2$-spaces, which will be used to establish the weighted $L^\infty$ a-priori bounds. Before entering Section \ref{sec9},  we give a sketch of the proof to solve in $(\phi^{in}, \psi^{in}, \phi^{out}, \psi^{out} )$ the whole inner-outer gluing system  \eqref{inner}-\eqref{out}-\eqref{boundary}.
	
	\subsection{Strategy for the rest of the proof} Let
	$$
	b_j (\beta_j , \psi^{out} , {\bf a} )=   (1+{\ve_j \over r_j} y_1 \chi )^2 \,  \eta_{4\ve} \, ( \hat \psi_j + \sum_{l=0,3}  \beta_{jl}(t) \mathcal Z_{1l} +   r_j \psi^{out} ) 
	$$
	where $\eta_{4\ve} $ is defined in \eqref{ccuts} and 
	 introduce the following operators, depending on a homotopy parameter $\la \in [0,1]$
	\begin{align}\label{Ejla}
		E_{j, \la}  & [\hat \phi_j , \beta_j \, \psi^{out} , {\bf a} ] (y,t) := |\log \ve |  \, \ve_j^2 \,  (1 +{\ve_j \over r_j } y_1 \chi ) \, \pp_t \hat \phi_j     \\
		& + \nabla^\perp \left( \Gamma_{0j} + \la \eta_{4 \ve} b_j^{*} + \la \eta_{4 \ve} b_j^{**} + \la b_j (\beta_j, \psi^{out}, {\bf a} ) \right)\cdot \nabla  \hat \phi_j \nonumber \\
		&+ \nabla^\perp ( (1 +{\ve_j \over r_j } y_1 \chi )^2 \hat  \psi_j) \cdot \nabla \left( f_0\left( 	\Gamma_0 + \la b_j^{*}  \right) + \eta_{4\ve}   U^* \right)\nonumber \\
		&+	 \nabla^\perp\left[\ve_j \left(  |\log \ve|  \pp_t {\bf a }_j +   D_x \nn_x\vp_j ( {\bf P}_j ; {\bf P} ) [{\bf a} ] + \la \eta_{4\ve} (1+{\ve_j y_1 \over r_1} )^2 r_j \psi^{out}  \right)\cdot y   \right] \nabla U \nonumber  \\
		& +\la 	\tilde {\mathcal E}_j  (\beta_j, \psi^{out}, {\bf a} ) + |\log \ve | \ve_j^2 (1 +{\ve_j \over r_j } y_1 \chi ) \sum_{\ell = 0,3} \pp_t (\beta_{j\ell}  \mathcal Z_{1\ell} )\quad \inn \R^2 \times [0,T]\nonumber 
	\end{align}
where $f_0 (s) = e^s$,
	\begin{align*}
		\tilde {\mathcal E}_j &=  |\log \ve | B_0 (\sum_{\ell=0,3} \beta_{j\ell} \mathcal Z_{1 \ell} )  + \nabla^\perp \left( \Gamma_{0j} + \eta_{4\ve}  b_j^* + \eta_{4 \ve} b_j^{**} + b_j   \right)\cdot \nabla  (\sum_{\ell=0,3} \beta_{j\ell} \mathcal Z_{1 \ell} )\\
		&+ \nabla^\perp \left((1+{\ve_j y_1 \over r_1} \chi )^2 (r_j \psi^{out} ) \right) \nabla  (   \eta_{4\ve} U^* ) \nonumber \\
			&+	 \nabla^\perp\left[\left( \ve_j  |\log \ve|  \dot {\bf a }_j + \ve_j  \nn_x\vp_j ( P_j ; P ) - \nn_x\vp_j ( {\bf P
		}; {\bf P} )-  D_x \nn_x\vp_j ( {\bf P}_j ; {\bf P}) [{\bf a} ] \right) \cdot y   \right] \nabla U \\
		&	+ {\mathcal E}_j  [\phi_j^* , \psi_j^* \, \psi^{*,out} , {\bf a} ] (y,t)
	\end{align*}
	where $ {\mathcal E}_j  [\phi_j^* , \psi_j^* \, \psi^{*,out} , P]$ is defined 
	and estimated in Proposition \ref{Approximation}.
	
	If  $\la =1$ and we restrict the problem to $B(0,3R)$, we get the operator $E_j^{in}$ defined in \eqref{Ej}, or equivalently \eqref{Ejn}.

	\medskip Recalling the notation introduced in \eqref{nota1}, we define
	\begin{align*}
		E_{1,\la}^{out}( \phi^{out}, \psi^{out} , \hat \phi ,  \beta,{\bf a} )
		& \, = \,  |\log \ve | \, r \,  \pp_t \phi^{out}   \, \\
		& + \, \nn^\perp  \big (r^2 (\Psi^* -r_0^{-1} |\log \ve | +\la \sum_{j=1}^k { \bar \eta_{j2}\over r_j} \psi_j ({x-P_j \over \ve_j }) + \psi^{out} ) )  \cdot \nn\phi^{out}  \\ &\
		+  \ttt \EE_{1}^{out}( \hat \phi ,\beta,     \psi^{out},{\bf a} )
		\quad (x,t) \in \Sigma \times [0,T) 
	\end{align*}
	where
	\be\label{EE1}
	\begin{aligned}
		& \ttt \EE_{\la }^{out}( \hat \phi ,\beta,     \psi^{out},{\bf a} ) \, :=\, \\
		& \sum_{j=1}^k \left[ r \, |\log \ve |  \,  \pp_t \bar \eta_{j1} + \nabla_x^\perp ( r^2 (\Psi^*  + \sum_{j=1}^k  { \bar \eta_{j2}\over r_j} \psi_j ({x-P_j \over \ve_j }) + \psi^{out}-r_0^{-1} |\log \ve |)) \nabla \bar \eta_{1j} \right] {\phi_j \over \ve_j^2 r_j}   \\
		&+ \left[ \sum_{j=1}^k (\bar \eta_{2j} - \bar \eta_{1j}) \nabla_x^\perp (r^2 ({\psi_j \over r_j} +\psi^{out}) ) + {r^2 \psi_j \over r_j} \nabla_x^\perp \bar \eta_{2j} \right] \nabla_x W^*   \\
		&+ (1-\sum_{j=1}^k \bar \eta_{2j} ) \nabla^\perp (r^2 \psi^{out} ) \cdot \nabla W^*   
		+ (1-\sum_{j=1}^k \bar \eta_{j1} ) S_1 (W^*,\Psi^*)=0 \quad (x,t) \in \Sigma \times [0,T), 
	\end{aligned}\ee
	We also define
	\begin{align*}
		&  E_{1,\la}^{out}(  \psi^{out}, \phi^{out}, \hat \phi, \beta , {\bf a} )\ = \
		\Delta_5 \psi^{out} + \la \phi^{out} + \la \sum_{j=1}^k (\bar \eta_{j1} -\bar \eta_{j2} ) {\phi_j \over r_j \ve_j^2} \\
		&+\la \sum_{j=1}^k ( {\psi_j \over r_j} \Delta_5 \bar \eta_{j2}  + 2 \nabla_x \bar \eta_{j2} \nabla_x {\psi_j \over r_j}  ), \quad (x,t) \in \Sigma \times [0,T)
	\end{align*}
	
	The key observation is that for
	$\phi_j$, $\psi_j$ given by \equ{desco1}-\equ{desco2} we have the identities, when $\la =1$
	\begin{align*}
		{E }_{j,1} ( \hat \phi, \beta , \phi^{out}, \psi^{out}, {\bf a} )   \, = &\, E_j^{in}( \phi_j ,  \psi_j  , \psi^{out} ; \beta, {\bf a} )\inn B (0,3R)\times [0,T] , \\
		E_{1}^{out}( \phi^{out}, \psi^{out}, \hat \phi, \beta , {\bf a} ) \, = &\, E^{out}(  \phi^{out}, \psi^{out}, \phi^{in} , \psi^{in}  ;\beta, {\bf a} ) \inn \Sigma \times [0,T) ,
		\\
		E_{1,1}^{out}( \psi^{out}, \phi^{out},  \hat \phi, \beta , {\bf a} ) \, = &\, E_1^{out}(  \psi^{out}, \psi^{in}  ,  \phi^{out} ; \beta, {\bf a}  )   \inn \Sigma \times [0,T) ,
	\end{align*}
	with boundary and decay conditions
	$$
	{\pp \over \pp r } \psi^{out} (x,t) = 0 , \quad (x,t) \in \pp \Sigma \times [0,T], \quad |\psi^{out} (x,t)| \to 0 , \quad \ass |x| \to \infty. 
	$$
	Here $E_j^{in}$, $E^{out}$ and $E_1^{out}$ are defined respectively in \eqref{Ej}, \eqref{Eoutn} and \eqref{Eout1}. In other words, solving the {\em inner-outer gluing} system \eqref{inner}-\eqref{out} coupled with the boundary and decay conditions \eqref{boundary} amounts to make the three quantities above equal to zero (keeping the boundary and decay conditions).
	We will do this by a continuation argument that involves finding uniform a-priori estimates for the corresponding equations along the deformation parameter $\la$  imposing in addition initial condition $0$ for all the parameter functions.
	
	\medskip
	We use the following strategy. We consider the functions $\beta_j, \psi^{out},{\bf a} $ as given
	and require that $\hat \phi_j$ satisfies an initial value problem of the form

	\be
		{E }_{j,\la}  ( \hat \phi_j, \beta_j , \phi^{out}, \psi^{out},{\bf a} )  =   \sum_{l=0}^3 c_{lj}(t) z_{1l}(y)
		\inn \R^2\times [0,T] , \quad
		\hat \phi_j(y,0 )  = 0\inn \R^2
	\label{equinner} \ee
	where
	\begin{equation}\label{defsmallz}
		\begin{aligned}
			z_{10}(y) = U_0(y),& \quad
			z_{11}(y)  = \pp_{y_1} U_0(y),\\
			z_{12}(y) = \pp_{y_2} U_0(y), &\quad
			z_{13}(y)  = 2U_0(y) + \nn_y U_0(y)\cdot y .
		\end{aligned}
	\end{equation}
	for some explicit functions $c_{lj}(t)$. We prove that   $c_{lj}(t)$ are linearly dependent on  $\hat \phi_j$ and  can be  computed after integrating the equation against ${\bf Z}_\ell$ in space variable. See \eqref{bernieb} for the definition of ${\bf Z}_\ell$.

	\medskip
	Solving $E_{\la,j} \equiv 0$
	is equivalent to 
	solving the initial value problems
	\be\label{equc} \begin{aligned}   c_{lj}[ \hat  \phi_j ,  \beta_j, \phi^{out},  \psi^{out},{\bf a} ,\la](t)& =0 \foral t\in [0,T], \, \ell =0,1,2,3\\
		{\bf a}_j(0)&=  \beta_j(0) =0  \end{aligned} \ee
	for all  $j=1, \ldots , k$.
	We require \be\label{equout1}  E_{\la}^{out}(\phi^{out},   \psi^{out}, \hat  \phi_j ,  \beta_j, {\bf a}) =0 \inn  \Sigma\times [0,T],\quad
		\phi^{out}(\cdot ,0) = 0 \inn \Sigma 
	 \ee
	and
	\be\label{equout2} \begin{aligned} E_{1,\la}^{out}(  \psi^{out},  \phi^{out},   \hat  \phi_j ,  \beta_j, {\bf a} ) &=0 \inn \ \ \Sigma \times [0,T],\\
		{\pp \over \pp r}\psi^{out} = 0 \onn \pp\Sigma\times [0,T], &\quad |\psi^{out} (x,t)| \to 0 , \quad \ass |x| \to \infty.
	\end{aligned} \ee
	We recall the form of $\phi_j$ and $\psi_j$, as in \equ{desco1}-\equ{desco2}
	\begin{align*}
		\phi_j(y,t) \, = & \, \hat \phi_j(y,t) + \sum_{l=0,3}  \beta_{jl}(t)  Z_{1l}(y,t),\\
		\psi_j(y,t)\, = & \, \hat \psi_j(y,t) + \sum_{l=0,3}  \beta_{jl}(t) \mathcal Z_{2l}(y).
	\end{align*}

	We can write the system of equations \equ{equinner}, \equ{equc}, \equ{equout1}, \equ{equout2} in the form of a fixed point problem for the variable
	$$
	\vec {p} = (\hat \phi, \beta,  \phi^{out},\psi^{out}, {\bf a}) .
	$$
	The fix point problem has the form
	\be\label{equp}
	\vec { p}  = {\mathcal F} ( \vec { p} , \la ), \quad \vec { p} \in \OO\, .   
	\ee
	Here $\OO$ designates a bounded open set in an appropriate Banach space with $\vec p= \vec 0\in \OO$ and
	${\mathcal F}( \cdot  , \la )$  is a homotopy of nonlinear compact operators on $\bar \OO$
	with ${\mathcal F} ( \cdot  , 0)$ linear.

	\medskip
	We shall prove that a suitable choice of a small $\OO$ yields that for all $\la\in [0,1]$
	no solution of \equ{equp} with  $\vec p\in \pp \OO $   exists.
	Existence  of a solution of \equ{equp} for $\la=1$ thus follows from standard degree theory. But this precisely corresponds to a solution of the original problem. The definition of the norm and the set $\OO$ will yield the desired properties of the solution of Euler equation thus obtained.

	\medskip
	In order to find the desired a priori estimates we need several preliminary considerations which we make in the next section.

	\section{Some a-priori estimates} \label{sec9}

	Let us consider functions  $\phi (y,t)$, $y \in \R^2$, $t\in [0,T]$ that satisfy
	$$
	\| \mathcal P \phi\, U^{-\frac{1}{2} } \|_{L^2(\R^2)}<\infty, \quad {\mbox {where}} \quad \mathcal P (y_1 , t) = p^3 (y_1, t) , \quad p(y_1, t) = (1 +{\ve_j \over r_j } y_1 \chi ).
	$$
	We recall that the function $U$ is defined in \eqref{ks},  $\mathcal P$, $p$ were introduced in \eqref{defP} and $\chi$ is in \eqref{ccuts}.
	We also assume  the orthogonality conditions on $\mathcal P \phi$
	\be
	\begin{aligned}
		\int_{\R^2} \mathcal P \phi(y,t){\bf Z}_\ell(y)\, dy \, = \, 0, \quad \ell=0,1,2,3, \quad \forall t \in [0,T]
		\quad
	\end{aligned}
	\label{orto40}\ee
	where the function ${\bf Z}_\ell$ , $\ell = 0,1,2,3$ are defined in \eqref{bernieb}.
	Let  $ \psi (y,t)$ be the solution, predicted by Lemma \ref{ii} of
	\begin{equation}\label{psi1}
		-\div (\mathcal P \nabla  \psi ) = \mathcal P  \phi, \quad \inn \quad \R^2.
	\end{equation}

	We let $f_0(v)= e^v $. This section is devoted to establish a series of a-priori estimates for solutions to 
	a linear transport equation of the form
	\be\label{innerL}
	\begin{aligned}
		|\log \ve | \,  \ve_j^2 \, & p  \,  \pp_t  \phi   
		+  \nabla_y^\perp \left( \Gamma_{0j} + a_* +a \right)\cdot  \nabla_y \left( \phi  - f_0'( \Gamma_{0j} +a_*)\,  p^2 \, \psi \right)
	 + \, E(y,t)\, = \, 0  \inn \R^2 \times (0,T),\\
		&\qquad\quad\phi(\cdot,0)\,  =\, 0 \inn \R^2
	\end{aligned}
	\ee
 These estimates will be used to treat \eqref{equinner}.

	On the functions
	$a_*(y,t)$, and $a(y,t)$  that appear in \eqref{innerL} we assume
	\begin{equation}\label{nu0}
		a_*(y,t), \, a (y,t) \, =0 \quad \hbox{for } |y|\ge 8R, \quad
		\Delta_y (a + a_*) \in L^\infty (\R^2\times (0,T))
	\end{equation}
	and
	for some numbers $C>0$, $\nu >0$,
	\begin{align}
	|\log \ve |^{1\over 2}	|\pp_t a_*(y,t)| + (1+ |y|) |  \,  \nn_y a_* (y,t)|  + |a_* (y,t)| \ \le & \  C \ve^2 \, (1+ |y|^2) \, \log (1+ |y|) \nonumber\\
		|\nn_y a(y,t)|   \ \le & \   \ve^{2+\nu}.\label{nu}\end{align}
	These assumptions are consistent with the description of Problem \eqref{equinner}, in the version contained in \eqref{Ejla}: we will take
	\begin{equation}\label{aastar}
	a_*= \la \eta_{4 \ve} \, b_j^*, \quad a= \la \left( \eta_{4\ve } b_j^{**} + b_j \right),
	\end{equation}
 where $b_j^*$, $b_j^{**}$ satisfy bounds similar to \eqref{similar} and $b_j$ is defined as in \eqref{e}.
 
 \medskip{}{}

	\subsection{An $L^2$-weighted a priori estimate}

	Our first result is an $L^2$-weighted a-priori estimate on a solution to \eqref{innerL}.
	We have the following

	\begin{lemma}\label{lin}
		There exists a constant $C>0$ such that for any
		$a$, $a_*$ satisfying $\equ{nu0}$-$\equ{nu}$,  $R$ given by \eqref{defR}, all sufficiently small $\ve$
		and any solution $\phi$ of $\equ{innerL}$-$\equ{orto40}$ with \be\label{aaa}\sup_{t\in [0,T]} \|U^{-\frac 12 }\, \mathcal P  \phi(\cdot , t)\|_{L^2(\R^2)} \ < \ +\infty\ee
		we have
		\be \label{pass1}
		\sup_{t\in [0,T]} \|U^{-\frac 12 } \, \mathcal P  \phi(\cdot , t)\|_{L^2(\R^2)} \ \le \ C\,\ve^{-2} \, |\log \ve |^{-{1\over 2}} \, \sup_{t \in [0,T]} \, \| E (\cdot,t)\,  U^{-1/2} \|_{L^2 (\R^2) }.
		\ee

	\end{lemma}
	
	\begin{proof} Let us assume that
		$$ \sup_{t \in [0,T]} \, \| E (\cdot,t)\,  U^{-1/2} \|_{L^2 (\R^2) } \, <\, +\infty $$
		and define the functions
		\[
		U_1  = f_0'(\Gamma_{0j} + a_*)  , \quad  U_1  g_1 =
		\phi  - f_0'(\Gamma_{0j} + a_*) \, p^2 \,  \psi  ,
		\]
		where $f_0(s) = e^s$
  and
  $$
  P (y_1 , t) = p^3 (y_1, t) , \quad p(y_1, t) = (1 +{\ve_j \over r_j } y_1 \chi ).$$
		We multiply 
		equation \eqref{innerL} against $g_1$ and integrate in $\R^2$. One term is
		\begin{align*}
		    \int_{\R^2} p \, \pp_t  \phi \, g_1 \, dy &= \int_{\R^2} p \pp_t  \phi  {\phi \over U_1} \, dy
		    -\int_{\R^2}  \pp_t  \phi \, \mathcal P \psi  dy\\
		    &= {1 \over 2} \pp_t \int_{\R^2} p U_1^{-1}  \phi^2   \, dy -{1\over 2} \int_{\R^2} \pp_t \left( p U_1^{-1} \right) \phi^2 \\
		    &-\int_{\R^2}  \pp_t \left( \mathcal P \phi \,  \right)  \psi  dy + \int_{\R^2}   \phi \,  \pp_t \mathcal P \psi  dy.
		\end{align*}
		Since $-\div (\mathcal P \nabla \psi ) = \mathcal P \phi$,
		\begin{align*}
		    -\int_{\R^2}  \pp_t \left( \mathcal P \phi \,  \right)  \psi  dy &= \int_{\R^2}  \pp_t \left( \div (\mathcal P \nabla \psi) \, \right)  \psi  dy =
		    \int_{\R^2}   \div (\pp_t \mathcal P \nabla \psi) \,  \psi  dy + \int_{\R^2}  \div (\mathcal P \nabla \pp_t \psi) \,   \psi  dy 
		\end{align*}
		On the other hand, using the symmetry of the form $ \int_{\R^2} \div (\mathcal P \nabla \psi_1 ) \psi_2  dy $ we have that
		\begin{align*}
		    \pp_t \int_{\R^2} \div (\mathcal P \nabla \psi ) \psi \, dy &= \int_{\R^2} \div (\pp_t \mathcal P \nabla \psi ) \psi \, dy + \int_{\R^2} \div (\mathcal P \nabla  \pp_t\psi ) \psi \, dy + \int_{\R^2} \div (\mathcal P \nabla \psi ) \pp_t \psi \, dy\\
		    &= \int_{\R^2} \div (\pp_t \mathcal P \nabla \psi ) \psi \, dy + 2\int_{\R^2} \div (\mathcal P \nabla  \pp_t\psi ) \psi \, dy
		\end{align*}
		from which we get that
		\begin{align*}
		    -\int_{\R^2}  \pp_t \left( \mathcal P \phi  \right)  \psi  dy &={1\over 2} \pp_t \int_{\R^2} \div (\mathcal P \nabla \psi ) \psi \, dy +{1\over 2} \int_{\R^2} \div (\pp_t \mathcal P \nabla \psi ) \psi \, dy.
		\end{align*}
		Thus we conclude that
			\begin{align*}
		    \int_{\R^2} p \, \pp_t  \phi \, g_1 \, dy &= {1 \over 2} \pp_t  \int_{\R^2} p  \phi g_1 \, dy   
		     -{1\over 2} \int_{\R^2} \pp_t \left( p\, U_1^{-1} \right) \phi^2 \\
		    & + \int_{\R^2}   \phi \,  \pp_t \mathcal P \psi  dy  +{1\over 2} \int_{\R^2} \div (\pp_t \mathcal P \nabla \psi ) \psi \, dy.
		\end{align*}
		Moreover
		\begin{align*}
			\int_{\R^2} &   \nabla_y^\perp( \Gamma_{0j} + a_* +a)\cdot  \nabla_y ( U_1 g_1  ) g_1 = 	\int_{\R^2}  U_1^{-1}  \nabla_y^\perp( \Gamma_{0j} + a_* +a)\cdot  \nabla_y ( { U_1^2 g_1^2 \over 2}   )\, dy \\
			&=-\int_{\R^2} \nabla \cdot  \left(  U_1^{-1}  \nabla_y^\perp( \Gamma_{0j} + a_* +a) \right) { U_1^2 g_1^2 \over 2} \, dy 
			= \int_{\R^2}  \nabla_y^\perp( \Gamma_{0j} + a_* +a) {\nabla U_1 \over U_1}  { U_1 g_1^2 \over 2} \, dy,
		\end{align*}
		and we conclude that
		\begin{equation}\label{ae}
			\begin{aligned}
				&{|\log \ve | \ve_j^2\over 2} \pp_t  \int_{\R^2} (1 +{\ve_j \over r_j } y_1 \chi )  \phi g_1 \, dy   
		     -{|\log \ve | \ve_j^2\over 2} \int_{\R^2} \pp_t \left( (1 +{\ve_j \over r_j } y_1 \chi ) U_1^{-1} \right) \phi^2 \\
		    & + |\log \ve | \ve_j^2 \int_{\R^2}   \phi \,  \pp_t \mathcal P \psi  dy
		    +{|\log \ve | \ve_j^2 \over 2} \int_{\R^2} \div (\pp_t \mathcal P \nabla \psi ) \psi \, dy\\
			 &+  \int_{\R^2}  \nabla_y^\perp( \Gamma_{0j} + a_* +a) {\nabla U_1 \over U_1}  { U_1 g_1^2 \over 2} \, dy + \int_{\R^2} E  g_1 \, dy =0.
			\end{aligned}
		\end{equation}
		Next, we estimate the last four terms in the above expression.
		Using estimates on $U_1$, the bounds on $\psi$ as in Lemma \ref{ii} and the explicit definition of $\mathcal P$ in \eqref{defP}, we get
		\begin{align*}
			&\left| \int_{\R^2}  \phi^2  \pp_t (  U_1^{-1} p ) \, dy \right| \lesssim |\log \ve |^{-{9\over 2} } \int_{\R^2} \phi^2 U^{-1}  \\
			&	\left|	\int_{\R^2} \phi \psi \pp_t \mathcal P  \, dy \right|\lesssim |\log \ve |^{-{9\over 2}} (\int_{\R^2} \phi^2 U^{-1})^{1\over 2} 
			(\int_{\R^2} \psi^2 U )^{1\over 2} \lesssim |\log \ve |^{-{9\over 2}} \int_{\R^2} \phi^2 U^{-1}\\
			&\left| \int_{\R^2} \div (\pp_t \mathcal P \nabla \psi ) \psi \, dy \right| \lesssim |\log \ve |^{-{9\over 2}} \int_{\R^2} \phi^2 U^{-1},
		\end{align*}
		To estimate $ \int_{\R^2}  \nabla_y^\perp( \Gamma_{0j} + a_* +a) {\nabla U_1 \over U_1}  { U_1 g_1^2 \over 2} \, dy$, we observe that
		\begin{align*}
			\int_{\R^2} & \nabla_y^\perp( \Gamma_{0j} + a_* +a) {\nabla U_1 \over U_1}  { U_1 g_1^2 \over 2} \, dy =  \int_{\R^2}  \nabla_y^\perp a {\nabla U_1 \over U_1}  { U_1 g_1^2 \over 2} \, dy .
		\end{align*}
		From \eqref{nu} we get
		$$
		\left| \int_{\R^2}  \nabla_y^\perp a {\nabla U_1 \over U_1}  { U_1 g_1^2 \over 2} \, dy \right| \lesssim \ve^{2+ \nu}  \int_{\R^2}    { U_1 g_1^2 \over 2} \, dy.
		$$
		 Finally 
		$$
		\left| \int_{\R^2} E g_1 \, dy \right| \leq \| E U^{-{1\over 2}} \|_{L^2(\R^2)} \, \| U^{1\over 2} g_1 \|_{L^2(\R^2)}.
		$$
		Since
		$
		\mathcal P (y_1, t) = 1+ O(\ve R)
		$
		uniformly as $\ve \to 0$ for $y_1 \in \R$, $t \in [0,T]$,
		from Lemma \ref{passl} and \eqref{pass} we  obtain
		$$
		\int_{\R^2} U_1 g_1^2 \leq C \int_{\R^2} \phi^2 U^{-1} \leq C |\log \ve | \int_{\R^2}  p\,   \phi g_1 \, dy,
		$$
		for some $C>0$. Define
		$$
		f (t) = \left( \int_{\R^2}  p\,   \phi g_1 \, dy \right)^{1\over 2} .
		$$
		We conclude that
		\begin{align*}
			&\ve_j^2 |\log \ve | \left| \int_{\R^2}  \phi^2  \pp_t (  U_1^{-1}  \, p ) \, dy \right| \lesssim \ve_j^2 |\log \ve | |\log \ve |^{-{1\over 2} }  f^2 (t),\\
			& \ve_j^2 |\log \ve |	\left|	\int_{\R^2} \phi \psi \pp_t \mathcal P  \, dy \right| \lesssim \ve_j^2 |\log \ve ||\log \ve |^{-{3\over 2}} f^2 (t) ,\\
			&\ve_j^2 |\log \ve | \left| \int_{\R^2} \div (\pp_t \mathcal P \nabla \psi ) \psi \, dy \right| \lesssim  \ve_j^2 |\log \ve ||\log \ve |^{-{3\over 2}} f^2 (t) ,\\
			&	\left| \int_{\R^2} E g_1 \, dy \right| \lesssim \| E U^{-{1\over 2}} \|_{L^2(\R^2)} f(t).
		\end{align*}
		Inserting these estimates in \eqref{ae}, we get
		$$
		\ve_j^2 |\log \ve |{d \over dt}  f^2 (t) \leq \ve_j^2 |\log \ve | A_\ve f^2 (t) + C \| E U^{-{1\over 2}} \|_{L^2(\R^2)} f(t),
		$$
		with $A_\ve \to 0$ as $\ve \to 0$. Then we find
		\[
		\varepsilon_j^2 |\log \ve |\frac{d f}{d t } (t)
		\, \leq\,  \ve_j^2 |\log \ve | A_\ve f (t) +C  \| E (\cdot,t)\,  U^{-1/2} \|_{L^2}.
		\]
		Performing the change of variable $		\varepsilon_j^2 |\log \ve | {d \over dt} = {d \over d\tau }$, Gronwall's inequality  yields
		\[
		f(t) \,\leq \,  C\,\ve^{-2} |\log \ve |^{-1}  \,
		\sup_{0\le t \leq T}  \| E (\cdot,t)\,  U^{-1/2} \|_{L^2}.
		\]
		This inequality and  Lemma \ref{passl} yield \eqref{pass1}.
	\end{proof}

	\subsection{An $L^\infty$-weighted a priori estimate}
	
	\medskip
	We consider now a class of functions $E(y,t)$ such that  for a number $0<\beta<1$ we have
	\be\label{Enu}
	\|E\|_{3+\beta}  :=   \sup_{(y,t)\in \R^2\times [0,T]} | (1+|y|)^{3+\beta} E(y,t) \| < + \infty .
	\ee
	We observe that there exists $C_* >0$ such that
	$$
	\sup_{t \in [0,T]} \, \| E (\cdot,t)\,  U^{-1/2} \|_{L^2 (\R^2) } \ \le \  C_* \,\|E\|_{3+\beta} .
	$$
	Hence Lemma \ref{lin} is applicable.

	\begin{lemma}\label{lin1} Under the assumptions of Lemma \ref{lin}, there exists a small $\sigma >0$ such that, for some $0<\alpha<1$ and  all small $\ve$, we have that 
		
		\be\label{pass2}
		\begin{aligned}
			& |\psi(y,t)| \, +\, (1+|y|) |\nn_y \psi (y,t)|  + (1+|y|)^{1+\alpha}  [\nn \psi(\cdot,t) ]_\alpha (y) \\ \, & \le \, \frac {\ve^{-2-\sigma} }{1+|y|}   \, \|E\|_{3+\beta} \foral (y,t) \in \R^2\times[0,T].
		\end{aligned}
		\ee
		where $\psi(y,t)$ is given by $\equ{psi1}$ where $\phi(y,t)$ is a solution of $\equ{innerL}$-$\equ{orto40}$ satisfying  $\equ{aaa}$.
	\end{lemma}

	\begin{proof}
		
		From \eqref{innerL} we get that $\phi(y,t)$ satisfies the transport equation
		\be\label{innerL1}
		\begin{aligned}
			\varepsilon_j^2 |\log \ve | \, p \, \partial_t \phi + & \nabla_y^\perp( \Gamma_{0j} + a_* +a)\cdot  \nabla_y  \phi + \, \ttt E(y,t)\, = \, 0  \inn \R^2 \times (0,T),\\
			&\quad\quad\ \phi(\cdot,0)\,  =\, 0 \inn \R^2
		\end{aligned}
		\ee
		where
		$$
		\ttt E(y,t) \,= \,  E(y,t)
		- \nabla_y^\perp \left( \Gamma_{0j} + a_* +a \right)\cdot  \nabla_y \left( f_0'( \Gamma_{0j} +a_*) p^2 \psi \right).
		$$
  The result of Lemma \ref{transport1-new} is still valid for a transport equation of the form \eqref{innerL1}.
		Let us fix a number $p$ with $2<p< \frac 2{1-\beta}$. Then we have that
		$$
		\sup_{t\in [0,T]}  \| U^{\frac 1p -1} E (\cdot, t)\|_{L^p(\R^2)} \ \le \ C \,\|E\|_{3+\beta}.
		$$
		Since $\phi$ solves equation \eqref{innerL1} and Lemma \ref{transport1-new}  apply to yield
		$$
			\| U^{{1\over p}-1} \phi \|_{L^p (\R^2) } \leq C \ve^{-2} |\log \ve |^{-1} \| U^{{1\over p}-1} \ttt E \|_{L^p (\R^2) }.
		$$
		Let us estimate this weighted $L^p$ norm for the second term in $\ttt E$. We have
		$$
		\big |\nabla_y^\perp( \Gamma_{0j} + a_* +a)\cdot  \nabla_y ( f_0'( \Gamma_{0j} +a_*) p^2\psi )\big |\ \le \  C\, \big [ \frac 1{1+ |y|^5} |\nn \psi| + \frac 1{1+ |y|^6}|\psi|\, \big ].
		$$
		Since
		$$
		\left \|\, U^{{1\over p}-1}  \frac {|\nn \psi|}{1+ |y|^5} \, \right \|_{L^p (\R^2) }\, \leq\, C \left \| U^{-{1\over 2} + {1\over p} } |\nabla \psi | \, \right \|_{L^p (\R^2) } \leq\,  C \,
		\left \| U^{-{1\over 2}} \phi  \right \|_{L^2 (\R^2) },
		$$
		where the last inequality follows from  \eqref{gradp}.
		From \eqref{cota1psi} we get
		$$
		\left \| U^{{1\over p}-1}  \frac {| \psi|}{1+ |y|^6} \, \right \|_{L^p (\R^2) } \,\leq\, C \left
		\| U^{-{1\over 2}} \phi  \right \|_{L^2 (\R^2) }.
		$$
		Combining the above estimates, we conclude
		$$
		\sup_{t\in [0,T]}  \| U^{\frac 1p -1} \ttt E (\cdot, t)\|_{L^p(\R^2)} \ \le C \left( \| E \|_{3+\beta} +
		\sup_{t\in [0,T]} \| U^{-{1\over 2}} \phi | \|_{L^2 (\R^2) } \right).
		$$
		We now apply estimate \eqref{pass1} and we conclude that
		$$
			\sup_{t\in [0,T]} \| U^{{1\over p}-1} \phi(\cdot, t) \|_{L^p (\R^2 )} \leq C \ve^{-4} \,  |\log \ve |^{- {3\over 2}} \,\sup_{t\in [0,T]}  \| U^{-{1\over 2}} E(\cdot, t)  \|_{L^2 (\R^2) } .
$$
		Next we use an interpolation argument to control $	\| U^{{1\over q}-1} \phi \|_{L^q (\R^2)}$ for some
		$q \in (2,p)$. Write $q$ in the form
		in the form   $$q= 2 (1-\la ) + \la p , \quad \la \in (0,1).$$
		Using H\"older's inequality we check that
		\begin{equation*}\begin{split}
				\| U^{{1\over q}-1} \phi \|_{L^q (\R^2)}^q &= \int_{\R^2} U^{1-q} |\phi |^q \, dy 
				\leq \left(\int_{\R^2}  U^{-1} |\phi|^2 \, dy \right)^{(1-\la)} \, \left( \int_{\R^2}  U^{1-p} |\phi|^p \, dy \right)^\la \\
				&= \| U^{-{1\over 2}} \phi \|_{L^2 (\R^2) }^{2(1-\la)} \, \| U^{{1\over p}-1} \phi \|_{L^p (\R^2)}^{p \la},
			\end{split}
		\end{equation*}
		We thus get that
		\begin{align*}
			\| U^{{1\over q}-1} \phi \|_{L^q (\R^2)} &\leq C \,  \ve^{-2 - 2 {p \lambda \over q}} \, |\log \ve |^{-{1-\lambda \over q} - {3p\lambda \over 2q} } \, \| E \|_{3+\beta}.
		\end{align*}
		Inserting this information in the above estimate we get
		\begin{align*}
			\| U^{{1\over q}-1} \phi \|_{L^q (\R^2)} 
			&\leq C \, \ve^{-2-\sigma} \, \| E \|_{3+\beta},
		\end{align*}
		for some $\sigma >0$.
		Estimate \eqref{pass2} follows from Lemma \ref{ii}. The number $\sigma >0$ can be taken arbitrarily small (asking $\lambda$ to be close to $0$), in particular satisfying $\sigma < \beta$.
		The proof is concluded.
	\end{proof}
	
	As a consequence of the above result we can also get an $L^\infty$-weighted estimate for $\phi$.
	\begin{corollary} \label{co71} Under the assumptions of Lemma \ref{lin}, we also have the estimate
		\be \label{pass3}
		|\phi(y,t)| \ \le \  C\, \Big[  \frac {\ve^{-2-\sigma}}  {1+|y|^{3+\beta}}  +  \frac {\ve^{-4- \sigma}}  {1+|y|^{7}}   \Big ]\,  \|E\|_{3+\beta}.
		\ee
	\end{corollary}
	\begin{proof}
		Recall that $\phi(y,t)$ solves \equ{innerL1}. From Lemma \ref{lin1}, we get
		$$
		|\ttt E (y,t) | \le C \Big[  \frac   1{1+|y|^{3+\beta}}  +  \frac {\ve^{-2-\sigma}}  {1+|y|^{7}}   \Big ]\,  \|E\|_{3+\beta}.
		$$
		Estimate \eqref{pass3} then follows as a direct application of Lemma \ref{transport1-new} for $p=+\infty$.
	\end{proof}

	\medskip
	
	We next revisit the estimates for $\psi$ given in Lemma \eqref{ii} and in Lemma \ref{lin1}, in view of Corollary \ref{co71}. 
	This will be useful in the sequel.
	
	\begin{lemma}\label{iii}
	Assume the validity of the assumptions of Lemma \ref{lin}, and let $\psi(y,t)$ be given by $\equ{psi1}$ where $\phi(y,t)$ is a solution of $\equ{innerL}$-$\equ{orto40}$ satisfying  $\equ{aaa}$. Then 
		$$
		\psi = \psi_1 + \psi_2 
		$$
		where, for $y=r \, e^{i\theta}$,
		\begin{equation}\label{es111}
		\begin{aligned}
		\psi_1 (y,t) &= A_1 (r, t) \cos \theta + A_2 (r,t) \sin \theta, \\
	(1+ |y|) |\pp_r A_j (r,t) |&+ 	|A_j (r,t)| \lesssim {1 \over 1+ |y|^{1 +\sigma}} \| U^{{1\over q}-1} \phi \|_{L^q (\R^2)} + \, 
			{R^{{2\over q} -1} \over 1+ |y|} \,  \| U^{{1\over q}-1} \phi \|_{L^q (\R^2)}, \quad j=1,2
		\end{aligned}
		\end{equation}
		and
		\begin{equation}\label{es1}
		\begin{aligned}
				(1+ |y|) |\nabla_y \psi_2 (y,t)|+ |\psi_2 (y,t) |& \lesssim  \Big[  \frac {\ve^{-2-\sigma}}  {1+|y|^{1+\beta}}  +  \frac {\ve^{-4- \sigma}}  {1+|y|^5  } \Big ]\,  \|E\|_{3+\beta},
			\end{aligned}
		\end{equation}
		for some $\sigma >0$. 
	\end{lemma}
	
	\begin{proof} 
		Under the orthogonality conditions \eqref{orto40} on $\mathcal P \phi$, the result in Lemma \ref{ii} gives the existence and uniqueness of a solution $\psi$ to \eqref{psi1} such that $|\psi (y,t) | \to 0$ as $|y| \to \infty$, for all $t \in [0,T]$.
	
	\medskip
	Decompose $h:= \mathcal P \phi$ in Fourier series
	$$
	h(r,\theta,t) = H_1 + H_2 , \quad  , \quad H_1= \sum_{n=\pm 1 } h_n (r,t) e^{in\theta} , \quad H_2=  \sum_{n\not= 0, \pm1 } h_n (r,t) e^{in\theta} .
	$$
	We let $\psi_1$ to be the solution to  $-\Delta \psi_1 = H_1$ with decay to $0$ at infinity. It has the form  in \eqref{es111}. Next we show the validity of the bounds in \eqref{es111}.

Using the orthogonality conditions \eqref{orto1} for $\ell=0,1,2$, we write
		\begin{align*}
			\psi_1 (y,t)  &={1\over 2\pi} \int_{\R^2} \left(  \log {1\over |y-z|} - \log {1\over |y|} - {y \cdot z \over |y|^2 }\right) \, H_1(z,t) \, dz \\
			& - {1\over 2\pi} \int_{\R^2 \setminus B(0,5R)}  {y \cdot z \over |y|^2 }\,  H_1(z,t) \, dz.
		\end{align*}
	The second integral can be easily estimated
	\begin{align*}
		\left| \int_{\R^2 \setminus B(0,5R)}  {y \cdot z \over |y|^2 }\,H_1 (z,t)  \, dz \right| & \leq {C \over 1+ |y|} \| U^{{1\over q}-1} \phi \|_{L^q (\R^2)}  \left( \int_{\rho >5R} {\rho^{q\over q-1} \over 1+ \rho^{{4q - 4 \over q-1}} }\rho \, d\rho \right)^{{q-1\over q}}\\&\leq 
		 R^{{2\over q} -1}\, 
		 {C \over 1+ |y|} \,  \| U^{{1\over q}-1} \phi \|_{L^q (\R^2)}  
	\end{align*}
		To estimate the first integral, we split the region of integration in $|z|< {|y| \over 2} $ and its complement.
		For $|z|< {|y| \over 2} $, we Taylor expand the quantity inside the bracket and get
		\begin{align*}
			\Bigl| \int_{|z|< {|y| \over 2}} \left(  \log {1\over |y-z|} - \log {1\over |y|} - {y \cdot z \over |y|^2 }\right) \,& \mathcal P (z,t) \phi (z,t) \, dz \Bigl|\leq {C \over 1+ |y|^2} \int_{|z|< {|y| \over 2}}
			|z|^2 \phi \, dz \\
			&\leq  {C \over 1+ |y|^2} \| U^{{1\over q}-1} \phi \|_{L^q (\R^2)}  \left( \int_{|z|< {|y| \over 2}} {\rho^{2q\over q-1} \over 1+ \rho^{{4q - 4 \over q-1}} }\rho \, d\rho \right)^{{q-1\over q}}\\
			&\leq {C \over 1+ |y|^{2-{2\over q}}} \| U^{{1\over q}-1} \phi \|_{L^q (\R^2)}.
		\end{align*}
		In the complementary region, we get
		\begin{align*}
			\Bigl| \int_{|z|> {|y| \over 2}} \left(  \log {1\over |y-z|} - \log {1\over |y|} - {y \cdot z \over |y|^2 }\right) \,& \mathcal P (z,t) \phi (z,t) \, dz \Bigl|\leq
				\Bigl| \int_{|z|> {|y| \over 2}} \left(  \log {1\over |y-z|}  - {y \cdot z \over |y|^2 }\right) \, \mathcal P (z,t) \phi (z,t) \, dz \Bigl|\\
			&\leq {C \over 1+ |y|^{2-{2\over q}}} \, \log (1+ |y| )  \| U^{{1\over q}-1} \phi \|_{L^q (\R^2)}.
		\end{align*}
	Thus we conclude that
	$$
	|\psi_1 (y,t)| \lesssim {1 \over 1+ |y|^{2-{2\over q} -\sigma'}} \| U^{{1\over q}-1} \phi \|_{L^q (\R^2)} + \, 
	{R^{{2\over q} -1} \over 1+ |y|} \,  \| U^{{1\over q}-1} \phi \|_{L^q (\R^2)},
	$$
		for any $\sigma' >0$. Proceeding in a similar way, we get the estimate for $\nabla \psi_1$ and the validity of \eqref{es111}.

	Consider the function
	$$
	H= H_2 + (\mathcal P -1 ) H_1 + \nabla \mathcal P \nabla \psi_1,
	$$
	and let $\psi_2$ solve $-\div (\mathcal P \nabla \psi_2 ) = H.$
		A direct computation gives
		\begin{align*}
			\int_{\R^2} & \left[ (\mathcal P -1) H_1 + \nabla \mathcal P \cdot \nabla \psi_1 \right] dy =- \int_{\R^2} H_1 +\int_{\R^2}  \mathcal P H_1 - \int_{\R^2}  \mathcal P \Delta \psi_1 =0 .
		\end{align*}
	A direct inspection gives that $H$ has no mode $1$ in its Fourier decomposition and
	$$
	 |H| \lesssim \Big[  \frac {\ve^{-2-\sigma}}  {1+|y|^{3+\beta}}  +  \frac {\ve^{-4- \sigma}}  {1+|y|^{7}}   \Big ]\,  \|E\|_{3+\beta}.
	$$
	Define $\bar \psi = (-\Delta )^{-1} \Big(  \frac {\ve^{-2-\sigma}}  {1+|y|^{3+\beta}}  +  \frac {\ve^{-4- \sigma}}  {1+|y|^{7}}   \Big )$  given by the Newtonnian potential in $\R^2$:
	$$
	\bar \psi (y,t) = {1\over 2\pi} \int_{\R^2} \log {1\over |y-z|} \, \Big(  \frac {\ve^{-2-\sigma}}  {1+|z|^{3+\beta}}  +  \frac {\ve^{-4- \sigma}}  {1+|z|^7  } \Big ) \, dz.
	$$
	One can show that
	$$
	(1+ |y| ) |\nabla \bar \psi |+ |\bar \psi|  \,  \lesssim \Big[  \frac {\ve^{-2-\sigma}}  {1+|y|^{1+\beta}}  +  \frac {\ve^{-4- \sigma}}  {1+|y|^{5}}   \Big ]\,  \|E\|_{3+\beta}.
	$$
	Then  $\hat \psi =: M \, \|E\|_{3+\beta} \, \bar \psi$ satisfies
	$$
	\div (\mathcal P \nabla \hat \psi ) + H \leq 0,
	$$
	provided the constant $M>0$ is taken large enough. This gives the bound \eqref{es1} on $\psi_2$.
	\end{proof}

	\medskip
	\subsection{Estimates for a projected problem}
	Here we consider the ``projected version'' of Problem  \equ{innerL},
	\be\label{innerL2}
	\begin{aligned}
		\varepsilon_j^2 |\log \ve | \, p \,  \pp_t \phi &+  \nabla_y^\perp \left( \Gamma_{0j} + a_* +a \right)\cdot  \nabla_y \left( \phi  - f_0'( \Gamma_{0j} +a_*) p^2  \psi \right)\, 
		 + \, E(y,t) = \, \sum_{l=0}^3 c_{l}(t) z_{1l}(y)  \inn \R^2 \times (0,T)\\
		&\qquad\quad\phi(\cdot,0)\,  =\, 0 \inn \R^2
	\end{aligned}
	\ee
	under the same assumptions on $a$ and $a_*$ as in \equ{innerL}, and 
		\be
	\|E\|_{3+\beta}\ <\ +\infty, \quad  
	\label{EEnu}\ee
	for some $0<\beta<1$ (see \equ{Enu} for the definition of this norm). We recall that
	$$
	p(y_1,t) = \left(1 +{\ve_j y_1 \over r_j }  \chi \right) \quad {\mbox {and}} \quad \mathcal P (y_1, t) =  \left(1 +{\ve_j y_1 \over r_j }  \chi \right)^3,
	$$
	as defined in \eqref{defP}.
	Besides   $\psi$ is given by \equ{psi1} and $\phi$ satisfies the orthogonality conditions \equ{orto40}. The functions $z_{1\ell}$ are given in \eqref{defsmallz}, we recall them here
	\begin{align*}
		z_{10}(y) = U_0(y),& \quad
		z_{11}(y)  = \pp_{y_1} U_0(y),\quad
		z_{12}(y) = \pp_{y_2} U_0(y), \quad
		z_{13}(y)  = 2U_0(y) + \nn_y U_0(y)\cdot y .
	\end{align*}
We next compute 	 $c_\ell (t)$ in \eqref{innerL2}. 
In accordance with \eqref{mm1} we write
 \begin{align*}
p^2 (\psi_j^0  - r_j \alpha_j  |\log \ve_j |  + \psi_j^* )  = \Gamma_{0j}(y) - (4-\alpha_j  r_j ) \log\varepsilon_j - \log 8 + K (P_j;P_j)
				+\eta_{4\ve} b_j^{*} 
 \end{align*}
 and we also write
  \begin{align*}
p^2 (\psi_j^0  - r_j \alpha_j  |\log \ve_j |  + \psi_j^* )  = p^2 \left( \psi_{0j} + g_*\right) .
 \end{align*}
  where $\psi_{0j}$ is defined in \eqref{mm} up to constant, and $g_*$ is
  $$
  g_*= - r_j \alpha_j  |\log \ve_j |  +g_{**}
  $$
with $g_{**}$  
 $$
|\log \ve |^{1\over 2}	|\pp_t g_{**}(y,t)| + (1+ |y|) |  \,  \nn_y g_{**} (y,t)|  + |g_{**} (y,t)| \ \le  \  C \ve^2 \, (1+ |y|^2) \, \log (1+ |y|).
$$
 Taking $a_*= \eta_{4\ve} b_j^*$ (see \eqref{aastar} with $\lambda=1$)
we get
\begin{equation}\label{newe}
\Gamma_{0j} + a_*  - (4-\alpha_j  r_j ) \log\varepsilon_j - \log 8 + K (P_j;P_j)= p^2 ( \psi_{0j} + g_*) .
\end{equation}
Setting again $f_0(s) = e^s$, we have
$$
f_0 \left( \Gamma_{0j} + a_*\right) = \tilde f \left( p^2 (\psi_{0j} + g_*)\right) := 8e^{- K(P_j;P_j) } \ve_j^{4-\alpha_j r_j} e^{p^2 (\psi_{0j} + g_*)}
$$
Let
\begin{equation}\label{defBL}
	\begin{aligned}
	{\bf B} (y,t) &:=    \div \left( \mathcal P \nabla (\psi_{0j} + g_* ) \right) +  \mathcal P \tilde f \left( p^2 \,  (\psi_{0j} + g_*) \right) \\
	&=  \div \left( \mathcal P \nabla (\psi_{0j} + g_* ) \right) +  \mathcal P \, f_0 \left( \Gamma_{0j} + a_*\right), \quad {\mbox {and}} \\
	{\bf L} (\psi ) &:= \div \left( \mathcal P \nabla \psi \right) +  \mathcal P \, f_0' \left( \Gamma_{0j} + a_*\right) \, p^2 \psi.
	\end{aligned}
	\end{equation}
	The construction in Proposition \ref{Approximation} gives that
	for some $\nu>0$, for all $(y,t)\in \R^2\times [0,T]$
	$$
	|{\bf B} (y,t)|    \le\  {\ve^{2+ \nu} \over (1+|y|)^{1+\nu}}. 
	$$
	We also have, after differentiating first with respect to $y_2$, then with respect to $y_1$,
\begin{equation}	\label{uuu}
	\begin{aligned}
&	|{\bf L} (\pp_2 (\psi_{0j} +  g_* ) ) | =	\left|  \div (\mathcal P \nabla \pp_2 (\psi_{0j} +  g_* ) ) +\mathcal P  f_0'(\Gamma_{0j} + a_*) p^2 \pp_2 (\psi_{0j}+  g_* ) \right|  \le\  {\ve^{2+ \nu} \over (1+|y|)^{2+\nu}}   \\
	&	 \Delta_{5,j}  \pp_1  (\psi_{0j} + g_*)     + f'_0 (\Gamma_{0j} + a_*) \,  \,    \pp_1 (p^2 (\psi_{0j} + g_*)) -{3 \ve^2 \over p^2} \pp_1  (\psi_{0j} + g_*) = R_1 , \quad {\mbox {with}}\\
	&\left| R_1 (y,t) \right| \le\  {\ve^{2+\nu} \over (1+|y|)^{2+\nu}}
	\end{aligned}
 \end{equation}
Besides, we choose the functions $b_3 $ in the definition of ${\bf Z}_3$ in \eqref{bernieb}, and $b_1$ in ${\bf Z}_1$:	we choose
	\begin{equation}\label{defb3}
	b_3(y,t)\, := \,  G'(\Gamma_{0j})\, a_*(y,t), \quad p^2 {\bf Z}_1 = \left(\int_0^{y_1} ( 1 +{\ve_j \over  r_j} s )^3 \, ds  \right)
	\end{equation}
	where $G$ has been introduced in \eqref{bernieb}. 
We will also assume conditions \eqref{nu} on $a_*$ and $a$.

	\medskip
Under these further assumptions, we prove the following 
	
	\begin{prop}\label{prop71}
For all sufficiently small $\ve>0$ and any functions $E(y,t)$ and $\phi(y,t)$ that satisfy  $\equ{orto40}$, $\equ{innerL2}$ and  $\equ{EEnu}$, we have that the numbers $c_\ell(t)$ define linear functionals of $E$ which satisfy the estimate
		$$
		\gamma _\ell c_\ell(t) =  \int_{\R^2}  E(\cdot ,t){\bf Z_\ell}\, dy \, +\, o(1) \|E\|_{3+\beta}   , \quad {\mbox {with}} \quad o(1) \to 0, \quad \ass \ve \to 0.
		$$
		Besides $\phi$ and $\psi$ satisfy the estimates $\equ{pass1}$, $\equ{pass2}$, $\equ{pass3}$.
	\end{prop}
	
	\begin{proof} 
		We define
		$$
		\tilde E(y,t) = E(y,t) - \, \sum_{l=0}^3 c_{l}(t) z_{1l}(y) ,
		 $$
		 so that $\| \tilde E \|_{3+\beta} <\infty$.
		Lemma \ref{lin} gives that
		$$
		\| U^{-{1\over 2}} \phi \|_{L^2 (\R^2)} \lesssim \ve^{-2} |\log \ve |^{-{1\over 2}} \mathcal M,
		$$
		where
			$$
		\mathcal M\, :=\,
		\|E\|_{3+\beta} +  \sum_{\ell =0}^3 \| c_\ell \|_{L^\infty (0,T)} \,.
		$$
		Arguing as in the proof of Lemma \ref{lin1} we have that for any $\sigma >0$ small, we can find $q>2$ such that  $\phi$ solution to \eqref{innerL2} satisfies
		$$
		\| U^{{1\over q}-1} \phi \|_{L^q (\R^2)} \lesssim \ve^{-2 - \sigma} \mathcal M.
		$$
		Thanks to Corollary \ref{co71} we also have the validity of the pointwise estimate
		$$
		|\phi(y,t)| \ \le \  C\, \Big[  \frac {\ve^{-2-\sigma}}  {1+|y|^{3+\beta}}  +  \frac {\ve^{-4-\sigma}}  {1+|y|^{7}}   \Big ]\, \mathcal M .
		$$
		Recalling the relations \equ{orto40}
		$$
		\int_{\R^2} \mathcal P \phi \, {\bf Z}_\ell = 0, \quad \ell =0,1,2,3,
		$$
		where ${\bf Z}_\ell$ are defined in \eqref{bernieb}, we  multiply \eqref{innerL2} against $p^2 {\bf Z}_\ell$ and integrate on $\R^2$ to get
		$$
		\begin{aligned}
			\gamma_l c_l(t)\, &= -\ve_j^2 |\log \ve | \int_{\R^2} \phi \, \pp_t (\mathcal P {\bf Z}_\ell ) \, dy   \,- \sum_{m \not= \ell } c_m (t) \int_{\R^2} p^2 z_{1m }  {\bf Z}_\ell \, dy \,  + \int_{\R^2}   E(\cdot ,t) \, p^2 {\bf Z}_\ell\, dy  \nonumber  \\
			&+ \int_{\R^2}
			\nabla_y^\perp \left( \Gamma_{0j} + a_* + a \right)\cdot  \nabla_y \left( \phi - f_0'(\Gamma_{0j} + a_*) \, p^2 \,  \psi \right)\, p^2 {\bf Z}_\ell \, dy. \nonumber \\  & \qquad  
			\nonumber \end{aligned}$$
		where
		$
		\gamma_\ell\, =\, \int_{\R^2} p^2\, z_{1\ell} {\bf Z}_\ell \, dy, \quad l=0,1,2,3.\\
		$
		A direct computation gives 
		$$
		\gamma_\ell = (1+ O(\ve R) ) \int_{\R^2} \, z_{1\ell} {\bf Z}_\ell \, dy,
		$$
		as $\ve \to 0$, for all $\ell$.
		Besides, for $m \not= \ell$,
	$
		     \int_{\R^2} p^2 z_{1m }  {\bf Z}_\ell \, dy = O(\ve R).
	$

Take $\ell =3$ and recall that
		 $p^2 \,  {\bf Z}_3 = G*\Gamma_{0j})+ b_3$.
		  It is straighforward to see that
		\begin{align*}
		\ve_j^2 |\log \ve | \int_{\R^2} \phi \, \pp_t (\mathcal P {\bf Z}_3 ) \, dy   &= O(\ve^3 |\log \ve |^{1\over 2} )   \|\phi(\cdot, t) U^{-\frac 12} \|_{L^2(\R^2)}.
		\end{align*}
		Integrating by parts we get
		\begin{align*}
			\int_{\R^2}
			\nabla_y^\perp \left( \Gamma_{0j} + a_* + a \right)\cdot  \nabla_y \left( \phi -f_0'(\Gamma_{0j} + a_*) \, p^2 \,  \psi \right)(\cdot, t) \, p^2 {\bf Z}_3 \, dy \, .
			\\
			= \ - \int _{\R^2}  \left( \phi  - f_0'(\Gamma_{0j} + a_*) \, p^2 \,  \psi \right) \nn_y (p^2 {\bf Z}_3 ) \cdot \nabla_y^\perp( \Gamma_{0j} + a_* + a).
		\end{align*}
Recalling that $p^2 \, {\bf Z}_3 = G(\Gamma_{0j}  (y) ) +  G'(\Gamma_{0j})\, a_*(y,t) $, we write
	      $$p^2 {\bf Z}_3(y,t) =  G(\Gamma_{0j} + a_*\, )  + \tilde {\bf Z}_3(y,t) ,  $$
and we get	
		\begin{align*}
			\nn_y (p^2  {\bf Z}_3\cdot \nabla_y^\perp( \Gamma_{0j} + a_* + a) )& =
			\nn_y {\bf Z}_3\cdot \nabla_y^\perp a  +  \nn_y \tilde  {\bf Z}_3\cdot \nabla_y^\perp( \Gamma_{0j} + a_* )
		\end{align*}
		and, 	using  \eqref{defb3}, 
		\begin{align*}
			\int _{\R^2} & \left( \phi  - f_0'(\Gamma_{0j} + a_*) \, p^2 \, \psi \right)  \,   \, \nn_y (p^2 {\bf Z}_3 ) \cdot \nabla_y^\perp( \Gamma_{0j} + a_* + a)  
			=  O( \ve^2) \|\phi(\cdot , t) U^{-\frac 12} \|_{L^2(\R^2)}  .
		\end{align*}
		We conclude that
		\begin{align*}
		\gamma _3c_3( t)   &=   \int_{\R^2} E  (\cdot, t) \, {\bf Z}_3 \, dy \,  +  O( \ve^{2}) \|\phi(\cdot , t) U^{-\frac 12} \|_{L^2(\R^2)} + O(\ve R) \sum_{\ell=0}^2  |c_\ell (t)|	\\
		&= \int_{\R^2} E  (\cdot, t) \, {\bf Z}_3 \, dy \,  + o(1) \mathcal M +   O(\ve R) \sum_{\ell=0}^2  |c_\ell (t)|, \quad o(1) \to 0, \quad \ass \ve \to 0.
		\end{align*}
		Take now $\ell =0$. We have
			\begin{align*}
		\ve_j^2 |\log \ve | \int_{\R^2} \phi \, \pp_t (\mathcal P {\bf Z}_0 ) \, dy   =\ve_j^2 |\log \ve | \int_{\R^2} \phi \, \pp_t (\mathcal P  ) \,  {\bf Z}_0 \, dy = O(\ve^3 |\log \ve |^{1\over 2} )   \|\phi(\cdot, t) U^{-\frac 12} \|_{L^2(\R^2)}
		\end{align*}
		and
		\begin{align*}
			&\int_{\R^2}  
			\nabla_y^\perp\left( \Gamma_{0j} + a_* + a \right)\cdot  \nabla_y \left( \phi- f_0'(\Gamma_{0j} + a_*) \, p^2 \,  \psi \right)\,  p^2 \, {\bf Z}_0\, dy \\
			& = \ -  \int_{\R^2}  \left( \phi   - f_0'(\Gamma_{0j} + a_*)  \, p^2 \,  \psi \right) \,  \nabla_y^\perp(\Gamma_{0j} +  a_* + a)\cdot  \nabla_y p^2 
			 \\
			&= \ - 2{\ve_j \over r_j}  \int_{\R^2}  \left( \phi   - f_0'(\Gamma_{0j} + a_*)  \, p^2 \,  \psi \right) \,  (1+ {\ve_j \over r_j} y_1 \chi ) \, (\chi + y_1 \chi') \pp_2 \Gamma_{0}  +
			 O( \ve^{2}) \|\phi(\cdot , t) U^{-\frac 12} \|_{L^2(\R^2)} 
		\end{align*}
		as
		$$
		\nabla p^2 = 2 {\ve_j \over r_j} (1+ {\ve_j \over r_j} \, y_1 \, \chi ) \, (\chi + y_1 \chi' ) {\bf e}_1
		$$
		with $\chi $ given in \eqref{defP}. We recall now that
		$$
		-\phi = \Delta \psi + {3 \ve_j \over (r_j + \ve_j y_1 \chi )  } \pp_1 \psi.
		$$
		Integrating by parts we get
		\begin{align*}
		-\int_{\R^2} \phi &(1+ {\ve_j \over r_j} y_1 \chi ) \, (\chi + y_1 \chi') \pp_2 \Gamma_{0}= \int_{\R^2} \Delta (\pp_2 \Gamma_{0} ) (1+ {\ve_j \over r_j} y_1 \chi ) \, (\chi + y_1 \chi') \psi \\
		&+ \int_{\R^2} \Delta \left(  (1+ {\ve_j \over r_j} y_1 \chi ) \, (\chi + y_1 \chi')  \right) \, \pp_2 \Gamma_{0}  \psi
		+{3 \ve_j \over r_j} \int_{\R^2}    \, (\chi + y_1 \chi') \pp_2 \Gamma_{0} \pp_1 \psi\\
		&= \int_{\R^2} \Delta (\pp_2 \Gamma_{0} )  \psi + O(\ve ) \|\phi(\cdot , t) U^{-\frac 12} \|_{L^2(\R^2)} .
		\end{align*}
		Since $\Delta (\pp_2 \Gamma_{0} ) + f'_0 (\Gamma_{0} ) \pp_2 \Gamma_{0} = 0$, we get
		\begin{align*}
		    \int_{\R^2}  \left( \phi   - f_0'(\Gamma_{0j} + a_*)  \, p^2 \,  \psi \right) \,  (1+ {\ve_j \over r_j} y_1 \chi ) \, (\chi + y_1 \chi') \pp_2 \Gamma_0  =
			 O( \ve) \|\phi(\cdot , t) U^{-\frac 12} \|_{L^2(\R^2)} 
		\end{align*}
	and we conclude that
		$$
		\gamma _0c_0( t)   =   \int_{\R^2} E  (\cdot, t) \, {\bf Z}_3 \, dy \,   + o(1) \mathcal M   +   O(\ve R) \sum_{\ell\not= 0 }  |c_\ell (t)|, \quad o(1) \to 0, \quad \ass \ve \to 0. 
		$$
	Consider now $l=1$. Recalling that $\Gamma_{0j} + a_* = p^2 (\psi_{0j} + g_*)$ in \eqref{newe} and that  $\nabla \left(  p^2 \, {\bf Z}_1 \right) =p^3 \,  {\bf 1}_{B_{8R}} \, {\bf e}_1 $ we get
		\begin{align*}
			&\int_{\R^2}  
			\nabla_y^\perp( \Gamma_{0j} + a_* + a)\cdot  \nabla_y \left( \phi - f_0'(\Gamma_{0j} + a_*)  \, p^2 \,  \psi \right) \,  p^2 \, {\bf Z}_1\, dy \\
			& =- \   \int_{{B_{8R}}}  \left( \phi    - f'_0(\Gamma_{0j}  + a_*) \, p^2 \,  \psi \right) \,  \nabla_y^\perp(\Gamma_{0j} +  a_* + a)\cdot \,   p^3  \,  {\bf e}_1\, dy   \\
			& + \,  \int_{\pp B_{8R}}  \left( \phi - f_0'(\Gamma_{0j} )  \psi \right) \,  \nabla_y^\perp \Gamma_{0j} \cdot \nu(y)  \, p^2  \, {\bf Z}_1\, d\sigma  \\
			& = \   \int_{{B_{8R}}}  \left( \phi    - f_0'(\Gamma_{0j} + a_*) \, p^2 \,  \psi \right) \mathcal P \pp_2 (\Gamma_{0j} +  a_* )      \, dy    +  \,  O(\ve^{2}) \|\phi(\cdot, t) U^{-\frac 12} \|_{L^2(\R^2)}.
		\end{align*}
		Since $-{1\over \mathcal P} \div (\mathcal P \nabla \psi ) = \phi$, we have, for an arbitrary $h$
		\begin{align*}
			\int_{{B_{8R}}} \mathcal P \phi  h &= -\int_{{B_{8R}}} \div (\mathcal P \nabla \psi ) h =  -\int_{{B_{8R}}} {1\over \mathcal P} \div (\mathcal P \nabla h ) \mathcal P \psi 
			+ \int_{{\pp B_{8R}}} \mathcal P ( \nabla \psi \cdot \nu h - \psi \nabla h \cdot \nu).
		\end{align*}
		We now take $h = \pp_2 (\Gamma_{0j} +  a_* )\, $. Using the estimates in Lemma \ref{iii}, \eqref{es111} and \eqref{es1}, we obtain 
		$$
		\int_{{\pp B_{8R}}} \mathcal P ( \nabla \psi \cdot \nu h - \psi \nabla h \cdot \nu) = o(1) \, \mathcal M.
		$$
	Moreover we get
	$$
			\int_{{B_{8R}}} \left( \phi    - f_0'(\Gamma_{0j} + a_*) \, p^2 \,   \psi \right) \,  \mathcal P \, \pp_2 (\Gamma_{0j} +  
			a_* )   \, dy 
			=  -  \int_{{B_{8R}}} {\bf L} \left( \pp_2 (\Gamma_{0j} +  a_* ) \right) \, \psi \, dy   + o(1) \mathcal M   ,
		$$
		where ${\bf L}$ is defined in \eqref{defBL}.
		From \eqref{newe} we observe that
		$\Gamma_{0j} +  a_*$ and $p^2 (\psi_{0j} + g_*)$ differ by a constant term, so
		$\pp_2 (\Gamma_{0j} +  a_* ) = \pp_2 \left( p^2 (\psi_{0j} + g_*) \right)$. Hence we conclude that
			$$
			\int_{{B_{8R}}}  \left( \phi    - f_0'(\Gamma_{0j} + a_*) \, p^2 \,   \psi \right) \,  \mathcal P \, \pp_2 (\Gamma_{0j} +  
			a_* )   \, dy 
			= 
			 O(\ve^{2+ \nu } ) \|U^{-\frac 12} \phi\|_{L^2(\R^2)} +  O(\ve^{2}) \|\phi(\cdot, t) U^{-\frac 12} \|_{L^2(\R^2)}+ o(1)  \mathcal M
		$$
		thanks to \eqref{uuu}. Besides
			\begin{align*}
		\ve_j^2 |\log \ve | \int_{\R^2} \phi \, \pp_t (\mathcal P {\bf Z}_1 ) \, dy    = O(\ve^3 |\log \ve |^{1\over 2} )   \|\phi(\cdot, t) U^{-\frac 12} \|_{L^2(\R^2)},
		\end{align*}
		so that
		$$
		\begin{aligned}
			\gamma_1 c_1(t)\, &=    \int_{\R^2}   E(\cdot ,t) \, p^2 {\bf Z}_1\, dy   + o(1) \mathcal M   +   O(\ve R) \sum_{\ell\not= 1 }  |c_\ell (t)|, \quad o(1) \to 0, \quad \ass \ve \to 0.
			\nonumber \end{aligned}$$
			The last case to consider is $l=2$. Recall that $ p^2 \, {\bf Z}_2 = \,  \left( y_2 + \ve y_1 y_2 \right) \, {\bf 1}_{B_{8R}}$, so 
					\begin{align*}
			&\int_{\R^2}  
			\nabla_y^\perp( \Gamma_{0j} + a_* + a)\cdot  \nabla_y \left( \phi - f_0'( \Gamma_{0j}  + a_* )  \, p^2 \,  \psi \right) \,  p^2 \, {\bf Z}_2\, dy \\
			& =- \   \int_{{B_{8R}}}  \left( \phi    - f_0'(\Gamma_{0j} + a_*) \, p^2\,   \psi \right) \,  \nabla_y^\perp( \Gamma_{0j}  +  a_* + a)\cdot \, \nabla (p^2 {\bf Z}_2 ) \, dy   \\
			& + \,  \int_{\pp B_{8R}}  \left( \phi - f_0'(\Gamma_{0j} )  \psi \right) \,  \nabla_y^\perp \Gamma_{0j} \cdot \nu(y)  \, p^2  {\bf Z}_2\, d\sigma  \\
			& =\  - \int_{{B_{8R}}}  \left( \phi    - f'_0 (\Gamma_{0j} + a_*) \, p^2 \,   \psi \right) \cdot \nabla_y^\perp( \Gamma_{0j}  +  a_* )\cdot \, \nabla (p^2 {\bf Z}_2 ) \, dy    +  \,  O(\ve^{2}) \|\phi(\cdot, t) U^{-\frac 12} \|_{L^2(\R^2)}\\
			&=   - \int_{{B_{8R}}}  \left(- \Delta_{5,j} \psi     - f'_0 (\Gamma_{0j} + a_*) \, p^2 \,   \psi \right) \,  \nabla_y^\perp( \Gamma_{0j}  +  a_* )\cdot \, \nabla (p^2 {\bf Z}_2 ) \, dy    +  \,  O(\ve^{2}) \|\phi(\cdot, t) U^{-\frac 12} \|_{L^2(\R^2)}.
		\end{align*}
		Use the fact that $\Gamma_{0j} + a_* = p^2 (\psi_{0j} + g_*)$ in \eqref{newe} to compute
		\begin{align*}
	h_0&:=	\nabla_y^\perp( \Gamma_{0j}  +  a_* )\cdot \, \nabla (p^2 {\bf Z}_2 ) = h_1 + h_2 \\
	h_1&= \pp_1 (p^2 (\psi_{0j} + g_*)) \\
	h_2&= \ve \, \left( y_1 \pp_1 (p^2 (\psi_{0j} + g_*)) - y_2 \pp_2 (p^2 (\psi_{0j} + g_*)) \right)
		\end{align*}
		Integrating by parts and using estimates \eqref{iii}, \eqref{es111} and \eqref{es1} we get
		\begin{align*}
			\int_{{B_{8R}}} & \left(- \Delta_{5,j} \psi     - f'_0 (\Gamma_{0j} + a_*) \, p^2 \,   \psi \right) \,  \nabla_y^\perp( \Gamma_{0j}  +  a_* )\cdot \, \nabla (p^2 {\bf Z}_2 ) \, dy \\
		      &=  \int_{{B_{8R}}}  \left(- \Delta_{5,j} h_0 + {6 \ve \over p} \pp_1 h_0 - {3 \ve^2 \over h^2} h_0      - f'_0 (\Gamma_{0j} + a_*) \, p^2 \,   h_0  \right) \,  \psi  \, dy \\
		      &= \int_{{B_{8R}}}  \left(- \Delta_{5,j} h_0 + 6 \ve  \pp_1 h_0       - f'_0 (\Gamma_{0j} + a_*) \, p^2 \,   h_0  \right) \,  \psi  \, dy +  \,  O(\ve^{2}) \|\phi(\cdot, t) U^{-\frac 12} \|_{L^2(\R^2)}.
		\end{align*}
		We have
		\begin{align*}
		\int_{{B_{8R}}}  &\left(- \Delta_{5,j} h_1 + 6 \ve  \pp_1 h_1       - f'_0 (\Gamma_{0j} + a_*) \, p^2 \,   h_1  \right) \,  \psi  \, dy \\
		&= \int_{{B_{8R}}}  p^2 \left(- \Delta_{5,j}  \pp_1  (\psi_{0j} + g_*)     - f'_0 (\Gamma_{0j} + a_*) \,  \,    \pp_1 (p^2 (\psi_{0j} + g_*))  \right) \,  \psi  \, dy \\
		&+ \int_{{B_{8R}}}  p^2 \left(4 \ve \pp_1^2  (\psi_{0j} + g_*)  - 2 \ve \Delta  (\psi_{0j} + g_*)  \right) \,  \psi  \, dy +  \,  O(\ve^{2}) \|\phi(\cdot, t) U^{-\frac 12} \|_{L^2(\R^2)} \\
		&= \int_{{B_{8R}}}  p^2 \left(- \Delta_{5,j}  \pp_1  (\psi_{0j} + g_*)     - f'_0 (\Gamma_{0j} + a_*) \,  \,    \pp_1 (p^2 (\psi_{0j} + g_*))  \right) \,  \psi  \, dy \\
		&+ \int_{{B_{8R}}}   \left(2 \ve \pp_1^2  \Gamma_0  - 2 \ve
		\pp_2^2  \Gamma_0   \right) \,  \psi  \, dy +  \,  O(\ve^{2}) \|\phi(\cdot, t) U^{-\frac 12} \|_{L^2(\R^2)}
		\end{align*}
	and
		\begin{align*}
		\int_{{B_{8R}}}  &\left(- \Delta_{5,j} h_2 + 6 \ve  \pp_1 h_2       - f'_0 (\Gamma_{0j} + a_*) \, p^2 \,   h_2  \right) \,  \psi  \, dy \\
		&= \ve \int_{{B_{8R}}}   \left(- \Delta (y_1 \pp_1  \Gamma_0 )  - U \,  \, y_1 \pp_1  \Gamma_0    +
		\Delta (y_2 \pp_2  \Gamma_0 )  - U \,  \, y_2 \pp_2  \Gamma_0\right) \,  \psi  \, dy +  \,  O(\ve^{2}) \|\phi(\cdot, t) U^{-\frac 12} \|_{L^2(\R^2)} \\
		&=\ve \int_{{B_{8R}}}   \left(- 2
		\pp_1^2  \Gamma_0 + 2 \pp_2^2  \Gamma_0 \right) \,  \psi  \, dy
		 +  \,  O(\ve^{2}) \|\phi(\cdot, t) U^{-\frac 12} \|_{L^2(\R^2)}.
		\end{align*}
		In the latter computation we use that
		$$
		\Delta (y_j \pp_j  \Gamma_0 )  +U \,  \, y_j \pp_j  \Gamma_0  = y_j \left( \Delta (\pp_j  \Gamma_0 )  +U \,  \,  \pp_j  \Gamma_0 \right) + 2 \pp_j^2 \Gamma_0 =  2 \pp_j^2 \Gamma_0, \quad j=1,2.
		$$
		We conclude that
			\begin{align*}
				 \int_{{B_{8R}}} & \left(- \Delta_{5,j} \psi     - f'_0 (\Gamma_{0j} + a_*) \, p^2 \,   \psi \right) \,  \nabla_y^\perp( \Gamma_{0j}  +  a_* )\cdot \, \nabla (p^2 {\bf Z}_2 ) \, dy = 
				\int_{{B_{8R}}}  \left(- \Delta_{5,j} \psi     - f'_0 (\Gamma_{0j} + a_*) \, p^2 \,   \psi \right) \,  h_0 \, dy \\
		     &= 
		      \int_{{B_{8R}}}  p^2 \left(- \Delta_{5,j}  \pp_1  (\psi_{0j} + g_*)     - f'_0 (\Gamma_{0j} + a_*) \,  \,    \pp_1 (p^2 (\psi_{0j} + g_*))  \right) \,  \psi  \, dy \\
		& +  \,  O(\ve^{2}) \|\phi(\cdot, t) U^{-\frac 12} \|_{L^2(\R^2)}.
		\end{align*}
We now use the second estimate in \eqref{uuu} to conclude that
		$$
		\int_{{B_{8R}}}  \left(- \Delta_{5,j} \psi     - f'_0 (\Gamma_{0j} + a_*) \, p^2 \,   \psi \right) \,  \nabla_y^\perp( \Gamma_{0j}  +  a_* )\cdot \, \nabla (p^2 {\bf Z}_2 ) \, dy =  O(\ve^{2}) \|\phi(\cdot, t) U^{-\frac 12} \|_{L^2(\R^2)},
		$$
		and hence
		$$
		\int_{\R^2}  
			\nabla_y^\perp( \Gamma_{0j} + a_* + a)\cdot  \nabla_y \left( \phi - f_0'( \Gamma_{0j}  + a_* )  \, p^2 \,  \psi \right) \,  p^2 \, {\bf Z}_2\, dy  =  O(\ve^{2}) \|\phi(\cdot, t) U^{-\frac 12} \|_{L^2(\R^2)}.
			$$
		Thus we get
		$$
		\begin{aligned}
			\gamma_2 c_2(t)\, &=    \int_{\R^2}   E(\cdot ,t) \, p^2 {\bf Z}_2\, dy   + o(1) \mathcal M   +   O(\ve R) \sum_{\ell\not= 1 }  |c_\ell (t)|, \quad o(1) \to 0, \quad \ass \ve \to 0.
			\nonumber \end{aligned}$$
		Combining the above estimates, we obtain the expected result.
	\end{proof}

 We want to apply Proposition \ref{prop71} 
	to  obtain a priori estimates for the projected non-linear problem  \eqref{equinner}. 	We take 
	$$
	a_*= \la \eta_{4 \ve} \, b_j^*, \quad a= \la \left( \eta_{4\ve } b_j^{**} + b_j \right) 
	$$
	in the inner operator defined in \eqref{Ejla}, that we write in the form
	$$
	\begin{aligned}
		\varepsilon_j^2 & |\log \ve | \, p\, \pp_t \phi +  \nabla_y^\perp( \Gamma_0 + a_* +a)\cdot  \nabla_y ( \phi  - f'( \Gamma_0 +a_*) p^2 \psi )+ \,  \Theta_{j,\la} \\ &
		 +\nabla^\perp ( (1 +{\ve_j \over r_j } y_1 \chi )^2 \hat  \psi_j) \cdot \nabla \left(  \la \, \eta_{4\ve}   U^* \right)
		= \, \sum_{l=0}^3 c_{l}(t) \mathcal  Z_{1l}(y)  \inn \R^2 \times (0,T)\\
		&\qquad\quad\phi(\cdot,0)\,  =\, 0 \inn \R^2
	\end{aligned}
	$$
	where
	\begin{align}
 \label{rrrr}
	    \Theta_{j,\la}&=	 \nabla^\perp\left[\ve_j \left(  |\log \ve|  \dot {\bf a }_j +   D_x \nn_x\vp_j (  {\bf P}_j ; {\bf P} ) [{\bf a} ] + \la \eta_{4\ve} (1+{\ve_j y_1 \over r_1} )^2 r_j \psi^{out}  \right)\cdot y   \right] \nabla U \nonumber  \\
		& +\la 	\tilde {\mathcal E}_j  (\beta_j, \psi^{out}, {\bf a} ) + |\log \ve | \ve_j^2 (1 +{\ve_j \over r_j } y_1 \chi ) \sum_{\ell = 0,3} \pp_t (\beta_{j\ell}  \mathcal Z_{1\ell} )\quad \inn \R^2 \times [0,T]
	\end{align}
	Besides we have
	 a priori bounds of the form
	\be \label{roger1}
	\gamma _\ell c_{l}(t) =  \int_{\R^2}   \Theta_{j,\la} (\cdot ,t){\bf Z_\ell}\, dy \, +\, o(1)  \| \Theta_{j,\la} \|_{3+\beta}   .
	\ee
	for some $o(1) \to 0$ as $\ve \to 0$ readily follow, of course provided that $a(y,t)$ satisfies the required smallness assumptions. This is guaranteed by the choice of spaces for the parameter functions $\beta_j, \psi^{out}, {\bf a} $. This is what we shall specify in the next section.

	\medskip

	\section{Fixed point formulation and conclusion of the proof}\label{sec10}
	
	In this section we set up the system \equ{equinner}, \equ{equc}, \equ{equout1}, \equ{equout2} as a fixed point problem in the form \equ{equp} in an appropriate Banach space for the parameter functions
	$$
	\vec {p} = (\hat \phi, \beta,  \phi^{out},\psi^{out}, {\bf a}) .
	$$

	\medskip
	\subsection{The space for the parameter functions.}
	We begin by defining an appropriate norm for the functions
	$\hat \phi_j(y,t)$, $\phi^{out} (x,t)$, $\psi^{out} (x,t)$, $\beta (t) $ and ${\bf a} (t)$.

	\medskip
	Let us fix a  small number  $0< \beta<1$.
	For and arbitrary functions $\phi(y,t)$  we define the inner norm
	\begin{align*}
		\| \phi\|_{i}\, :=  &\,  \sup_{t\in [0,T]} \| \mathcal{P}\, \phi(\cdot,t) U^{-\frac 12} \|_{L^2(\R^2)}\\
		&+   \sup_{(y,t)\in \R^2\times [0,T]} |\, (1+|y|)^{3+\beta}\min \{ 1, \ve^{2+ \frac \beta 4 } (1+|y|)^{4-\beta} \} \phi(y,t)   \, |
	\end{align*}
	For the outer functions $\psi^{out}$, $\phi^{out}$ we consider the following norms
	for functions $\phi(x,t)$, $\psi(x,t)$ defined in $\Sigma\times [0,T]$.
	\begin{align*}
		\| \phi\|_{o1}\, := &\,   \| (1+ |x|)^{2+ \nu} \phi \|_ {L^\infty (\Sigma\times [0,T])} ,  \\
		\|\psi \|_{o2} \, :=&\,  \|  (1+ |x|)^{\nu} \psi \|_ {L^\infty (\Sigma\times [0,T])} + \| (1+ |x|)^{1+ \nu} \nn_x \psi \|_ {L^\infty (\Sigma\times [0,T])}
		\, .\end{align*}
	We consider the space $X$ of all continuous functions
	$\vec { p}  = (\hat \phi, \beta , \phi^{out},\psi^{out} , {\bf a} ) $
	such that
	$$\nn_y\hat \psi(y,t),\quad \nn_x \psi^{out}(x,t),\quad {d \over dt} \beta (t), \quad {d \over dt}  {\bf a}(t)  $$ exist and are continuous and such that
	$$
	\|\vec p \,\| _X \,:=\,     \| \phi^{out}\|_{o1}+  \|\psi^{out}\|_{o2} + \sum_{j=1}^k \big( \| \hat \phi_j\|_{i} + \|\beta_j\|_{C^1[0,T]} + \|{\bf a} \|_{C^1[0,T]}\big)    \, <\, +\infty .
	$$
	We define the set $\OO$ as a ``deformed ball'' centered at $\vec p = \vec 0$. We fix $\nu =1$ in the definition of $\|\phi^{out}\|_{o1}-\|\psi^{out}\|_{o2}$ and let $\OO $ be the set of all functions
	$\vec { p}  = (\bphi, \alpha, \ttt\xi, \phi^{out},\psi^{out})\in X $ such that
	\be\label{OO} \left\{ \begin{aligned}
		& \sum_{j=1}^k  \| \hat \phi_j\|_{i}
		\ < \ \ve^{3- 3\sigma_*} ,\\
		&\sum_{j=1}^k \|\beta_j\|_{C^1[0,T]}  \, < \, \ve^{3- 3\sigma_*} ,\quad  \sum_{j=1}^k \|{\bf a} \|_{C^1[0,T]}  \, < \, \ve^{4-3\sigma_*}  ,\\
		&\|\phi^{out}\|_{o1}  \ < \ \ve^{4- 3\sigma_*} ,\quad   \|\psi^{out}\|_{o2}  \ < \ \ve^{4- 3\sigma_* }.
	\end{aligned}
	\right.
	\ee
	The number $\sigma_*$ is introduced in Proposition \ref{Approximation}.

	\subsection{Fixed point formulation}
	Let us express System  \equ{equinner}, \equ{equc}, \equ{equout1}, \equ{equout2} in the fixed point form \equ{equp}
	for a suitable operator $ {\mathcal F} (\cdot , \la )$, in a region of
	the form $\equ{OO}$.

	\medskip
	We start with \equ{equinner}.
	For a given function $a(y,t)$   let us consider the transport operator
	$$
	\TT_j (a)[\phi] :=  |\log \ve |  \, \ve_j^2 \,  p \, \pp_t \phi + \nn_y^\perp ( \Gamma_{0j} + a) \cdot \nn_y \phi ,
	$$
	and for a bounded function $E(y,t)$ the linear equation
	$$
		\TT_j (a)[\phi] + E = 0 \inn \R^2\times [0,T]\quad
		\phi(\cdot,0) = 0 \inn \R^2  .
	$$
The result of Lemma \ref{transport1-new} is still valid for a transport equation of the form 
		\begin{align*}
			|\log \ve | \ve_j^2 &p\, \pp_t  \phi  
			+  \nabla_y^\perp( \Gamma_{0j} + a)\cdot  \nabla_y  \phi 
			+ \, E(y,t)\, = \, 0  \inn \R^2 \times (0,T),\\
			&\qquad\quad\phi(\cdot,0)\,  =\, 0 \inn \R^2
		\end{align*}
		Here the function $\chi (y_1) $ is defined in \eqref{ccuts} and the function
		 $a(y,t)$ satisfy
		$$
	 a (y,t) \, =0 \quad \hbox{for } |y|\ge 8R, \quad
		\Delta_y (a ) \in L^\infty (\R^2\times (0,T))
		$$
		and
		for some numbers $C>0$, $\nu >0$,
		\begin{align*}
			(1+ |y|)^{-1} |\log \ve |^{1\over 2}	|\pp_t a (y,t)| + |  \,  \nn_y a (y,t)|  \ \le & \  C \ve^2 (1+ |y|)  \, \log (1+ |y|) \nonumber .\end{align*}

\medskip{}{}
	We call
	$
	\phi =  \mathcal \TT^{-1}_j (a) [ E] $
	the unique solution of this problem, through the representation formula \equ{phi}, which defines a linear operator of $E$.
	Let us write the operator $E_{j,\la}$ in \equ{Ejla} in the form
	\begin{align*}
		E_{j,\la}  (  \hat \phi_j  , \beta_j,  \psi^{out};{\bf a} )  & =  \TT_j (\la \eta_{4\ve} b_j^* ({\bf a} ) + \la \eta_{4\ve} b_j^{**} ({\bf a} ) + \,\la {b}_{j}( \hat \psi_j , \beta_j, \psi^{out},{\bf a}  \,) )[\hat \phi_j]  \\
		& +\nabla^\perp ( p^2\hat  \psi_j) \cdot \nabla \left(  \la \, \eta_{4\ve}  U^* \right)+
		 \Theta_{j \la}  ( \hat \psi_j , \beta_j, \psi^{out};{\bf a} ).
	\end{align*}
	and we reformulate equations \equ{equinner} as
	\be\label{fixed1} 
	\begin{aligned}
		\hat \phi_j &=
		\mathcal F_{\la}^{in} (\hat \phi_j , \beta_j , \psi^{out},{\bf a}  ),\\
		&{\mbox {where}} \\
		\mathcal F_{\la}^{in} (\hat \phi_j , \beta_j , \psi^{out},{\bf a}  )&= 
		\TT_j^{-1}  (\la \eta_{4\ve} b_j^* ({\bf a} ) + \la \eta_{4\ve} b_j^{**} ({\bf a} ) + \,\la {b}_{j}( \hat \psi_j , \beta_j, \psi^{out},{\bf a}  \,) ) \circ \\
		&\left[  \nabla^\perp ( p^2 \hat  \psi_j) \cdot \nabla \left(  \la \, \eta_{4\ve}   U^* \right)+
		 \Theta_{j \la}  ( \hat \psi_j , \beta_j, \psi^{out};{\bf a} ) - \sum_{l=0}^3 c_{lj}  z_{1l} \right]
	\end{aligned}
	\ee

	
	We reformulate the outer equations \equ{equout1}-\equ{equout2} in  a similar way.
	For a given function $e(x,t)$ with $r\Delta_x e , \, \nabla e \in L^\infty (\Sigma\times [0,T])$  
	let us consider the transport operator
	$$
	\TT_o (e)[\phi] :=  |\log \ve | \, r \, \pp_t \phi +   \nn_y^\perp ( r^2 (\Psi_* -r_0^{-1} |\log \ve |  + e) )  \cdot \nn_x \phi ,
	$$
	and for a bounded function $E(x,t)$, the linear equation
	$$
		\TT_o (e)[\phi] + E = 0 \inn \Sigma \times [0,T], \quad 
		\phi(\cdot,0) = 0 \inn \Sigma  .
	$$
	We call
	$
	\phi =  \TT_o^{-1} (e) [ E] $
	the unique solution of this problem, through the representation formula \equ{phi3}.
	We write \equ{equout1}-\equ{equout2} in the form
	\be\label{fixed2}  
	\begin{aligned}
		\phi^{out} = & \mathcal F_{1\la}^{out}  (\hat \phi ,  \beta, \psi^{out},{\bf a}  )\\
		\psi^{out} = & \mathcal F_{2\la}^{out}   (\hat \phi ,  \beta, \psi^{out},{\bf a}  )
	\end{aligned}
	\ee
	where
	$$
	\mathcal F_{1\la}^{out}   (\hat \phi ,  \beta, \psi^{out},{\bf a}  ) :=
	\TT_o^{-1}(\la \sum_{j=1}^k { \bar \eta_{j2}\over r_j} \psi_j ({x-P_j \over \ve_j }) + \la \psi^{out}  ) \left[ \la \ttt \EE_{1}^{out}( \hat \phi ,  \beta,     \psi^{out};{\bf a} ) \right]
	$$
	with $ \ttt \EE_{1}^{out}$ given by \equ{EE1}
	and
	\begin{align*}
		\mathcal F_{2\la}^{out}   (\hat \phi ,  \beta, \psi^{out},{\bf a}  ):= {\mathcal T}^{-1} &\Biggl[ \la\,  \phi^{out} 	+  \la \sum_{j=1}^k (\bar \eta_{j1} -\bar \eta_{j2} ) {\phi_j \over r_j \ve_j^2}  
		+\la \sum_{j=1}^k ( {\psi_j \over r_j} \Delta_5 \bar \eta_{j2}  + 2 \nabla_x \bar \eta_{j2} \nabla_x {\psi_j \over r_j}  )\Biggl]
	\end{align*}
	where $\psi = {\mathcal T}^{-1} h$ is the unique solution of Problem \eqref{ext0}
	$$
	\Delta_5  \psi + h =0, \quad  {\mbox {in }} \Sigma , \quad {\partial  \psi \over \partial r}  = 0 \quad {\mbox {on}} \, \,  \pp \Sigma , \quad \psi (x ) \to 0, \, \,  \ass |x|\to \infty,
	$$
	for a smooth function $h$ satisfying \eqref{decayh}.
	Recall that $\Delta_5 = {\pp^2 \over \pp r^2} + {3\over r} {\pp \over \pp r} + {\pp^2 \over \pp z^2}$, for $x= (r,z) \in \Sigma$.

	Using \eqref{roger1}, we write equations
	\equ{equc} $c_{\ell j } = 0$ as 
	\be\label{pico71} 
	\begin{aligned}
		|\log \ve | \dot {\bf a}_j  & =  -  D_x \nn_x\vp_j (  {\bf P}_j ; {\bf P} ) [{\bf a} ] + \la \ve_j^{-1}\mathcal G_j( \hat \phi , \beta, \psi^{out};{\bf a} ),\\
		|\log \ve | \dot  {\beta}_j    &= \la \ve_j^{-2}\mathcal H_j
		( \hat \phi , \beta, \psi^{out};{\bf a} ),  \quad j=1,\ldots, k, \\
		{\bf a} (0) &=0,  \quad  \beta(0)=0 ,
	\end{aligned}
	\ee
	where 
	$$\beta_j = (\alpha_{0j}, \alpha_{3j}),\quad  {\bf a}_j= (a_{j1},a_{j2}),
	\quad
	\beta = (\beta_1,\ldots,\beta_k), \quad {\bf a}  = ({\bf a}_1,\ldots,{\bf a}_k)$$
	and, for
	$$ \mathcal G_j = (\mathcal G_{1j}, \mathcal G_{2j}) ,\quad \mathcal H_j = (\mathcal H_{0j}, \mathcal H_{3j}) ,
	$$
	\begin{align*}
		\mathcal G_{\ell j}    ( \hat \phi , \beta , \psi^{out},{\bf a} )(t)\ := &\  \gamma_\ell \la \int_{\R^2} \big[ \,  \nabla^\perp [\eta_{4\ve} p^2 r_j \psi^{out}  \cdot y ]  \nabla U   +	\tilde {\mathcal E}_j  (\beta_j, \psi^{out}, {\bf a} )  \, \big ] {\bf Z}_\ell\, dy\\
		&{\mbox {for }} \ell =1,2, \\
		\mathcal H_{m j} ( \hat \phi , \beta , \psi^{out},{\bf a} )(t)\ := & \gamma_m \la \int_{\R^2} \big[ \,  \nabla^\perp [\eta_{4\ve} p^2 r_j \psi^{out}  \cdot y ]  \nabla U   +	\tilde {\mathcal E}_j  (\beta_j, \psi^{out}, {\bf a} )  \, \big ] {\bf Z}_m\, dy\\
		&+  \gamma_m \int_{\R^2}  p \sum_{i = 0,3}  \beta_{ji }  \pp_t \mathcal Z_{1i} ) {\bf Z}_m \\
		&{\mbox {for }} m=0,3.
	\end{align*}
	Recall that $p= 1+{\ve_j \over r_j } \, y_1 \, \chi$.
	Equations \equ{pico71} can be written in fixed point form as
	\be\label{sisalfa} 
	\begin{aligned}
		{\bf a} (t)  &=   \mathcal F_{1\la} ( \hat \phi , \beta,  \psi^{out},{\bf a}  )  = |\log \ve |^{-1} \int_0^t \big( B (s) [{\bf a} ]  + \la \ve_j^{-1}\mathcal G_j( \hat \phi , \beta, \psi^{out};{\bf a} ) \big)\, ds \\
		\beta(t)& = \mathcal F_{0\la} ( \hat \phi , \beta,  \psi^{out},{\bf a}) = |\log \ve |^{-1}  \int_0^t  \la \ve_j^{-2}\mathcal H_j
		( \hat \phi , \beta, \psi^{out};{\bf a} ) \, ds
	\end{aligned}
	\ee
	where
	$
	\big( B (t) [{\bf a} ] \big)_j =  - D_x \nn_x\vp_j ( {\bf P}_j ; {\bf P} ) [{\bf a} ].$
	
	\medskip
	System  \equ{equinner}, \equ{equc}, \equ{equout1}, \equ{equout2} can be written as a fixed point problem  \equ{fixed1}-\equ{fixed2}-\equ{sisalfa}, which we write as
	\begin{align*}
		\hat \phi = &
		\mathcal F_{\la}^{in} (\hat \phi , \beta_j , \psi^{out},{\bf a}  ), \quad 
		\phi^{out} =  \tilde {\mathcal F}_{1\la}^{out}  (\hat \phi ,  \beta, \psi^{out},{\bf a}  ), \quad 
		\psi^{out} =  \tilde {\mathcal F}_{2\la}^{out}   (\hat \phi ,  \beta, \psi^{out},{\bf a}  )\\
		{\bf a} (t)  =  & \tilde {\mathcal F}_{1\la} ( \hat \phi , \beta,  \psi^{out},{\bf a}  ), \quad 
		\beta(t) = \tilde {\mathcal F}_{0\la} ( \hat \phi , \beta,  \psi^{out},{\bf a})
	\end{align*}
	where
	\begin{align*}
		\tilde {\mathcal F}_{1\la}^{out}  (\hat \phi ,  \beta, \psi^{out},{\bf a}  ) &= F_{1\la}^{out}  (\mathcal F_{\la}^{in} (\hat \phi , \beta_j , \psi^{out},{\bf a}  ) ,  \beta, \psi^{out},{\bf a}  )\\
		\tilde {\mathcal F}_{2\la}^{out}   (\hat \phi ,  \beta, \psi^{out},{\bf a}  )&	=  {\mathcal F}_{2\la}^{out}   (\mathcal F_{\la}^{in} (\hat \phi , \beta_j , \psi^{out},{\bf a}  ),  \beta,\tilde {\mathcal F}_{1\la}^{out}  (\hat \phi ,  \beta, \psi^{out},{\bf a}  ) ,{\bf a}  )\\
		\tilde {\mathcal F}_{1\la} ( \hat \phi , \beta,  \psi^{out},{\bf a}  )& =  {\mathcal F}_{1\la} (\mathcal F_{\la}^{in} (\hat \phi , \beta_j , \psi^{out},{\bf a}  ),  \beta,\tilde {\mathcal F}_{1\la}^{out}  (\hat \phi ,  \beta, \psi^{out},{\bf a}  ) ,{\bf a}  )\\
		\tilde {\mathcal F}_{0\la} ( \hat \phi , \beta,  \psi^{out},{\bf a}) &=  {\mathcal F}_{0\la} (\mathcal F_{\la}^{in} (\hat \phi , \beta_j , \psi^{out},{\bf a}  ),  \beta,\tilde {\mathcal F}_{1\la}^{out}  (\hat \phi ,  \beta, \psi^{out},{\bf a}  ) ,{\bf a}  ).
	\end{align*}

	In a more compact form, the above system can be expressed as 
	\be\label{ecc}
	\vec p \ =\ \tilde {\mathcal F_\la} (\vec p) , \quad \vec p\in\bar\OO
	\ee
	where $\OO$ is defined in \eqref{OO} and 
	\be \label{opera}
	\left\{
	\begin{aligned}
		\tilde {\mathcal F_\la} (\vec p) \, := &\,   ( {\mathcal F}^{in}_\la(\vec p), \tilde {\mathcal F}_{0\la}(\vec p), \tilde {\mathcal F}_{1\la}(\vec p) , \tilde {\mathcal F}^{out}_{1\la}(\vec p), \tilde {\mathcal F}^{out}_{2\la}(\vec p)
		), \\
		\vec p\, = &\, (\hat \phi,\beta ,  \phi^{out}, \psi^{out}, {\bf a} ) .
	\end{aligned}  \right.
	\ee

	\begin{lemma}
		The operator $\tilde {\mathcal F} : \OO\times [0,1] \to X $ given by
		$\tilde {\mathcal F} ( \cdot , \la )=   \tilde {\mathcal F}_\la$  in $\equ{opera}$
		is compact.
	\end{lemma}
	
	\begin{proof}
		We check that each of the five operators defining $\ttt {\mathcal F }_\la(\vec p)$ is compact in $\OO $ (uniformly in $\la$).
		This operator $\mathcal F^{in}_\la(\vec p)$ is defined through 
		$ g = \mathcal T^{-1}_j (b)[h] $ which has the property in Lemma \ref{modc2}. This gives that a uniform bound in $\Delta_y b $ and a control of the modulus of continuity in $y$ of $h(y,t)$ uniformly in $t$ yields a uniform control of the modulus of  continuity of $g$  in both variables $(y,t)$.
		We see in \equ{fixed1} that  for a certain $C_\ve>0$ we have
		$$ \|\Delta_y (b^*_j + b_j) \|_ {L^\infty(\R^2\times [0,T])} \le  C_\ve  \foral \vec p\in \bar\OO. $$
		and it vanishes outside a compact set. Moreover, we have a uniform H\"older
		control in space variables on the corresponding arguments $h$ for $\vec p\in \bar\OO$ as it follows from the H\"older estimates for the gradients of $\hat \psi_j$ and $\psi^{out}$  inherited from the uniform bounds holding for
		$\ttt\psi$ and $\phi^{out}$ in the definition of $\OO$ (see the argument in the proof of \equ{pass2}). Also, the numbers
		$c_{lj}(t)$ have a uniform bound, thanks to \equ{roger1}.
		Uniform Lipschitz bounds hold for the remaining errors, as it follows in particular from   the control of the terms involving
		$\nn_y \phi_j^*$.  Lemma \ref{modc2} then implies that
		$\ttt{\mathcal F}^{in}_\la(\bar\OO)$ is a set of continuous functions
		$g:\R^2\times [0,T]\to \R^N$ whose restrictions to any compact set defines a uniformly bounded, equicontinuous set. Hence, any sequence $\phi_n \in \ttt{\mathcal F}^{in}_\la(\OO)$
		has a subsequence $\phi_{n'}$ which is uniformly convergent on each compact set.
		Finally, we observe that $\|\phi_n\|_{4}\le C_\ve$ since the argument of
		the transport operator has this property. This implies that $\phi_{n'}$ is actually convergent in the space of continuous functions with finite
		$\| \cdot \|_{3+\beta} $-norm, since $0<\beta <1$. Hence
		$\ttt{\mathcal F}^{in}_\la(\OO)$ is precompact in this space.
		The compactness of the operator $\tilde {\mathcal F}_{1\la}^{out}$ into $C(\bar\Omega\times [0,T])$ follows directly from Arzela-Ascoli's theorem, again from the corresponding control for the transport equation and the
		uniform controls on space and time variables valid for the operator $\tilde {\mathcal F}_{1\la}^{in}$. From here the compactness for $\tilde {\mathcal F}_{2\la}^{out}$ follows in similar manner. Finally, the compactness of the operators
		$
		\tilde {\mathcal F}_{0\la}(\vec p), \tilde {\mathcal F}_{1\la}(\vec p)
		$
		into $C^1([0,T])$ follows again from the equicontinuity in $t$ inherited for the
		different terms involved in their definition. The proof is concluded.
	\end{proof}

	\subsection{Conclusion of the proof of Theorem \ref{teo}}
	The original problem has been so far reduced  to finding a solution of the fixed point problem
	\equ{ecc} for $\la= 1$. To do this, we will prove that for all $\la\in [0,1]$ this equation  has no solution $\vec p\in \pp\OO$, at least whenever $\ve$ is chosen sufficiently small.  Let us assume that $\vec p\in \bar\OO$ satisfies
	\equ{ecc} for some $\la$. We claim that actually $\vec p\in \OO$.
 Take $0<\beta <{\sigma_* \over 2} $, with $\sigma_*$ given in Proposition \ref{Approximation}. From \equ{rrrr} we get
	$$
	\| \Theta_j( \beta, \psi^{out}, {\bf a} ) \|_{3+\beta }  \le \ve^{5- \frac 32 \sigma_*}
	$$
	Corollary \ref{co71} and Lemma \ref{lin}  then yield, by definition of the inner norm,
	$$
	\|\hat \phi\|_i \le \ve^{3-2\sigma_* } \ll \ve^{3-3\sigma_*},
	$$
	the latter number being that involved in the definition of $\OO$ in \equ{OO}. We recall here that $\sigma$ appearing in Corollary \ref{co71} and Lemma \ref{lin} can be made as small as wished.
	Let us consider the outer equations.
	Examining expression \equ{EE1} that determines the size of $\phi^{out}$,
	we see that its magnitude does not exceed the order $O(\ve^{4-{\sigma_* \over 2} })$. Here we have used the remote size of $\hat \phi$ implicit in the norm
	$\|\hat \phi_j\|_i$. Indeed using the size induced in $\hat \psi$, we find
	that
	$$
	\| \phi^{out}\|_{o1}  +   \| \psi^{out}\|_{o2} \le \ve^{4-2\sigma_*} \ll \ve^{4-3\sigma_*}.
	$$
	Finally from the size of $\Theta_j$ we readily see that
	$$\|{\bf a}_j\|_{C^1[0,T]} +  \ve\|\beta_j\|_{C^1[0,T]} \le \ve^{4-2\sigma_* } \ll \ve^{4-3\sigma_*}.
	$$
	As a conclusion, we get that $\vec p\in \OO$ and the claim has been proven.
	
	\medskip
	Standard degree theory applies then to yield that the degree
	$\deg ( I -\ttt{ \mathcal F}(\cdot, \la), \OO, 0)$ is well-defined and it is constant in $\la\in [0,1]$. Since $\ttt {\mathcal F}(\cdot, 0)$ is a linear compact operator, this constant is actually non-zero.
	Existence of a solution in $\OO$ for $\la=1$ then follows. The proof is concluded. \qed

	\bigskip\noindent
	{\bf Acknowledgements:}
	J.~D\'avila has been supported  by  a Royal Society  Wolfson Fellowship, UK and  Fondecyt grant 1170224, Chile.
	M.~del Pino has been supported by a Royal Society Research Professorship, UK.
	M. Musso has been supported by EPSRC research Grant EP/T008458/1. The  research  of J.~Wei is partially supported by NSERC of Canada.

\end{document}